\documentclass[]{article}
\usepackage{graphicx}
\usepackage{amsmath}
\usepackage{amssymb}
\usepackage{amsthm}
\usepackage{pxfonts}
\usepackage{enumerate}
\usepackage{color}
\usepackage{mathdots}
\usepackage{sectsty}
\usepackage[hidelinks]{hyperref}
\usepackage{tikz}
\usepackage{caption}

\sectionfont{\scshape\centering\fontsize{11}{14}\selectfont}
\subsectionfont{\scshape\fontsize{11}{14}\selectfont}
\usepackage{fancyhdr}

\newcommand\shorttitle{Random Surface Growth and Karlin-McGregor Polynomials}
\newcommand\authors{T. Assiotis}

\newtheorem{thm}{Theorem}[section]
\newtheorem{cor}[thm]{Corollary}
\newtheorem{lem}[thm]{Lemma}
\newtheorem{defn}[thm]{Definition}
\newtheorem{rmk}[thm]{Remark}
\newtheorem{prop}[thm]{Proposition}

\date{}

\title{
\sc Random Surface Growth and Karlin-McGregor Polynomials}
\author{ \sc Theodoros Assiotis }

\begin{document}

\maketitle

\begin{abstract}
We consider consistent dynamics for non-intersecting birth and death chains, originating from dualities of stochastic coalescing flows and one dimensional orthogonal polynomials. As corollaries, we obtain unified and simple probabilistic proofs of certain key intertwining relations between multivariate Markov chains on the levels of some branching graphs. Special cases include the dynamics on the Gelfand-Tsetlin graph considered in the seminal work of Borodin and Olshanski in \cite{BorodinOlshanski} and the ones on the BC-type graph recently studied by Cuenca in \cite{Cuenca}. Moreover, we introduce a general inhomogeneous random growth process with a wall that includes as special cases the ones considered by Borodin and Kuan \cite{BorodinKuan} and Cerenzia \cite{Cerenzia}, that are related to the representation theory of classical groups and also the Jacobi growth process more recently studied by Cerenzia and Kuan \cite{CerenziaKuan}. Its most important feature is that, this process retains the determinantal structure of the ones studied previously and for the fully packed initial condition we are able to calculate its correlation kernel explicitly in terms of a contour integral involving orthogonal polynomials. At a certain scaling limit, at a finite distance from the wall, one obtains for a single level discrete determinantal ensembles associated to continuous orthogonal polynomials, that were recently introduced by Borodin and Olshanski in \cite{BorodinOlshanskiAsep}, and that depend on the inhomogeneities.
\end{abstract}
\tableofcontents

\section{Introduction}
\subsection{Determinantal structures in inhomogeneous random growth models}
This work revolves around two sets of closely related problems and ideas. One of them is, the construction of consistent dynamics on the levels of certain branching graphs and the other is, the exact computation of correlations in random stepped surface growth processes. 

These probabilistic models can be viewed as dynamics on (discrete) interlacing arrays, namely multilevel configurations of particles that satisfy some constraints (that we make precise below), see Figure \ref{figuresteppedsurface} below for an illustration. Such (2+1)-dimensional dynamics (2 space and 1 time dimensions) have been extensively studied in the past decade, see \cite{BorodinFerrari},\cite{BorodinKuan},\cite{CerenziaKuan},\cite{MacdonaldProcesses},\cite{BorodinPetrov},\cite{StochasticVertex}. In addition to being interesting in its own right a further motivation for this study is the following phenomenon: the exact solvability of a wide class of (1+1)-dimensional models such as the Totally Asymmetric Simple Exclusion Process (TASEP) is a by-product of the fact that they appear as projections of these higher dimensional models, see \cite{BorodinFerrari}. 

In many of these papers (see \cite{BorodinFerrari}, \cite{BorodinKuan}, \cite{CerenziaKuan}) the models considered give rise to determinantal point processes: for a point process on a discrete space $\mathfrak{X}$ we say that it is determinantal if for all $n\ge 1$ its correlation functions $\rho_n$ are given as determinants of a two variable function $\mathsf{K}:\mathfrak{X}\times \mathfrak{X} \to \mathbb{C}$:
\begin{align*}
\rho_n\left(\mathfrak{x}_1,\cdots,\mathfrak{x}_n\right)=\det \left[\mathsf{K}\left(\mathfrak{x}_i,\mathfrak{x}_j\right)\right]^n_{i,j=1}.
\end{align*}
Thus, all probabilistic information about the model is encoded in the function $\mathsf{K}$ and questions about its limit behaviour reduce to asymptotic analysis of $\mathsf{K}$.

In all aforementioned papers, the jump rates of particles on each level had a rather special algebraic dependence on their positions. The main novelty of our contribution is that we allow (essentially) arbitrary jump rates for individual particles depending on the position in the horizontal direction, while retaining the determinantal point process structure.

For many of the works in Integrable Probability the exact solvability of the models can be traced down to a rich duality structure, see \cite{BorodinCorwinDuality}, \cite{KuanDuality1}, \cite{KuanDuality2}. In this paper a key role is played by the famous Siegmund duality for birth and death chains, going back to Karlin and McGregor, see \cite{KarlinMcGregorClassification}, \cite{KarlinMcGregorDifferential}: Consider a birth and death chain in $I=\mathbb{N}$ (reflecting at 0) or a bilateral chain in $I=\mathbb{Z}$  with generator $\mathcal{D}$ given by the birth rates $\lambda(x)$ and death rates $\mu(x)$ (the positive functions $\lambda(\cdot), \mu(\cdot)$ can be essentially arbitrary modulo technicalities). Then we define its Siegmund dual (which is absorbed at $-1$ in the birth and death chain case) with generator $\hat{\mathcal{D}}$ and birth rates given by $\hat{\lambda}(x)=\mu(x+1)$ and death rates by $\hat{\mu}(x)=\lambda(x)$. The key property these dual chains satisfy is the following: if we consider two independent copies $X(t)$ and $\hat{X}(t)$ with generators $\mathcal{D}$ and $\hat{\mathcal{D}}$ respectively then for $x,y \in I$ and $t\ge 0$ we have:
\begin{align*}
\mathbb{P}_x\left(X(t) \le y\right)=\mathbb{P}_y\left(\hat{X}(t)\ge x\right).
\end{align*}

Then, from considering a coalescing flow of birth and death chains we obtain an explicit formula in terms of block determinants, describing a joint evolution $(X,Y)$ of interacting $\mathcal{D}$ and $\hat{\mathcal{D}}$-chains. To explain this further we need some notation. Let us denote the $n$-dimensional (discrete) Weyl chamber, where all the $x_i$ are either in $\mathbb{N}$ or $\mathbb{Z}$, by
\begin{align*}
W^n=\{(x_1,\cdots,x_n):x_1<\cdots<x_n\}.
\end{align*}
 Then, for $x\in W^{n+1}$ and $y\in W^n$ we will say that $x$ and $y$ \textit{interlace} and write $y\prec x$ if (note the position of $<$ and $\le$):
 \begin{align*}
x_1\le y_1<x_2\le \cdots \le y_n<x_{n+1}
 \end{align*}
 and denote by $W^{n,n+1}$ the space of such pairs $(x,y)$. 
 
 The joint evolution $(X,Y)$ takes values in $W^{n,n+1}$ and can be described as follows: $Y$ is autonomous and evolves as $n$ $\hat{\mathcal{D}}$-chains conditioned not to intersect by a Doob's $h$-transform and $X$ as $n+1$ $\mathcal{D}$-chains \textit{pushed} and \textit{blocked} by the $Y$-particles, when the process is on the boundary of $W^{n,n+1}$ (the interactions are local), in order for the interlacing to remain true. In particular, the $X$-particles never intersect. As a by-product of the special structure of these formulae, we obtain as part of our first set of results, that under special initial conditions of $(X,Y)$ the non-autonomous $X$-component is in fact distributed as a Markov chain. Its evolution being that of $n+1$ $\mathcal{D}$-chains conditioned not to intersect by an explicit Doob's transformation, given in terms of the original transform of the $Y$-component. 
 
 Analogous formulae, having essentially the same structure, are also obtained for $(X,Y)$ taking values in $W^{n,n}$ given by interlacing sequences of the form $y_1 \le x_1 <y_2 \le \cdots \le x_n$ (again we write $y\prec x$). It is then possible to concatenate such two-level couplings in a consistent fashion, to build a multilevel process with interlacing components such that, if started according to certain initial conditions each level evolves as a Markov chain in its own right with an explicit distribution.
 
 Then, we go on to consider a particular choice of such consistent multilevel dynamics, that we call the alternating construction. Its distribution at time $t\ge 0$ gives rise to a determinantal point process. To compute its correlation kernel explicitly we make heavy use of the spectral theory for birth and death chains and their associated orthogonal polynomials, developed by Karlin and McGregor in \cite{KarlinMcGregorClassification}, \cite{KarlinMcGregorDifferential}.

We now proceed to explain our main results in detail and how they relate to other works in the field of Integrable Probability.

\subsection{Intertwinings and consistent multilevel dynamics}
We make precise our first set of results. To begin, we need some definitions. We write $p_t(x,y)$ for the transition density of the $\mathcal{D}$-chain and $\pi(\cdot)$ for the measure with respect to which it is reversible. Similarly we write $\hat{p}_t(x,y)$ and $\hat{\pi}(\cdot)$ for the ones associated to its Siegmund dual, the $\hat{\mathcal{D}}$-chain.
We shall denote the Karlin-McGregor semigroup associated to $n$ $\mathcal{D}$-chains killed when they intersect by $\left(P_t^n;t\ge 0\right)$. This is given by the determinantal transition kernel:
\begin{align*}
p^n_t(x,y)=\det(p_t(x_i,y_j))_{i,j=1}^n.
\end{align*}
Similarly, we will write $\left(\hat{P}_t^n;t\ge0\right)$ for the one associated to $n$ $\hat{\mathcal{D}}$-chains. We also define the positive kernels:
\begin{align*}
\left(\Lambda_{n,n+1}f\right)(x)=\sum_{y\prec x}\prod_{i=1}^{n}\hat{\pi}(y_i)f(y), \ \
\left(\Lambda_{n,n}f\right)(x)=\sum_{y\prec x}^{}\prod_{i=1}^{n}\pi(y_i)f(y).
\end{align*}
Then, we have the following Theorem, proven as part of more general results in Section \ref{sectionintertwinins} (for the shortest path to a proof of this particular statement see Remark \ref{Proofremark}).
\begin{thm}\label{TheoremInterIntro} For $t \ge 0$:
\begin{align} 
P_t^{n+1}\Lambda_{n,n+1}&=\Lambda_{n,n+1}\hat{P}_t^{n},\\
\hat{P}_t^{n}\Lambda_{n,n}&=\Lambda_{n,n}P_t^{n}.
\end{align}
\end{thm}
After a Doob's $h$-transformation (see Section \ref{sectionintertwinins}), by a strictly positive eigenfunction $h(\cdot)$ of either $\hat{P}_t^{n}$ or $P_t^{n}$, the relations above take the form:
\begin{align}\label{introconsistency}
P_{N+1}(t) \Lambda_N^{N+1}=\Lambda_N^{N+1}P_N(t) \ ,\ \forall t \ge 0.
\end{align}
Here, the semigroups $P_N(t),P_{N+1}(t)$ are Markov on $W^{\mathsf{n}(N)}$ and $W^{\mathsf{n}(N+1)}$ respectively, where $\mathsf{n}(N+1)\in \{\mathsf{n}(N),\mathsf{n}(N)+1\}$. For example, in \cite{BorodinFerrari} the case $\mathsf{n}(N)=N$ is considered while in \cite{BorodinKuan} one has $\mathsf{n}(N)=\lfloor \frac{N+1}{2}\rfloor$, see figure below for an illustration. Moreover, $\Lambda_N^{N+1}$ is a Markov kernel from $W^{\mathsf{n}(N+1)}$ to $W^{\mathsf{n}(N)}$. 

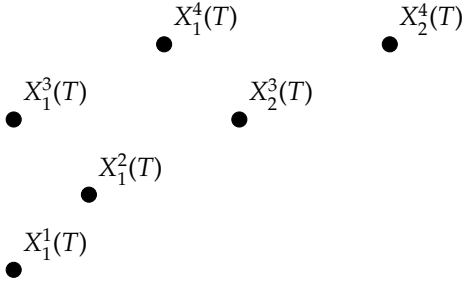
\begin{figure}
\captionsetup{singlelinecheck = false, justification=justified}
\begin{tikzpicture}

\draw[fill] (1,1) circle [radius=0.1];
\node[above right] at (1,1) {$X_1^{2}(T)$};
\draw[fill] (0,2) circle [radius=0.1];
\node[above right] at (0,2) {$X_1^{3}(T)$};
\draw[fill] (3,2) circle [radius=0.1];
\node[above right] at (3,2) {$X_2^{3}(T)$};
\draw[fill] (2,3) circle [radius=0.1];
\node[above right] at (2,3) {$X_1^{4}(T)$};
\draw[fill] (5,3) circle [radius=0.1];
\node[above right] at (5,3) {$X_2^{4}(T)$};
\draw[fill] (0,0) circle [radius=0.1];
\node[above right] at (0,0) {$X_1^{1}(T)$};

\end{tikzpicture}
\caption{Here $\mathsf{n}(1)=1, \mathsf{n}(2)=1, \mathsf{n}(3)=2, \mathsf{n}(4)=2$. The figure is a snapshot, at fixed time $T$ of the particle configuration.}
\end{figure}

As mentioned in the subsection above, these results are a by-product of a two-level coupling of interacting $\mathcal{D}$-chains and $\hat{\mathcal{D}}$-chains, coming from considering a coalescing stochastic flow, which remarkably admits an explicit transition kernel in terms of a block determinant. Although, aspects of this probabilistic argument appeared in the seminal work of Warren \cite{Warren} in the context of the Brownian web or Arratia flow, our exposition in Section \ref{SubsectionFlow} is new, being both elementary and completely self-contained.

\paragraph{Branching graphs} Sequences of stochastic evolutions satisfying (\ref{introconsistency}) can be recast in the framework of coherent dynamics on branching graphs. Let us briefly and informally describe this, all notions are made precise in Section \ref{sectionbranching}. We consider a graded graph $\Gamma$, with vertex set $\sqcup_{N}V_N$ such that $V_N=W^{\mathsf{n}(N)}$, where $\mathsf{n}(1)\le \mathsf{n}(2) \le \cdots,\mathsf{n}(i+1)-\mathsf{n}(i)\in \{0,1\}$. Two vertices $y\in V_N$ and $x \in V_{N+1}$ are connected by an edge if and only if $x$ and $y$ interlace (more precisely if $y\prec x$). We assign certain multiplicities (positive weights) to each edge and from this, see Section \ref{sectionbranching}, we can associate for all $N$ a natural Markov kernel $\Lambda_N^{N+1}$ from $V_{N+1}$ to $V_N$.

The semigroups $P_N(t)$ can be viewed as dynamics on the individual levels $V_N$ of $\Gamma$. We will moreover say that they are coherent with respect to the $\Lambda_N^{N+1}$ if (\ref{introconsistency}) holds. 

A motivation for studying such relations comes from the method of intertwiners of Borodin and Olshanski: it takes as input a tower of relations (\ref{introconsistency}) for all $N$ and produces a Markov process with semigroup $P_{\infty}(t)$ on the boundary $\Omega_{\Gamma}$ of the graph $\Gamma$. Informally the boundary $\Omega_{\Gamma}$ of $\Gamma$ is the space parametrizing extremal coherent sequences of probability measures: $\{\mu_N\}_{N\ge 1}$ on $\{V_N\}_{N \ge 1}$ that satisfy,
\begin{align*}
\mu_{N+1} \Lambda_{N}^{N+1}=\mu_N
\end{align*}
and that cannot be decomposed into convex combinations of other such sequences.

The method was applied in the context of two well known branching graphs related to representation theory, the Gelfand-Tsetlin graph in \cite{BorodinOlshanski} by Borodin and Olshanski and the type-BC graph in \cite{Cuenca} by Cuenca. In Section \ref{sectionExamples} we give alternative proofs of their main results, showing how they follow from Theorem \ref{TheoremInterIntro}. For a brief comparison between the proofs see Remark \ref{Proofremark}.

\paragraph{Push-Block dynamics}

Now, suppose we are given a sequence of processes with semigroups $\{ P_n(t)\}_{n=1}^N$ and Markov kernels $\{\Lambda^{n+1}_n\}^{N-1}_{n=1}$ satisfying (\ref{introconsistency}) derived from Theorem \ref{TheoremInterIntro}. We then construct a multilevel process $\left(\left(X^1(t),\cdots,X^N(t)\right);t \ge 0\right)$, taking values in $W^{\mathsf{n}(1)}\times \cdots \times W^{\mathsf{n}(N)}$ so that consecutive levels interlace and which satisfies a Gibbs property with respect to the Markov kernels $\Lambda_n^{n+1}$. The interactions between particles (coordinates $X_i^n$) are through the so called push-block dynamics. These dynamics were first introduced in the seminal paper \cite{BorodinFerrari} and the continuous Brownian analogue, where pushing and blocking are no longer distinguishable, in \cite{Warren}.

They are described informally as follows (see Section \ref{subsectionmultilevelconstruction} for the rigorous description): Each particle has two independent exponential clocks with (not necessarily equal) rates depending only on its position (and not of other particles) for jumping to the right and to the left respectively by one. Suppose the clock for jumping to the right of particle $X_i^n$ at level $n$ rings first. Then the particle will attempt to jump to the right by one; if the interlacing with level $n-1$ is no longer satisfied this jump is not allowed and we say the particle is blocked. Otherwise, it moves by one to the right, possibly triggering some pushing moves. Namely, if the interlacing is no longer preserved with the next level then the particle at level $n+1$ with respect to which the interlacing is broken also moves (instantaneously) to the right by one. This pushing is propagated to higher levels. 

Then, in Section \ref{subsectionconsistent} we prove a result of the following sort, that we state informally here (see Propositions \ref{CoherentdynamicsproptypeA} and \ref{CoherentdynamicsproptypeB} for the precise statements):

\begin{prop}[Informal statement]\label{IntroPropositionInformal}
Suppose the Markov kernels $\{\Lambda^{n+1}_n\}^{N-1}_{n=1}$ and semigroups $\left(P_n(t);t\ge0\right)^N_{n=1}$ obtained from Theorem \ref{TheoremInterIntro} satisfy the intertwining relations (\ref{introconsistency}) for $n=1,\cdots, N-1$. Then, there exists a Markovian coupling $\left(\left(X^1(t),\cdots,X^N(t)\right);t \ge 0\right)$ (taking values in an interlacing array) evolving according to push-block dynamics with explicit rates (see Subsections \ref{subsectionmultilevelconstruction} and \ref{subsectionconsistent} ) such that the following hold: Assume the process $\left(\left(X^1(t),\cdots,X^N(t)\right);t \ge 0\right)$ is initialized according to the Gibbs measure, where $\mathfrak{M}_N(\cdot)$ is an arbitrary probability measure on $W^{\mathsf{n}(N)}$:
\begin{align*}
\mathsf{Prob}^{\mathfrak{M}_N}_{1,\cdots,N}\left(x^1,\cdots,x^N\right)=\mathfrak{M}_N(x^N)\Lambda_{N-1}^{N}\left(x^N,x^{N-1}\right)\cdots\Lambda_{1}^{2}\left(x^2,x^{1}\right).
\end{align*}
Then, the distribution at time $T$ of $\left(X^1(T),\cdots,X^N(T)\right)$ is given by the evolved Gibbs measure:
\begin{align*}
\mathsf{Prob}^{\mathfrak{M}_NP_N(T)}_{1,\cdots,N}\left(x^1,\cdots,x^N\right)=\left[\mathfrak{M}_NP_N(T)\right](x^N)\Lambda_{N-1}^{N}\left(x^N,x^{N-1}\right)\cdots\Lambda_{1}^{2}\left(x^2,x^{1}\right).
\end{align*}
Moreover, for each $n=1,\cdots,N$ the projection to $\left(X^n(t);t\ge 0\right)$ is a Markov process evolving according to $\left(P_n(t);t\ge0\right)$.
\end{prop}

By the special structure of the Markov kernels and semigroups involved (for certain initial measures $\mathfrak{M}_N$) we get that the evolved measure $\mathsf{Prob}^{\mathfrak{M}_NP_N(T)}_{1,\cdots,N}\left(x^1,\cdots,x^N\right)$ is given as a certain product of determinants. Such measures, by the celebrated Eynard-Mehta Theorem (see \cite{BorodinRains}), give rise to determinantal point processes with an extended correlation kernel $\mathsf{K}$, which can in principle be computed (see Remark \ref{RemarkEynardMehtaForEvolved}).

\subsection{Alternating construction}

We will now consider in some detail a particular choice of consistent dynamics. These dynamics give rise via Proposition \ref{IntroPropositionInformal} to an interacting interlacing particle system with a wall. Such a system can be mapped to a random growth and decay model of a stepped surface under a certain correspondence between particles and lozenges/cubes, see Figure \ref{figuresteppedsurface} for an illustration. We first need a definition. Denote by $\mathbb{GT}_\textbf{s}(\infty)$ the set of infinite symplectic Gelfand-Tsetlin patterns, interlacing arrays of the form, where all particles live in $\mathbb{N}$, with the \textit{origin} playing the role of a wall,
 \begin{align*}
\mathbb{GT}_\textbf{s}(\infty)=\bigg\{\mathbb{X}=\left(\mathbb{X}^{(0,1)},\mathbb{X}^{(1,1)},\mathbb{X}^{(1,2)},\cdots\right):\mathbb{X}^{(i,i)}\in W^i,\mathbb{X}^{(i,i+1)}\in W^{i+1};\mathbb{X}^{(i-1,i)}\prec\mathbb{X}^{(i,i)}\prec\mathbb{X}^{(i,i+1)}\bigg\}.
\end{align*}
Here, the similar notations $\mathbb{X}^{(i,i+1)}$ and $W^{i,i+1}$ should not be confused, $\mathbb{X}^{(i,i+1)}$ denotes a configuration on a single level (in fact it is the top level of $\left(\mathbb{X}^{(i,i)},\mathbb{X}^{(i,i+1)}\right)$ which takes values in $W^{i,i+1}$). The dynamics go as follows: Particles at level $\mathbb{X}^{(i,i+1)}$ evolve as $i+1$ independent $\mathcal{D}$-chains which are pushed and blocked by particles at level $\mathbb{X}^{(i,i)}$, which themselves evolve as $i$ independent $\hat{\mathcal{D}}$-chains that are in turn pushed and blocked by particles at level $\mathbb{X}^{(i-1,i)}$ and so forth, see Figure \ref{figuresteppedsurface} for an example. We call this the \textit{alternating construction}, since we alternate between using the jump rates for $\mathcal{D}$ and $\hat{\mathcal{D}}$-chains on odd and even levels. We think of the position-dependent jump, equivalently growth and decay, rates as inhomogeneities of the surface. 

To make the connection with Proposition \ref{IntroPropositionInformal}, the projection on level $\mathbb{X}^{(n,n+1)}$ evolves according to the semigroup with transition kernel:
\begin{align*}
\frac{h_{n,n+1}(y_1,\cdots,y_{n+1})}{h_{n,n+1}(x_1,\cdots,x_{n+1})}\det(p_t(x_i,y_j))_{i,j=1}^{n+1}
\end{align*}
and on level $\mathbb{X}^{(n,n)}$ according to:
\begin{align*}
\frac{h_{n,n}(y_1,\cdots,y_{n})}{h_{n,n}(x_1,\cdots,x_{n})}\det(\hat{p}_t(x_i,y_j))_{i,j=1}^n.
\end{align*}
Moreover, the harmonic functions $h_{n,n+1}(\cdot)$ and $h_{n,n}(\cdot)$ are given by:
\begin{align*}
h_{n,n+1}(\cdot)=(\Lambda_{n,n+1}\Lambda_{n,n}\cdots \Lambda_{1,1}\textbf{1})(\cdot) \ , \
h_{n,n}(\cdot)=(\Lambda_{n,n}\Lambda_{n-1,n}\cdots \Lambda_{1,1}\textbf{1})(\cdot).
\end{align*}

The distribution of this particle system at time $t$ determines a point process denoted by $\Xi^{t}$. Assume that all particles are initially \textit{fully packed} i.e. at levels $(i-1,i)$ and $(i,i)$ we have our $i$ particles at positions $0<1<2<\cdots<i-1$ (see Figure \ref{figuresteppedsurface}). We shall also use the following notation throughout: the variable $z=((n_1,n_2),x)$ will denote the level $(n_1,n_2)$ and (horizontal) position $x$ of the particle. 

\begin{figure}
\captionsetup{singlelinecheck = false, justification=justified}
\begin{tikzpicture}

\draw (0,0) to (8/3,8);
\draw (1,0) to (11/3,8);
\draw (2,0) to (14/3,8);
\draw (3,0) to (17/3,8);
\draw (0,0) to (3,0);
\draw (1/3,1) to (10/3,1);
\draw (2/3,2) to (11/3,2);
\draw (1,3) to (4,3);
\draw (4/3,4) to (13/3,4);
\draw (5/3,5) to (14/3,5);
\draw (2,6) to (5,6);
\draw (7/3,7) to (16/3,7);
\draw (8/3,8) to (17/3,8);
\draw (0,0) to (-1/3,1);
\draw (1/3,1) to (-1/3,3);
\draw (2/3,2) to (-1/3,5);
\draw (1,3) to (-1/3,7);
\draw (4/3,4) to (0,8);
\draw (5/3,5) to (1/3,9);
\draw (2,6) to (1,9);
\draw (7/3,7) to (5/3,9);
\draw (8/3,8) to (7/3,9);
\draw (-1/3,5) to (1,9);
\draw (-1/3,7) to (1/3,9);
\draw (-1/3,3) to (5/3,9);
\draw (-1/3,1) to (7/3,9);
\draw[fill,blue] (0,1) circle [radius=0.1];
\draw[fill] (1/3,2) circle [radius=0.1];
\draw[fill] (0,3) circle [radius=0.1];
\draw[fill] (2/3,3) circle [radius=0.1];
\draw[fill] (0,3) circle [radius=0.1];
\draw[fill] (2/3,3) circle [radius=0.1];
\draw[fill] (1/3,4) circle [radius=0.1];
\draw[fill] (1,4) circle [radius=0.1];
\draw[fill] (0,5) circle [radius=0.1];
\draw[fill] (2/3,5) circle [radius=0.1];
\draw[fill] (4/3,5) circle [radius=0.1];
\draw[fill] (1/3,6) circle [radius=0.1];
\draw[fill] (1,6) circle [radius=0.1];
\draw[fill] (5/3,6) circle [radius=0.1];
\draw[fill] (0,7) circle [radius=0.1];
\draw[fill] (2/3,7) circle [radius=0.1];
\draw[fill] (4/3,7) circle [radius=0.1];
\draw[fill] (2,7) circle [radius=0.1];
\draw[fill] (1/3,8) circle [radius=0.1];
\draw[fill] (1,8) circle [radius=0.1];
\draw[fill] (5/3,8) circle [radius=0.1];
\draw[fill] (7/3,8) circle [radius=0.1];
\draw[dotted,thick, red] (0,0)--(0,1) -- (1/3,2)--(0,3)--(1/3,4)--(0,5)--(1/3,6)--(0,7)--(1/3,8)--(0,9);
\draw[dotted,thick, red] (2/3,0)--(2/3,3) -- (1,4)--(2/3,5)--(1,6)--(2/3,7)--(1,8)--(2/3,9);
\draw[dotted,thick, red] (4/3,0)--(4/3,5) -- (5/3,6)--(4/3,7)--(5/3,8)--(4/3,9);
\draw[dotted,thick, red] (2,0)--(2,7) --(7/3,8)--(2,9);
\node[below] at (0,0) {$0$};
\node[below] at (2/3,0) {$1$};
\node[below] at (4/3,0) {$2$};
\node[below] at (2,0) {$3$};
\node[left] at (-1/3,1) {$(0,1)$};
\node[left] at (-1/3,2) {$(1,1)$};
\node[left] at (-1/3,3) {$(1,2)$};
\node[left] at (-1/3,4) {$(2,2)$};
\node[left] at (-1/3,5) {$(2,3)$};
\node[left] at (-1/3,6) {$(3,3)$};
\node[left] at (-1/3,7) {$(3,4)$};
\node[left] at (-1/3,8) {$(4,4)$};
\draw[->,ultra thick,blue] (0,1) to (0.5,1);

\end{tikzpicture}
\begin{tikzpicture}

\draw (1,0) to (11/3,8);
\draw (2,0) to (14/3,8);
\draw (3,0) to (17/3,8);
\draw (1,0) to (3,0);
\draw (4/3,1) to (10/3,1);
\draw (5/3,2) to (11/3,2);
\draw (2,3) to (4,3);
\draw (7/3,4) to (13/3,4);
\draw (8/3,5) to (14/3,5);
\draw (3,6) to (5,6);
\draw (10/3,7) to (16/3,7);
\draw (11/3,8) to (17/3,8);
\draw (0,0) to (-1/3,1);
\draw (0,2) to (-1/3,3);
\draw (1/3,3) to (-1/3,5);
\draw (2/3,4) to (-1/3,7);
\draw (1,5) to (0,8);
\draw (4/3,6) to (1/3,9);
\draw (5/3,7) to (1,9);
\draw (2,8) to (5/3,9);

\draw (-1/3,5) to (1,9);
\draw (-1/3,7) to (1/3,9);
\draw (-1/3,3) to (5/3,9);
\draw (-1/3,1) to (7/3,9);

\draw (1,0) to (2/3,1);
\draw (4/3,1) to (1,2);
\draw (5/3,2) to (4/3,3);
\draw (2,3) to (5/3,4);
\draw (7/3,4) to (2,5);
\draw (8/3,5) to (7/3,6);
\draw (3,6) to (8/3,7);
\draw (10/3,7) to (3,8);
\draw (11/3,8) to (10/3,9);
\draw (2/3,1) to (10/3,9);

\draw (0,0) to (1,0);
\draw (-1/3,1) to (2/3,1);
\draw (0,2) to (1,2);
\draw (1/3,3) to (4/3,3);
\draw (2/3,4) to (5/3,4);
\draw (1,5) to (2,5);
\draw (4/3,6) to (7/3,6);
\draw (5/3,7) to (8/3,7);
\draw (2,8) to (3,8);
\draw (7/3,9) to (10/3,9);

\draw[fill] (0,3) circle [radius=0.1];
\draw[fill] (5/3,3) circle [radius=0.1];
\draw[fill] (1,1) circle [radius=0.1];
\draw[fill] (4/3,2) circle [radius=0.1];

\draw[fill] (1/3,4) circle [radius=0.1];
\draw[fill] (2,4) circle [radius=0.1];
\draw[fill] (0,5) circle [radius=0.1];
\draw[fill] (2/3,5) circle [radius=0.1];
\draw[fill,blue] (7/3,5) circle [radius=0.1];
\draw[fill] (1/3,6) circle [radius=0.1];
\draw[fill] (1,6) circle [radius=0.1];
\draw[fill] (8/3,6) circle [radius=0.1];
\draw[fill] (0,7) circle [radius=0.1];
\draw[fill] (2/3,7) circle [radius=0.1];
\draw[fill] (4/3,7) circle [radius=0.1];
\draw[fill] (3,7) circle [radius=0.1];
\draw[fill] (1/3,8) circle [radius=0.1];
\draw[fill] (1,8) circle [radius=0.1];
\draw[fill] (5/3,8) circle [radius=0.1];
\draw[fill] (10/3,8) circle [radius=0.1];
\draw[->,ultra thick, blue] (7/3,5) to (17/6,5);
\end{tikzpicture}
\begin{tikzpicture}

\draw (1,0) to (7/3,4);
\draw (2,0) to (14/3,8);
\draw (3,0) to (17/3,8);
\draw (1,0) to (3,0);
\draw (4/3,1) to (10/3,1);
\draw (5/3,2) to (11/3,2);
\draw (2,3) to (4,3);
\draw (7/3,4) to (13/3,4);
\draw (11/3,5) to (14/3,5);
\draw (4,6) to (5,6);
\draw (13/3,7) to (16/3,7);
\draw (14/3,8) to (17/3,8);
\draw (0,0) to (-1/3,1);
\draw (0,2) to (-1/3,3);
\draw (1/3,3) to (-1/3,5);
\draw (2/3,4) to (-1/3,7);
\draw (1,5) to (0,8);
\draw (4/3,6) to (1/3,9);
\draw (5/3,7) to (1,9);
\draw (2,8) to (5/3,9);

\draw (-1/3,5) to (1,9);
\draw (-1/3,7) to (1/3,9);
\draw (-1/3,3) to (5/3,9);
\draw (-1/3,1) to (7/3,9);
\draw (3,5) to (13/3,9);

\draw (1,0) to (2/3,1);
\draw (4/3,1) to (1,2);
\draw (5/3,2) to (4/3,3);
\draw (2,3) to (5/3,4);
\draw (7/3,4) to (2,5);
\draw (10/3,4) to (3,5);
\draw (11/3,5) to (10/3,6);
\draw (4,6) to (11/3,7);
\draw (13/3,7) to (4,8);
\draw (14/3,8) to (13/3,9);
\draw (2/3,1) to (10/3,9);

\draw (0,0) to (1,0);
\draw (-1/3,1) to (2/3,1);
\draw (0,2) to (1,2);
\draw (1/3,3) to (4/3,3);
\draw (2/3,4) to (5/3,4);
\draw (1,5) to (3,5);
\draw (4/3,6) to (10/3,6);
\draw (5/3,7) to (11/3,7);
\draw (2,8) to (4,8);
\draw (7/3,9) to (13/3,9);

\draw[fill] (0,3) circle [radius=0.1];
\draw[fill] (5/3,3) circle [radius=0.1];
\draw[fill] (1,1) circle [radius=0.1];
\draw[fill] (4/3,2) circle [radius=0.1];

\draw[fill] (1/3,4) circle [radius=0.1];
\draw[fill] (2,4) circle [radius=0.1];
\draw[fill] (0,5) circle [radius=0.1];
\draw[fill] (2/3,5) circle [radius=0.1];
\draw[fill] (10/3,5) circle [radius=0.1];
\draw[fill] (1/3,6) circle [radius=0.1];
\draw[fill] (1,6) circle [radius=0.1];
\draw[fill] (11/3,6) circle [radius=0.1];
\draw[fill] (0,7) circle [radius=0.1];
\draw[fill] (2/3,7) circle [radius=0.1];
\draw[fill] (4/3,7) circle [radius=0.1];
\draw[fill] (4,7) circle [radius=0.1];
\draw[fill] (1/3,8) circle [radius=0.1];
\draw[fill] (1,8) circle [radius=0.1];
\draw[fill] (5/3,8) circle [radius=0.1];
\draw[fill] (13/3,8) circle [radius=0.1];

\end{tikzpicture}
\caption{The visualisation of a particle configuration of $\mathbb{GT}_\textbf{s}(\infty)$ as a stepped surface. In the first figure the fully packed initial condition is depicted. Particle $x_1^{(0,1)}$ wants to jump to the right and in doing so, pushes all the particles indexed $x_i^{(i-1,i)}$ and $x_i^{(i,i)}$ to the right by one as well, resulting in the surface shown in the second figure. Next, particle $x_3^{(2,3)}$ jumps to the right by one and this produces the stepped surface of the last figure.}\label{figuresteppedsurface}
\end{figure}
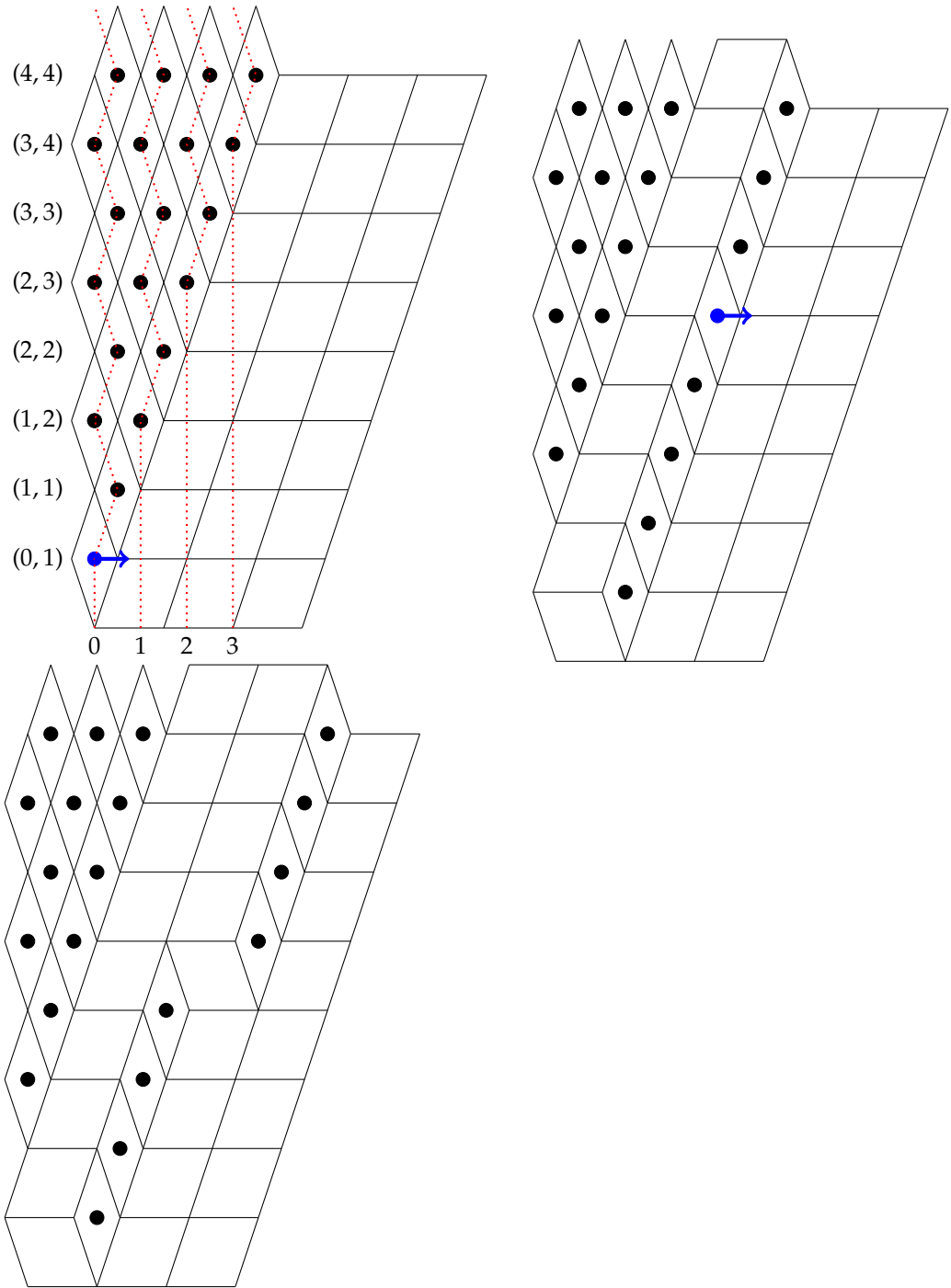

We now explain how the model in the seminal work of Borodin and Kuan \cite{BorodinKuan} related to the representation theory of the infinite dimensional orthogonal group $O(\infty)$ and its recent generalization by Cerenzia and Kuan \cite{CerenziaKuan} are special cases of this construction: they simply correspond to a particular choice of the rate functions $\lambda(\cdot), \mu(\cdot)$. The rates considered by Cerenzia and Kuan in \cite{CerenziaKuan}, depending on two real parameters $\alpha,\beta>-1$, are the following:
\begin{align*}
\lambda(n)&=\frac{n+\alpha+\beta+1}{2n+\alpha+\beta+1}\frac{2(n+\alpha+1)}{2n+\alpha+\beta+2},\\
\mu(n)&=\frac{n+\beta}{2n+\alpha+\beta}\frac{2n}{2n+\alpha+\beta+1}.
\end{align*}
For $\alpha=\beta=-\frac{1}{2}$ these specialize to the model studied by Borodin and Kuan in \cite{BorodinKuan} while for $-\alpha=\beta=\frac{1}{2}$ they specialize to the model studied by Cerenzia \cite{Cerenzia} related to the infinite dimensional symplectic group $Sp(\infty)$. 

We now briefly mention two simple examples of rates that as far as we know cannot be obtained at present from some other integrable model with determinantal structure. These correspond to regions with different speeds (with $r$ and $n^*$ positive constants):
\begin{align*}
\mu(n)\equiv 1  \ , \ \lambda (n)= r \textbf{1} (n\le n^*)+\textbf{1} (n > n^*)
\end{align*}
and periodic (in space) rates:
\begin{align*}
\mu(n)\equiv 1  \ , \ \lambda (n)= \begin{cases}
1, \ & n \ \textnormal{even},\\
r, \ & n  \ \textnormal{odd}
\end{cases}.
\end{align*}
Of course, much more complicated variations for the rates can be considered and our main result below still applies.

Finally, as explained in the previous subsection, the evolved Gibbs measures for these dynamics are given as products of determinants and by making use of (one of the many variants of) the famous Eynard-Mehta Theorem, in particular a generalization to interlacing particles first studied by Borodin and Rains (see \cite{BorodinRains}) it is standard that there is an underlying determinantal structure for this point process. However, to compute the correlation kernel $\mathcal{K}^t$ explicitly one needs to either invert a Gram matrix or solve a biorthogonalization problem, which is usually a formidable task.

\subsection{Explicit computation of correlation kernel and scaling limit}
It is at this point that a further insight is required in order to proceed. We make use of the spectral theory for one-dimensional birth and death chains first developed by Karlin and McGregor in \cite{KarlinMcGregorClassification} and \cite{KarlinMcGregorDifferential}. More precisely we define the polynomials $Q_i(x)$ through the three term recurrence:
 \begin{align*}
    Q_0(x)=1, -xQ_0(x)=-(\lambda(0)+\mu(0))Q_0(x)+\lambda(0)Q_1(x),\\
    -xQ_n(x)=\mu(n)Q_{n-1}(x)-(\lambda(n)+\mu(n))Q_n(x)+\lambda(n)Q_{n+1}(x).
\end{align*}
These are orthogonal with respect to the \textit{spectral measure} $\textnormal{d}\mathfrak{w}(x)$ on $\mathbb{R}_+$ with support $\mathfrak{I}$,
    \begin{align*}
    \int_{0}^{\infty}Q_i(x)Q_j(x)\textnormal{d}\mathfrak{w}(x)=\frac{1}{\pi(j)}\delta_{ij}.
    \end{align*}
Here and onwards we use the notational convention that $x$ is a continuous (spectral as we see below) variable and not the discrete horizontal position of particles as before (also $n$ in the three term recurrence above should not be confused with the index of vertical levels which in the setting of the alternating construction will always be denoted by a pair of indices).

If we view $\mathcal{D}_{k}$, the generator of the birth and death chain with rates $(\lambda(\cdot),\mu(\cdot))$, as a difference operator in the discrete variable $k$, then the three term recurrence takes the form of an eigenfunction relation, with eigenvalue $-x\le 0$:
\begin{align*}
\mathcal{D}_kQ_k(x)=-xQ_k(x).
\end{align*}
 These ingredients provide the following spectral expansion for the transition density of the chain:
    \begin{align*}
    p_t(i,j)=\pi(j)\int_{0}^{\infty}e^{-tx}Q_i(x)Q_j(x)\textnormal{d}\mathfrak{w}(x).
    \end{align*}
We now give one of the simplest explicit examples. It corresponds to the model of \cite{BorodinKuan}, with $\lambda(0)=1, \lambda(n)\equiv\mu(n)\equiv \frac{1}{2}, n \ge 1$ and $\pi(0)=1, \pi(n)\equiv 2, n \ge 1$. Then, the $Q_i$ are the Chebyshev polynomials of the first kind, orthogonal with respect to the spectral measure $\textnormal{d}\mathfrak{w}(x)=\frac{1}{\pi}x^{-\frac{1}{2}}(2-x)^{-\frac{1}{2}}dx$ on $[0,2]$:
\begin{align*}
    \int_{0}^{2}\frac{1}{\pi}x^{-\frac{1}{2}}(2-x)^{-\frac{1}{2}}Q_i(x)Q_j(x)dx=\begin{cases}
    \delta_{ij},  & i=0,\\
    \frac{1}{2}\delta_{ij}, & i\ge 1
    \end{cases}.
\end{align*}   
Moreover, the polynomials $Q_n(x)$ are given explicitly:
\begin{align*}
Q_n(x)=\sum_{k=0}^{n}\frac{(-n)_k(n)_k}{\left(\frac{1}{2}\right)_k}\frac{\left(\frac{x}{2}\right)^k}{k!}
\end{align*}
where $(a)_k=\prod_{i=1}^{k}(a+i-1)$ with $(a)_0=1$, is the Pochhammer symbol.

Coming back to the general discussion, one can also define the polynomials $\hat{Q}_k$ and measure $\hat{\mathfrak{w}}$ associated to the Siegmund dual chain and many relations exist between these dual polynomials, which can be found in Section \ref{sectionorthogonalpolynomials}.

We then go on to introduce and study in detail, from a probabilistic perspective in Sections \ref{multivariatepolynomialssection} to \ref{sectionevolutionoperators} (see also Subsection \ref{SubsectionFurtherResultsIntro} below), their multivariate versions: For $\nu \in W^{n}$, we consider the $n$-variate polynomials given by, with $x=\left(x_1, \cdots , x_n\right)$ in $\mathbb{R}^n$,
\begin{align*}
\mathfrak{Q}_{\nu}(x)=\frac{\det\left(Q_{\nu_i}(x_j)\right)_{i,j=1}^n}{\det\left(x_j^{i-1}\right)_{i,j=1}^n},\ \ \hat{\mathfrak{Q}}_{\nu}(x)=\frac{\det\left(\hat{Q}_{\nu_i}(x_j)\right)_{i,j=1}^n}{\det\left(x_j^{i-1}\right)_{i,j=1}^n}.
\end{align*}
We call these the Karlin-McGregor polynomials, since they were first introduced by Karlin and McGregor, in their original study of intersection probabilities of birth and death chains in \cite{KarlinMcGregorCoincidence}. We remark that the Karlin-McGregor polynomials have the same structure as the well-known Schur polynomials, indeed we arrive at them by replacing the monomials $x^{\nu_i}$ in the numerator determinant by general orthogonal polynomials $Q_{\nu_i}(x)$ on $[0,\infty)$. Note however that the Schur polynomials are not special cases of our construction; the monomials are not orthogonal on $[0,\infty)$ (they are orthogonal on the circle).

We can obtain the Markov kernels associated to the alternating construction from branching rules of these multivariate polynomials (see Section \ref{multivariatepolynomialssection}, also Remark \ref{RemarkMarkovKernelFromBranchingRule}). In particular the harmonic functions $h_{n,n+1}(\nu)$ and $h_{n,n}(\nu)$ are given by, here $\lambda_0=\lambda(0)$ is the rate of jumping from 0 to 1:
\begin{align*}
h_{n,n+1}(\nu)=(-1)^{\binom{n}{2}}\lambda_0^{\binom{n}{2}}\mathfrak{Q}_{\nu}(\vec{0}) \ , \ h_{n,n}(\nu)=(-1)^{\binom{n-1}{2}}\lambda_0^{\binom{n-1}{2}}\hat{\mathfrak{Q}}_{\nu}(\vec{0}).
\end{align*}

Moving on, it is only after expressing the entries of the determinants appearing in the distribution of the growth process starting from the fully packed initial condition in terms of these one dimensional orthogonal polynomials and the spectral measures, that it is possible to see/guess what the solution to the biorthogonalization problem is. Then we proceed to carefully check that it is indeed the solution. All of this is done in detail in Section \ref{sectioncorrelation}. 

The biorthogonalization problem could possibly admit a concise enough solution for other initial conditions as well, other than the fully packed, although we do not attempt to do this here; see for example \cite{BrownianStationaryInitial} where this is done for the stationary case of the Brownian motion analogue. Pursuing this further would be an interesting line of investigation.

Finally, after some more algebraic manipulations we arrive at the following result, proven as a special case of the more general Theorem \ref{correlationkernelmain} in the text:

\begin{thm}\label{IntroductionCorrelationKernelTheorem}
 Let $\mathfrak{I}$ be compact then the correlation functions $\{\rho^{t}_k\}_{k\ge0}$ of the point process $\Xi^{t}$, associated to the alternating construction starting from the fully packed initial condition, are determinantal:
 \begin{align}
  \rho^{t}_k(z_1,\cdots,z_k)\overset{\textnormal{def}}{=}\Xi^{t}(\{E \in \mathbb{GT}_{\mathbf{s}}(\infty) \textnormal{ s.t. } \{z_1,\cdots,z_k\} \subset E\})=\det\left(\mathcal{K}^{t}(z_i,z_j)\right)^k_{i,j=1}
 \end{align}
 where $\mathcal{K}^{t}$ is given by,
 \begin{align}
 \mathcal{K}^{t}\left(\left((n_1,n_2),i),(m_1,m_2),j)\right)\right)=\frac{1}{2\pi \mathsf{i}}\oint_{u \in\mathsf{C}(\mathfrak{I})}\int_{x\in\mathfrak{I}}^{} \tilde{\mathcal{P}_j}(u) \bar{\mathcal{P}}_i(x) \frac{x^{n_2}}{u^{m_2}}\frac{e^{-tx}}{(x-u)e^{-tu}} d\mathfrak{m}(x)du\nonumber\\
 +\textbf{1}\left((n_1,n_2)\ge (m_1,m_2)\right)\int_{\mathfrak{I}}^{}\bar{\mathcal{P}}_i(x) x^{n_2-m_2}\tilde{\mathcal{P}_j}(x) d\mathfrak{m}(x)
 \end{align}
 and,
 \begin{align}
 (\bar{\mathcal{{P}}},\tilde{\mathcal{P}},\mathfrak{m})=\begin{cases}
 (\pi_iQ_i,Q_j,\mathfrak{w}) \textnormal{ if } (n_1,n_2),(m_1,m_2)=(n,n+1),(m,m+1)\\
 (\pi_iQ_i,\hat{Q}_j,\mathfrak{w}) \textnormal{ if } (n_1,n_2),(m_1,m_2)=(n,n+1),(m,m)\\
 (\hat{\pi}_i\hat{Q}_i,Q_j,\hat{\mathfrak{w}}) \textnormal{ if } (n_1,n_2),(m_1,m_2)=(n,n),(m,m+1)\\
 (\hat{\pi}_i\hat{Q}_i,\hat{Q}_j,\hat{\mathfrak{w}}) \textnormal{ if } (n_1,n_2),(m_1,m_2)=(n,n),(m,m)
 \end{cases}.
 \end{align}
 The contour $\mathsf{C}(\mathfrak{I})$ is positively oriented and encircles the support $\mathfrak{I}$ and $0$.
 \end{thm}

That the interval of orthogonality $\mathfrak{I}$ needs to be compact appears to be a technical analytic requirement and presumably can be removed. Compactness ensures uniform convergence for orthogonal decompositions in terms of the polynomials $Q_i$ and allows for the interchange of summation and integration (see e.g. Remark 9.4); it is essentially only used in the computations in Sections 9 and 10. Because of the length of the paper we have not tried to remove this assumption. 

In the model of Cerenzia and Kuan \cite{CerenziaKuan} mentioned above the $Q_i(x)$ are the Jacobi polynomials, which specialize to the Chebyshev polynomials of the earlier works \cite{BorodinKuan}, \cite{Cerenzia}.

We then go on to consider a particular scaling limit, at a finite distance from the wall and make the connection with Borodin and Olshanski's work in \cite{BorodinOlshanskiAsep} on discrete determinantal ensembles associated to continuous orthogonal polynomials. 

More precisely, suppose we scale time as $t(N)=N \tau$ and the arguments of the kernel as $\left(\tilde{m}_1(N),\tilde{m}_2(N)\right)=\left(\lfloor N \eta\rfloor+m_1,\lfloor N \eta\rfloor+m_2\right)$ and $\left(\tilde{n}_1(N),\tilde{n}_2(N)\right)=\left(\lfloor N \eta\rfloor+n_1,\lfloor N \eta\rfloor+n_2\right)$
and let $\alpha=\frac{\eta}{\tau}$. Note that, we do not scale the horizontal positions which avoids hard asymptotics involving the orthogonal polynomials $Q_i,\hat{Q}_i$ or the spectral measures $\mathfrak{w}, \hat{\mathfrak{w}}$.
Then, we have the following theorem whose proof, based on a simple steepest descent analysis, can be found in subsection \ref{SubsectionScalingLimit}:
\begin{thm} With the notations above and with $I^+$ and $I^-$ denoting the upper and lower endpoints of $\mathfrak{I}$ respectively, we have:
\begin{align*}
\lim_{N\to \infty}\mathcal{K}^{t(N)}\left(\left(\left(\tilde{n}_1(N),\tilde{n}_2(N)\right),i),\left(\tilde{m}_1(N),\tilde{m}_2(N)\right),j)\right)\right)=\mathfrak{K}_{\alpha}\left(\left(\left(n_1,n_2\right),i),\left(m_1,m_2\right),j)\right)\right)\\
=\int_{I^-}^{I^+}\left[-\textbf{1}(x\ge \alpha)+\textbf{1}\left((n_1,n_2)\ge (m_1,m_2)\right)\right]\bar{\mathcal{P}}_i(x) x^{n_2-m_2}\tilde{\mathcal{P}_j}(x)d\mathfrak{m}(x).
\end{align*}
\end{thm}

Now, to a weight $\mathcal{W}(dx)$ on (some subset of) $\mathbb{R}$ for which the moment problem is determinate and a point $r\in \mathbb{R}$  one can associate a discrete determinantal point process with kernel denoted by $\mathsf{K}^{\mathcal{W}}_{r}\left(i,j\right)$ (see Remark \ref{DiscreteEnsemblesRemark} and \cite{BorodinOlshanskiAsep} for the exact details). Then, as explained in subsection \ref{SubsectionScalingLimit}, if restricted to single levels $\mathfrak{K}_{\alpha}\left(\left(\left(n,n+1\right),i),\left(n,n+1\right),j)\right)\right)$ gives rise to the determinantal ensemble with kernel $\mathsf{K}_{\alpha}^{\mathfrak{w}}(i,j)$ and also $\mathfrak{K}_{\alpha}\left(\left(\left(n,n\right),i),\left(n,n\right),j)\right)\right)$ gives rise to the ensemble governed by the kernel $\mathsf{K}_{\alpha}^{\hat{\mathfrak{w}}}(i,j)$. Thus, $\mathfrak{K}_{\alpha}\left(((n_1,n_2),i),(m_1,m_2),j)\right)$ provides a novel multilevel determinantal extension of these discrete ensembles, so that particles on consecutive levels interlace (by construction). Moreover, in this generality, it is the first time that these ensembles appear in a concrete interacting particle system.

\subsection{Further results}\label{SubsectionFurtherResultsIntro}
En route, to our computation of the correlation kernel $\mathcal{K}^t$ we introduce in Section \ref{sectioncoherentmeasures} a large class of coherent probability measures, with respect to the Markov kernels corresponding to the alternating construction. These depend on a set of parameters $(t;\alpha_1, \cdots, \alpha_{\mathfrak{N}})$ with $\mathfrak{N}\in \mathbb{N}$ such that  $t \ge 0$ and $\textnormal{Const}\ge \alpha_1 \ge \alpha_2 \ge \cdots \ge 0$, where $\textnormal{Const}$ has a natural interpretation in terms of the interval of orthogonality $\mathfrak{I}$ (see Remark \ref{remarkendpoint}). Their description is through the spectral theory explained above and the single variable function:
\begin{align} 
\psi(x)=\psi_{t,\vec{\alpha}}(x)=\prod_{i=1}^{\mathfrak{N}}(1-\alpha_ix)e^{-tx}.
\end{align}
By Kolmogorov's Theorem these measures give rise to a stochastic point process in $\mathbb{GT}_\textbf{s}(\infty)$ denoted by $\Xi^{\psi}$, which specializes to $\Xi^t$ when all the $\vec{\alpha}$ coordinates are identically zero. In Theorem \ref{correlationkernelmain} we show that $\Xi^{\psi}$ is a determinantal point process with an explicit kernel $\mathcal{K}^{\psi}$.

Moreover, combining the results of Proposition \ref{factorizationforcoherent} and Section \ref{SubsectionFactorizationImpliesExtremality} in the Appendix, we obtain that under a positive definiteness assumption (see Remark \ref{RemarkPositiveDefiniteness}) for the corresponding Karlin-McGregor polynomials these sequences of measures are actually extremal.

Finally, we observe that an inhomogeneous, with position dependent jumps, two species analogue of PushASEP (see \cite{BorodinFerrariPushASEP}), with at most two particles per site, arises if one looks at the rightmost particles in the interlacing array above. In particular the evolution of the particles $\left(X_1^{(0,1)}(t),X_1^{(1,1)}(t),X_2^{(1,2)}(t),\cdots;t\ge 0\right)$ is autonomous. Of course, the distribution of this $(1+1)$-dimensional model is completely characterized by Theorem \ref{IntroductionCorrelationKernelTheorem}.

\subsection{Outlook and questions}
Many directions and questions arise from this work. We indicate and briefly discuss the ones we find the most interesting:

\begin{itemize}
\item(\textbf{Scaling limits}) Study different scaling regimes for the inhomogeneous process introduced above. It is clear from simulations, performed by the author, that interesting behaviour arises when one introduces for example slow or fast regions, periodic or trigonometric rates. The analysis of course boils down to the associated one-dimensional orthogonal polynomials. Another question is whether in any of the possible scaling regimes perturbations of the rates still give the same asymptotic behaviour. This again will come down to universality statements for orthogonal polynomials.

\item(\textbf{Inhomogeneous TASEP}) Borodin-Ferrari studied in \cite{BorodinFerrari} a $(2+1)$-dimensional growth model taking values in a Gelfand-Tsetlin pattern. Each particle has an independent exponential clock of rate one for jumping to the right by one ($\mu(n)\equiv 0$) and particles as before interact through the push-block dynamics. The projection to the left most particles gives the Totally Asymmetric Simple Exclusion Process (TASEP). In fact the construction in \cite{BorodinFerrari} is much more general and allows for level dependent (constant in space) jump rates, which in terms of TASEP corresponds to particle dependent speeds. A natural question is to find the right, namely integrable, inhomogeneous (in space) generalization of the model in \cite{BorodinFerrari}. This could provide a route to some exact solvability in inhomogeneous TASEP which has thus far resisted many efforts. For a particular case, the slow bond problem (the bond from -1 to 0 is slow, in this paper's notation $\mu(n)\equiv 0, \lambda(n)=1-\epsilon \textbf{1}(n=-1)$), a breakthrough was achieved for the leading order behaviour using non-exactly solvable techniques, see \cite{SlowBondProblem}.

\item (\textbf{Boundary of generalized type BC-graph}) As mentioned above one can associate a branching graph to the alternating construction, that we call generalized type-BC branching graph, its multiplicities are given by general product form weights. Is it possible, at least for certain multiplicities, to describe its boundary? Moreover, what is the relation of such extreme coherent measures with dynamics on the graph. In the case of both the Gelfand-Tsetlin graph and the type-BC graph there is an exact correspondence with continuous time birth and death chain dynamics, discrete time Bernoulli and also geometric jumps. A more ambitious direction would be to develop some kind of perturbation theory for these graphs.
\end{itemize}

\subsection{Contents of the paper}
We quickly describe the contents of each section. In Section \ref*{sectioncoalescing}, we introduce all the relevant material on birth and death (or bilateral) chains that we need. We then introduce the coalescing flows and give our two-level couplings formulae. We moreover obtain our intertwining and Markov functions results. In Section \ref*{sectionpushblock}, we prove that the formulae describe the push-block dynamics by showing that they solve the corresponding backwards equations and that these are unique. Furthermore, we spell out a procedure for concatenating such two-level processes in order to build an interlacing array in a consistent manner. In Section \ref*{sectionbranching}, we define and collect some facts about branching graphs along with two classical examples, the Gelfand-Tsetlin graph and the BC-type branching graph, and the graph corresponding to the alternating construction. We also state the theorem of Borodin and Olshanski known as the method of intertwiners. In Section \ref*{sectionExamples}, we show how known and new examples of consistent dynamics can be obtained as corollaries of our first main result, including the ones in \cite{BorodinOlshanski} and \cite{Cuenca} and moreover, we characterize the ones arising from the coupling studied here that are coherent for the Gelfand-Tsetlin graph. In Sections \ref*{sectionorthogonalpolynomials} and \ref{multivariatepolynomialssection}, we introduce the Karlin-McGregor polynomials associated to $\mathcal{D}$ and $\hat{\mathcal{D}}$-chains and their multivariate analogues and prove some of their properties. In Section \ref*{sectioncoherentmeasures}, we introduce coherent measures (with respect to the Markov kernels associated to the alternating construction) $\mathcal{M}^{\psi}_{n-1,n},\mathcal{M}^{\psi}_{n,n}$ indexed by a function $\psi$ and investigate some of their properties. For $\psi_t(x)=e^{-tx}$ these correspond to the distribution at time $t$ of the push-block dynamics started from the fully packed initial condition as described in the paragraphs above. In Section \ref*{sectionevolutionoperators}, we introduce "evolution operators" for coherent measures denoted by $\mathfrak{P}^{g}_{n-1,n},\mathfrak{P}^{g}_{n,n}$, which when applied to $\mathcal{M}^{\psi}_{n-1,n},\mathcal{M}^{\psi}_{n,n}$ "evolve" these measures to $\mathcal{M}^{g\psi}_{n-1,n},\mathcal{M}^{g\psi}_{n,n}$. We also obtain some sufficient conditions for functions $\psi$ to give rise to bona fide probability measures (with positivity being the non-trivial issue here). In Section \ref{sectioncorrelation}, we finally prove our second main result, the explicit computation of the correlation kernel of the process described previously, this being an application of the Eynard-Mehta theorem (see e.g. \cite{BorodinRains}) along with some preliminaries. Finally, in the Appendix we collect a couple of technical proofs along with; essentially reproducing for our own and the reader's convenience, an argument of Okounkov and Olshanski that we found in \cite{OkounkovOlshanskiJack}, that uses de Finetti's theorem to give a sufficient condition for coherent measures with multiplicative "generating functions" to be extremal, based on a kind of positive definiteness property (an assumption) for the associated orthogonal polynomials.

\paragraph{Acknowledgements} I would like to thank Jon Warren for generously sharing his ideas during several very useful conversations. I would like to thank the anonymous referees for many useful comments and suggestions which led to a much improved version of the paper. Financial support from EPSRC through the MASDOC DTC grant number EP/HO23364/1 is gratefully acknowledged.

\section{Coalescing birth and death chains and intertwinings}\label{sectioncoalescing}
\subsection{General facts on birth and death chains and their duals}

We consider a birth and death chain on $I=\mathbb{N}$, or bilateral birth and death chain on $I=\mathbb{Z}$, denoted by $X$, given by the infinitesimal \textit{birth} $(\lambda(x))_{x\in I}$ and \textit{death} $(\mu(x))_{x\in I}$ rates and with matrix of transition rates denoted by $\mathcal{D}$,
\begin{align*}
\mathcal{D}(x,y)=\begin{cases}
\lambda(x) \ \ & y=x+1\\
-\lambda(x)-\mu(x) \ \ & y=x \\
\mu(x) \ \ & y=x-1
\end{cases}.
\end{align*}
We assume that $\lambda(x),\mu(x)>0$, for all $x \in \mathbb{Z}$ in the bilateral case and $\mu(0)=0$ in case of $I=\mathbb{N}$, i.e. that $0$ is reflecting ($\lambda(x)$ for $x\ge 0$ and $\mu(x)>0$ for $x\ge 1$). We moreover assume that, $\infty$ is a \textit{natural} boundary point, namely a process can neither reach in finite time or be started from such a point (similarly $-\infty$ is assumed \textit{natural} in case $I=\mathbb{Z}$), so that the rates uniquely determine our chain. Sufficient conditions for this, will be given later on below in this subsection. In order to be more concise, we will frequently refer to such a Markov chain with generator $\mathcal{D}$, as a $\mathcal{D}$-chain. Now, define the forward and backward discrete derivatives by,
\begin{align*}
(\nabla f)(x)=f(x+1)-f(x),
(\bar{\nabla} f)(x)=f(x-1)-f(x),\  x \in I,
\end{align*}
and observe that $\mathcal{D}$ can be regarded as a difference operator acting on functions, $f:I \to \mathbb{C}$ as follows,
\begin{align*}
(\mathcal{D}f)(x)=\lambda(x) (\nabla f)(x)+\mu(x) (\bar{\nabla} f)(x), \ x\in I.
\end{align*}
Denote by $p_t(x,y)$ the transition density of the $\mathcal{D}$-chain i.e. with $\left(X(t);t\ge0\right)$ denoting a realization of this chain governed by the family of measures indexed by starting positions, $\{\mathbb{P}_x\}_{x \in I}$ then, $p_t(x,y)=\mathbb{P}_x(X(t)=y)$. Furthermore, we denote by $\left(P_t;t\ge0\right)$ the Feller semigroup (that maps the space of functions vanishing at infinity to itself), it gives rise to (the fact that all these are well defined is discussed next). In particular, we will often use the notation:
\begin{align*}
P_t\textbf{1}_{[l,y]}(x)=\sum_{l\le z \le y}^{}p_t(x,z).
\end{align*}
We note that, under the conditions $(\ref{birthdeathcondition1})$ and $(\ref{birthdeathcondition2})$ below, $p_t(x,y)$ will be the unique solution to the Kolmogorov backward differential equation given by, $\forall t>0,x,y \in I$,
\begin{align*}
\frac{d}{dt}p_t(x,y)=\mathcal{D}_xp_t(x,y),
\end{align*}
subject further to the initial condition, positivity and sub-stochasticity assumptions:
\begin{align*}
p_0(x,y)=\delta_{x,y}\textnormal{ , } p_t(x,y) \ge 0 \textnormal{ and } \sum_{y \in I}^{} p_t(x,y) \le 1.
\end{align*}
Here, $\mathcal{D}_x$ acts as $\mathcal{D}$ on a (possibly multivariate) function in the variable labelled $x$. Now, define the \textit{symmetrizing} measure of the $\mathcal{D}$-chain (the measure with respect to which it is reversible) which we denote by $\pi$ as follows,
\begin{align*}
\pi(x)=\prod_{i=1}^{x}\frac{\lambda(i-1)}{\mu(i)} \ \ x\ge 1, \pi(0)=1, \pi(x)=\prod_{i=1}^{-x}\frac{\mu(x+i)}{\lambda(x+i-1)} \ , \ x\le -1.
\end{align*}

In the case of $I=\mathbb{N}$, we will enforce throughout this paper, the following two conditions (see \cite{Kemperman}, \cite{Vandoorn}),
\begin{align}
\sum_{j=0}^{\infty}\frac{1}{\lambda(j)\pi(j)}\sum_{i=0}^{j}\pi(i)&=\infty \label{birthdeathcondition1},\\ \sum_{j=0}^{\infty}\frac{1}{\lambda(j)\pi(j)}\sum_{i=j+1}^{\infty}\pi(i)&=\infty\label{birthdeathcondition2}.
\end{align}
Then, under conditions $(\ref{birthdeathcondition1})$ and $(\ref{birthdeathcondition2})$ the chain with generator $\mathcal{D}$ is uniquely determined by its rates, it is non-explosive and $p_t(x,y)$ is the unique (stochastic) solution to both the backwards and forwards equations (for proofs of these statements see for example \cite{Kemperman} or \cite{Vandoorn} and the references therein). Moreover, we have $p_t(x,y)\to 0$ as $y\to \infty$ and $p_t(x,y)\to 0$ as $x \to \infty $.

In the case of a bilateral chain, in order for both $-\infty$ and $+\infty$ to be natural boundaries, which in particular, ensures the uniqueness of solutions to both the backwards and forwards equation and non-explosiveness, we need the following four conditions.  The first two, (\ref{bilateralcondition1}) and (\ref{bilateralcondition2}), govern the behaviour at $+\infty$ and the last two, (\ref{bilateralcondition3}) and (\ref{bilateralcondition4}), at $-\infty$, for a proof see Theorem 2.5 and the discussion at the end of page 511 of \cite{Pruitt},
\begin{align}
\sum_{j=1}^{\infty}\frac{1}{\lambda(j)\pi(j)}\sum_{i=1}^{j}\pi(i)&=\infty \label{bilateralcondition1},\\ \sum_{j=1}^{\infty}\pi(j)\sum_{i=1}^{j-1}\frac{1}{\lambda(i)\pi(i)}&=\infty\label{bilateralcondition2},\\
\sum_{j=-\infty}^{-1}\frac{1}{\lambda(j)\pi(j)}\sum_{i=j+1}^{-1}\pi(i)&=\infty \label{bilateralcondition3},\\ \sum_{j=-\infty}^{-1}\pi(j)\sum_{i=j}^{-1}\frac{1}{\lambda(i)\pi(i)}&=\infty\label{bilateralcondition4}.
\end{align}

We now come to the definition (going back to the papers of Karlin and McGregor \cite{KarlinMcGregorClassification}, \cite{KarlinMcGregorDifferential}) of the dual chain $\hat{X}$, on $\mathbb{N}^-=\mathbb{N}\cup \{-1\}$ and $\mathbb{Z}$ respectively, that is given by the infinitesimal rates $\hat{\lambda}(x)=\mu(x+1)$ and $\hat{\mu}(x)=\lambda(x)$ and with generator:
\begin{align*}
\hat{\mathcal{D}}(x,y)=\begin{cases}
\hat{\lambda}(x)=\mu(x+1) \ \ & y=x+1\\
-\mu(x+1)-\lambda(x) \ \ & y=x \\
\hat{\mu}(x)=\lambda(x) \ \ & y=x-1
\end{cases}.
\end{align*}
Note that, in the case of $\mathbb{N}^-$ then, $-1$ is an \textit{absorbing state}. As before, in order to be concise and to emphasise the role of duality in this work, we will sometimes refer to this Markov chain as the $\hat{\mathcal{D}}$-chain and denote its transition density by $\hat{p}_t(x,y)$ (in case of a birth and death chain we only consider the transition density in $\mathbb{N}$, i.e it is the same as that of the process \textit{killed} at -1), its semigroup by $\left(\hat{P}_t;t\ge 0\right)$ and symmetrizing measure by $\hat{\pi}$. 

Now, it is not hard to check that, conditions (\ref{birthdeathcondition1}), (\ref{birthdeathcondition2}) and (\ref{bilateralcondition1}),(\ref{bilateralcondition2}),(\ref{bilateralcondition3}),(\ref{bilateralcondition4}) respectively hold for the rates $\left(\lambda,\mu\right)$, if and only if they hold for the dual rates $\left(\hat{\lambda},\hat{\mu}\right)$ and thus the dual chain is well posed with natural boundaries at $\pm \infty$ as well. 

With the above definitions in place, we arrive at the following key duality relation for birth and death chains, going back to Karlin's and McGregor's classic works \cite{KarlinMcGregorClassification} and \cite{KarlinMcGregorDifferential} (see also \cite{Siegmund}, \cite{CoxRosler}). The relation is also true for bilateral chains and we present, the admittedly almost identical, proof in the Appendix because we could not locate it in the literature. We also give a "graphical" proof in the next subsection.

\begin{lem}[Siegmund duality] \label{ConjugacyLemma}
For $x,y \in I$ and $t\ge 0$ we have, 
\begin{align}
P_t \textbf{1}_{[l,y]}(x)=\hat{P}_t \textbf{1}_{[x,\infty)}(y),
\end{align}
where $l=0$ if $I=\mathbb{N}$ or $l=-\infty$ if $I=\mathbb{Z}$ respectively.
\end{lem}

\begin{rmk}
Note that, the $\ \hat{} \ $ operation is not an involution even in the case of $I=\mathbb{Z}$, unlike the diffusion process setting, see \cite{InterlacingDiffusions}. This is an artefact of the discrete world and will complicate things a little bit, since these asymmetries make keeping track of the positions of $\le$ and $<$ below important.
\end{rmk}

\subsection{Discrete coalescing flow and two-level process}\label{SubsectionFlow}

 First, we define the interlacing spaces our processes will take values in, with $I$ being either $\mathbb{N}$ or $\mathbb{Z}$, in particular all coordinates are integers, and with $l=0$ or $-\infty$ respectively, as follows,
\begin{align*}
W^{n}(I)&=\{x=(x_1,\cdots,x_n)\in I^n:l \le x_1 < \cdots < x_{n} < \infty \},\\
W^{n,n+1}(I)&=\{(x,y)=(x_1,\cdots,x_{n+1},y_1,\cdots,y_n)\in I^{2n+1}:l \le x_1 \le y_1 < x_2 \le \cdots < x_{n+1}< \infty \},\\
W^{n,n}(I)&=\{(x,y)=(x_1,\cdots,x_n,y_1,\cdots,y_n)\in I^{2n}:l \le y_1 \le x_1 < y_2 \le \cdots \le x_{n} < \infty \}.
\end{align*}
Also, define for $x\in W^{n}(I)$,
\begin{align*}
 W^{\bullet,n}(x)=\{y\in W^{\bullet}(I):(x,y)\in W^{\bullet,n}(I)\}.
\end{align*}
Similarly, define $W^{n,\bullet}(y)$,
\begin{align*}
 W^{n,\bullet}(y)=\{x\in W^{\bullet}(I):(x,y)\in W^{n,\bullet}(I)\}.
\end{align*}

\paragraph{Graphical construction of coalescing flow} We now describe the "graphical" construction of the coalescing flow of birth and death (or bilateral) chains. For each site of the lattice $x\in I$, we have independent Poisson processes, indexed by time $t \in \mathbb{R}$, of up $\uparrow$ arrows denoted by $\{N_x^{\uparrow}(t):t \in \mathbb{R}\}$ of (constant) rate $\lambda(x)$ and down $\downarrow$ arrows denoted by $\{N_x^{\downarrow}(t):t \in \mathbb{R}\}$ of (constant) rate $\mu(x)$. 

We now define the family of random maps $\{\boldsymbol{\Phi}_{s,t}:I \to I; s\le t\}$ as follows. For $x\in I$ and $s \le t$, the value $\boldsymbol{\Phi}_{s,t}(x)$ is arrived at by starting at time $s$ at site $x$ and following the direction of the arrows until time $t$. The site you are on at time $t$ is defined to be $\boldsymbol{\Phi}_{s,t}(x)$. There is a slight ambiguity in this definition at arrival times of the arrows and by convention we take the right continuous (in time) version of this map. See Figure \ref{figurecoalescing} for an illustration.

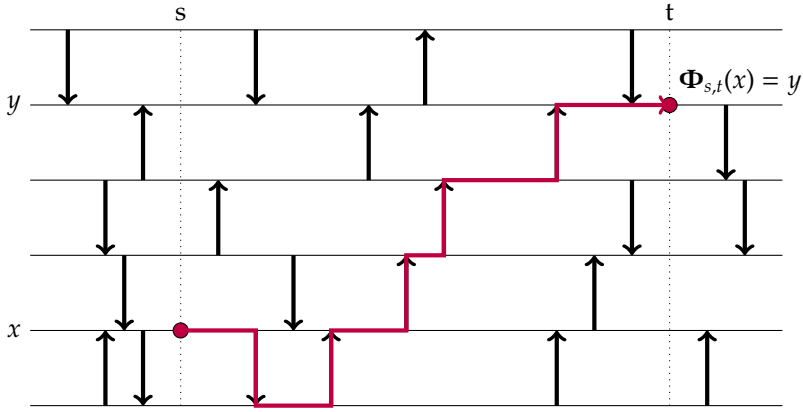
\begin{figure}[h]
\captionsetup{singlelinecheck = false, justification=justified}
\begin{tikzpicture}
\draw[dotted] (2,0) -- (2,5);
\node[above ] at (2,5) { s};
\draw[dotted] (8.5,0) -- (8.5,5);
\node[above ] at (8.5,5) {t};
\node[left ] at (0,1) {$x$};
\node[left ] at (0,4) {$y$};
\node[above right ] at (8.5,4) {$\boldsymbol{\Phi}_{s,t}(x)=y$};
\draw[thin] (0,0) -- (10,0);
\draw[thin] (0,1) -- (10,1);
\draw[thin] (0,2) -- (10,2);
\draw[thin] (0,3) -- (10,3);
\draw[thin] (0,4) -- (10,4);
\draw[thin] (0,5) -- (10,5);
\draw[fill=purple] (2,1) circle [radius=0.1];
\draw[fill=purple] (8.5,4) circle [radius=0.1];
\draw[->, ultra thick] (1,0) -- (1,1);
\draw[->, ultra thick] (1.5,1) -- (1.5,0);
\draw[->, ultra thick] (3,1) -- (3,0);
\draw[->, ultra thick] (4,0) -- (4,1);
\draw[->, ultra thick] (7,0) -- (7,1);
\draw[->, ultra thick] (9,0) -- (9,1);
\draw[->, ultra thick] (1.25,2) -- (1.25,1);
\draw[->, ultra thick] (3.5,2) -- (3.5,1);
\draw[->, ultra thick] (5,1) -- (5,2);
\draw[->, ultra thick] (7.5,1) -- (7.5,2);
\draw[->, ultra thick] (1,3) -- (1,2);
\draw[->, ultra thick] (2.5,2) -- (2.5,3);
\draw[->, ultra thick] (5.5,2) -- (5.5,3);
\draw[->, ultra thick] (8,3) -- (8,2);
\draw[->, ultra thick] (9.5,3) -- (9.5,2);
\draw[->, ultra thick] (1.5,3) -- (1.5,4);
\draw[->, ultra thick] (4.5,3) -- (4.5,4);
\draw[->, ultra thick] (7,3) -- (7,4);
\draw[->, ultra thick] (9.25,4) -- (9.25,3);
\draw[->, ultra thick] (0.5,5) -- (0.5,4);
\draw[->, ultra thick] (3,5) -- (3,4);
\draw[->, ultra thick] (5.25,4) -- (5.25,5);
\draw[->, ultra thick] (8,5) -- (8,4);
\draw[->, ultra thick, purple] (2,1)-- (3,1) -- (3,0) --(4,0) -- (4,1)--(5,1)--(5,2)--(5.5,2)--(5.5,3)--(7,3)--(7,4)--(8.5,4);
\end{tikzpicture}
\caption{The graphical construction of the coalescing flow $\left(\boldsymbol{\Phi}_{s,t}(\cdot);s \le t\right)$.}\label{figurecoalescing}

\end{figure}

It is clear from the construction, namely from the properties of the independent Poisson processes $\{N_x^{\uparrow},N_x^{\downarrow}:x \in I \}$, that almost surely $\boldsymbol{\Phi}_{\cdot,\cdot}(\cdot)$ satisfies: $\forall u\le s \le t \in \mathbb{R}$ and $h \in \mathbb{R}$ , $\boldsymbol{\Phi}_{t,t}=Id$,  $\boldsymbol{\Phi}_{s,t}\circ \boldsymbol{\Phi}_{u,s}=\boldsymbol{\Phi}_{u,t}$, $\boldsymbol{\Phi}_{s,t}\overset{\textnormal{law}}{=}\boldsymbol{\Phi}_{s+h,t+h}$ and $\boldsymbol{\Phi}_{s,t}$ and $\boldsymbol{\Phi}_{u,s}$ are independent. Moreover, $\boldsymbol{\Phi}_{s,t}(x)$ is distributed as a $\mathcal{D}$-chain ran from time $s$ to time $t$ starting from $x$ and the joint distribution of $\left((\boldsymbol{\Phi}_{s,t}(x_1),\boldsymbol{\Phi}_{s,t}(x_2));t\ge s\right)$ is that of two independent $\mathcal{D}$-chains starting from sites $x_1$ and $x_2$ at time $s$, that coalesce when they meet, since once they are at the same site they will follow the same arrows.

Now, define the dual flow for $s\le t$ by:
\begin{align*}
\boldsymbol{\Phi}_{s,t}^*(x)=\boldsymbol{\Phi}^{-1}_{-t,-s}(x)=\sup\{w \in I:\boldsymbol{\Phi}_{-t,-s}(w) \le x \}.
\end{align*}
Note that,
\begin{align*}
\boldsymbol{\Phi}_{s,t}^*\left(\boldsymbol{\Phi}_{u,s}^*(x)\right)=\sup\{w \in I:\boldsymbol{\Phi}_{-t,-s}(w)\le\boldsymbol{\Phi}_{u,s}^*(x)\}=\sup\{w \in I: \boldsymbol{\Phi}_{-s,-u} \circ \boldsymbol{\Phi}_{-t,-s}(w) \le x\}=\boldsymbol{\Phi}_{u,t}^*(x).
\end{align*}
More generally, the fact that this again satisfies the stochastic flow properties will be implied immediately from the pathwise construction below, which also identifies the dynamics of the random maps $\{\boldsymbol{\Phi}_{s,t}^*;s \le t \}$.
 
The following statements are purely deterministic. Suppose that on each site of the lattice $x\in I$ we have a countable number, with no accumulation points, of up $\uparrow$ and down $\downarrow$ arrows arriving at (distinct) time points $\{\cdots <t_{-1}^{x,\uparrow}<t_0^{x,\uparrow}<t_1^{x,\uparrow}<t_2^{x,\uparrow}<\cdots\}$ and $\{\cdots<t_{-1}^{x,\downarrow}<t_{0}^{x,\downarrow}<t_{1}^{x,\downarrow}<t_2^{x,\downarrow}<\cdots\}$ respectively (by convention, $t_{0}^{x,\cdot}$ denotes the first arrival after time-0). Define the maps $F_{\cdot,\cdot}(\cdot)$ as before: Start at time $s$ at site $x$ and follow the direction of the arrows until time $t$. The site you are at is defined to be $F_{s,t}(x)$. As before, there is some ambiguity in this definition at the arrival times $t_{\cdot}^{x,\uparrow},t_{\cdot}^{x,\downarrow}$ of arrows and again by convention we take the right continuous (in time) version of this map. In particular, if $t_l^{x,\uparrow}$ is the first arrow after time $s$ at site $x$ then $F_{s,t}(x)=x$ for $s\le t<t_l^{x,\uparrow}$ while $F_{s,t_l^{x,\uparrow}}(x)=x+1$ and so on.

Consider $F^{-1}_{s,t}(x)=\sup\{w \in I:F_{s,t}(w) \le x \}$ and our aim is to obtain a pathwise description for this map. We introduce the following two operations on the original/black arrows to get new/$\color{red}{red}$ arrows. It is important to note the minor asymmetry (coming from our choice of $\le$ in the definition of $F^{-1}_{s,t}$) in the operations below.

$\mathbf{1}.$ An up arrow $\uparrow$ at time $t$ from site $x$ to site $x+1$, becomes a $\color{red}{red}$ down arrow $\color{red}{\downarrow}$ from site $x$ to site $x-1$ at time $t$. See Figure \ref{figureuparrows} for an illustration.

\begin{figure}[h]
\captionsetup{singlelinecheck = false, justification=justified}
\begin{tikzpicture}

\node[left ] at (0,0) {x-1};
\node[left ] at (0,1) {x};
\node[left ] at (0,2) {x+1};
\node[above ] at (1.5,2) {t};
\draw[thin] (0,0) -- (3,0);
\draw[thin] (0,1) -- (3,1);
\draw[thin] (0,2) -- (3,2);
\draw[->, ultra thick] (1.5,1) -> (1.5,2);

\draw[->, ultra thick, red] (4,1) to [out=60, in=120] (6,1);

\node[above ] at (8.5,2) {t};
\draw[thin] (7,0) -- (10,0);
\draw[thin] (7,1) -- (10,1);
\draw[thin] (7,2) -- (10,2);
\draw[->, ultra thick,red] (8.5,1) -> (8.5,0);
\node[left ] at (7,0) {x-1};
\node[left ] at (7,1) {x};
\node[left ] at (7,2) {x+1};

\end{tikzpicture}
\caption{The transformation of up arrows.}\label{figureuparrows}

\end{figure}
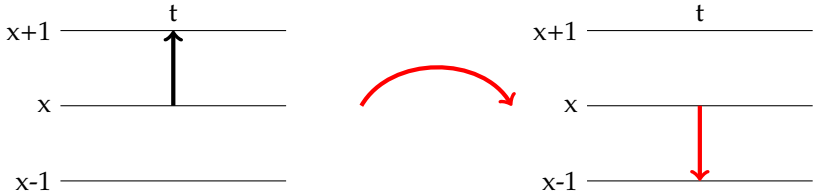

$\mathbf{2}.$ A down arrow $\downarrow$ at time $t$ from site $x+1$ to site $x$, becomes a $\color{red}{red}$ up arrow $\color{red}{\uparrow}$ from site $x$ to site $x+1$ at time $t$. See Figure \ref{figuredownarrows} for an illustration.

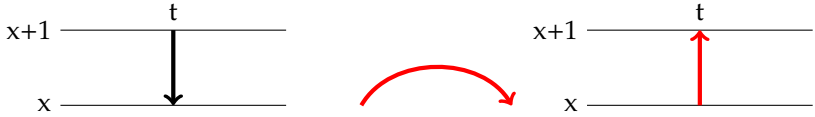
\begin{figure}[h]
\captionsetup{singlelinecheck = false, justification=justified}
\begin{tikzpicture}

\node[left ] at (0,0) {x};
\node[left ] at (0,1) {x+1};
\node[above ] at (1.5,1) {t};
\draw[thin] (0,0) -- (3,0);
\draw[thin] (0,1) -- (3,1);
\draw[->, ultra thick] (1.5,1) -> (1.5,0);

\draw[->, ultra thick, red] (4,0) to [out=60, in=120] (6,0);

\node[above ] at (8.5,1) {t};
\draw[thin] (7,0) -- (10,0);
\draw[thin] (7,1) -- (10,1);

\draw[->, ultra thick,red] (8.5,0) -> (8.5,1);
\node[left ] at (7,0) {x};
\node[left ] at (7,1) {x+1};

\end{tikzpicture}
\caption{The transformation of down arrows.}\label{figuredownarrows}

\end{figure}

Moreover, define the maps $G_{\cdot,\cdot}(\cdot)$, when evaluated at $G_{s,t}(x)$ as follows: Start at time $t$ at site $x$ and follow the direction of the $\color{red}{red}$ up and down arrows backwards until time $s$. The site you are at, is defined to be $G_{s,t}(x)$.

 We then have the following proposition, whose proof is deferred to the Appendix.

\begin{prop}\label{Pathwise}
For $x \in I $ and $s \le t$, we have $F_{s,t}^{-1}(x)=G_{s,t}(x)$.
\end{prop}

 Observe that, if the processes $N_x^{\uparrow}$ of up arrows are independent Poisson processes of rate $\lambda(x)$ and down arrows $N_x^{\downarrow}$ are of rate $\mu(x)$ then the processes of $\color{red}{red}$ arrows $\color{red}N_x^{{\uparrow}},\color{red}N_x^{{\downarrow}}$, that are followed by $\boldsymbol{\Phi}^*$,  are independent Poisson processes with rates $\mu(x+1)$ and $\lambda(x)$ respectively. Thus, this construction identifies the dual flow as that of coalescing $\hat{\mathcal{D}}$-chains ran backwards in time. In particular, this also gives a graphical proof of the Siegmund duality Lemma \ref{ConjugacyLemma}.

\begin{rmk}
It is possible, and equivalent, to consider the dual flow $\boldsymbol{\Phi}^*$ on the (dual) lattice $I\pm \frac{1}{2}$. Then, the operations performed to obtain arrows followed by this flow backwards in time become symmetric.
\end{rmk}
 
We arrive at the following proposition for the finite dimensional distributions of the coalescing flow. The result is stated for times $0$ and $t$, but by stationarity it extends to arbitrary pairs of times.
\begin{prop}\label{propflowfdd}
For $z,z'\in W^n(I)$,
\begin{align*}
\mathbb{P}\big(\boldsymbol{\Phi}_{0,t}(z_i)\le z_i' \ , \ \textnormal{for} \ 1 \le i \le n\big)=\det \big(P_t\textbf{1}_{[l,z_j']}(z_i)-\textbf{1}(i<j)\big)_{i,j=1}^n.
\end{align*}
\end{prop}
\begin{proof}
By translating the non-intersection probability found in display (3) in \cite{KarlinMcGregorCoincidence} and the paragraph following it, to our setting we get for $(y_1,\cdots,y_n)\in W^n(I)$:
\begin{align*}
\mathbb{P}\big(\boldsymbol{\Phi}_{0,t}(z_i)=y_i,\ \textnormal{for} \ 1 \le i \le n \big)=\det\left(p_t(z_i,y_j)\right)^n_{i,j=1}.
\end{align*}
This is because of the following observation: the fact that the $\boldsymbol{\Phi}_{0,t}(z_i)$ are equal to distinct points $y_i$ is equivalent to non-coalescence/non-intersection in the time interval $[0,t]$ of the underlying independent $\mathcal{D}$-chains. Then, summing over $(y_1,\cdots,y_n)$ in $\{  l \le y_1 \le z'_1, z'_1+1\le y_2 \le z'_2, \cdots,z'_{n-1}+1\le y_n \le z_n' \}$ and successively adding column $j$ to column $j+1$ we obtain,
\begin{align*}
\mathbb{P}\big(\boldsymbol{\Phi}_{0,t}(z_1)\le z_1'<\boldsymbol{\Phi}_{0,t}(z_2)\le z_2'<\cdots<\boldsymbol{\Phi}_{0,t}(z_n)\le z_n'\big)=\det\big(P_t\textbf{1}_{[l,z_j']}(z_i)\big)_{i,j}^n.
\end{align*}
The result will then follow, by writing the indicator function of the event,
\begin{align*}
\{\boldsymbol{\Phi}_{0,t}(z_1)\le z_1',\boldsymbol{\Phi}_{0,t}(z_2)\le z_2',\cdots,\boldsymbol{\Phi}_{0,t}(z_n)\le z_n'\},
\end{align*}
in terms of an expansion of indicator functions of events of the form,
\begin{align*}
\big\{\boldsymbol{\Phi}_{0,t}(z_{i_1})\le z_{j_1}'<\boldsymbol{\Phi}_{0,t}(z_{i_2})\le z_{j_2}'<\cdots<\boldsymbol{\Phi}_{0,t}(z_{i_{k}})\le z_{i_{k}}'\big\},
\end{align*}
for increasing subsequences $i_1,\cdots,i_k$ and $j_1,\cdots,j_k$. This combinatorial fact is presented in detail in Proposition 9 of \cite{Warren}, to which the reader is referred to.
\end{proof}

We now come to the key definition of the time-dependent block determinant kernel, $\mathsf{q}_t^{n,n+1}((x,y),(x',y'))$ on $W^{n,n+1}(I)$.

\begin{defn}\label{blockdetdef1}
For $(x,y),(x',y')\in W^{n,n+1}(I)$ and $t\ge0$, define $\mathsf{q}_t^{n,n+1}((x,y),(x',y'))$ by,
\begin{align*}
& \mathsf{q}_t^{n,n+1}((x,y),(x',y'))=\\
&=\frac{\prod_{i=1}^{n}\hat{\pi}(y'_i)}{\prod_{i=1}^{n}\hat{\pi}(y_i)}(-1)^n\nabla_{y_1}\cdots\nabla_{y_n}(-1)^{n+1}\bar{\nabla}_{x'_1}\cdots\bar{\nabla}_{x'_{n+1}}\mathbb{P}\big(\boldsymbol{\Phi}_{0,t}(x_i)\le x_i',\boldsymbol{\Phi}_{0,t}(y_j)\le y_j' \ \ \textnormal{for all} \ \ i,j \big).
\end{align*}
\end{defn}
Note that,
\begin{align}\label{flowtransition1}
 \mathsf{q}_t^{n,n+1}((x,y),(x',y'))=\frac{\prod_{i=1}^{n}\hat{\pi}(y'_i)}{\prod_{i=1}^{n}\hat{\pi}(y_i)}\mathbb{P}\big(\boldsymbol{\Phi}_{0,t}(x_i)=x_i',\boldsymbol{\Phi}^{*}_{-t,0}(y'_j)=y_j \ \ \text{for all} \ \ i,j \big)
\end{align}
and that, using Proposition \ref{propflowfdd}, $\mathsf{q}_t^{n,n+1}$ can be written out explicitly,
\begin{align}
\mathsf{q}_t^{n,n+1}((x,y),(x',y'))=\det\
 \begin{pmatrix}
\mathsf{A}_t(x,x') & \mathsf{B}_t(x,y')\\
  \mathsf{C}_t(y,x') & \mathsf{D}_t(y,y') 
 \end{pmatrix},
\end{align}
where, using reversibility with respect to $\hat{\pi}$,
\begin{align*}
\mathsf{A}_t(x,x')_{ij} &= p_t(x_i,x_j')=-\bar{\nabla}_{x'_j}P_t \textbf{1}_{[l,x_j']} (x_i),  \\
\mathsf{B}_t(x,y')_{ij}&=\hat{\pi}(y'_j)(P_t \textbf{1}_{[l,y'_j]}(x_i) -\textbf{1}(j\ge i)),\\
\mathsf{C}_t(y,x')_{ij}&=\hat{\pi}^{-1}(y_i)\nabla_{y_i}\bar{\nabla}_{x'_j}P_t \textbf{1}_{[l,x'_j]}(y_i),\\
\mathsf{D}_t(y,y')_{ij}&=-\frac{\hat{\pi}(y'_j)}{\hat{\pi}(y_i)}\nabla_{y_i} P_t \textbf{1}_{[l,y'_j]}(y_i)=\hat{p}_t(y_i,y_j').
\end{align*}

We define the family of operators $\left(\mathsf{Q}_t^{n,n+1};t \ge 0\right)$, acting on bounded Borel functions on $W^{n,n+1}(I)$ by,
\begin{align*}
(\mathsf{Q}_t^{n,n+1}f)(x,y)=\sum_{(x',y') \in W^{n,n+1}(I)}^{}\mathsf{q}_t^{n,n+1}((x,y),(x',y'))f(x',y').
\end{align*}

We will say that the family of operators $\left(\mathfrak{P}(t);t\ge 0\right)$ defined on bounded Borel functions on a space $\mathcal{X}$ forms a sub-Markov semigroup on $\mathcal{X}$ if the following hold:
\begin{align}\label{definitionsubmarkov}
& \mathfrak{P}(0)=Id,\nonumber\\
&\mathfrak{P}(t) 1\le 1 \ , \textnormal{ for }t \ge 0,\nonumber\\
& \mathfrak{P}(t)f \ge 0 \ , \textnormal{ for} f \ge 0,\nonumber\\
& \mathfrak{P}(t+s)=\mathfrak{P}(t)\mathfrak{P}(s),\textnormal{ for }s,t \ge 0.
\end{align}

\begin{prop}\label{propositionsubmarkov}
$\left(\mathsf{Q}_t^{n,n+1};t \ge 0\right)$ forms a sub-Markov semigroup on $W^{n,n+1}(I)$. We can thus associate to it a Markov process  $(X,Y)=\left(\left(X(t),Y(t)\right);t\ge 0\right)$, with possibly finite lifetime, with state space $ W^{n,n+1}(I)$.
\end{prop}
\begin{proof}
We proceed to check the items found in display (\ref{definitionsubmarkov}).
The initial, or \textit{time-0}, condition follows immediately from the representation (\ref{flowtransition1}). The second property, follows from performing the sum $\sum_{x'\in W^{\star,n}(y')}^{}$ and then we are left with the sum,
\begin{align*}
\sum_{y'\in W^n(I)}^{}\det(\hat{p}_t(y_i,y'_j))^{n}_{i,j}\le 1, \forall y \in W^n, t \ge 0.
\end{align*}
The quite non-trivial at first sight \textit{positivity} preserving property again follows from representation (\ref{flowtransition1}). The semigroup property for the transition kernels $\mathsf{q}_t^{n,n+1}$, can be got in the following fashion. First, by making use of the composition identity $\boldsymbol{\Phi}_{0,s+t}=\boldsymbol{\Phi}_{s,s+t}\circ\boldsymbol{\Phi}_{0,s}$, then using the independence of $\boldsymbol{\Phi}_{s,s+t}$ and $\boldsymbol{\Phi}_{0,s}$, noting that $\boldsymbol{\Phi}_{s,s+t}\overset{\textnormal{law}}{=}\boldsymbol{\Phi}_{0,t}$ and conditioning on the values of $\boldsymbol{\Phi}_{0,s}(x_i)$ and $\boldsymbol{\Phi}^{*}_{-(s+t),-s}(y''_j)$ we obtain,
\begin{align*}
 &\mathsf{q}_{s+t}^{n,n+1}((x,y),(x'',y''))=\frac{\prod_{i=1}^{n}\hat{\pi}(y''_i)}{\prod_{i=1}^{n}\hat{\pi}(y_i)}\mathbb{P}\big(\boldsymbol{\Phi}_{0,s+t}(x_i)=x_i'',\boldsymbol{\Phi}^{*}_{-(s+t),0}(y''_j)=y_j \ \ \text{for all} \ \ i,j \big) \\
 &=\frac{\prod_{i=1}^{n}\hat{\pi}(y''_i)}{\prod_{i=1}^{n}\hat{\pi}(y_i)}\sum_{(x',y')\in W^{n,n+1}(I)}^{}\mathbb{P}\big(\boldsymbol{\Phi}_{0,s}(x_i)=x_i',\boldsymbol{\Phi}_{s,s+t}(x_i')=x_i'',\boldsymbol{\Phi}^{*}_{-s,0}(y'_j)=y_j,\boldsymbol{\Phi}^{*}_{-(s+t),-s}(y''_j)=y_j' \big)
 \\ 
& =\sum_{(x',y')\in W^{n,n+1}(I)}^{} \frac{\prod_{i=1}^{n}\hat{\pi}(y'_i)}{\prod_{i=1}^{n}\hat{\pi}(y_i)}\mathbb{P}\big(\boldsymbol{\Phi}_{0,s}(x_i)=x_i',\boldsymbol{\Phi}^{*}_{-s,0}(y'_j)=y_j \big)\\
&\times \frac{\prod_{i=1}^{n}\hat{\pi}(y''_i)}{\prod_{i=1}^{n}\hat{\pi}(y'_i)}\mathbb{P}\big(\boldsymbol{\Phi}_{s,s+t}(x'_i)=x_i'',\boldsymbol{\Phi}^{*}_{-(s+t),-s}(y''_j)=y_j' \big)\\
&=\sum_{(x',y')\in W^{n,n+1}(I)}^{} \mathsf{q}_s^{n,n+1}((x,y),(x',y'))\mathsf{q}_t^{n,n+1}((x',y'),(x'',y'')).
\end{align*} 
The reason we are restricting our sum, in the second line onwards, over $(x',y') \in W^{n,n+1}(I)$ is because by the coalescing property for $(x,y)\in W^{n,n+1}(I)$ we have that almost surely $\{\boldsymbol{\Phi}_{s,t}(x_i)=x_i',\boldsymbol{\Phi}^{*}_{-t,-s}(y'_i)=y_i\}$ is empty unless $(x',y') \in W^{n,n+1}(I)$. This then, concludes the proof of the proposition.
\end{proof}

We now aim to define a family of time-dependent kernels, $\mathsf{q}_t^{n,n}((x,y),(x',y'))$ on $W^{n,n}(I)$. We again, consider in a similar fashion a (discrete) \textit{stochastic coalescing flow} $\boldsymbol{\hat{\Phi}}_{s,t}$, now consisting of coalescing $\hat{\mathcal{D}}$-chains. Now, define its dual as follows (\textbf{note well} the minor but important asymmetry to the above considerations) $\boldsymbol{\hat{\Phi}}^{*}_{s,t}(y)=\inf \{w:\boldsymbol{\hat{\Phi}}_{-t,-s}(w)\ge y \}$. As before, we have an explicit formula for its finite dimensional distributions (also by stationarity the proposition extends to arbitrary pairs of times $s \le t$).
\begin{prop}
For $z,z'\in W^n(I)$,
\begin{align*}
\mathbb{P}\big(\boldsymbol{\hat{\Phi}}_{0,t}(z_i)\ge z_i' \ , \ \textnormal{for} \ 1 \le i \le n\big)=\det \big(\hat{P}_t\textbf{1}_{[z'_j,\infty)}(z_i)-\textbf{1}(j<i)\big)_{i,j=1}^n.
\end{align*}
\end{prop}
\begin{proof}
The proof is entirely analogous to the proof of the Proposition \ref{propflowfdd} for $\boldsymbol{\Phi}$.
\end{proof}

As before, we define the following kernels:

\begin{defn}\label{blockdetdef2}
For $(x,y),(x',y')\in W^{n,n}(I)$ and $t\ge0$, define $\mathsf{q}_t^{n,n}((x,y),(x',y'))$ by,
\begin{align*}
& \mathsf{q}_t^{n,n}((x,y),(x',y'))=\\
&=\frac{\prod_{i=1}^{n}\pi(y'_i)}{\prod_{i=1}^{n}\pi(y_i)}(-1)^n\bar{\nabla}_{y_1}\cdots\bar{\nabla}_{y_n}(-1)^n\nabla_{x'_1}\cdots\nabla_{x'_{n}}\mathbb{P}\big(\boldsymbol{\hat{\Phi}}_{0,t}(x_i)\ge x_i',\boldsymbol{\hat{\Phi}}_{0,t}(y_j)\ge y_j' \ \ \textnormal{for all} \ \ i,j \big).
\end{align*}
\end{defn}
Observe that,
\begin{align}\label{flowtransition2}
 \mathsf{q}_t^{n,n}((x,y),(x',y'))=\frac{\prod_{i=1}^{n}\pi(y'_i)}{\prod_{i=1}^{n}\pi(y_i)}\mathbb{P}\big(\boldsymbol{\hat{\Phi}}_{0,t}(x_i)=x_i',\boldsymbol{\hat{\Phi}}^{*}_{-t,0}(y'_j)=y_j \ \ \text{for all} \ \ i,j \big).
\end{align}
and that $\mathsf{q}_t^{n,n}$ can be written out explicitly,
\begin{align}
\mathsf{q}_t^{n,n}((x,y),(x',y'))=\det\
 \begin{pmatrix}
\mathsf{A}_t(x,x') & \mathsf{B}_t(x,y')\\
  \mathsf{C}_t(y,x') & \mathsf{D}_t(y,y') 
 \end{pmatrix},
\end{align}
where,
\begin{align*}
\mathsf{A}_t(x,x')_{ij} &= \hat{p}_t(x_i,x_j')=-\nabla_{x'_j}\hat{P}_t \textbf{1}_{[x_j',\infty)} (x_i) , \\
\mathsf{B}_t(x,y')_{ij}&=\pi(y'_j)(\hat{P}_t \textbf{1}_{[y'_j,\infty)}(x_i) -\textbf{1}(j \le i)),\\
\mathsf{C}_t(y,x')_{ij}&=\pi^{-1}(y_i)\bar{\nabla}_{y_i}\nabla_{x'_j}\hat{P}_t \textbf{1}_{[x'_j,\infty)}(y_i),\\
\mathsf{D}_t(y,y')_{ij}&=-\frac{\pi(y'_j)}{\pi(y_i)}\bar{\nabla}_{y_i} \hat{P}_t \textbf{1}_{[y'_j,\infty)}(y_i)=p_t(y_i,y_j').
\end{align*}

Define the family of operators $\left(\mathsf{Q}_t^{n,n};t\ge0\right)$, acting on bounded Borel functions on $W^{n,n}(I)$ by,
\begin{align*}
(\mathsf{Q}_t^{n,n}f)(x,y)=\sum_{W^{n,n}(I)}^{}\mathsf{q}_t^{n,n}((x,y),(x',y'))f(x',y').
\end{align*}
Then, with completely analogous considerations as for $\left(\mathsf{Q}_t^{n,n};t\ge0\right)$, we get that:
\begin{prop}\label{propositionsubmarkov2}
$\left(\mathsf{Q}_t^{n,n};t \ge 0\right)$ forms a sub-Markov semigroup on $W^{n,n}(I)$. We can thus associate to it a Markov process  $(X,Y)=\left(\left(X(t),Y(t)\right);t\ge 0\right)$, with possibly finite lifetime, with state space $ W^{n,n}(I)$.
\end{prop}

\subsection{Intertwinings}\label{sectionintertwinins}

We first denote the Karlin-McGregor semigroup associated to $n$ $\mathcal{D}$-chains by $\left(P_t^n;t\ge 0\right)$, that is given by the following transition density, with $x,y \in W^n(I)$ and $t\ge 0$,
\begin{align*}
p^n_t(x,y)=\det(p_t(x_i,y_j))_{i,j=1}^n.
\end{align*}
Similarly, define the Karlin-McGregor semigroup $\left(\hat{P}_t^n;t\ge0\right)$ associated to $n$ $\hat{\mathcal{D}}$-chains (\textit{killed} at $-1$ if $-1$ is an absorbing boundary point) given by its transition density, with $x,y \in W^n(I)$ and $t\ge0$,
\begin{align*}
\hat{p}^n_t(x,y)=\det(\hat{p}_t(x_i,y_j))_{i,j=1}^n.
\end{align*}

Now, define the positive kernels $\Lambda_{n,\star}$ (not necessarily of finite mass in the case of $\Lambda_{n,n}$) acting on Borel functions on $W^{n,\star}(I)$, whenever $f$ is summable, by where $\star \in \{n,n+1\}$,
\begin{align*}
(\Lambda_{n,n+1}f)(x)&=\sum_{y\in W^{n,n+1}(x)}^{}\prod_{i=1}^{n}\hat{\pi}(y_i)f(x,y), \ x \in W^{n+1}(I),\\
(\Lambda_{n,n}f)(x)&=\sum_{y\in W^{n,n}(x)}^{}\prod_{i=1}^{n}\pi(y_i)f(x,y), \ x \in W^n(I).
\end{align*}
Note that $\Lambda_{n,n+1}$ involves $\hat{\pi}$ while $\Lambda_{n,n}$ involves $\pi$. Moreover, observe that we can alternatively view $\Lambda_{n,\star}$ as kernels from $W^{\star}$ to $W^{n,\star}$, assigning to each $\mathfrak{x}\in W^{\star}$ a positive measure $\Lambda_{n,\star}(\mathfrak{x},\cdot)$ on $W^{n,\star}$ supported on $\{(x,y)\in W^{n,\star}:x=\mathfrak{x}\}$. Finally, abusing notation it is obvious that we can also view $\Lambda_{n,\star}$ as kernels from $W^{\star}$ to $W^{n}$ or as operators acting on Borel functions on $W^{n}$.

Now, consider the projection operators $\Pi_{\star,n}$, acting on bounded Borel functions on $W^{\star}$, induced by the projections on the $Y$-level, with $\star \in \{n-1,n\}$,
\begin{align*}
(\Pi_{\star,n}f)(x,y)=f(y), \ (x,y)\in W^{\star,n}.
\end{align*}

\begin{prop}\label{PropInter1} For $t \ge 0$, we have the following equalities,
\begin{align}
\Pi_{n-1,n}\hat{P}_t^{n-1}&=\mathsf{Q}_t^{n-1,n}\Pi_{n-1,n},\\
\Pi_{n,n}P_t^{n}&=\mathsf{Q}_t^{n,n}\Pi_{n,n}.
\end{align}
\end{prop}
\begin{proof}
These follow directly from the probabilistic representations (\ref{flowtransition1}) and (\ref{flowtransition2}); essentially we are taking the marginal. 

Alternatively, we can take the sum $\sum_{x'\in W^{\star,n}(y')}^{}$ in the explicit form of the transition kernels and use multilinearity of the determinant. For example, in the case of $\mathsf{Q}_t^{n-1,n}$ the statement of the proposition is a consequence of the following:
\begin{align*}
\sum_{x_j'=y'_{j-1}+1}^{y'_j} \mathsf{A}_t\left(x,x'\right)_{ij}&=P_t\textbf{1}_{[l,y'_j]}(x_i)-P_t\textbf{1}_{[l,y'_j]}(x_i),\\
\sum_{x_j'=y'_{j-1}+1}^{y'_j}\mathsf{C}_{t}(y,x')_{ij} &=-\hat{\pi}^{-1}(y_i)\nabla_{y_i} P_t \textbf{1}_{[l,y'_j]}(y_i)+\hat{\pi}^{-1}(y_{i})\nabla_{y_i} P_t \textbf{1}_{[l,y'_{j-1}]}(y_i).
\end{align*}
The case of $\mathsf{Q}_t^{n,n}$ is analogous.
\end{proof}

\begin{rmk}
This, being an instance of Dynkin's criterion, has the following probabilistic interpretation. The evolution of the $Y$-level is Markovian with respect to the filtration generated by the process $(X,Y)$. In the case of $W^{n-1,n}$, $Y$ evolves as $n-1$ $\hat{\mathcal{D}}$-chains killed when they intersect or when they hit $-1$ if $-1$ is absorbing and in the case of $W^{n,n}$ it evolves as $n$ $\mathcal{D}$-chains killed when they intersect. In particular, the finite lifetime of the joint process $(X,Y)$ corresponds to the killing time of $Y$.
\end{rmk}

Moreover, the following (intermediate) intertwining relations hold.

 \begin{prop}\label{PropInter2}For $t\ge 0$, we have the equalities of positive kernels,
 \begin{align}
 P_t^{n+1}\Lambda_{n,n+1}&=\Lambda_{n,n+1}\mathsf{Q}_t^{n,n+1},\label{intermediateintertwining1}\\
 \hat{P}_t^{n}\Lambda_{n,n}&=\Lambda_{n,n}\mathsf{Q}_t^{n,n}.\label{intermediateintertwining2}
 \end{align}
 \end{prop}
\begin{proof}
 This, similarly to the Proposition above, directly follows from the probabilistic representations (\ref{flowtransition1}) and (\ref{flowtransition2}).
 
 Otherwise, we can take the sum $\sum_{y\in W^{n,\star}(x)}^{}$ using the explicit form of the transition densities and multilinearity. In particular, (\ref{intermediateintertwining1}) is a consequence of:
 \begin{align*}
\sum_{y_i=x_i}^{x_{i+1}-1}\hat{\pi}(y_i)\mathsf{C}_t(y,x')_{ij}&=\nabla_{x'_j} P_t \textbf{1}_{[l,x'_j]}(x_{i+1})-\nabla_{x'_j} P_t \textbf{1}_{[l,x'_j]}(x_{i}),\\
\sum_{y_i=x_i}^{x_{i+1}-1}\hat{\pi}(y_i)\mathsf{D}_t(y,y')_{ij}&=-\hat{\pi}(y'_j) P_t \textbf{1}_{[l,x'_j]}(x_{i+1})+ \hat{\pi}(y'_j)P_t \textbf{1}_{[l,y'_j]}(x_{i}).
 \end{align*}
 The proof of (\ref{intermediateintertwining2}) is analogous.
\end{proof}
Combining the two preceding propositions, we straightforwardly obtain the following intertwining relations for the Karlin-McGregor semigroups (where as remarked above we simply write $\Lambda_{n,\star}$ for $\Lambda_{n,\star}\Pi_{n,\star}$), for $t\ge0$,
\begin{align}
P_t^{n+1}\Lambda_{n,n+1}&=\Lambda_{n,n+1}\hat{P}_t^{n}\label{KMintertwining1},\\
\hat{P}_t^{n}\Lambda_{n,n}&=\Lambda_{n,n}P_t^{n}\label{KMintertwining2}.
\end{align}

This gives us a machine, for constructing positive eigenfunctions for these semigroups; in particular it is immediate that, with $\textbf{1}(\cdot)$ denoting the function which is constant and equal to 1 on $I$,
\begin{align}
h_{n,n+1}(\cdot)&=(\Lambda_{n,n+1}\Lambda_{n,n}\cdots \Lambda_{1,1}\textbf{1})(\cdot)\label{alternatingharmonic1},\\
h_{n,n}(\cdot)&=(\Lambda_{n,n}\Lambda_{n-1,n}\cdots \Lambda_{1,1}\textbf{1})(\cdot)\label{alternatingharmonic2},
\end{align}
are \textit{positive harmonic} functions for $P_t^{n+1}$ and $\hat{P}_t^{n}$ respectively. In the case of birth and death chains, these functions will come up in terms of the multivariate Karlin-McGregor polynomials, in relation to a general random growth process with a wall, in section \ref{multivariatepolynomialssection}.

Before proceeding, we need to make precise one more notion, referenced several times already. For a sub-Markovian semigroup $\left(\mathfrak{P}(t);t\ge0\right)$, with a strictly positive eigenfunction $\mathfrak{h}$, with eigenvalue $e^{\mathfrak{c}t}$, we define its Doob's $h$-transform,  $\left(\mathfrak{P}^{\mathfrak{h}}(t);t\ge0\right)$ by,
\begin{align*}
\left(\mathfrak{P}^{\mathfrak{h}}(t);t\ge0\right)\overset{\textnormal{def}}{=}\left(e^{-\mathfrak{c}t}\mathfrak{h}^{-1}\circ\mathfrak{P}(t)\circ \mathfrak{h};t\ge0\right),
\end{align*}
which now, a fact which can be readily checked, forms an honest Markov semigroup, $\mathfrak{P}^{\mathfrak{h}}(t)1=1$ (the definition extends to non time-dependent positive kernels).

Now, coming back to our two-level process, suppose $\hat{h}_{n}$ is a strictly positive eigenfunction for $\hat{P}_t^{n}$ namely, $\hat{P}_t^{n}\hat{h}_{n}=e^{\lambda_{n}t}\hat{h}_{n}$ then,
\begin{align*}
\left(P_t^{n+1}\Lambda_{n,n+1}\hat{h}_{n}\right)(\cdot)&=e^{\lambda_{n}t}\left(\Lambda_{n,n+1}\hat{h}_{n}\right)(\cdot),
\end{align*}
so that, $\Lambda_{n,n+1}\hat{h}_{n}$ is a strictly positive eigenfunction of $P_t^{n+1}$. Moreover, observe that if $\hat{h}_{n}$ is a positive eigenfunction for $\hat{P}_t^{n}$ then it is an eigenfunction (with the same eigenvalue) for $\mathsf{Q}_t^{n,n+1}$. We can thus define an honest Markov process, with semigroup $\left(\mathsf{Q}_t^{n,n+1,\hat{h}_{n}};t \ge 0\right)$, which is the $h$-transform of $\left(\mathsf{Q}_t^{n,n+1};t\ge 0\right)$ by $\hat{h}_{n}$. Also, define the strictly positive function $h_{n+1}(\cdot)$ by,
\begin{align*}
h_{n+1}(x)=(\Lambda_{n,n+1}\hat{h}_{n})(x), \ x\in W^{n+1}(I),
\end{align*}
and the Markov kernel $\Lambda^{\hat{h}_{n}}_{n,n+1}$ by (from the definition $h_{n+1}(x)=(\Lambda_{n,n+1}\hat{h}_{n})(x)$ it is immediate that $\Lambda^{\hat{h}_{n}}_{n,n+1}\textbf{1}=\textbf{1}$),
\begin{align*}
(\Lambda^{\hat{h}_{n}}_{n,n+1}f)(x)=\frac{1}{h_{n+1}(x)}\sum_{y\in W^{n,n+1}(x)}^{}\prod_{i=1}^{n}\hat{\pi}(y_i)\hat{h}_{n}(y)f(x,y), \ x \in W^{n+1}(I).
\end{align*}
Finally, defining $\left(P_t^{n+1,h_{n+1}};t\ge 0\right)$ to be the Karlin-McGregor semigroup $\left(P_t^{n+1};t\ge 0\right)$  that is $h$-transformed by $h_{n+1}$, we arrive at our first main result.
\begin{thm} \label{Master1} Let $\hat{h}_{n}$ be a strictly positive eigenfunction of $\hat{P}_t^{n}$, then with the notations of the paragraph above, we have the intertwining relations, for $t\ge 0$,
\begin{align} 
P_t^{n+1,h_{n+1}}\Lambda^{\hat{h}_{n}}_{n,n+1}&=\Lambda^{\hat{h}_{n}}_{n,n+1}\mathsf{Q}_t^{n,n+1,\hat{h}_{n}}\label{MasterIntertwining1eq1},\\
P_t^{n+1,h_{n+1}}\Lambda^{\hat{h}_{n}}_{n,n+1}&=\Lambda^{\hat{h}_{n}}_{n,n+1}\hat{P}_t^{n,\hat{h}_{n}}\label{MasterIntertwining1eq2}.
\end{align}
\end{thm}
\begin{proof}
These are immediate consequences of relations (\ref{intermediateintertwining1}) and (\ref{KMintertwining1}) respectively and the discussion above.
\end{proof}

Moreover, using the theorem just obtained and the Rogers and Pitman Markov functions theory (see Theorem 2 in \cite{RogersPitman} for example) we immediately get the following proposition as a corollary.

\begin{prop}\label{MarkovFunctionProposition}
Consider a Markov process $(X,Y)$ with semigroup $\left(\mathsf{Q}_t^{n,n+1,\hat{h}_{n}};t \ge 0\right)$. Then, the projection on the $X$-components evolves as a Markov process with semigroup $\left(P_t^{n+1,h_{n+1}};t \ge 0\right)$ started from $x$, if $(X,Y)$ is initialized according to $\Lambda^{\hat{h}_{n}}_{n,n+1}(x,\cdot)$. Moreover, in such case, for any fixed $T\ge 0$, the conditional distribution of $\left(X(T),Y(T)\right)$ given $X(T)$ is $\Lambda^{\hat{h}_{n}}_{n,n+1}(X(T),\cdot)$
\end{prop}
\begin{proof}
This is a straightforward application of Theorem 2 of \cite{RogersPitman}, by virtue of the intertwining relation (\ref{MasterIntertwining1eq1}) above, the Markov function $\phi$, being the projection on the $X$-component, $\phi(x,y)=x$. For the conditional distribution statement see Remark (ii) on page 575 immediately after Theorem 2 of \cite{RogersPitman}.
\end{proof}

Similarly, in the setting of having an equal number of particles for the two levels (i.e. for a process in $W^{n,n}(I)$); if $g_n$ is a positive eigenfunction of $P_t^n$ and assuming $\hat{g}_{n}(x)=(\Lambda_{n,n}g_{n})(x)$ is finite, with the analogous definitions as above, we obtain the following theorem.

\begin{thm} \label{Master2} Let $g_n$ be a strictly positive eigenfunction of $P_t^n$. Then, for $t\ge 0$,
\begin{align} 
\hat{P}_t^{n,\hat{g}_n}\Lambda^{g_{n}}_{n,n}&=\Lambda^{g_{n}}_{n,n}\mathsf{Q}_t^{n,n,g_{n}}\label{MasterIntertwining2eq1},\\
\hat{P}_t^{n,\hat{g}_n}\Lambda^{g_{n}}_{n,n}&=\Lambda^{g_{n}}_{n,n}P_t^{n,g_{n}}\label{MasterIntertwining2eq2}.
\end{align}
In particular, the projection on the $X$-components evolves as a Markov process with semigroup $\left(\hat{P}_t^{n,\hat{g}_n};t\ge0\right)$ started from $x$, if $(X,Y)$ is initialized according to $\Lambda^{g_{n}}_{n,n}(x,\cdot)$. Furthermore, for any fixed time $T\ge 0$, the conditional distribution of $\left(X(T),Y(T)\right)$ given $X(T)$ is $\Lambda^{g_{n}}_{n,n}(X(T),\cdot)$.
\end{thm}

\begin{rmk}\label{Proofremark}
We now explain the shortest path to a complete proof of the single level intertwining relations (\ref{MasterIntertwining1eq2}), (\ref{MasterIntertwining2eq2}), or more precisely to the proof of (\ref{KMintertwining1}), (\ref{KMintertwining2}). There are two essential ingredients, the Siegmund duality Lemma \ref{ConjugacyLemma} and the rather ingenious introduction of the $\mathsf{q}_t^{\bullet,\star}((x,y),(x',y'))$ transition kernels. Once, we define $\mathsf{q}_t^{\bullet,\star}((x,y),(x',y'))$ by (\ref{flowtransition1}) or (\ref{flowtransition2}), none of its probabilistic properties or the coalescing flows picture are needed. We can then proceed as in the proofs of Propositions \ref{PropInter1} and \ref{PropInter2} by taking the sums over $x'$ and $y$, \textbf{assuming these sums converge}. Of course, if  $\mathsf{q}_t^{\bullet,\star}((x,y),(x',y'))$ is positive we can make use of Tonelli's theorem to interchange the sums, however with the possibility that both sides are infinite.

We also comment on the relation to Borodin and Olshanski's approach in \cite{BorodinOlshanski} (also Cuenca's in \cite{Cuenca}). Their proof checks the intertwining relation at the multivariate infinitesimal level and then concludes by a lift to semigroups. Both of our proofs of the Siegmund duality Lemma \ref{ConjugacyLemma} in the Appendix also contain such a lift, but in the single variable setting. The introduction of the explicit coupling, equivalently of $\mathsf{q}_t^{\bullet,\star}((x,y),(x',y'))$, is what allows us to essentially check such a relation only in a single variable.
\end{rmk}

\begin{rmk}
By the methods presented above, we have identified the finite lifetime of the process $Z=(X,Y)$ as the lifetime of the autonomous component $Y$, which we have described explicitly. Moreover, under special initial conditions we have proven that the projection on the $X$-level turns out to be a Markov process as well, but the interaction between $X$ and $Y$ still remains unclear. It is natural to guess, from the locality of the coalescing flow (namely that particles only interact whence they meet) and the fact that the $Y$-level is autonomous, that the $X$-particles should be blocked and pushed, in order for the interlacing to remain. This turns out to be exactly the case and we pursue it next.
\end{rmk}

\section{Push-block dynamics}\label{sectionpushblock}
\subsection{Push-block dynamics for the two-level process}
In this subsection, we prove that the $\mathsf{q}^{n,n+1}_t$ transition matrix governs the dynamics of a continuous time, possibly with finite lifetime, Markov chain $(X,Y)$ in $W^{n,n+1}$ described informally as follows: The $Y$-level consists of $n$ independent $\hat{\mathcal{D}}$-chains and the $X$-level of $n+1$ independent $\mathcal{D}$-chains that are \textit{"pushed"} and \textit{"blocked"} by the $Y$-particles, when the process is at the \textit{boundary} (precised below) $\partial W^{n,n+1}$, in order for it to remain in $W^{n,n+1}$. The chain is \textit{killed} when two $Y$-particles collide or hit $l^*=l-1$ i.e. at the stopping time,
\begin{align*}
\mathfrak{T}_{W^{n,n+1}}=\inf\{t>0:\exists \ 1\le i<j \le n \ ,\textnormal{ such that } Y_i(t)=Y_j(t) \textnormal{ or } Y_i(t)=l^*\}.
\end{align*}
See Figures \ref{figureXYinteraction1}-\ref{figureXYinteraction4} for an illustration of the four possible types (pushing and blocking from the left and from the right) of interaction between $X$-particles and $Y$-particles in $W^{n,n+1}$.

Similarly, the $\mathsf{q}^{n,n}_t$ transition matrix governs the dynamics of a continuous time, possibly with finite lifetime, Markov chain $(X,Y)$ in $W^{n,n}$ with the following informal description: The $Y$-level consists of $n$ independent $\mathcal{D}$-chains and the $X$-level of $n$ independent $\hat{\mathcal{D}}$-chains that are \textit{"pushed"} and \textit{"blocked"} by the $Y$-particles, when the process is at $\partial W^{n,n}$, in order for it to remain in $W^{n,n}$. The chain is \textit{killed} when two $Y$-particles collide i.e. at the stopping time (note that compared to $W^{n,n+1}$, now $Y_1(t)$ never reaches $l^*=l-1$ since the $\mathcal{D}$-chain is reflecting at $l$),
\begin{align*}
\mathfrak{T}_{W^{n,n}}=\inf\{t>0:\exists \ 1\le i<j \le n \ ,\textnormal{ such that } Y_i(t)=Y_j(t) \}.
\end{align*}
See Figures \ref{figureXYinteraction5}-\ref{figureXYinteraction8} for an illustration of the possible interactions in $W^{n,n}$ and also note the asymmetry (again related to the locations of $\le$ and strict $<$ in the definitions of $W^{n,n},W^{n,n+1}$) compared to the dynamics in $W^{n,n+1}$.

\begin{figure}
\captionsetup{singlelinecheck = false, justification=justified}
\begin{tikzpicture}
\draw[dotted] (0,0) grid (3,1);
\draw[fill] (1,1) circle [radius=0.1];
\node[above right] at (1,1) {$x_{i+1}$};

\draw[fill,blue] (0,0) circle [radius=0.1];
\node[above right] at (0,0) {$y_i$};

\draw[rotate around={135:(0.5,0.5)}] (0.5,0.5) ellipse (0.6cm and 1cm);
\draw[->, very thick,] (0,0) to [out=45, in=135] (1,0);
\node[below] at (0,0) {$z$};
\node[below] at (1,0) {$z+1$};
\node[below] at (2,0) {$z+2$};
\node[below] at (3,0) {$z+3$};

\draw[->,dashed] (4,0.5) to (6,0.5);

\draw[dotted] (7,0) grid (10,1);
\draw[fill] (9,1) circle [radius=0.1];
\node[above right] at (9,1) {$x_{i+1}$};

\draw[fill] (8,0) circle [radius=0.1];
\node[above right] at (8,0) {$y_i$};

\node[below] at (7,0) {$z$};
\node[below] at (8,0) {$z+1$};
\node[below] at (9,0) {$z+2$};
\node[below] at (10,0) {$z+3$};

\end{tikzpicture}
\caption{In $W^{n,n+1}$, a jump of $y_i$ pushes (induces a simultaneous jump of) $x_{i+1}$ to the right so that the interlacing remains. Here, the jump happens with rate $\hat{\lambda}(z)=\mu(z+1)$.}\label{figureXYinteraction1}
\end{figure}
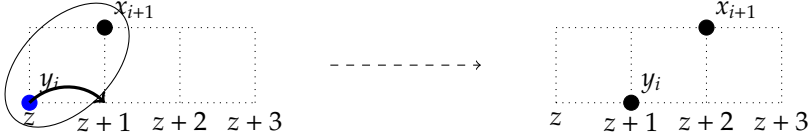

\begin{figure}
\captionsetup{singlelinecheck = false, justification=justified}
\begin{tikzpicture}
\draw[dotted] (0,0) grid (3,1);
\draw[fill,blue] (1,1) circle [radius=0.1];
\node[above right] at (1,1) {$x_{i+1}$};

\draw[fill] (0,0) circle [radius=0.1];
\node[above right] at (0,0) {$y_i$};

\draw[->, very thick] (1,1) to [out=135, in=45] (0,1);
\draw[very thick] (0.75,1) to (0.25,1.5);
\draw[very thick] (0.25,1) to (0.75,1.5);
\node[below] at (0,0) {$z$};
\node[below] at (1,0) {$z+1$};
\node[below] at (2,0) {$z+2$};
\node[below] at (3,0) {$z+3$};

\draw[->,dashed] (4,0.5) to (6,0.5);

\draw[dotted] (7,0) grid (10,1);
\draw[fill] (8,1) circle [radius=0.1];
\node[above right] at (8,1) {$x_{i+1}$};

\draw[fill] (7,0) circle [radius=0.1];
\node[above right] at (7,0) {$y_i$};

\node[below] at (7,0) {$z$};
\node[below] at (8,0) {$z+1$};
\node[below] at (9,0) {$z+2$};
\node[below] at (10,0) {$z+3$};

\end{tikzpicture}
\caption{In $W^{n,n+1}$, a jump of $x_{i+1}$ to the left is blocked by $y_{i}$ so that the interlacing remains. Here, the clock of $x_{i+1}$ rings with rate $\mu(z+1)$.}\label{figureXYinteraction2}
\end{figure}
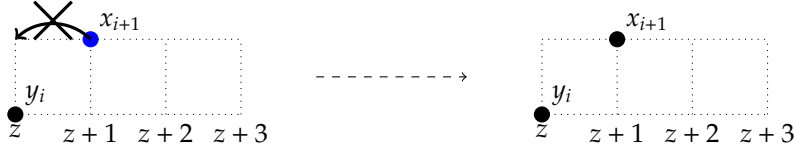

\begin{figure}
\captionsetup{singlelinecheck = false, justification=justified}
\begin{tikzpicture}
\draw[dotted] (0,0) grid (3,1);
\draw[fill] (1,1) circle [radius=0.1];
\node[above left] at (1,1) {$x_{i}$};

\draw[fill,blue] (1,0) circle [radius=0.1];
\node[above right] at (1,0) {$y_i$};

\draw[->, very thick] (1,0) to [out=135, in=45] (0,0);
\draw (1,0.5) ellipse (0.5cm and 1cm);
\node[below] at (0,0) {$z-1$};
\node[below] at (1,0) {$z$};
\node[below] at (2,0) {$z+1$};
\node[below] at (3,0) {$z+2$};

\draw[->,dashed] (4,0.5) to (6,0.5);

\draw[dotted] (7,0) grid (10,1);
\draw[fill] (7,1) circle [radius=0.1];
\node[above right] at (7,1) {$x_{i}$};

\draw[fill] (7,0) circle [radius=0.1];
\node[above right] at (7,0) {$y_i$};

\node[below] at (7,0) {$z-1$};
\node[below] at (8,0) {$z$};
\node[below] at (9,0) {$z+1$};
\node[below] at (10,0) {$z+2$};

\end{tikzpicture}
\caption{In $W^{n,n+1}$, a jump of $y_i$ pushes (induces a simultaneous jump of) $x_{i+1}$ to the left so that the interlacing remains. Here, the jump happens with rate $\hat{\mu}(z)=\lambda(z)$.}\label{figureXYinteraction3}
\end{figure}
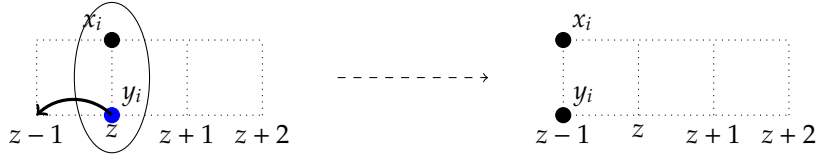

\begin{figure}
\captionsetup{singlelinecheck = false, justification=justified}
\begin{tikzpicture}
\draw[dotted] (0,0) grid (3,1);
\draw[fill,blue] (0,1) circle [radius=0.1];
\node[above left] at (0,1) {$x_{i}$};

\draw[fill] (0,0) circle [radius=0.1];
\node[above right] at (0,0) {$y_i$};

\draw[->, very thick] (0,1) to [out=45, in=135] (1,1);
\draw[very thick] (0.75,1) to (0.25,1.5);
\draw[very thick] (0.25,1) to (0.75,1.5);
\node[below] at (0,0) {$z$};
\node[below] at (1,0) {$z+1$};
\node[below] at (2,0) {$z+2$};
\node[below] at (3,0) {$z+3$};

\draw[->,dashed] (4,0.5) to (6,0.5);

\draw[dotted] (7,0) grid (10,1);
\draw[fill] (7,1) circle [radius=0.1];
\node[above right] at (7,1) {$x_{i}$};

\draw[fill] (7,0) circle [radius=0.1];
\node[above right] at (7,0) {$y_i$};

\node[below] at (7,0) {$z$};
\node[below] at (8,0) {$z+1$};
\node[below] at (9,0) {$z+2$};
\node[below] at (10,0) {$z+3$};

\end{tikzpicture}

\caption{In $W^{n,n+1}$, a jump of $x_{i}$ to the right is blocked by $y_{i}$ so that the interlacing remains. Here, the clock of $x_{i}$ rings with rate $\lambda(z)$.}\label{figureXYinteraction4}
\end{figure}
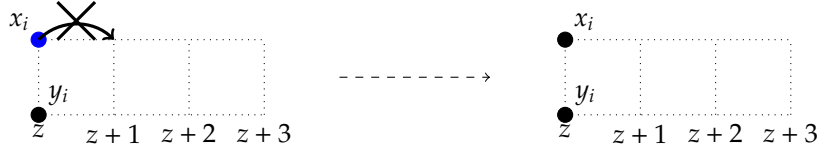

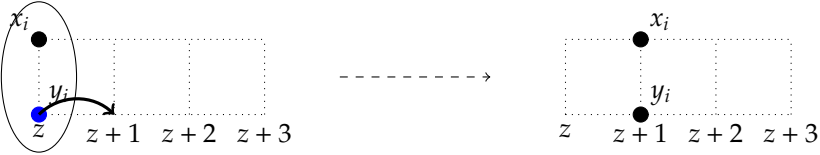
\begin{figure}
\captionsetup{singlelinecheck = false, justification=justified}
\begin{tikzpicture}
\draw[dotted] (0,0) grid (3,1);
\draw[fill] (0,1) circle [radius=0.1];
\node[above left] at (0,1) {$x_{i}$};

\draw[fill,blue] (0,0) circle [radius=0.1];
\node[above right] at (0,0) {$y_i$};

\draw[->, very thick] (0,0) to [out=45, in=135] (1,0);
\draw (0,0.5) ellipse (0.5cm and 1cm);
\node[below] at (0,0) {$z$};
\node[below] at (1,0) {$z+1$};
\node[below] at (2,0) {$z+2$};
\node[below] at (3,0) {$z+3$};

\draw[->,dashed] (4,0.5) to (6,0.5);

\draw[dotted] (7,0) grid (10,1);
\draw[fill] (8,1) circle [radius=0.1];
\node[above right] at (8,1) {$x_{i}$};

\draw[fill] (8,0) circle [radius=0.1];
\node[above right] at (8,0) {$y_i$};

\node[below] at (7,0) {$z$};
\node[below] at (8,0) {$z+1$};
\node[below] at (9,0) {$z+2$};
\node[below] at (10,0) {$z+3$};

\end{tikzpicture}
\caption{In $W^{n,n}$, a jump of $y_i$ pushes (induces a simultaneous jump of) $x_{i}$ to the right so that the interlacing remains. Here, the jump happens with rate $\lambda(z)$.}\label{figureXYinteraction5}
\end{figure}

\begin{figure}
\captionsetup{singlelinecheck = false, justification=justified}
\begin{tikzpicture}
\draw[dotted] (0,0) grid (3,1);
\draw[fill,blue] (1,1) circle [radius=0.1];
\node[above right] at (1,1) {$x_{i}$};

\draw[fill] (1,0) circle [radius=0.1];
\node[above right] at (1,0) {$y_i$};

\draw[->, very thick] (1,1) to [out=135, in=45] (0,1);
\draw[very thick] (0.75,1) to (0.25,1.5);
\draw[very thick] (0.25,1) to (0.75,1.5);
\node[below] at (0,0) {$z-1$};
\node[below] at (1,0) {$z$};
\node[below] at (2,0) {$z+1$};
\node[below] at (3,0) {$z+2$};

\draw[->,dashed] (4,0.5) to (6,0.5);

\draw[dotted] (7,0) grid (10,1);
\draw[fill] (8,1) circle [radius=0.1];
\node[above right] at (8,1) {$x_{i}$};

\draw[fill] (8,0) circle [radius=0.1];
\node[above right] at (8,0) {$y_i$};

\node[below] at (7,0) {$z-1$};
\node[below] at (8,0) {$z$};
\node[below] at (9,0) {$z+1$};
\node[below] at (10,0) {$z+2$};

\end{tikzpicture}
\caption{In $W^{n,n}$, a jump of $x_{i}$ to the left is blocked by $y_{i}$ so that the interlacing remains. Here, the clock of $x_{i}$ rings with rate $\hat{\mu}(z)=\lambda(z)$.}\label{figureXYinteraction6}
\end{figure}
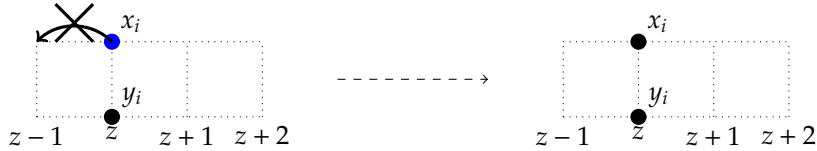

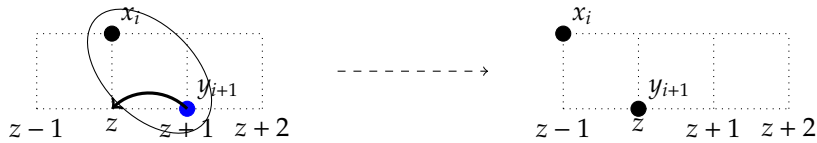
\begin{figure}
\captionsetup{singlelinecheck = false, justification=justified}
\begin{tikzpicture}
\draw[dotted] (0,0) grid (3,1);
\draw[fill] (1,1) circle [radius=0.1];
\node[above right] at (1,1) {$x_{i}$};

\draw[fill,blue] (2,0) circle [radius=0.1];
\node[above right] at (2,0) {$y_{i+1}$};

\draw[rotate around={45:(1.5,0.5)}] (1.5,0.5) ellipse (0.6cm and 1cm);
\draw[->, very thick,] (2,0) to [out=135, in=45] (1,0);
\node[below] at (0,0) {$z-1$};
\node[below] at (1,0) {$z$};
\node[below] at (2,0) {$z+1$};
\node[below] at (3,0) {$z+2$};

\draw[->,dashed] (4,0.5) to (6,0.5);

\draw[dotted] (7,0) grid (10,1);
\draw[fill] (7,1) circle [radius=0.1];
\node[above right] at (7,1) {$x_{i}$};

\draw[fill] (8,0) circle [radius=0.1];
\node[above right] at (8,0) {$y_{i+1}$};

\node[below] at (7,0) {$z-1$};
\node[below] at (8,0) {$z$};
\node[below] at (9,0) {$z+1$};
\node[below] at (10,0) {$z+2$};

\end{tikzpicture}
\caption{In $W^{n,n}$, a jump of $y_{i+1}$ pushes (induces a simultaneous jump of) $x_{i}$ to the left so that the interlacing remains. Here, the jump happens with rate $\mu(z+1)$.}\label{figureXYinteraction7}
\end{figure}

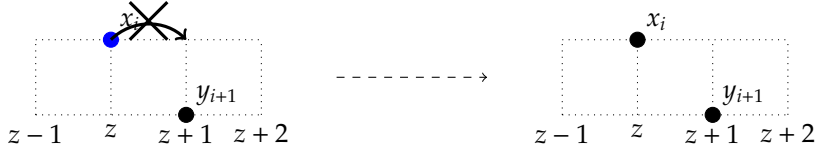
\begin{figure}
\captionsetup{singlelinecheck = false, justification=justified}
\begin{tikzpicture}
\draw[dotted] (0,0) grid (3,1);
\draw[fill,blue] (1,1) circle [radius=0.1];
\node[above right] at (1,1) {$x_{i}$};

\draw[fill] (2,0) circle [radius=0.1];
\node[above right] at (2,0) {$y_{i+1}$};

\draw[->, very thick,] (1,1) to [out=45, in=135] (2,1);
\draw[very thick] (1.75,1) to (1.25,1.5);
\draw[very thick] (1.25,1) to (1.75,1.5);
\node[below] at (0,0) {$z-1$};
\node[below] at (1,0) {$z$};
\node[below] at (2,0) {$z+1$};
\node[below] at (3,0) {$z+2$};

\draw[->,dashed] (4,0.5) to (6,0.5);

\draw[dotted] (7,0) grid (10,1);
\draw[fill] (8,1) circle [radius=0.1];
\node[above right] at (8,1) {$x_{i}$};

\draw[fill] (9,0) circle [radius=0.1];
\node[above right] at (9,0) {$y_{i+1}$};

\node[below] at (7,0) {$z-1$};
\node[below] at (8,0) {$z$};
\node[below] at (9,0) {$z+1$};
\node[below] at (10,0) {$z+2$};

\end{tikzpicture}
\caption{In $W^{n,n}$, a jump of $x_{i}$ to the right is blocked by $y_{i+1}$ so that the interlacing remains. Here, the clock of $x_{i}$ rings with rate $\hat{\lambda}(z)=\mu(z+1)$.}\label{figureXYinteraction8}
\end{figure}

We will only consider the dynamics in $W^{n,n+1}$ in detail, as the case of $W^{n,n}$ is entirely analogous (but see Remark \ref{equalnumberrates} below for a discussion). We define the \textit{boundary of $W^{n,n+1}$} denoted by $\partial W^{n,n+1}$, as follows,
\begin{align*}
\partial W^{n,n+1}=\{(x,y)\in W^{n,n+1}: \exists 1 \le i \le n+1 ,\textnormal{ such that with } x'_i=x_i\pm 1 \textnormal{ then }(x',y) \notin W^{n,n+1} \}.
\end{align*}
Also, define the \textit{interior of $\mathring{W}^{n,n+1}$} by $\mathring{W}^{n,n+1}=W^{n,n+1}\backslash \partial W^{n,n+1}$. Finally, define the following indexing sets, $I^{n,n+1,+}_{adm}(x,y)$ and $I^{n,n+1,-}_{adm}(x,y)$ for $(x,y)\in W^{n,n+1}$ ("adm" stands for admissible jump),
\begin{align*}
I^{n,n+1,+}_{adm}(x,y)&=\{1 \le i \le n+1: (x',y)\in W^{n,n+1} \textnormal{ with } x_i'=x_i+1 \},\\
I^{n,n+1,-}_{adm}(x,y)&=\{1 \le i \le n+1: (x',y)\in W^{n,n+1} \textnormal{ with } x_i'=x_i-1 \}.
\end{align*}

We begin, by observing that we have the following time-$0$ initial condition,
\begin{align}
\mathsf{q}_0((x,y),(x',y'))=\delta_{(x,y),(x',y')}.
\end{align}
This follows directly from the form of $\mathsf{q}_t((x,y),(x',y'))$, by noting that as $t \downarrow 0$, the diagonal entries converge to $\delta_{x_i,x_i'},\delta_{y_i,y_i'}$, while all other contributions to the determinant vanish (or see proof of Proposition \ref*{propositionsubmarkov}).

Moreover, note that the entries of each matrix in the block determinant $\mathsf{q}^{n,n+1}_t$ namely $\mathsf{A}_t(x,x')$, $\mathsf{B}_t(x,x')$, $\mathsf{C}_t(x,x')$, $\mathsf{D}_t(x,x')$ (we are abusing notation slightly by using the same notation for both the matrices and their scalar entries) solve the following differential equations in the backwards variable $x$, for any $x,x' \in I$ fixed and $t>0$,
\begin{align}
\frac{d}{dt}\mathsf{A}_t(x,x')&=\mathcal{D}_x\mathsf{A}_t(x,x') \label{diffA},\\
\frac{d}{dt}\mathsf{B}_t(x,x')&=\mathcal{D}_x\mathsf{B}_t(x,x') \label{diffB},\\
\frac{d}{dt}\mathsf{C}_t(x,x')&=\hat{\mathcal{D}}_x\mathsf{C}_t(x,x')\label{diffC},\\
\frac{d}{dt}\mathsf{D}_t(x,x')&=\hat{\mathcal{D}}_x\mathsf{D}_t(x,x')\label{diffD}.
\end{align}
Observe that the differential equation (\ref{diffC}) for $\mathsf{C}_t$ follows from the Siegmund duality Lemma \ref{ConjugacyLemma} and reversibility with respect to $\hat{\pi}$ of the $\hat{\mathcal{D}}$-chain.

Now, we consider the discrete generator  $\mathfrak{D}^{n,n+1}$, the matrix that gives the rates of the push-block dynamics in $W^{n,n+1}$ (see Figures \ref{figureXYinteraction1}-\ref{figureXYinteraction4} to help visualize the meaning of these rates; also see Remark \ref{equalnumberrates} below for the rates in $W^{n,n}$),
\begin{align*}
\mathfrak{D}^{n,n+1}((x,y),(x',y'))=\begin{cases}
\lambda(x_i)  & x_i'=x_i+1 \textnormal{ and } i \in I^{n,n+1,+}_{adm}(x,y)\\
 \mu(x_i)  & x_i'=x_i-1 \textnormal{ and } i \in I^{n,n+1,-}_{adm}(x,y)  \\
 \hat{\lambda}(y_i)=\mu(y_i+1)  & y_i'=y_i+1 \textnormal{ and } i+1 \in I^{n,n+1,-}_{adm}(x,y) \\
 \hat{\mu}(y_i)=\lambda(y_i)  & y_i'=y_i-1 \textnormal{ and } i \in I^{n,n+1,+}_{adm}(x,y)\\
 \hat{\lambda}(y_i)=\mu(y_i+1) & (x_{i+1},y_i)=(x+1,x), (x'_{i+1},y'_i)=(x+2,x+1)  \\
 \hat{\mu}(y_i)=\lambda(y_i)  & (x_{i},y_i)=(x,x), (x'_{i},y'_i)=(x-1,x-1) \\
 S^{n,n+1}_{(x,y)} & (x',y')=(x,y)\\
  0 & \textnormal{ otherwise }
\end{cases},
\end{align*}
where $S^{n,n+1}_{(x,y)}$ is given by,
\begin{align*}
S^{n,n+1}_{(x,y)}=-\sum_{i \in I^{n,n+1,+}_{adm}(x,y)}^{}\lambda(x_i)-\sum_{i \in I^{n,n+1,-}_{adm}(x,y)}^{}\mu(x_i)-\sum_{i=1}^{n}\left[\hat{\lambda}(y_i)+\hat{\mu}(y_i)\right].
\end{align*}
Observe that, there is a non-zero rate for the transition $(x,y) \in W^{n,n+1} \to (x',y') \notin W^{n,n+1}$, which corresponds to the chain being killed (in the sequel we will identify all such configurations with a cemetery/absorbing state $\dagger$); this of course coincides with the rate of $y \in W^n(I) \to y'\notin W^n(I)$, which is non-zero only for $y \in \partial W^n(I)$ and is furthermore given by,
\begin{align*}
k_{(x,y)}^{n,n+1}=\sum_{i=1}^{n-1}\textbf{1}\left(y_i+1=y_{i+1}\right)\left[\hat{\lambda}(y_i)+\hat{\mu}(y_i+1)\right]+\textbf{1}\left(y_1=l\right)\hat{\mu}(l).
\end{align*}
Moreover, note that the first four conditions, given in terms of the indexing sets $I_{adm}^{n,n+1,+}$ and $I_{adm}^{n,n+1,-}$, in $\mathfrak{D}^{n,n+1}$ above could have been replaced by, $(x',y)\in W^{n,n+1}$ and $(x,y')\in W^{n,n+1}$ respectively. Also, observe that in the definition of $\mathfrak{D}^{n,n+1}$ the first two rates correspond to the free evolution of the $X$-particles as $\mathcal{D}$-chains, the next two to the evolution of the $Y$-particles as $\hat{\mathcal{D}}$-chains and the last two to the pushing mechanism (obviously, blocking corresponds to the $0$ rate).

\begin{lem}
Then, $\mathsf{q}_t^{n,n+1}$ solves the (backwards) differential equation, for $(x,y),(x',y')\in W^{n,n+1}$ and $t>0$:
\begin{align*}
\frac{d}{dt}\mathsf{q}^{n,n+1}_t((x,y),(x',y'))=(\mathfrak{D}^{n,n+1}\mathsf{q}^{n,n+1}_t)((x,y),(x',y')).
\end{align*}
\end{lem}

\begin{proof}
For $(x,y)\in \mathring{W}^{n,n+1}$, the claim follows immediately from (\ref{diffA}), (\ref{diffB}), (\ref{diffC}), (\ref{diffD}) and the multilinearity of the determinant. We will hence, now concentrate on the case of $(x,y)\in \partial W^{n,n+1}$. We will only consider the case $x_1=y_1=x$, as all others are completely analogous. Moreover, in order to ease notation and make the gist of the simple argument clear we will further restrict our attention to the rows containing $x_1,y_1$ and in fact it is easy to see that it suffices to consider the $2\times 2$ matrix given by, with $x',y'\in I$ fixed,
\begin{align*}
\det\
 \begin{pmatrix}
\mathsf{A}_t(x,x') & \mathsf{B}_t(x,y')\\
  \mathsf{C}_t(x,x') & \mathsf{D}_t(x,y') 
 \end{pmatrix}.
\end{align*}

By taking the $\frac{d}{dt}$-differential of the determinant, we easily see from the differential equations (\ref{diffA}), (\ref{diffB}), (\ref{diffC}), (\ref{diffD}) that we get,

\begin{align*}
\frac{d}{dt}\det\
 \begin{pmatrix}
\mathsf{A}_t(x,x') & \mathsf{B}_t(x,y')\\
  \mathsf{C}_t(x,x') & \mathsf{D}_t(x,y') 
 \end{pmatrix}=\lambda(x)\left[\det\
  \begin{pmatrix}
 \mathsf{A}_t(x+1,x') & \mathsf{B}_t(x+1,y')\\
   \mathsf{C}_t(x,x') & \mathsf{D}_t(x,y') 
  \end{pmatrix}-\det\
   \begin{pmatrix}
 \mathsf{A}_t(x,x') & \mathsf{B}_t(x,y')\\
    \mathsf{C}_t(x,x') & \mathsf{D}_t(x,y') 
   \end{pmatrix}\right]\\
 +\mu(x)\left[\det\
   \begin{pmatrix}
  \mathsf{A}_t(x-1,x') & \mathsf{B}_t(x-1,y')\\
    \mathsf{C}_t(x,x') & \mathsf{D}_t(x,y') 
   \end{pmatrix}-\det\
    \begin{pmatrix}
   \mathsf{A}_t(x,x') & \mathsf{B}_t(x,y')\\
     \mathsf{C}_t(x,x') & \mathsf{D}_t(x,y') 
    \end{pmatrix}\right]\\
    \mu(x+1)\left[\det\
      \begin{pmatrix}
     \mathsf{A}_t(x,x') & \mathsf{B}_t(x,y')\\
       \mathsf{C}_t(x+1,x') & \mathsf{D}_t(x+1,y') 
      \end{pmatrix}-\det\
       \begin{pmatrix}
     \mathsf{A}_t(x,x') & \mathsf{B}_t(x,y')\\
        \mathsf{C}_t(x,x') & \mathsf{D}_t(x,y') 
       \end{pmatrix}\right]\\
       \lambda(x)\left[\det\
         \begin{pmatrix}
        \mathsf{A}_t(x,x') & \mathsf{B}_t(x,y')\\
          \mathsf{C}_t(x-1,x') & \mathsf{D}_t(x-1,y') 
         \end{pmatrix}-\det\
          \begin{pmatrix}
         \mathsf{A}_t(x,x') & \mathsf{B}_t(x,y')\\
           \mathsf{C}_t(x,x') & \mathsf{D}_t(x,y') 
          \end{pmatrix}\right].
\end{align*}
On the other hand, what we would like to have, according to the rates of $\mathfrak{D}^{n,n+1}$, is the following,
\begin{align*}
\frac{d}{dt}\det\
 \begin{pmatrix}
\mathsf{A}_t(x,x') & \mathsf{B}_t(x,y')\\
  \mathsf{C}_t(x,x') & \mathsf{D}_t(x,y') 
 \end{pmatrix}=\mu(x)\left[\det\
   \begin{pmatrix}
  \mathsf{A}_t(x-1,x') & \mathsf{B}_t(x-1,y')\\
    \mathsf{C}_t(x,x') & \mathsf{D}_t(x,y') 
   \end{pmatrix}-\det\
    \begin{pmatrix}
   \mathsf{A}_t(x,x') & \mathsf{B}_t(x,y')\\
     \mathsf{C}_t(x,x') & \mathsf{D}_t(x,y') 
    \end{pmatrix}\right]\\
    \mu(x+1)\left[\det\
      \begin{pmatrix}
     \mathsf{A}_t(x,x') & \mathsf{B}_t(x,y')\\
       \mathsf{C}_t(x+1,x') & \mathsf{D}_t(x+1,y') 
      \end{pmatrix}-\det\
       \begin{pmatrix}
      \mathsf{A}_t(x,x') & \mathsf{B}_t(x,y')\\
        \mathsf{C}_t(x,x') & \mathsf{D}_t(x,y') 
       \end{pmatrix}\right]\\
       \lambda(x)\left[\det\
         \begin{pmatrix}
        \mathsf{A}_t(x-1,x') & \mathsf{B}_t(x-1,y')\\
          \mathsf{C}_t(x-1,x') & \mathsf{D}_t(x-1,y') 
         \end{pmatrix}-\det\
          \begin{pmatrix}
        \mathsf{A}_t(x,x') & \mathsf{B}_t(x,y')\\
           \mathsf{C}_t(x,x') & \mathsf{D}_t(x,y') 
          \end{pmatrix}\right].
\end{align*}
We are thus, required to show that,
\begin{align}\label{equality1}
\det\
  \begin{pmatrix}
 \mathsf{A}_t(x+1,x') & \mathsf{B}_t(x+1,y')\\
   \mathsf{C}_t(x,x') & \mathsf{D}_t(x,y') 
  \end{pmatrix}=\det\
   \begin{pmatrix}
  \mathsf{A}_t(x,x') & \mathsf{B}_t(x,y')\\
    \mathsf{C}_t(x,x') & \mathsf{D}_t(x,y') 
   \end{pmatrix},
\end{align}
which corresponds to $x_1$ being blocked when $x_1=y_1$ and $x_1$ tries to jump to the right (see the configuration in Figure \ref{figureXYinteraction4}) and also,
\begin{align} \label{equality2}
\det\
  \begin{pmatrix}
 \mathsf{A}_t(x,x') & \mathsf{B}_t(x,y')\\
   \mathsf{C}_t(x-1,x') & \mathsf{D}_t(x-1,y') 
  \end{pmatrix}=\det\
   \begin{pmatrix}
  \mathsf{A}_t(x-1,x') & \mathsf{B}_t(x-1,y')\\
    \mathsf{C}_t(x-1,x') & \mathsf{D}_t(x-1,y') 
   \end{pmatrix},
\end{align}
which corresponds to $x_1$ being pushed to the left when $x_1=y_1$ and $y_1$ jumps to the left (see the configuration in Figure \ref{figureXYinteraction3}). Observe that, this latter equality in display (\ref{equality2}) is the same as the one above in display (\ref{equality1}), after replacing $x$ with $x-1$. Both of these equalities follow from simple row and column operations. First recall,
\begin{align*}
\mathsf{A}_t(x,x')&=p_t(x,x')=-\bar{\nabla}_{x'}P_t\textbf{1}_{[l,x']}(x),\\
\mathsf{B}_t(x,y')&=\hat{\pi}(y')\left(P_t\textbf{1}_{[l,y']}(x)-1\right),\\
\mathsf{C}_t(y,x')&=\hat{\pi}^{-1}(y)\nabla_{y}\bar{\nabla}_{x'}P_t\textbf{1}_{[l,x']}(y),\\
\mathsf{D}_t(y,y')&=-\frac{\hat{\pi}(y')}{\hat{\pi}(y)}\nabla_yP_t\textbf{1}_{[l,y']}(y)=\hat{p}_t(y,y').
\end{align*}
In order to obtain (\ref{equality1}) and hence (\ref{equality2}) as well, we work on the RHS and we multiply the second row by $-\hat{\pi}(x)$ and add it to the first row to obtain,
\begin{align*}
\mathsf{A}_t(x,x')-\hat{\pi}(x)\mathsf{C}_t(x,x')&=-\bar{\nabla}_{x'}P_t\textbf{1}_{[l,x']}(x)-\bar{\nabla}_{x'}P_t\textbf{1}_{[l,x']}(x+1)+\bar{\nabla}_{x'}P_t\textbf{1}_{[l,x']}(x)=-\bar{\nabla}_{x'}P_t\textbf{1}_{[l,x']}(x+1)\\
&=\mathsf{A}_t(x+1,x'),
\end{align*}
and similarly for the second column, which then gives us the LHS of  (\ref{equality1}).
\end{proof}

We now add a \textit{cemetery state} $\dagger$ to the state space and to (the transition matrix) $\mathsf{q}^{n,n+1}_t$, to make it an honest (i.e. stochastic) transition matrix, denoted by $\tilde{\mathsf{q}}^{n,n+1}_t$. This corresponds to the process with infinite lifetime, that instead of being killed, gets absorbed at $\dagger$ and stays there forever. Observe that, $\dagger=\{(x,y):y \notin W^n(I) \}$ and so $\tilde{\mathsf{q}}^{n,n+1}_t$ and $\tilde{\mathfrak{D}}^{n,n+1}$ are given by,
\begin{align*}
\tilde{\mathsf{q}}^{n,n+1}_t(z,w)&=\mathsf{q}^{n,n+1}_t(z,w) ,\textnormal{  \ for \ }z,w \ne \dagger,\\
\tilde{\mathsf{q}}^{n,n+1}_t(\dagger,w)&=\delta_{\dagger,w},\\
\tilde{\mathsf{q}}^{n,n+1}_t(z,\dagger)&=1-\sum_{w}^{}\mathsf{q}^{n,n+1}_t(z,w) 
\end{align*}
and,
\begin{align*}
\tilde{\mathfrak{D}}^{n,n+1}(z,w)&=\mathfrak{D}^{n,n+1}(z,w) ,\textnormal{  \ for \ }z,w \ne \dagger,\\
\tilde{\mathfrak{D}}^{n,n+1}(\dagger,w)&=0, \ w \ne \dagger, \\
\tilde{\mathfrak{D}}^{n,n+1}(z,\dagger)&=\textnormal{ rate of transition: } y\in W^n(I) \to y' \notin W^n(I), \textnormal{ for } z=(x,y)\\
&=k_{(x,y)}^{n,n+1}=\sum_{i=1}^{n-1}\textbf{1}\left(y_i+1=y_{i+1}\right)\left[\hat{\lambda}(y_i)+\hat{\mu}(y_i+1)\right]+\textbf{1}\left(y_1=l\right)\hat{\mu}(l).
\end{align*}
Then, from our previous considerations we get: 

\begin{prop}
For fixed $z,w \in W^{n,n+1}\cup \dagger$ we have for $t>0$,
\begin{align}\label{BackwardsEquationCemeteryState}
\frac{d}{dt}\tilde{\mathsf{q}}^{n,n+1}_t(z,w)=(\tilde{\mathfrak{D}}^{n,n+1}\tilde{\mathsf{q}}^{n,n+1}_t)(z,w).
\end{align}
Moreover, $\tilde{\mathsf{q}}^{n,n+1}_0=Id$ and also for $t\ge0, \tilde{\mathsf{q}}^{n,n+1}_t$ is positive.
\end{prop}

We proceed to prove uniqueness of solutions:

\begin{prop}\label{UniquenessBackwardsWithCemetery}
The solution to the backwards equation (\ref{BackwardsEquationCemeteryState}) is unique.
\end{prop}

\begin{proof}
Following \cite{BorodinOlshanski} we write $\tilde{\mathfrak{D}}^{n,n+1}=-\textnormal{diag}(\tilde{\mathfrak{D}}^{n,n+1})+\bar{\mathfrak{D}}^{n,n+1}$ where $\textnormal{diag}(\tilde{\mathfrak{D}}^{n,n+1})(z,w)=-\tilde{\mathfrak{D}}^{n,n+1}(z,w)\textbf{1}_{zw}$ and $\bar{\mathfrak{D}}^{n,n+1}(z,w)=\hat{\mathfrak{D}}^{n,n+1}(z,w)$ if $z \ne w$ and $0$ otherwise. We define the following recursion $\big\{\left(\mathcal{P}^{(k)}(t);t\ge 0\right)\big \}_{k\ge 1}$, of operators (matrices) by, for $t\ge0$,
\begin{align*}
\mathcal{P}^{(0)}(t)&=e^{-\textnormal{diag}(\tilde{\mathfrak{D}}^{n,n+1})t},\\
\mathcal{P}^{(k)}(t)&=\int_{0}^{t}e^{-\textnormal{diag}(\tilde{\mathfrak{D}}^{n,n+1})s}\bar{\mathfrak{D}}^{n,n+1}\mathcal{P}^{(k-1)}(t-s)ds
\end{align*}
and also let $\left(\tilde{\mathcal{P}}(t);t\ge 0\right)$ be given by, for $t \ge 0$,
\begin{align*}
\tilde{\mathcal{P}}(t)=\sum_{k=0}^{\infty}\mathcal{P}^{(k)}(t).
\end{align*}
Then (see Theorem 4.1, Corollary 4.2 of \cite{BorodinOlshanski}), $\left(\tilde{\mathcal{P}}(t);t\ge 0\right)$ is the \textit{minimal} solution of the backwards equation,  $\frac{d}{dt}S(t)=\tilde{\mathfrak{D}}^{n,n+1}S(t)$ for $t>0$ and $S(0)=Id$ and if it is \textit{stochastic} then, it is the \textit{unique} one. So, in such a case it must necessarily coincide with $\tilde{\mathsf{q}}_t^{n,n+1}$. 

By Proposition 4.3 of \cite{BorodinOlshanski}, in order to show that the minimal solution is indeed stochastic it suffices to prove that for $w\in W^{n,n+1}$, we have $\mathbb{P}_w\left((X(t),Y(t))\notin w+[-N,N]^{2n+1}\right)\to 0$ as $N\to \infty$, for fixed $t\ge 0$.

Note that,
\begin{align*}
\mathbb{P}_w\left((X(t),Y(t))\notin w+[-N,N]^{2n+1}\right)\le 2(n+1) \max \{\mathbb{P}_w\left(X_{n+1}(t)> x_{n+1}+N\right),\\ \mathbb{P}_w\left(X_1(t)<x_1-N\right)\}.
\end{align*}

So it suffices to show that the probabilities on the right hand side go to $0$ as $N \to \infty$ and since both cases are completely similar, we will show that,
\begin{align*}
\mathbb{P}\left(X_{n+1}(t)> x_{n+1}+N\right)
\end{align*}
vanishes as $N \to \infty$. This is intuitively obvious, since away from $(Y_n(t);t \ge 0)$, the top particle $(X_{n+1}(t);t \ge 0)$ follows the non-explosive $\mathcal{D}$-chain dynamics and so the only way for it to explode is if $Y_n$ drives it to $+\infty$, which does not happen (since $Y_n$ is itself an autonomous non-exploding $\hat{\mathcal{D}}$-chain ). More formally, we have (the notation is made precise below),
\begin{align*}
\mathbb{P}\left(X_{n+1}(t)> x_{n+1}+N\right)\le \mathbb{E}\left[
\mathbb{P}\left(\bar{D}(t)>x_{n+1}+N\bigg\rvert \bar{D}(0)=\underset{s \le t}{\sup}\hat{D}(s)\vee x_{n+1}\right)\right]
\end{align*}
where $\hat{D}$ is a realization of a $\hat{\mathcal{D}}$-chain and the outer expectation is taken over this. Also note that,
\begin{align*}
M=\underset{s \le t}{\sup}\hat{D}(s)<\infty, \textnormal{ a.s. }
\end{align*}
and conditioned on the realization of $\hat{D}$, the chain $\bar{D}$ is defined as follows: it moves as a $\mathcal{D}$-chain except that, jumps below $M$ are suppressed, namely its rates $(\bar{\lambda},\bar{\mu})$ are given by,
\begin{align*}
\bar{\lambda}(M)=\lambda(M), \bar{\mu}(M)=0 \textnormal{ and }\bar{\lambda}(k)=\lambda(k),\bar{\mu}(k)=\mu(k), \textnormal{ for } k\ge M+1.
\end{align*}
 This is again, non-explosive and hence,
\begin{align*}
\mathbb{P}\left(\bar{D}(t)>x_{n+1}+N\bigg\rvert \bar{D}(0)=\underset{s \le t}{\sup}\hat{D}(s)\vee x_{n+1}\right) \to 0 ,\textnormal{ as } N \to \infty.
\end{align*}
The result now, follows from the dominated convergence theorem.
\end{proof}

Finally, after a Doob's $h$-transform, by a strictly positive eigenfunction $\mathfrak{h}$ of $(\hat{P}_t^n;t \ge 0)$, the rates for the two-level Markov process, evolving according to $\left(\mathsf{Q}^{n,n+1,\mathfrak{h}}_t;t \ge 0\right)$ are given by,
\begin{align*}
\mathfrak{D}^{n,n+1}((x,y),(x',y'))=\begin{cases}
\lambda(x_i)  & x_i'=x_i+1 \textnormal{ and } i \in I^{n,n+1,+}_{adm}(x,y)\\
 \mu(x_i)  & x_i'=x_i-1 \textnormal{ and } i \in I^{n,n+1,-}_{adm}(x,y)  \\
 \hat{\lambda}^i_{\mathfrak{h}}(y_1,\cdots,y_n)  & y_i'=y_i+1 \textnormal{ and } i+1 \in I^{n,n+1,-}_{adm}(x,y) \\
 \hat{\mu}^i_{\mathfrak{h}}(y_1,\cdots,y_n)  & y_i'=y_i-1 \textnormal{ and } i \in I^{n,n+1,+}_{adm}(x,y)\\
 \hat{\lambda}^i_{\mathfrak{h}}(y_1,\cdots,y_n) & (x_{i+1},y_i)=(x+1,x), (x'_{i+1},y'_i)=(x+2,x+1)  \\
 \hat{\mu}^i_{\mathfrak{h}}(y_1,\cdots,y_n)  & (x_{i},y_i)=(x,x), (x'_{i},y'_i)=(x-1,x-1) \\
 S^{n,n+1,\mathfrak{h}}_{(x,y)} & (x',y')=(x,y)\\
  0 & \textnormal{ otherwise }
\end{cases},
\end{align*}
where for $1\le i \le n$,
\begin{align*}
\hat{\lambda}^i_{\mathfrak{h}}(y_1,\cdots,y_n)&=\frac{\mathfrak{h}(y_1,\cdots,y_{i-1},y_i+1,y_{i+1},\cdots,y_n)}{\mathfrak{h}(y_1,\cdots,y_n)}\hat{\lambda}(y_i),\\
\hat{\mu}^i_{\mathfrak{h}}(y_1,\cdots,y_n)&=\frac{\mathfrak{h}(y_1,\cdots,y_{i-1},y_i-1,y_{i+1},\cdots,y_n)}{\mathfrak{h}(y_1,\cdots,y_n)}\hat{\mu}(y_i)
\end{align*}
and $S^{n,n+1,\mathfrak{h}}_{(x,y)}$ is given by,
\begin{align*}
S^{n,n+1,\mathfrak{h}}_{(x,y)}=-\sum_{i \in I^{n,n+1,+}_{adm}(x,y)}^{}\lambda(x_i)-\sum_{i \in I^{n,n+1,-}_{adm}(x,y)}^{}\mu(x_i)-\sum_{i=1}^{n}\left[\hat{\lambda}^i_{\mathfrak{h}}(y_1,\cdots,y_n)+\hat{\mu}^i_{\mathfrak{h}}(y_1,\cdots,y_n)\right].
\end{align*}

\begin{rmk} \label{equalnumberrates}
We list here the rates for the push-block dynamics in $W^{n,n}$, described informally in the second paragraph of this subsection. With the analogous (with minor modifications due to the positions of the $\le$ and $<$ signs, see also Figures \ref{figureXYinteraction5}-\ref{figureXYinteraction8}) definitions for $\partial W^{n,n}$, $\mathring{W}^{n,n}$, $I^{n,n,+}_{adm}(x,y)$ and $I^{n,n,-}_{adm}(x,y)$ we have,
\begin{align*}
\mathfrak{D}^{n,n}((x,y),(x',y'))=\begin{cases}
\hat{\lambda}(x_i)  & x_i'=x_i+1 \textnormal{ and } i \in I^{n,n,+}_{adm}(x,y)\\
 \hat{\mu}(x_i)  & x_i'=x_i-1 \textnormal{ and } i \in I^{n,n,-}_{adm}(x,y)  \\
 \lambda(y_i)  & y_i'=y_i+1 \textnormal{ and } i \in I^{n,n,-}_{adm}(x,y) \\
 \mu(y_i)  & y_i'=y_i-1 \textnormal{ and } i-1 \in I^{n,n,+}_{adm}(x,y)\\
 \lambda(y_i) & (x_{i},y_i)=(x,x), (x'_{i},y'_i)=(x+1,x+1)  \\
 \mu(y_i)  & (x_{i-1},y_i)=(x-1,x), (x'_{i-1},y'_i)=(x-2,x-1) \\
 S^{n,n}_{(x,y)} & (x',y')=(x,y)\\
  0 & \textnormal{ otherwise }
\end{cases},
\end{align*}
where $S^{n,n}_{(x,y)}$ is given by,
\begin{align*}
S^{n,n}_{(x,y)}=-\sum_{i \in I^{n,n,+}_{adm}(x,y)}^{}\hat{\lambda}(x_i)-\sum_{i \in I^{n,n,-}_{adm}(x,y)}^{}\hat{\mu}(x_i)-\sum_{i=1}^{n}\left[\lambda(y_i)+\mu(y_i)\right].
\end{align*}
Again observe that, there is a non-zero rate $(x,y) \in W^{n,n} \to (x',y') \notin W^{n,n}$, which corresponds to killing the chain; this of course coincides with the rate of $y \in W^n(I) \to y'\notin W^n(I)$, which is only non-zero for $y \in \partial W^n(I)$ and is given by,
\begin{align*}
k_{(x,y)}^{n,n}=\sum_{i=1}^{n-1}\textbf{1}\left(y_i+1=y_{i+1}\right)\left[\lambda(y_i)+\mu(y_i+1)\right].
\end{align*}
The scheme of proof for the fact that $\mathsf{q}_t^{n,n}$ describes the dynamics above is exactly the same as the one followed for $W^{n,n+1}$.
\end{rmk}

\begin{rmk}\label{ComparisonToBorodinMultilevel}
Note that $\mathsf{q}_t^{n_1,n_2}$ is the transition kernel of the push-block dynamics in $W^{n_1,n_2}$ starting from \textbf{any} initial distribution $\nu(x,y)$, that is supported in $W^{n_1,n_2}$. One should compare with the "multilevel transition operator" for central or Gibbs measures denoted here by $\mathfrak{A}_t$, considered in Theorem 3.12 of \cite{BorodinKuan} and later used in \cite{Cerenzia} Proposition 5.3 and \cite{CerenziaKuan} section 5.3, that forms a semigroup when restricted to such measures. For the two-level dynamics these correspond to a measure on $W^{n_1,n_2}$ of the form $m_{n_2}(x)\Lambda^{h_{n_1}}_{n_1,n_2}(x,y)$, where $m_{n_2}$ is a measure on $W^{n_2}$ and $\Lambda^{h_{n_1}}_{n_1,n_2}(x,y)$ is a normalized (Markov) intertwining kernel from section \ref{sectionintertwinins}. It is of course clear that, $\mathsf{q}_t^{n_1,n_2,h_{n_1}}$ and $\mathfrak{A}_t$ coincide on such measures. Currently, we have no explicit analogue of the transition kernel for at least $3$ levels starting from any initial condition.
\end{rmk}

\subsection{Multilevel process construction}\label{subsectionmultilevelconstruction}
Let the state space $I$, be fixed. Suppose that, we are given a sequence of positive integers, $\mathsf{n}(1)\le \mathsf{n}(2) \le \cdots \le \mathsf{n}(N)\le \cdots$, so that $\mathsf{n}(k)-\mathsf{n}(k-1)\le 1$. Moreover, we have the following (off-diagonal) jump rates (their purpose is explained below),
\begin{align*}
r_j^+:W^{\mathsf{n}(1)}\to \mathbb{R}_+ &, r_j^-:W^{\mathsf{n}(1)}\to \mathbb{R}_+ , \textnormal{ for } 1 \le j \le \mathsf{n}(1),\\
\lambda^i:I\to \mathbb{R}_+ &, \mu^i:I \to \mathbb{R}_+, \textnormal{ for } i\ge 2.
\end{align*}
For, $k\ge 1$, the $k^{th}$ level will consist of $\mathsf{n}(k)$ (ordered) particles, i.e. will be taking values in $W^{\mathsf{n}(k)}$. We assume that, the rates for the first level, $(r_j^+,r_j^-)$, with $1\le j \le \mathsf{n}(1)$, which correspond to increasing or decreasing the $j^{th}$-coordinate by $1$ respectively (equivalently the $j^{th}$-particle jumping to the right or to the left), give rise to non-explosive dynamics in $W^{\mathsf{n}(1)}$. In the setting studied in this work, these are given by a conditioning, using a Doob's $h$-transformation, of $\mathsf{n}(1)$ independent birth and death chains (see discussion after proof of Proposition \ref*{UniquenessBackwardsWithCemetery} above for example). Furthermore, assume that the rates $(\lambda^i,\mu^i)_{i\ge 2}$ give rise to non-explosive (one-dimensional) birth and death chains in $I$.

Our goal is to construct, for each $N\ge 1$, a multilevel interlaced Markov process $\left(X^1(t),\cdots, X^{N}(t);t \ge 0\right)$ with generator $\mathfrak{D}_{1,\cdots, N}$, such that for each $k\ge 1 $, $\left(X^{k+1}(t);t \ge 0\right)$ consists of $\mathsf{n}(k+1)$ independent birth and death chains, each moving with rates $(\lambda^{k+1},\mu^{k+1})$, pushed and blocked, when at the boundary of $W^{\mathsf{n}(k),\mathsf{n}(k+1)}$ by the (particles of the) process $\left(X^{k}(t);t \ge 0\right)$, as in our two-level couplings from the previous subsection. We do this by induction. For the first level define,
\begin{align*}
\mathfrak{D}_{1}\left(x^1,z^1\right)=\begin{cases}
r^+_i\left(x^1\right)  & z^1_i=x^1_i+1, \ 1 \le i \le \mathsf{n}(1) \\
r^-_i\left(x^1\right)  & z^1_i=x^1_i-1,  \ 1 \le i \le \mathsf{n}(1) \\
-\sum_{i=1}^{\mathsf{n}(1)}\left[r^+_i\left(x^1\right)+r^-_i\left(x^1\right)\right] & x^1=z^1\\
  0 & \textnormal{ otherwise }
\end{cases}.
\end{align*}
 Suppose that we have constructed a process $\left(X^1(t),\cdots, X^{N-1}(t);t \ge 0\right)$, with rates of a transition $\left(x^1,\cdots,x^{N-1}\right)\to\left(z^1,\cdots,z^{N-1}\right)$ given by,
\begin{align*}
\mathfrak{D}_{1,\cdots, N-1}\left(\left(x^1,\cdots,x^{N-1}\right),\left(z^1,\cdots,z^{N-1}\right)\right)
\end{align*}
where, for $i\ge 1$, $x^i$ and $x^{i+1}$, $z^i$ and $z^{i+1}$, interlace. We proceed to define the rates $\mathfrak{D}_{1,\cdots, N}$ giving rise to $\left(X^1(t),\cdots, X^{N}(t);t \ge 0\right)$. First, suppose that $\mathsf{n}(N)=\mathsf{n}(N-1)+1$. Then we let the jump rates $\left(x^1,\cdots,x^{N}\right)\to\left(z^1,\cdots,z^{N}\right)$ ,
\begin{align*}
\mathfrak{D}_{1,\cdots,N}\left(\left(x^1,\cdots,x^{N}\right),\left(z^1,\cdots,z^{N}\right)\right)
\end{align*}
be given by,
\begin{align*}
\begin{cases}
\lambda^N(x^N_i)  & z^N_i=x^N_i+1 \textnormal{ and } i \in I^{\mathsf{n}(N)-1,\mathsf{n}(N),+}_{adm}(x^N,x^{N-1})\\
\mu^N(x^N_i)  & z^N_i=x^N_i-1 \textnormal{ and } i \in I^{\mathsf{n}(N)-1,\mathsf{n}(N),-}_{adm}(x^N,x^{N-1}) \\
 \mathfrak{D}_{1,\cdots, N-1}\left(\left(x^1,\cdots,x^{N-1}\right),\left(z^1,\cdots,z^{N-1}\right)\right)  & x^N=z^N \textnormal{ and } (x^N,z^{N-1})\in W^{\mathsf{n}(N)-1,\mathsf{n}(N)} \\
\mathfrak{D}_{1,\cdots, N-1}\left(\left(x^1,\cdots,x^{N-1}\right),\left(z^1,\cdots,z^{N-1}\right)\right) & (x^N_{i+1},x^{N-1}_i)=(x+1,x), (z^N_{i+1},z^{N-1}_i)=(x+2,x+1)  \\
 \mathfrak{D}_{1,\cdots, N-1}\left(\left(x^1,\cdots,x^{N-1}\right),\left(z^1,\cdots,z^{N-1}\right)\right)  & (x^N_{i},x^{N-1}_i)=(x,x), (z^N_{i},z^{N-1}_i)=(x-1,x-1) \\
S_{1,\cdots,N}^{(x^1,\cdots,x^N)} & \left(x^1,\cdots,x^{N}\right)=\left(z^1,\cdots,z^{N}\right)\\
  0 & \textnormal{ otherwise }
\end{cases},
\end{align*}
where $S_{1,\cdots,N}^{(x^1,\cdots,x^N)}$ is given by,
\begin{align*}
S_{1,\cdots,N}^{(x^1,\cdots,x^N)}=&-\sum_{i \in I^{\mathsf{n}(N)-1,\mathsf{n}(N),+}_{adm}(x^N,x^{N-1})}^{}\lambda^N(x^N_i)-\sum_{i \in I^{\mathsf{n}(N)-1,\mathsf{n}(N),-}_{adm}(x^N,x^{N-1})}^{}\mu^N(x^N_i)\\
&-\sum_{z^1,\cdots,z^{N-1}}\mathfrak{D}_{1,\cdots, N-1}\left(\left(x^1,\cdots,x^{N-1}\right),\left(z^1,\cdots,z^{N-1}\right)\right).
\end{align*}
Similarly, if $\mathsf{n}(N)=\mathsf{n}(N-1)$ we then define $\mathfrak{D}_{1,\cdots,N}\left(\left(x^1,\cdots,x^{N}\right),\left(z^1,\cdots,z^{N}\right)\right)$ as follows,
\begin{align*}
\begin{cases}
\lambda^N(x^N_i)  & z^N_i=x^N_i+1 \textnormal{ and } i \in I^{\mathsf{n}(N),\mathsf{n}(N),+}_{adm}(x^N,x^{N-1})\\
\mu^N(x^N_i)  & z^N_i=x^N_i-1 \textnormal{ and } i \in I^{\mathsf{n}(N),\mathsf{n}(N),-}_{adm}(x^N,x^{N-1}) \\
 \mathfrak{D}_{1,\cdots, N-1}\left(\left(x^1,\cdots,x^{N-1}\right),\left(z^1,\cdots,z^{N-1}\right)\right)  & x^N=z^N \textnormal{ and } (x^N,z^{N-1})\in W^{\mathsf{n}(N),\mathsf{n}(N)} \\
\mathfrak{D}_{1,\cdots, N-1}\left(\left(x^1,\cdots,x^{N-1}\right),\left(z^1,\cdots,z^{N-1}\right)\right) & (x^N_{i},x^{N-1}_i)=(x,x), (z^N_{i+1},z^{N-1}_i)=(x+1,x+1)  \\
 \mathfrak{D}_{1,\cdots, N-1}\left(\left(x^1,\cdots,x^{N-1}\right),\left(z^1,\cdots,z^{N-1}\right)\right)  & (x^N_{i-1},x^{N-1}_i)=(x-1,x), (z^N_{i-1},z^{N-1}_i)=(x-2,x-1) \\
\tilde{S}_{1,\cdots,N}^{(x^1,\cdots,x^N)} & \left(x^1,\cdots,x^{N}\right)=\left(z^1,\cdots,z^{N}\right)\\
  0 & \textnormal{ otherwise }
\end{cases},
\end{align*}
where $\tilde{S}_{1,\cdots,N}^{(x^1,\cdots,x^N)}$ is given by,
\begin{align*}
\tilde{S}_{1,\cdots,N}^{(x^1,\cdots,x^N)}=&-\sum_{i \in I^{\mathsf{n}(N),\mathsf{n}(N),+}_{adm}(x^N,x^{N-1})}^{}\lambda^N(x^N_i)-\sum_{i \in I^{\mathsf{n}(N),\mathsf{n}(N),-}_{adm}(x^N,x^{N-1})}^{}\mu^N(x^N_i)\\
&-\sum_{z^1,\cdots,z^{N-1}}\mathfrak{D}_{1,\cdots, N-1}\left(\left(x^1,\cdots,x^{N-1}\right),\left(z^1,\cdots,z^{N-1}\right)\right).
\end{align*}
Observe that, by construction for any $1 \le k\le N$, the process consisting of the first $k$ levels, $\left(X^1(t),\cdots, X^{k}(t);t \ge 0\right)$ is autonomous, governed by the transition rates $\mathfrak{D}_{1,\cdots,k}$. Moreover, given the trajectories of $\left(X^k(t);t \ge 0\right)$, the very next $(k+1)^{st}$ level $\left(X^{k+1}(t);t \ge 0\right)$, simply moves according to the corresponding push-block dynamics in either $W^{\mathsf{n}(k),\mathsf{n}(k)+1}$ or $W^{\mathsf{n}(k),\mathsf{n}(k)}$.

The fact that, the process with transition matrix $\mathfrak{D}_{1,\cdots,N}$ just defined, is well-posed can be seen inductively as follows. Assume that $\left(X^1(t),\cdots, X^{N-1}(t);t \ge 0\right)$ is almost surely non-explosive. Then by definition, adding level-$N$, $\left(X^{N}(t);t \ge 0\right)$ means introducing $\mathsf{n}(N)$ further independent birth and death chains (particles) each moving according to the non-explosive jump rates $(\lambda^N,\mu^N)$ that only interact with $\left(X^{N-1}(t);t \ge 0\right)$ via the pushing and blocking mechanism. Hence, this new enlarged process is seen to be non-explosive by the exact same argument used at the end of the preceding subsection.

\subsection{Consistent dynamics for multilevel processes}\label{subsectionconsistent}
We will discuss consistency relations under which if the multilevel process, whose construction we have just described, is started according to certain \textit{Gibbs} or \textit{central} initial conditions, then each level evolves as a Markov process and the fixed time $T>0$ distribution of the whole process retains the explicit Gibbs structure. We restrict our attention to multilevel processes taking values in triangular arrays known as Gelfand-Tsetlin patterns. The consistency relations and Propositions \ref{CoherentdynamicsproptypeA} and \ref{CoherentdynamicsproptypeB} below have analogues, with rather obvious modifications, to arbitrary multilevel interlaced processes, so that the number of particles from one level to the next increases by at most 1. We do not spell this out, since the already heavy notation becomes quite cumbersome.

Before we continue, we note that none of the results of this subsection are essentially new. In recent years Borodin and collaborators have many variations of constructions of multilevel processes (see Remark \ref{BorodinMultilevel}). In particular Propositions \ref{CoherentdynamicsproptypeA} and \ref{CoherentdynamicsproptypeB} follow as corollaries, after setting things up carefully, of the results found in Section 8 of \cite{BorodinOlshanski} (see also Section 9 therein). The reader familiar with those constructions can safely skip to the statements of the propositions (or skip the current subsection altogether). The reason we decided to include this rather detailed section, other than for completeness of the paper and sake of exposition, is because our method of proof is different; in particular the explicit form of the transition kernel of the two-level dynamics (cf. Remark \ref{ComparisonToBorodinMultilevel}) does not appear in any of those works (and is special to our setting).

We first consider the Gelfand-Tsetlin patterns of type-A, with $N$ levels. These are defined as follows,
\begin{align}
\mathbb{GT}(N)=\big\{\left(x^1,\cdots,x^N\right): x^i \in W^{i,i+1}(x^{i+1}),\textnormal{ for } 1 \le i\le N-1 \big\}.
\end{align}
See Figure \ref{figuretypeA} for an example.

\begin{figure}
\captionsetup{singlelinecheck = false, justification=justified}
\begin{tikzpicture}
\draw[dotted] (0,0) grid (7,3);
\draw[fill] (1,1) circle [radius=0.1];
\node[above right] at (1,1) {$x_1^{2}$};
\draw[fill] (3,1) circle [radius=0.1];
\node[above right] at (3,1) {$x_2^{2}$};
\draw[fill] (0,2) circle [radius=0.1];
\node[above right] at (0,2) {$x_1^{3}$};
\draw[fill] (2,2) circle [radius=0.1];
\node[above right] at (2,2) {$x_2^{3}$};
\draw[fill] (4,2) circle [radius=0.1];
\node[above right] at (4,2) {$x_3^{3}$};
\node[above right] at (0,3) {$x_1^{4}$};
\draw[fill] (0,3) circle [radius=0.1];
\node[above right] at (2,3) {$x_2^{4}$};
\draw[fill] (2,3) circle [radius=0.1];
\node[above right] at (3,3) {$x_3^{4}$};
\draw[fill] (3,3) circle [radius=0.1];
\node[above right] at (6,3) {$x_4^{4}$};
\draw[fill] (6,3) circle [radius=0.1];
\draw[fill] (1,0) circle [radius=0.1];
\node[above right] at (1,0) {$x_1^{1}$};
\node[below] at (0,0) {$-1$};
\node[below] at (1,0) {$0$};
\node[below] at (2,0) {$1$};
\node[below] at (3,0) {$2$};
\node[below] at (4,0) {$3$};
\node[below] at (5,0) {$4$};
\node[below] at (6,0) {$5$};
\node[below] at (7,0) {$6$};
\end{tikzpicture}
\caption{An example of a Gelfand-Tsetlin pattern of depth $4$ for $I=\mathbb{Z}$, with $x^1=0, x^2=(0,2),x^3=(-1,1,3),x^4=(-1,1,2,5)$.}\label{figuretypeA}
\end{figure}
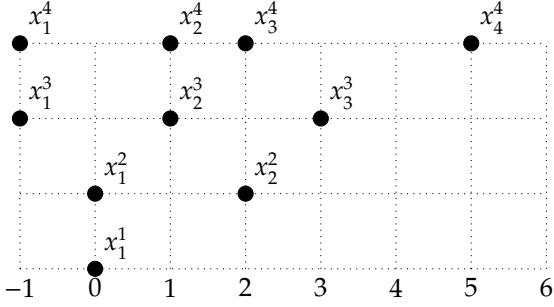

Suppose we have, for $1\le k \le N$, rates $(\lambda^k(\cdot),\mu^k(\cdot))$ governing modulo interactions the $k$ independent birth and death chains of the $k^{th}$ level. Denote by, $p_t^{k}(\cdot,\cdot)$ the transition density of this chain, also let $\hat{p}_t^{k}(\cdot,\cdot)$ be the transition density and $\hat{\pi}^k(\cdot)$ the symmetrizing measure of its Siegmund dual chain (with rates $(\hat{\lambda}^k(\cdot),\hat{\mu}^k(\cdot)$). Finally, with these rates as input, construct the process $\left(X^1(t),\cdots, X^{N}(t);t \ge 0\right)$ via the procedure detailed in subsection \ref{subsectionmultilevelconstruction} above.

We want to be able to apply Proposition \ref{MarkovFunctionProposition} (and Theorem \ref{Master1}) repeatedly recursively, for $k\ge 2$, to each pair $\left(X^{k-1},X^k\right)$. Towards this end, suppose $X^{k-1}$ is distributed as a Markov process in $W^{k-1}$, evolving according to the Doob's $h$-transformed Karlin-McGregor semigroup, by the strictly positive eigenfunction $h_{k-1}$, with eigenvalue $e^{c_{k-1}t}$, having transition density,
\begin{align*}
e^{-c_{k-1}t}\frac{h_{k-1}(y_1,\cdots,y_{k-1})}{h_{k-1}(x_1,\cdots,x_{k-1})}\det \left(\hat{p}_t^{k}\left(x_i,y_j\right)\right)_{i,j=1}^{k-1}.
\end{align*}
Moreover, define for $k \ge 2$ the following strictly positive function on $W^{k}$,
\begin{align}\label{consistencyrelationtypeA1}
H_{k-1}\left(x_1,\cdots,x_k\right)=\sum_{y\in W^{k-1,k}(x)}^{}\prod_{i=1}^{k-1}\hat{\pi}^k(y_i)h_{k-1}(y_1,\cdots,y_{k-1}).
\end{align}
Then, the basic consistency relation at the level of transition densities, which guarantees that the two descriptions of $X^k$ as the non-autonomous component of the coupling $\left(X^{k-1},X^k\right)$ and the autonomous component of the coupling $\left(X^{k},X^{k+1}\right)$ match, becomes for $k\ge 2$,
\begin{align}\label{consistencyrelationtypeA2}
e^{-c_{k-1}t}\frac{H_{k-1}(y_1,\cdots,y_{k})}{H_{k-1}(x_1,\cdots,x_{k})}\det \left(p_t^{k}\left(x_i,y_j\right)\right)_{i,j=1}^{k}=e^{-c_{k}t}\frac{h_{k}(y_1,\cdots,y_{k})}{h_{k}(x_1,\cdots,x_{k})}\det \left(\hat{p}_t^{k+1}\left(x_i,y_j\right)\right)_{i,j=1}^{k}.
\end{align}
For $k=1$ we put by definition $H_0 \equiv 1$ and so,
\begin{align}
p_t^{1}\left(x,y\right)=e^{-c_{1}t}\frac{h_{1}(y)}{h_{1}(x)}\hat{p}_t^{2}\left(x,y\right).
\end{align}
Let $\left(\mathfrak{P}^{k}(t);t \ge 0\right)$, denote the Markov semigroup that these densities give rise to and also define the Markov kernel,
\begin{align*}
\mathfrak{L}^{k}_{k-1}\left(x,y\right)=\frac{\prod_{i=1}^{k-1}\hat{\pi}^k(y_i)h_{k-1}(y_1,\cdots,y_{k-1})}{H_{k-1}(x_1,\cdots,x_{k})}\textbf{1}\left(y \in W^{k-1,k}(x)\right).
\end{align*}
Then, we have the following proposition.

\begin{prop}\label{CoherentdynamicsproptypeA}
Let  $\left(X^1(t),\cdots, X^{N}(t);t \ge 0\right)$ be the Markov process with transition matrix $\mathfrak{D}_{1,\cdots,N}$, built from the non-explosive rates $(\lambda^i(\cdot),\mu^i(\cdot))_{1\le i \le N}$. Suppose the consistency relations (\ref{consistencyrelationtypeA1}) and (\ref{consistencyrelationtypeA2}) hold for $1\le k \le N-1$. Let $\left(\mathfrak{P}^{k}(t);t \ge 0\right)$ and $\mathfrak{L}^{k}_{k-1}$ denote the semigroups and Markov kernels defined above and let $\mathfrak{M}^N(\cdot)$ be a probability measure on $W^N$. Finally, suppose that, $\left(X^1(t),\cdots, X^{N}(t);t \ge 0\right)$ is initialized according to the Gibbs measure with density in $\mathbb{GT}(N)$,
\begin{align}\label{GibbsTypeA}
\mathfrak{M}^N(x^N)\mathfrak{L}^{N}_{N-1}\left(x^N,x^{N-1}\right)\cdots \mathfrak{L}^{2}_{1}\left(x^2,x^{1}\right).
\end{align}
Then, $\left(X^{k}(t);t \ge 0\right)$ for $1 \le k \le N$ is distributed as a Markov process evolving according to $\left(\mathfrak{P}^{k}(t);t \ge 0\right)$ and moreover, for fixed $T>0$, the law of $\left(X^1(T),\cdots, X^{N}(T)\right)$ is given by the evolved Gibbs measure, with density in $\mathbb{GT}(N)$,
\begin{align}\label{EvolvedGibbsTypeA}
\left[\mathfrak{M}^N\mathfrak{P}^{N}(T)\right](x^N)\mathfrak{L}^{N}_{N-1}\left(x^N,x^{N-1}\right)\cdots \mathfrak{L}^{2}_{1}\left(x^2,x^{1}\right).
\end{align}
\end{prop}
\begin{proof}
The proof is by induction. For $N=2$, this is Proposition \ref{MarkovFunctionProposition} (see Theorem \ref{Master1} as well). Assume the result is true for $N-1$. Then, $\left(X^{N-1}(t);t \ge 0\right)$ is a Markov process with semigroup $\left(\mathfrak{P}^{N-1}(t);t \ge 0\right)$. Moreover, from the consistency relation (\ref{consistencyrelationtypeA2}) for $k=N-1$, the joint dynamics of $\left(X^{N-1}(t),X^N(t);t \ge 0\right)$ are those considered in Proposition \ref{MarkovFunctionProposition} and thus, we obtain that $\left(X^N(t);t \ge 0\right)$ is distributed as a Markov process with semigroup $\left(\mathfrak{P}^{N}(t);t \ge 0\right)$. Furthermore, for fixed $T>0$, the conditional law of $X^{N-1}(T)$ given $X^{N}(T)$ is $\mathfrak{L}^{N}_{N-1}\left(X^N(T),\cdot\right)$. Hence, since the distribution of $X^N(T)$ has density $\left[\mathfrak{M}^N\mathfrak{P}^{N}(T)\right](\cdot)$, we get by the induction hypothesis, that the fixed time $T>0$, distribution of $\left(X^1(T),\cdots, X^{N}(T)\right)$ is given by (\ref{EvolvedGibbsTypeA}).
\end{proof}

\begin{rmk}\label{RemarkEynardMehtaForEvolved}
If there exist (positive) functions $\{f_k(\cdot)\}^N_{k=2}$ such that, for $2 \le k \le N$,
\begin{align*}
h_k(x_1,\cdots,x_k)=\prod_{i=1}^{k}f_k(x_i)H_{k-1}(x_1,\cdots,x_k)
\end{align*}
and moreover functions  $\{G_k(T,\cdot)\}^N_{k=1}$ so that,
\begin{align*}
\left[\mathfrak{M}^N\mathfrak{P}^{N}(T)\right](x_1,\cdots,x_N)=H_{N-1}(x_1,\cdots,x_N) \det \left(G_i(T,x_j)\right)_{i,j=1}^N
\end{align*}
then (\ref{EvolvedGibbsTypeA}) simplifies to:
\begin{align*}
\det \left(G_i\left(T,x^N_j\right)\right)_{i,j=1}^N \prod_{k=2}^{N}\prod_{i=1}^{k-1}\hat{\pi}^k\left(x_i^{k-1}\right)h_1\left(x_1^1\right)\prod_{k=2}^{N}\prod_{i=1}^{k}f_k\left(x_i^k\right)\prod_{k=1}^{N-1}\textbf{1}\left(x^k \in W^{k,k+1}(x^{k+1})\right).
\end{align*}
Hence, since the interlacing constraints can be written as a determinant, for some function $g(\cdot,\cdot)$ of two variables, see section \ref*{multivariatepolynomialssection} for the details, the display above becomes,
\begin{align*}
\det \left(G_i\left(T,x^N_j\right)\right)_{i,j=1}^N \prod_{k=2}^{N}\prod_{i=1}^{k-1}\hat{\pi}^k\left(x_i^{k-1}\right)h_1\left(x_1^1\right)\prod_{k=2}^{N}\prod_{i=1}^{k}f_k\left(x_i^k\right)\prod_{k=1}^{N-1}\det\left(g\left(x_i^k,x_j^{k+1}\right)\right)_{i,j=1}^{k+1}.
\end{align*}
These types of measures, by the celebrated Eynard-Mehta Theorem (see \cite{BorodinRains}), give rise to determinantal point processes with an extended correlation kernel $\mathsf{K}$, which can in principle be computed. 

In order to obtain this explicitly however, one has to invert a certain matrix or do some kind of bi-orthogonalization which is usually a very daunting task. For a particular, but still quite general, solution of the consistency relations, in the setting of a symplectic Gelfand-Tsetlin pattern, see the discussion after Proposition \ref*{CoherentdynamicsproptypeB} below, we are able to perform such a computation in Section \ref*{sectioncorrelation} later on. In fact these computations carry over to a large class of consistent probability measures, that include the ones corresponding to the dynamics considered in this section as special cases, the reader is referred to sections \ref{sectioncoherentmeasures} to \ref*{sectioncorrelation} for these developments.
\end{rmk}

We shall now consider coherent dynamics in symplectic Gelfand-Tsetlin patterns of depth $N$ defined by,
\begin{align}
\mathbb{GT}_{\textbf{s}}(N)=\big\{\left((x^{(0,1)},x^{(1,1)}\cdots,x^{(N-1,N)}\right): x^{(i-1,i)} \in W^{i,i}(x^{(i,i)}),  x^{(i,i)} \in W^{i,i+1}(x^{(i,i+1)})\big\},
\end{align}
with the notation convention of using two superscript indices to indicate the number of particles at both the preceding and current levels. See Figure \ref{figuretypeBC} for a simple example.

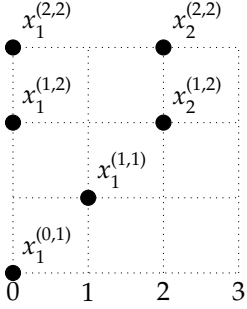
\begin{figure}
\captionsetup{singlelinecheck = false, justification=justified}
\begin{tikzpicture}
\draw[dotted] (0,0) grid (3,3);
\draw[fill] (1,1) circle [radius=0.1];
\node[above right] at (1,1) {$x_1^{(1,1)}$};
\draw[fill] (0,2) circle [radius=0.1];
\node[above right] at (0,2) {$x_1^{(1,2)}$};
\draw[fill] (2,2) circle [radius=0.1];
\node[above right] at (2,2) {$x_2^{(1,2)}$};

\node[above right] at (0,3) {$x_1^{(2,2)}$};
\draw[fill] (0,3) circle [radius=0.1];
\node[above right] at (2,3) {$x_2^{(2,2)}$};
\draw[fill] (2,3) circle [radius=0.1];

\draw[fill] (0,0) circle [radius=0.1];
\node[above right] at (0,0) {$x_1^{{(0,1)}}$};
\node[below] at (0,0) {$0$};
\node[below] at (1,0) {$1$};
\node[below] at (2,0) {$2$};
\node[below] at (3,0) {$3$};

\end{tikzpicture}
\caption{An example of a symplectic Gelfand-Tsetlin pattern of depth $2$ (note that it has 4 levels), for $I=\mathbb{N}$, with $x^{(0,1)}=0, x^{(1,1)}=1,x^{(1,2)}=(0,2),x^{(2,2)}=(0,2)$.}\label{figuretypeBC}
\end{figure}

Suppose that, for each level of $\mathbb{GT}_{\textbf{s}}(N)$ we are given (non-explosive) birth and death rates $(\lambda^{(k,k)}(\cdot),\mu^{(k,k)}(\cdot))$ and  $(\lambda^{(k,k+1)}(\cdot),\mu^{(k,k+1)}(\cdot))$ and from these we construct a Markov process $\left(X^{(0,1)}(t),X^{(1,1)}(t)\cdots, X^{(N-1,N)}(t);t \ge 0\right)$, using the recipe detailed in subsection \ref{subsectionmultilevelconstruction}. In order to proceed and be able to state the basic consistency relations, we need one more piece of notation. Define the operation $\check{{\cdot}}$ on transition matrices of birth and death (or bilateral) chains, as the inverse of the $\hat{{\cdot}}$ operation, i.e. as the inverse of taking the Siegmund dual. More explicitly, for a chain with birth rates $b(\cdot)$ and death rates $d(\cdot)$ this is given by:
\begin{align*}
\left(\check{b}(z),\check{d}(z)\right)\overset{\textnormal{def}}{=}\left(d(z),b(z-1)\right),\ z \in I.
\end{align*}
Observe that, in case $I=\mathbb{N}$ this is only defined on chains absorbed at $-1$. Finally, we shall use the same notations as before, with obvious modifications, for the transition densities and symmetrizing measures of the chains with rates $(\lambda^{(k,k)}(\cdot),\mu^{(k,k)}(\cdot))$, $(\lambda^{(k,k+1)}(\cdot),\mu^{(k,k+1)}(\cdot))$ and their various transforms.

We would like Proposition \ref{MarkovFunctionProposition} (see also Theorem \ref{Master1}) to be applicable, for $1\le k \le N-1$, to each pair $\left(X^{(k,k)},X^{(k,k+1)}\right)$ and Theorem \ref{Master2} to be applicable, for $1\le k \le N-1$, to each pair of the form $\left(X^{(k-1,k)},X^{(k,k)}\right)$, respectively.

Towards this end, suppose that $X^{(k-1,k-1)}$ evolves according to the $h$-transformed, by the strictly positive eigenfunction $h_{k-1,k-1}$ with eigenvalue $e^{c_{k-1,k-1}t}$, Karlin-McGregor semigroup with transition kernel in $W^{k-1}$,
\begin{align*}
e^{-c_{k-1,k-1}t}\frac{h_{k-1,k-1}(y_1,\cdots,y_{k-1})}{h_{k-1,k-1}(x_1,\cdots,x_{k-1})}\det \left(\hat{p}_t^{(k-1,k)}\left(x_i,y_j\right)\right)_{i,j=1}^{k-1}
\end{align*}
and moreover, define for $k \ge 2$ the following strictly positive function on $W^{k}$,
\begin{align}
H_{k-1,k-1}\left(x_1,\cdots,x_k\right)=\sum_{y\in W^{k-1,k}(x)}^{}\prod_{i=1}^{k-1}\hat{\pi}^{(k-1,k)}(y_i)h_{k-1,k-1}(y_1,\cdots,y_{k-1}).
\end{align}
We also define, $H_{0,0}\equiv 1$. Similarly, suppose that $X^{(k-1,k)}$ evolves according to the following $h$-transformed, by the strictly positive eigenfunction $h_{k-1,k}$ with eigenvalue $e^{c_{k-1,k}t}$, Karlin-McGregor semigroup with transition kernel in $W^{k}$,
\begin{align*}
e^{-c_{k-1,k}t}\frac{h_{k-1,k}(y_1,\cdots,y_{k})}{h_{k-1,k}(x_1,\cdots,x_{k})}\det \left(\check{p}_t^{(k,k)}\left(x_i,y_j\right)\right)_{i,j=1}^{k}
\end{align*}
and also, define for $k \ge 1$ the following strictly positive function on $W^{k}$,
\begin{align}
H_{k-1,k}\left(x_1,\cdots,x_k\right)=\sum_{y\in W^{k,k}(x)}^{}\prod_{i=1}^{k}\check{\pi}^{(k,k)}(y_i)h_{k-1,k}(y_1,\cdots,y_{k}).
\end{align}

Then, the basic consistency relations at the level of transition densities, which ensure that the descriptions of the levels $X^{(k-1,k)}$ and $X^{(k,k)}$ in two consecutive two-level couplings match, become,
\begin{align}
e^{-c_{k-1,k-1}t}\frac{H_{k-1,k-1}(y_1,\cdots,y_{k})}{H_{k-1,k-1}(x_1,\cdots,x_{k})}\det \left(p_t^{(k-1,k)}\left(x_i,y_j\right)\right)_{i,j=1}^{k}&=e^{-c_{k-1,k}t}\frac{h_{k-1,k}(y_1,\cdots,y_{k})}{h_{k-1,k}(x_1,\cdots,x_{k})}\det \left(\check{p}_t^{(k,k)}\left(x_i,y_j\right)\right)_{i,j=1}^{k}\label{consistencyrelationtypeB1},\\
e^{-c_{k-1,k}t}\frac{H_{k-1,k}(y_1,\cdots,y_{k})}{H_{k-1,k}(x_1,\cdots,x_{k})}\det \left(p_t^{(k,k)}\left(x_i,y_j\right)\right)_{i,j=1}^{k}&=e^{-c_{k,k}t}\frac{h_{k,k}(y_1,\cdots,y_{k})}{h_{k,k}(x_1,\cdots,x_{k})}\det \left(\hat{p}_t^{(k,k+1)}\left(x_i,y_j\right)\right)_{i,j=1}^{k}\label{consistencyrelationtypeB2}.
\end{align}

Let, $\left(\mathfrak{P}^{(k-1,k)}(t);t \ge 0\right)$ and $\left(\mathfrak{P}^{(k,k)}(t);t \ge 0\right)$ denote the corresponding semigroups these transition densities give rise to and finally define the Markov kernels,
\begin{align*}
\mathfrak{L}^{(k-1,k)}(x,y)&=\frac{\prod_{i=1}^{k-1}\hat{\pi}^{(k-1,k)}(y_i)h_{k-1,k-1}(y_1,\cdots,y_{k-1})}{H_{k-1,k-1}(x_1,\cdots,x_{k})}\textbf{1}\left(y \in W^{k-1,k}(x)\right),\\
\mathfrak{L}^{(k,k)}(x,y)&=\frac{\prod_{i=1}^{k}\check{\pi}^{(k,k)}(y_i)h_{k-1,k}(y_1,\cdots,y_{k})}{H_{k-1,k}(x_1,\cdots,x_{k})}\textbf{1}\left(y \in W^{k,k}(x)\right).
\end{align*}

Then, with similar considerations as in Proposition \ref{CoherentdynamicsproptypeA} above, by inductively applying Proposition \ref{MarkovFunctionProposition} and Theorem \ref{Master2} interchangeably we obtain:

\begin{prop}\label{CoherentdynamicsproptypeB}
Let $\left(X^{(0,1)}(t),X^{(1,1)}(t)\cdots, X^{(N-1,N)}(t);t \ge 0\right)$ be the multilevel Markov process in $\mathbb{GT}_\textbf{s}(N)$ built from the (non-explosive) rates  $(\lambda^{(k,k)}(\cdot),\mu^{(k,k)}(\cdot))$ and  $(\lambda^{(k,k+1)}(\cdot),\mu^{(k,k+1)}(\cdot))$. Suppose that, for all $k$ the consistency relations (\ref{consistencyrelationtypeB2}) hold. Let $\mathfrak{M}^{(N-1,N)}\left(\cdot\right)$ be a probability measure on $W^N$. Suppose that, $\left(X^{(0,1)}(t),X^{(1,1)}(t)\cdots, X^{(N-1,N)}(t);t \ge 0\right)$ is initialized according to the Gibbs measure with density in $\mathbb{GT}_{\textbf{s}}(N)$,
\begin{align}
\mathfrak{M}^{(N-1,N)}(x^{(N-1,N)})\mathfrak{L}^{N}_{N-1}\left(x^{(N-1,N)},x^{(N-1,N-1)}\right)\cdots \mathfrak{L}^{2}_{1}\left(x^{(1,2)},x^{(1,1)}\right)\mathfrak{L}^{1}_{1}\left(x^{(1,1)},x^{(0,1)}\right).
\end{align}
Then, for each $k$ the projections $\left(X^{(k,k)}(t);t \ge 0\right)$ and $\left(X^{(k,k+1)}(t);t \ge 0\right)$ are distributed as Markov processes, evolving according to the semigroups $\left(\mathfrak{P}^{(k,k)}(t);t \ge 0\right)$ and $\left(\mathfrak{P}^{(k,k+1)}(t);t \ge 0\right)$ respectively. Moreover, for fixed times $T>0$, the law of $\left(X^{(0,1)}(T),X^{(1,1)}(T)\cdots, X^{(N-1,N)}(T)\right)$ has density in $\mathbb{GT}_{\textbf{s}}(N)$ given by,
\begin{align}
\left[\mathfrak{M}^{(N-1,N)}\mathfrak{P}^{(N-1,N)}(T)\right](x^{(N-1,N)})\mathfrak{L}^{N}_{N-1}\left(x^{(N-1,N)},x^{(N-1,N-1)}\right)\cdots \mathfrak{L}^{2}_{1}\left(x^{(1,2)},x^{(1,1)}\right)\mathfrak{L}^{1}_{1}\left(x^{(1,1)},x^{(0,1)}\right).
\end{align}
\end{prop}
The most natural solution (this fact is readily checked) to the consistency relations (\ref{consistencyrelationtypeB1}) and (\ref{consistencyrelationtypeB2}) in a symplectic Gelfand-Tsetlin pattern, for $I=\mathbb{N}$, is given by, with $(\lambda(\cdot),\mu(\cdot))$ being the rates of a reflecting at the origin (non-exploding) birth and death chain,
\begin{align}
\left(\lambda^{(k,k+1)}(\cdot),\mu^{(k,k+1)}(\cdot)\right)&=\left(\lambda(\cdot),\mu(\cdot)\right), \textnormal{ for } k \ge 0,\\
\left(\lambda^{(k,k)}(\cdot),\mu^{(k,k)}(\cdot)\right)&=\left(\hat{\lambda}(\cdot),\hat{\mu}(\cdot)\right), \textnormal{ for } k\ge 1.
\end{align}
As already stated several times, this particular construction and its intimate relation to orthogonal polynomials will be studied in detail in later sections.

\begin{rmk}\label{BorodinMultilevel}
As already mentioned, a related approach for constructing continuous-time consistent multivariate/multilevel dynamics on countable spaces, which partly inspired our exposition, can be found in Section 8 of \cite{BorodinOlshanski}. This takes as input the following: a sequence $E_1,\cdots, E_N$ of countable sets, $Q_1,\cdots, Q_N$ (regular) matrices of transition rates on these sets (equivalently $(P_1(t);t\ge 0),\cdots, (P_N(t);t\ge 0)$ the Markovian semigroups corresponding to them) and Markov kernels $\Lambda_1^2,\cdots, \Lambda_{N-1}^N$:
\begin{align*}
\Lambda_{k-1}^k:E_k\times E_{k-1}\to [0,1] \ , \ \sum_{y\in E_{k-1}}^{}\Lambda_{k-1}^k(x,y)=1, \forall x \in E_k, \ k=2,\cdots, N.
\end{align*}
Finally, it is assumed that the intertwining/coherency relations between the (single level) semigroups/transition matrices hold, for $k=2,\cdots,N$:
\begin{align*}
Q_k\Lambda_{k-1}^k&=\Lambda_{k-1}^kQ_{k-1},\\
P_k(t)\Lambda_{k-1}^k&=\Lambda_{k-1}^kP_{k-1}(t), \ t\ge 0.
\end{align*}
Then, from this data a consistent coupling is provided, with the analogous consequences of Proposition \ref{CoherentdynamicsproptypeA} and \ref{CoherentdynamicsproptypeB} above, see Proposition 8.6 in \cite{BorodinOlshanski}. In particular, using only the single level intertwining relations (\ref{MasterIntertwining1eq2}) and (\ref{MasterIntertwining2eq2}), which are elementary to obtain c.f. Remark \ref{Proofremark}, we could have made use of the theory developed in Section 8 of \cite{BorodinOlshanski} to construct consistent multilevel dynamics. However, since we already have a two-level coupling, from which as we tried to stress throughout this work (\ref{MasterIntertwining1eq2}) and (\ref{MasterIntertwining2eq2}) originate after all, and for completeness of this paper, we decided to present and discuss in detail the multilevel construction in subsections \ref{subsectionmultilevelconstruction} and \ref{subsectionconsistent}.
\end{rmk}

\section{Branching graphs and Markov processes on their boundaries}\label{sectionbranching} 

\subsection{General setup of branching graphs}
We assume that we are given a set of vertices $V$, decomposed into levels $V=\sqcup^{\infty}_{N=1}V_N$, where each $V_N$ is countable. We moreover, assume that for each $x \in V_{N+1}$ there is at least one edge but not infinitely many connecting it to a vertex in $V_N$ and for each $y\in V_N$ there is at least one edge connecting it to a vertex in $V_{N+1}$. There are no edges between vertices of non-consecutive levels. 

For $N\ge 1$ and each $x\in V_{N+1}$ and $y \in V_N$, let $\mathsf{mult}(x,y)\in \mathbb{R}_+$ denote the multiplicity or weight of the edge connecting $x$ and $y$. If there is no such edge then this is 0. Define inductively the dimension of $x \in V_{N+1}$ by,
\begin{align*}
\mathsf{dim}_{N+1}(x)=\sum_{y\in V_N}^{}\mathsf{mult}(x,y)\mathsf{dim}_N(y).
\end{align*}
Note that, we need to stipulate $\mathsf{dim}_1(\cdot)$ for vertices at the first level. In all the examples that we consider, this will always be 1. We can then define the Markov kernel or link $\Lambda_N^{N+1}:V_{N+1}\to V_N$ (note that this is a generalized map, that maps a point in $V_{N+1}$ to a probability measure on $V_N$) as follows,
\begin{align*}
\Lambda_{N}^{N+1}(x,y)=\frac{\mathsf{mult}(x,y)\mathsf{dim}_N(y)}{\mathsf{dim}_{N+1}(x)}.
\end{align*}
Denoting by $\mathcal{M}_p(E)$ the space of probability measures on a measurable space $E$ ($\mathcal{M}_p(E)$ is a Banach space with the total variation norm), the kernels $\{\Lambda^{N+1}_{N}\}_{N\ge1}$ induce the following projective chain,
\begin{align*}
\mathcal{M}_p(V_1)\leftarrow \mathcal{M}_p(V_2) \leftarrow \cdots \mathcal{M}_p(V_N) \leftarrow \cdots .
\end{align*}
The projective limit $\underset{\leftarrow}{\lim}\mathcal{M}_p(V_N)$, is by definition the convex set consisting of sequences of probability measures $\{\mu_N\}^{\infty}_{N=1}$ that are coherent with respect to the links,
\begin{align*}
\mu_{N+1}\Lambda^{N+1}_N=\mu_N ,
\end{align*}
or more explicitly, for $y \in V_N$,
\begin{align*}
\mu_N(y)=\sum_{x \in V_{N+1}}^{}\mu_{N+1}(x)\Lambda_N^{N+1}(x,y).
\end{align*}
This space is equipped with the projective limit topology. Now, we will call the \textit{extreme points} of $\underset{\leftarrow}{\lim}\mathcal{M}_p(V_N)$ denoted by $V_{\infty}=\textnormal{Ex} \left(\underset{\leftarrow}{\lim}\mathcal{M}_p(V_N)\right)$, the \textit{boundary} of the branching graph (or more generally of the projective chain) with the topology inherited from $\underset{\leftarrow}{\lim}\mathcal{M}_p(V_N)$. Then, from Theorem 9.2 of \cite{OlshanskiHarmonic} we get that if $V_{\infty} \ne 0$, then there exists a natural map,
\begin{align}\label{bijection}
\mathcal{M}_p(V_{\infty})\to \underset{\leftarrow}{\lim}\mathcal{M}_p(V_N),
\end{align} 
that is an isomorphism of measurable spaces. More precisely, $V_{\infty}$ comes along with a family of (abstract) Markov kernels $\Lambda_{N}^{\infty}:V_{\infty} \to V_N$, which induce a map $\mathcal{M}_p(V_{\infty})\to \underset{\leftarrow}{\lim}\mathcal{M}_p(V_N)$, which is an isomorphism of measurable spaces. It is a remarkable fact that in certain concrete situations the (abstract) Markov kernels $\Lambda_{N}^{\infty}:V_{\infty} \to V_N$ can be given explicitly. Moreover, we will say that a Markov kernel from a locally compact space $X$ to a locally compact space $Y$ is Feller if the induced contraction that maps $C(Y)$ to $C(X)$ in fact maps $C_0(Y)$ into $C_0(X)$ , the continuous functions vanishing at infinity. We finally come to the following definition. 

We shall say that, $V_{\infty}$ is the \textit{Feller boundary} of the branching graph if $V_{\infty}$ is locally compact, for all $N\ge 1$ the Markov kernels $\Lambda_N^{N+1}$,$\Lambda_N^{\infty}$ are Feller and furthermore the map (\ref{bijection}) is an isomorphism of measurable spaces.

\subsection{Method of intertwiners and semigroups on the boundary} \label{subsectionintertwiners}

The following theorem is known as the method of intertwiners, first proven by Borodin and Olshanski in \cite{BorodinOlshanski}:

\begin{thm}
Assume that $V_{\infty}$ is the Feller boundary of the branching graph described above. Assume that, $\forall N \ge N_0$ we have Feller semigroups $\left(P_N(t);t\ge 0\right)$ on the levels $V_N$, that satisfy the following intertwining relations, for all $t\ge0$ and $N \ge N_0$,
\begin{align*}
P_{N+1}(t)\Lambda_{N}^{N+1}=\Lambda_N^{N+1}P_{N}(t).
\end{align*}
Then, there exists a unique Feller semigroup $\left(P_{\infty}(t);t\ge 0\right)$ on $V_{\infty}$ such that,
\begin{align*}
P_{\infty}(t)\Lambda_{N}^{\infty}=\Lambda_N^{\infty}P_{N}(t) ,\textnormal{ for } t \ge0 , N \ge N_0.
\end{align*}
Furthermore, if $\mu_N$ is the unique invariant probability measure for $\left(P_N(t);t\ge 0\right)$ then there exists a unique probability measure $\mu_{\infty}$ on $V_{\infty}$ that is invariant with respect to $\left(P_{\infty}(t);t\ge 0\right)$.
\end{thm}

\subsection{Examples of branching graphs} \label{examplesofbranchingsubsection}

In this subsection, we describe three examples of branching graphs. The first two are classical and originated from the representation theory of Lie groups. The third one, the generalized BC-type branching graph, is new and is related to the two-step branching rules for the multivariate Karlin-McGregor polynomials. We will provide rather complete information for the Gelfand-Tsetlin graph, since we will mainly focus on it in Section \ref{sectionExamples}. The same kind of information is available for the BC-type graph, although the notation gets a bit more cumbersome, while for the generalized BC-type branching graph much less is known.

\paragraph{The Gelfand-Tsetlin graph} The vertices at level $N$ of this branching graph are given by \textit{signatures} of length $N$, i.e. integer sequences  $\kappa=(\kappa_1,\cdots,\kappa_N)$ so that $\kappa_1\ge \kappa_2 \ge \cdots \ge \kappa_N$. Moreover, vertices $\kappa$ at level $N$ and $\nu$ at level $N+1$ are connected if they interlace in the following way, $\nu_1\ge\kappa_1 \ge \nu_2 \ge \cdots \ge\kappa_N\ge\nu_{N+1}$, the multiplicity $\mathsf{mult}(\nu,\kappa)$ being equal to $1$ in such a case. To transform this into our notation, note that there is a bijection,
\begin{align*}
(\kappa_1\ge \cdots \ge \kappa_N) \mapsto (y_1<y_2<\cdots <y_N),
\end{align*}
given by,
\begin{align*}
\tilde{\kappa}_i=\kappa_i+N-i \textnormal{ and } y_i=\tilde{\kappa}_{N-i}.
\end{align*}
Observe that, under this bijection if,
\begin{align*}
\nu=(\nu_1\ge \cdots \ge \nu_{N+1})&\mapsto x= (x_1<x_2<\cdots<x_{N+1}),\\
\kappa=(\kappa_1\ge \cdots \ge \kappa_{N})&\mapsto y= (y_1<y_2<\cdots<y_{N}),
\end{align*}
then, $\nu_1\ge\kappa_1 \ge \nu_2 \ge \cdots \ge\kappa_N\ge\nu_{N+1}$ if and only if $y\in W^{N,N+1}(x)$. Hence, observe that a path of length $N$ is given by a Gelfand-Tsetlin pattern (of type-A) of depth $N$. See Figure \ref{figuregraphpatternexample} for an example.

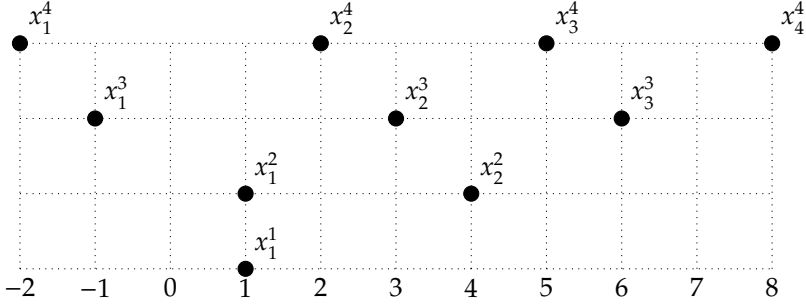
\begin{figure}
\captionsetup{singlelinecheck = false, justification=justified}
\begin{tikzpicture}
\draw[dotted] (0,0) grid (10,3);
\draw[fill] (3,1) circle [radius=0.1];
\node[above right] at (3,1) {$x_1^{2}$};
\draw[fill] (6,1) circle [radius=0.1];
\node[above right] at (6,1) {$x_2^{2}$};
\draw[fill] (1,2) circle [radius=0.1];
\node[above right] at (1,2) {$x_1^{3}$};
\draw[fill] (5,2) circle [radius=0.1];
\node[above right] at (5,2) {$x_2^{3}$};
\draw[fill] (8,2) circle [radius=0.1];
\node[above right] at (8,2) {$x_3^{3}$};
\node[above right] at (0,3) {$x_1^{4}$};
\draw[fill] (0,3) circle [radius=0.1];
\node[above right] at (4,3) {$x_2^{4}$};
\draw[fill] (4,3) circle [radius=0.1];
\node[above right] at (7,3) {$x_3^{4}$};
\draw[fill] (7,3) circle [radius=0.1];
\node[above right] at (10,3) {$x_4^{4}$};
\draw[fill] (10,3) circle [radius=0.1];
\draw[fill] (3,0) circle [radius=0.1];
\node[above right] at (3,0) {$x_1^{1}$};
\node[below] at (0,0) {$-2$};
\node[below] at (1,0) {$-1$};
\node[below] at (2,0) {$0$};
\node[below] at (3,0) {$1$};
\node[below] at (4,0) {$2$};
\node[below] at (5,0) {$3$};
\node[below] at (6,0) {$4$};
\node[below] at (7,0) {$5$};
\node[below] at (8,0) {$6$};
\node[below] at (9,0) {$7$};
\node[below] at (10,0) {$8$};
\end{tikzpicture}
\caption{An example of a path of length 4 in the Gelfand-Tsetlin graph, given by a Gelfand-Tsetlin pattern of depth 4. Here the path in terms of signatures $\kappa^1 \to \kappa^2 \to \kappa^3 \to \kappa^4$ is given by $\kappa^1=1,\kappa^2=(3,1),\kappa^3=(4,2,-1),\kappa^4=(5,3,1,-2)$, which transformed into our notation gives, $x^1=1,x^2=(1,4),x^3=(-1,3,6),x^4=(-2,2,5,8)$.}\label{figuregraphpatternexample}
\end{figure}

The Gelfand-Tsetlin graph has a representation theoretic origin, vertices at level $N$ parametrize the irreducible characters of $\mathbb{U}(N)$, the $N$-dimensional unitary group. The edges correspond to how an irreducible representation of $\mathbb{U}(N)$ when restricted to $\mathbb{U}(N-1)$ splits into irreducibles (since when restricted it becomes reducible).

It is a remarkable Theorem, originally due to Edrei \cite{Edrei} (in an equivalent form) and Voiculescu \cite{Voiculescu} (see also \cite{VershikKerov},  \cite{OkounkovOlshanskiJack}, \cite{BorodinOlshanskiBoundary}) that the boundary of the Gelfand-Tsetlin graph can be described explicitly. In order to do this, we need some more definitions.

 Let $\mathbb{R}^{\infty}_+$ denote the product of countably many copies of $\mathbb{R}_+$ and also write $\mathbb{R}_+^{4\infty+2}=\mathbb{R}^{\infty}_+\times\mathbb{R}^{\infty}_+\times\mathbb{R}^{\infty}_+\times\mathbb{R}^{\infty}_+\times\mathbb{R}_+\times\mathbb{R}_+$, equipped with the product topology. Then, consider $\Omega \subset \mathbb{R}_+^{4\infty+2}$ the set of sextuples,
\begin{align*}
\omega=(\alpha^+,\beta^+;\alpha^-,\beta^-;\delta^+,\delta^-),
\end{align*}
so that,
\begin{align*}
\alpha^{\pm}=(\alpha_1^{\pm}\ge \alpha_2^{\pm}\ge \cdots \ge 0)\in \mathbb{R}^{\infty}_+ \ &\textnormal{and} \ \beta^{\pm}=(\beta_1^{\pm}\ge \beta_2^{\pm}\ge \cdots \ge 0)\in \mathbb{R}^{\infty}_+, \\
\sum_{i=1}^{\infty}(\alpha_i^{\pm}+\beta_i^{\pm})\le \delta^{\pm} \ &\textnormal{and} \ \beta_1^{+}+\beta_1^-\le 1.
\end{align*}
Note that, $\Omega$ is locally compact under the induced topology. Then set,
\begin{align*}
\gamma^{\pm}=\delta^{\pm}-\sum_{i=1}^{\infty}(\alpha_i^{\pm}+\beta_i^{\pm})
\end{align*}
and observe that $\gamma^{\pm}\ge 0$ and define for $u \in \mathbb{C}^*$ and $\omega\in \Omega$ the function $\Phi\left(\omega;u\right)$ given by,
\begin{align*}
\Phi\left(\omega;u\right)=e^{\gamma^+(u-1)+\gamma^{-}(u^{-1}-1)}\prod_{i=1}^{\infty}\frac{1+\beta_i^+(u-1)}{1-\alpha_i^+(u-1)}\frac{1+\beta_1^ -(u^{-1}-1)}{1-\alpha_i^-(u^{-1}-1)}.
\end{align*}
As its poles do not accumulate to $1$, the function $\Phi\left(\omega;u\right)$ is holomorphic in a neighbourhood of the unit circle $\mathbb{T}=\{u\in \mathbb{C}:|u|=1\}$. For $n \in \mathbb{Z}$, we denote its Laurent coefficient by,
\begin{align*}
\phi_{n}(\omega)=\frac{1}{2\pi \mathsf{i}}\oint_{\mathbb{T}}\Phi\left(\omega;u\right)\frac{du}{u^{n+1}}
\end{align*}
and for a signature $\nu=(\nu_1,\cdots,\nu_N)$ of length $N$ define,
\begin{align*}
\phi_{\nu}(\omega)=\det\left(\phi_{\nu_i-i+j}(\omega)\right)^N_{i,j=1}
\end{align*}
and the Markov kernels $\Lambda_N^{\infty}:\Omega \to V_N$ by,
\begin{align*}
\Lambda_N^{\infty}(\omega,\nu)=\mathsf{dim}_N(\nu)\phi_{\nu}(\omega) \ ,\forall { N \ge 1, \omega \in \Omega,\nu=(\nu_1,\cdots,\nu_N)},
\end{align*}
where $\mathsf{dim}_N(\nu)=\prod_{1\le i <j \le N}\frac{\nu_i-\nu_j+j-i}{j-i}$ is the dimension of a level-$N$ signature $\nu=(\nu_1,\cdots,\nu_N)$.

Then, $\Omega$ is the \textit{Feller boundary} of the Gelfand-Tsetlin graph with link from $\Omega$ to level $N$ given by $\Lambda_N^{\infty}$ (for the Feller property in particular, see Corollary 2.11 of \cite{BorodinOlshanskiBoundary}).

\paragraph{BC-type branching graph} This graph has a representation theoretic origin as well. For certain values of its multiplicities it describes the branching of the irreducible characters of the Lie groups $\{\textnormal{SO}(2N+1)\}_{N\ge 1}$, $\{\textnormal{Sp}(2N)\}_{N\ge 1}$ and $\{\textnormal{O}(2N)\}_{N\ge 1}$. Vertices at level $N$ are now given by \textit{positive} signatures of length $N$, namely $\kappa=(\kappa_1\ge \cdots \ge\kappa_N \ge 0)$ with two vertices $\kappa=(\kappa_1\ge \cdots \ge\kappa_N \ge 0)$ and $\nu=(\nu_1\ge \cdots \ge\nu_{N+1} \ge 0)$ being connected by an edge and we write $\kappa\prec_{\textnormal{BC}}\nu$, if and only if there exists an "intermediate" signature $\rho=(\rho_1\ge \cdots \ge\rho_N\ge0)$ such that,
\begin{align*}
\rho_1\ge \kappa_1 \ge \cdots \ge \rho_N \ge \kappa_N \textnormal{ and } \nu_1\ge \rho_1 \ge \cdots \ge \rho_N \ge \nu_{N+1},
\end{align*}
or equivalently in our notation, under the transformation described previously in the context of the Gelfand-Tsetlin graph $\kappa \mapsto y$, $\rho \mapsto z$ and $\nu \mapsto x$,
\begin{align*}
y\in W^{N,N}(z) \textnormal{ and } z\in W^{N,N+1}(x).
\end{align*}
The multiplicities are now given in terms of certain coefficients associated to the multivariate $\theta=1$ Jacobi polynomials, so they depend on two real parameters $a,b$; see Section $3$ of \cite{Cuenca} for more details. It is a theorem, originally of Okounkov and Olshanski \cite{OkounkovOlshanskiJacobi}, but also see Section 3 of \cite{Cuenca} for a nice exposition and a proof of the Feller property, that the boundary of the BC-type branching graph can be parametrized by the space $\Omega_{\textnormal{BC}}$ (which \textit{does not} depend on $a,b$) being the closed subspace of $\mathbb{R}_+^{2\infty+1}$ consisting of points $\omega_{\textnormal{BC}}=(\alpha^{\textnormal{BC}},\beta^{\textnormal{BC}},\delta^{\textnormal{BC}})$ such that,
\begin{align*}
\alpha^{\textnormal{BC}}=(\alpha_1^{\textnormal{BC}}\ge \alpha_2^{\textnormal{BC}}\ge \cdots \ge 0)\in \mathbb{R}^{\infty}_+ \ , \ \beta^{\textnormal{BC}}=(1\ge\beta_1^{\textnormal{BC}}\ge \beta_2^{\textnormal{BC}}\ge \cdots \ge 0)\in \mathbb{R}^{\infty}_+  \textnormal { and }
\sum_{i=1}^{\infty}(\alpha_i^{\textnormal{\textnormal{BC}}}+\beta_i^{\textnormal{BC}})\le \delta^{\textnormal{BC}}.
\end{align*}

\paragraph{Alternating construction and generalized BC-type branching graph} This corresponds to the construction of a general random growth process with a wall in later sections, which we call the \textit{alternating construction}. The graph consists of the vertices and edges of the BC-type branching graph described above, but with more general multiplicities (in particular the BC-type graph is a special case). Of course, these multiplicities are not arbitrary but arise from the \textit{consistent dynamics} between Karlin-McGregor semigroups namely (\ref{KMintertwining1}) and (\ref{KMintertwining2}), or from the branching rules for multivariate Karlin-McGregor polynomials. These polynomials arise as follows: to any family $\{Q_i\}_{i\ge 1}$ of orthogonal polynomials in $[0,\infty)$ we can associate a multivariate determinantal version, indexed by $\nu \in W^N$, by $\det\left(Q_{\nu_i}(x_j)\right)_{i,j=1}^N/\det\left(x^{i-1}_j\right)_{i,j=1}^N$. Then using the branching rules for these polynomials, see Section \ref{multivariatepolynomialssection} (also the Appendix) one can obtain the following general multiplicities. In the notation of this paper, if we define the following (positive) \textit{weight functions} by,
\begin{align*}
(z,y) \in W^{N,N}(\mathbb{N}) \ ,\ w_{N,N}(z,y)=\prod_{i=1}^{N}\pi(y_i),\\
(x,z) \in W^{N,N+1}(\mathbb{N}) \ ,\ w_{N,N+1}(x,z)=\prod_{i=1}^{N}\hat{\pi}(z_i),
\end{align*}
then, the multiplicities are given by,
\begin{align*}
\mathsf{mult}(x,y)=\sum_{z:y\in W^{N,N}(z),z \in W^{N,N+1}(x)}^{}w_{N,N}(z,y)w_{N,N+1}(x,z).
\end{align*}
Moreover, observe that for $x\in W^{N+1}$, its dimension in the branching graph is given by the harmonic function from (\ref*{alternatingharmonic1}),
\begin{align*}
\mathsf{dim}_{N+1}(x)=h_{N,N+1}(x)=(\Lambda_{N,N+1}\Lambda_{N,N}\cdots \Lambda_{1,1}\textbf{1})(x).
\end{align*}
Under a certain \textit{positive definiteness} assumption, which admittedly can be non-trivial to check (see Appendix), our results from sections \ref{sectioncoherentmeasures} and \ref{sectionevolutionoperators} \textit{partially} describe the boundary of these graphs. More precisely, we first introduce a large class of coherent measures for this graph in Section \ref{sectioncoherentmeasures}. Combining Lemma \ref{factorizationforcoherent} (see also subsection \ref{subsectionpositivityofcoherent}) and the results of subsection \ref*{SubsectionFactorizationImpliesExtremality} in the Appendix (under this positive definiteness assumption, see Remark \ref*{RemarkPositiveDefiniteness}) we show that these coherent sequences are actually extremal.

\begin{rmk}
The projective chains associated to all these graphs can also be recast in terms of branching coefficients of certain families of (symmetric) functions (see Appendix).
\end{rmk}

\section{Examples of consistent dynamics}\label{sectionExamples}

Before giving any examples we first record some useful facts and fix notation. Throughout this section we will denote the Vandermonde determinant by,
\begin{align*}
\Delta_n(x)=\prod_{1\le i < j\le n}^{}(x_j-x_i),\  x\in W^n(I).
\end{align*}
We will consider a difference operator $\mathsf{L}$ that is the generator of a birth and death chain or a bilateral birth and death chain with quadratic rates, i.e. so that with $x \in I$,
\begin{align*}
\mathsf{L}=(ax^2+bx+c)\nabla+(ax^2+\bar{b}x+\bar{c})\bar{\nabla}.
\end{align*}
We assume throughout that,  $a,b,c,\bar{b},\bar{c}$ are such that the rates are positive namely,
\begin{align*}
\lambda(x)=(ax^2+bx+c)>0 \textnormal{ and } \mu(x)=(ax^2+\bar{b}x+\bar{c})>0 \ , \ \forall x \in I
\end{align*}
and that conditions (\ref{birthdeathcondition1}),(\ref{birthdeathcondition2}) or (\ref{bilateralcondition1}),(\ref{bilateralcondition2}),(\ref{bilateralcondition3}) and (\ref{bilateralcondition4}) respectively are always satisfied for all chains considered in this subsection. Finally, observe that we need the leading coefficient $a$ to be the same for both rates.

Now, with all these requirements in place a direct computation (see e.g. \cite{Doumerc} Proposition 6.2.1) gives that,
\begin{align*}
\sum_{i=1}^{n}\mathsf{L}_{x_i}\Delta_n(x)=\bigg(a\frac{n(n-1)(n-2)}{3}+(b-\bar{b})\frac{n(n-1)}{2}\bigg) \Delta_n(x), \ x \in W^n(I),
\end{align*}
where each $\mathsf{L}_{x_i}$ is a copy of the difference operator $\mathsf{L}$ acting in the $x_i$ variable. So that, we can $h$-transform $n$ independent copies of $\mathsf{L}$-chains by $\Delta_n$ to stay in $W^{n}(I)$.

Define the following operator from functions on $W^{n}(I)$ to functions on $W^{n+1}(I)$, these when viewed as Markov kernels from $W^{n+1}(I)$ to $W^{n}(I)$ are the links that appear in the Gelfand-Tsetlin graph by,
\begin{align*}
(\mathfrak{L}^{\textnormal{Vnd}}_{n\to n+1}f)(x)=\frac{n!}{\Delta_{n+1}(x)}\sum_{y\in W^{n,n+1}(x)}^{}\Delta_n(y)f(y), \ x \in W^n(I).
\end{align*}
Then, we have the following lemma.
\begin{lem}
For $n\ge 1$, the kernels $\mathfrak{L}^{\textnormal{Vnd}}_{n\to n+1}$ are Feller.
\end{lem}
\begin{proof}
In order to prove this, it suffices to apply the kernel $\mathfrak{L}^{\textnormal{Vnd}}_{n\to n+1}$ to a delta function $\delta_y$ and show that $(\mathfrak{L}^{Vnd}_{n\to n+1}\delta_y)(x)$ vanishes as $x \to \infty$. This can be readily checked, see e.g. Proposition 3.3 of \cite{BorodinOlshanski} for the details.
\end{proof}

Now, suppose that we are given as above the following birth and death (reflecting at the origin, $\mu(0)=0$) or bilateral ($I=\mathbb{Z}$) chain with generator $\mathcal{D}=\mathsf{L}$ so that,
\begin{align*}
\mathcal{D}(x,y)=\begin{cases}
ax^2+bx+c\ \ & y=x+1\\
-(ax^2+bx+c)-(ax^2+\bar{b}x+\bar{c}) \ \ & y=x \\
ax^2+\bar{b}x+\bar{c} \ \ & y=x-1
\end{cases}.
\end{align*}
Then, a simple computation gives us that the $h$-transform of the chain with generator $\mathcal{D}$ by the strictly positive function $\hat{\pi}^{-1}$ (which is an eigenfunction with eigenvalue $b-\bar{b}$) is the (reflecting) birth and death (or bilateral birth and death chain) with generator $\tilde{\mathcal{D}}$ with rates,
\begin{align*}
\tilde{\mathcal{D}}(x,y)=\begin{cases}
a(x+1)^2+b(x+1)+c\ \ & y=x+1\\
-(a(x+1)^2+b(x+1)+c)-(ax^2+\bar{b}x+\bar{c}) \ \ & y=x \\
ax^2+\bar{b}x+\bar{c} \ \ & y=x-1
\end{cases}.
\end{align*}

Moreover, we define $\left(P^{\Delta_{n+1}}_{n+1}(t);t\ge 0\right)$ to be the Karlin-McGregor semigroup of $n+1$ copies of $\mathcal{D}$-chains $h$-transformed by $\Delta_{n+1}$ and similarly $\left(\tilde{P}_n^{\Delta_n}(t);t\ge 0\right)$ to be the Karlin-McGregor semigroup of $n$ copies of $\tilde{\mathcal{D}}$-chains $h$-transformed by $\Delta_n$. Then as expected these possess the Feller property.

\begin{lem}
The semigroups $\left(P^{\Delta_{n+1}}_{n+1}(t);t \ge 0\right)$ and $\left(\tilde{P}_n^{\Delta_n}(t);t\ge 0\right)$ are Feller for any $n$.
\end{lem}
\begin{proof}
This, again easily follows by applying these semigroups to $\delta_y$ and making use of the fact that the one dimensional transition densities in the Karlin-McGregor semigroups satisfy $p_t(x_i,y_j),\tilde{p}_t(x_i,y_j)\to 0$ as $x_i\to \infty$ (or $-\infty$) and that moreover $\Delta_n(x) \ge 1$.
\end{proof}

Then, Theorem \ref*{Master1} and in particular, the intertwining relation (\ref{MasterIntertwining1eq2}) immediately gives the following proposition which is the main result of this subsection.
\begin{prop}\label{VandermondeIntertwinings}
 $P^{\Delta_{n+1}}_{n+1}(t)\mathfrak{L}^{\textnormal{Vnd}}_{n\to n+1}f=\mathfrak{L}^{\textnormal{Vnd}}_{n\to n+1}\tilde{P}_n^{\Delta_n}(t)f$, for  $n\ge 1$, $f\in C_0(W^{n}(I))$ and $t\ge 0$ .
\end{prop}

We now, list several interesting applications of this proposition. For $a=b=\bar{b}=0$ and $c,\bar{c}>0$, we obtain the well known intertwining between non-colliding (asymmetric) continuous time random walks. 

For a linear birth and death chain, i.e. with a parameter $\theta>0$ and rates given by,
\begin{align*}
\mathcal{D}_{\theta}(x,y)=\begin{cases}
x+\theta\ \ & y=x+1\\
-2x-\theta \ \ & y=x \\
x \ \ & y=x-1
\end{cases},
\end{align*}
we get that,
\begin{align*}
\tilde{\mathcal{D}}_{\theta}(x,y)=\begin{cases}
x+\theta+1\ \ & y=x+1\\
-2x-\theta-1 \ \ & y=x \\
x \ \ & y=x-1
\end{cases}.
\end{align*}
Observe that $\tilde{\mathcal{D}}_{\theta}=\mathcal{D}_{\theta+1}$,the birth rate or equivalently the drift to the right of the preceding level increased by $1$, in particular such a construction cannot be iterated indefinitely. Moreover, Proposition \ref{VandermondeIntertwinings} gives the discrete analogue of the intertwining between $n+1$ non-intersecting squared Bessel processes of dimension $d$ abbreviated by $\textnormal{BESQ}(d)$ and $n$ non-intersecting $\textnormal{BESQ}(d+2)$ (see Proposition 3.14 of \cite{InterlacingDiffusions}). 

We can also consider the Meixner process, which is the analogue of the Laguerre diffusion (a $\textnormal{BESQ}$ process with a restoring drift towards the origin, for certain choices of parameters the modulus of Ornstein Uhlenbeck processes, for more details see \cite{InterlacingDiffusions}) with parameters $r,\theta>0$,
\begin{align*}
\mathcal{D}^{\textnormal{Me}}_{r,\theta}(x,y)=\begin{cases}
r(x+\theta)\ \ & y=x+1\\
-r(x+\theta)-(r+1)x \ \ & y=x \\
(r+1)x \ \ & y=x-1
\end{cases},
\end{align*}
then,
\begin{align*}
\tilde{\mathcal{D}}^{\textnormal{Me}}_{r,\theta}(x,y)=\begin{cases}
r(x+\theta+1)\ \ & y=x+1\\
-r(x+\theta+1)-(r+1)x \ \ & y=x \\
(r+1)x \ \ & y=x-1
\end{cases}.
\end{align*}
Similarly as above, we see that $\tilde{\mathcal{D}}^{\textnormal{Me}}_{r,\theta}=\mathcal{D}^{\textnormal{Me}}_{r,\theta+1}$, so that the drift to the right has decreased from the preceding level, or when thinking in terms of the couplings, the birth rate for the autonomous particles is greater by 1.

As a final example, we consider the bilateral birth and death chain studied by Borodin and Olshanski in \cite{BorodinOlshanski}, with $u,u',v,v'\in \mathbb{C}$ satisfying the assumptions in section $5.1$ therein (these ensure well-posedness and non-explosion, moreover note that although the parameters can be complex, they really correspond to $4$ free real parameters ),
\begin{align*}
\mathcal{D}^{\mathbb{U}(\infty)}_{u,u',v,v'}(x,y)=\begin{cases}
(x-u)(x-u')\ \ & y=x+1\\
-(x-u)(x-u')-(x+v)(x+v')\ \ & y=x \\
(x+v)(x+v') \ \ & y=x-1
\end{cases},
\end{align*}
so that,
\begin{align*}
\tilde{\mathcal{D}}^{\mathbb{U}(\infty)}_{u,u',v,v'}(x,y)=\begin{cases}
(x+1-u)(x+1-u')\ \ & y=x+1\\
-(x+1-u)(x+1-u')-(x+v)(x+v')\ \ & y=x \\
(x+v)(x+v') \ \ & y=x-1
\end{cases}.
\end{align*}

As before note the following fact, $\tilde{\mathcal{D}}^{\mathbb{U}(\infty)}_{u,u',v,v'}=\mathcal{D}^{\mathbb{U}(\infty)}_{u-1,u'-1,v,v'}$. Then, Proposition \ref*{VandermondeIntertwinings} above immediately gives as a corollary Theorem 6.1 of \cite{BorodinOlshanski}. This along
with the \textit{method of intertwiners} (see Subsection \ref*{subsectionintertwiners}), constructs a Feller process on the boundary $\Omega$ of the Gelfand-Tsetlin graph. We note that the motivation behind these specific rates stems from the fact that the corresponding semigroups leave invariant the so called $zw$-measures, which are consistent measures on the Gelfand-Tsetlin graph and whose decomposition into extremal coherent measures is the \textit{problem of harmonic analysis} on the infinite dimensional unitary group $\mathbb{U}(\infty)$ (for more details see \cite{OlshanskiHarmonic}).

\paragraph{Characterization of Vandermonde intertwiners for push-block dynamics} The choice of quadratic rates might have seemed a bit arbitrary. We now proceed to briefly explain its significance. More specifically,  we show that in order for the Vandermonde links,
\begin{align*}
(\mathfrak{L}^{\textnormal{Vnd}}_{n\to n+1}f)(x)=\frac{n!}{\Delta_{n+1}(x)}\sum_{y\in W^{n,n+1}(x)}^{}\Delta_n(y)f(y), \ x \in W^n(I),
\end{align*}
to intertwine the levels of the (type-A) Gelfand-Tsetlin pattern valued process moving according to the push-block dynamics considered in the two-level couplings of this paper (or c.f. equality (\ref{MasterIntertwining1eq2}), for the semigroups for each level to be consistent with these links) then, the rates $\lambda(x)$ and $\mu(x)$ must be quadratic functions of $x\in I$, with coefficients related as shown below in displays (\ref{vandermonderate1}) and (\ref{vandermonderate2}).  

 Starting from the process of the two first levels, taking values in $W^{1,2}$, it is easy to see from relation (\ref{MasterIntertwining1eq1}) that we need $\hat{\pi}^{-1}$ to be an eigenfunction of the generator $\hat{\mathcal{D}}$ for the resulting intertwining kernel to be given by,
 \begin{align*}
 \frac{1}{x_2-x_1}\textbf{1}(x_1\le y <x_2).
 \end{align*}
 Since $\hat{\mathcal{D}}$ is reversible with respect  to $\hat{\pi}$, this requirement is equivalent to the fact that the transpose (when viewed as an infinite matrix indexed by $\mathbb{N}$ or $\mathbb{Z}$) of $\hat{\mathcal{D}}$ minus some constant times the identity matrix ( $\hat{\mathcal{D}}^{\textnormal{T}}-\textnormal{const}\times Id$) is the generator of a birth and death (or bilateral) chain with rates,
\begin{align*}
\tilde{\mathcal{D}}(x,y)=\begin{cases}
\tilde{\lambda}(x)=\lambda(x+1) \ \ & y=x+1\\
-\lambda(x+1)-\mu(x) \ \ & y=x \\
\tilde{\mu}(x)=\mu(x) \ \ & y=x-1
\end{cases}.
\end{align*}
Now this is true, if and only if, for some constant $c_0$,
\begin{align*}
\lambda(x+1)+\mu(x)-\mu(x+1)-\lambda(x)=c_0 \ , \forall x \in \mathbb{Z}.
\end{align*}
Then, moving to the two-level process taking values in $W^{2,3}$, an analogous consideration (with $\lambda,\mu$ still denoting the birth and death rates of the chains on the $2^{nd}$ level) leads to the extra requirement that,
\begin{align*}
\lambda(x+2)+\mu(x)-\mu(x+1)-\lambda(x+1)=c_1 \ ,\forall x \in \mathbb{Z}.
\end{align*}
These two conditions are now sufficient to characterize $\lambda(x)$ and $\mu(x)$ as quadratic functions of $x$. Let $\Lambda(x)=(\nabla \lambda)(x)$ and $M(x)=(\nabla \mu)(x)$ so that,
\begin{align*}
\Lambda(x)-M(x)&=c_0,\\
\Lambda(x+1)-M(x)&=c_1.
\end{align*}
Observe that, with $n\ge 0$ we have
$\Lambda(x+n)-M(x)=\Lambda(x+n)-\Lambda(x+n-1)+\Lambda(x+n-1)-M(x)=c_1-c_0+\Lambda(x+n-1)-M(x)=\cdots=n(c_1-c_0)+c_0$ and similarly for $n$ negative. Thus,
\begin{align*}
\Lambda(y)&=y(c_1-c_0)+c_0+M(0),\\
M(y)&=y(c_1-c_0)+M(0).
\end{align*}
From these, we obtain,
\begin{align}
\mu(y)&=\frac{y(y-1)}{2}(c_1-c_0)+(\mu(1)-\mu(0))y+\mu(0) \label{vandermonderate1},\\
\lambda(y)&=\frac{y(y-1)}{2}(c_1-c_0)+(c_0+\mu(1)-\mu(0))y+\lambda(0)\label{vandermonderate2},
\end{align}
where $\lambda(1)=c_0+\mu(1)-\mu(0)+\lambda(0)$ so that $c_0=\mu(1)-\mu(0)+\lambda(0)-\lambda(1)$ and $\lambda(2)=c_1+\lambda(0)+\mu(1)+\mu(0)$ so that $c_1=\lambda(2)-\lambda(0)-\mu(1)-\mu(0)$. 

In conclusion, at an \textit{algebraic level} we need to specify five positive real parameters $\lambda(0),\lambda(1),\lambda(2),\mu(0),\mu(1)$. Of course in addition to that, we need $\mu(y),\lambda(y)>0$ and that the well-posedness conditions (\ref{birthdeathcondition1}), (\ref{birthdeathcondition2}) or (\ref{bilateralcondition1}), (\ref{bilateralcondition2}), (\ref{bilateralcondition3}) and (\ref{bilateralcondition4}) respectively are satisfied. Finally, if we denote by $\mathsf{r}^+_1(x),\mathsf{r}^-_1(x)$ the quadratic birth and death rates respectively of the single chain at level $1$ then, the rates for the chains at level $n$ are given by $\mathsf{r}^+_n(x)=\mathsf{r}^+_1(x+n-1)$ and $\mathsf{r}^-_n(x)=\mathsf{r}^-_1(x)$.

\paragraph{Intertwining relations for dynamics on BC-type graphs} The aim of this subsection is to prove Proposition \ref{CuencaIntertwining} below, first proven as Theorem 5.1 in \cite{Cuenca} by Cuenca. We will use the following notation. In all that follows, $I=\mathbb{N}$ and we define,
\begin{align*}
W_{\textnormal{BC}}^{n,n+1}=\{(x,y)\in (W^{n+1},W^n): \exists \ z\in W^n, \ \textnormal{ such that }  \ y\in W^{n,n}(z), z \in  W^{n,n+1}(x)\}.
\end{align*}
Analogously to $W^{n,n+1}$ we define $W_{\textnormal{BC}}^{n,n+1}(x)$ for $x \in W^{n+1}$. 

Moreover, we consider the following rates for a $\mathcal{D}$-chain depending on 4 parameters $(u,u',a,b)$, which satisfy the relations (5.1) in \cite{Cuenca} (these conditions ensure positivity of the rates and non-explosivity of the chain and will not be recalled since they don't affect the essentially algebraic arguments below), with $\beta_{u,u'}$ denoting the \textit{birth rate} and $\delta_{u,u'}$ the \textit{death rate}, for $x \in \mathbb{N}$,
\begin{align*}
\beta_{u,u'}(x)&=\frac{(x+a+b+1)(x+a+1)(x-u)(x-u')}{(2x+a+b+1)(2x+a+b+2)},\\
\delta_{u,u'}(x)&=\frac{x(x+b)(x+u+a+b+1)(x+u'+a+b+1)}{(2x+a+b+1)(2x+a+b)}.
\end{align*}
The parameters $(a,b)$ will be fixed throughout so we suppress any dependence of $\beta_{u,u'}$ and $\delta_{u,u'}$ on them. Now, define the following functions $\mathsf{f},\mathsf{g},\mathsf{B}$ again depending on $(a,b)$ but \textit{not} on $u$ and $u'$ by,
\begin{align*}
\mathsf{f}(x)&=\frac{(2x+a+b+2)x!\Gamma(x+b+1)}{\Gamma(x+a+b+2)\Gamma(x+a+2)},\  x \in \mathbb{N},\\
\mathsf{g}(y)&=\frac{(2y+a+b+1)\Gamma(y+a+b+1)\Gamma(y+a+1)}{y!\Gamma(y+b+1)}, \ y \in \mathbb{N},\\
\mathsf{B}(x,y)&=\frac{1}{2}\mathsf{f}(x)\mathsf{g}(y), \ x ,y\in \mathbb{N}.
\end{align*}
Define the function $\mathsf{F}_n$ on $W^n$ by,
\begin{align*}
\mathsf{F}_n(x)=\prod_{i<j}^{n}\left(\left(x_j+\frac{a+b+1}{2}\right)^2-\left(x_i+\frac{a+b+1}{2}\right)^2\right).
\end{align*}
Furthermore, define the following kernel,
\begin{align*}
(\mathfrak{L}^{\textbf{BC}}_{n\to n+1}f)(x)=\frac{2^nn!\Gamma(n+a+1)}{\Gamma(a+1)\mathsf{F}_{n+1}(x)}\sum_{y\in W_{\textnormal{BC}}^{n,n+1}(x)}^{}\mathsf{F}_n(y)f(y)\sum_{z:y\in W^{n,n}(z), z \in  W^{n,n+1}(x)}^{}\prod_{i=1}^{n}\mathsf{B}(z_i,y_i), \ x \in W^{n+1}.
\end{align*}
Then, we have the following lemma originally proven in \cite{Cuenca}.
\begin{lem}
For $n\ge 1$, the kernels $\mathfrak{L}^{\textbf{BC}}_{n\to n+1}$ are Feller.
\end{lem}
\begin{proof}
The fact that these are Markov, i.e. correctly normalized, comes from the branching of the normalized Jacobi polynomials, see Section 3 of \cite{Cuenca}. Moreover, to show that they are Feller, it again suffices to check it for a delta function; however the situation is a bit more involved than for $\mathfrak{L}^{Vnd}_{n\to n+1}$, see Proposition 3.1 of \cite{Cuenca} for the details.
\end{proof}

 Denote by $\left(P_n^{u,u'}(t);t\ge 0\right)$ the Karlin-McGregor semigroup associated to  $n$ $\mathcal{D}$-chains with birth and death rates $\beta_{u,u'}$ and $\delta_{u,u'}$ respectively. It can be checked, see Lemma 4.12 of \cite{Cuenca}, that $\mathsf{F}_n$ is a positive eigenfunction of $P_n^{u,u'}(t)$ with eigenvalue $e^{c_nt}$, where $c_n=\frac{n(n-1)(n-2)}{3}-\frac{n(n-1)}{2}(u+u'+b)$ (this fact can also be obtained via iteration of the results below) so that in particular, we can define the honest Markov semigroup $\left(P_n^{u,u',\mathsf{F}_n}(t);t\ge 0\right)$ given by the $h$-transform of $\left(P_n^{u,u'}(t);t \ge 0\right)$ by $\mathsf{F}_n$. Then, under the assumptions on $(u,u',a,b)$ referred to above we have:
 
 \begin{lem}
For $n\ge1$, the semigroups $\left(P_n^{u,u',\mathsf{F}_n}(t);t \ge 0\right)$ are Feller.
 \end{lem}
 \begin{proof}
This as before, immediately follows from the fact that the one dimensional transition densities that go in the Karlin-McGregor semigroups are Feller along with the fact that $F_n(x) \ge 1$.
 \end{proof}
 
 Finally, the following proposition along with the method of intertwiners immediately gives a Feller process on the boundary $\Omega_{\textnormal{BC}}$ of the type-BC branching graph.

\begin{prop}\label{CuencaIntertwining}
$P^{u+1,u'+1,\mathsf{F}_{n+1}}_{n+1}(t)\mathfrak{L}^{\textbf{BC}}_{n\to n+1}f=\mathfrak{L}^{\textbf{BC}}_{n\to n+1}P_n^{u,u',\mathsf{F}_n}(t)f$, for $n\ge 1$, $f\in C_0(W^{n})$, $t\ge 0$.
\end{prop}

 Again, the interest in these specific rates stems from the fact that they preserve the so called $z$-measures, which are the analogues of the $zw$-measures mentioned previously, for the problem of harmonic analysis on infinite dimensional BC-type groups. For more details and a complete study of the $z$-measures see the recent paper \cite{Cuenca}.

Proposition \ref{CuencaIntertwining} will follow from the two relations given in Proposition \ref{BCIntertwinings}  below, which reveal a "hidden" dynamic on "intermediate signatures" (see Okounkov's paper \cite{OkounkovIntermediate} and the references therein for more about these). In fact, this is exactly the dynamic followed by the projection on the even levels ($x^{(i,i)}$ in our notation), if one constructs a symplectic Gelfand-Tsetlin pattern valued process, that links (on odd levels) the semigroups $\left(P^{u+1,u'+1,\mathsf{F}_{n+1}}_{n+1}(t);t\ge 0\right)$ and $\left(P^{u,u',\mathsf{F}_{n}}_{n}(t);t\ge 0\right)$ and initializes it according to a Gibbs measure (see Proposition \ref{CoherentdynamicsproptypeB}). 

Some more definitions are necessary. Let the functions $\mathsf{\hat{F}}_n$ and $\mathsf{\bar{F}}_{n+1}$ on $W^n$ and $W^{n+1}$ respectively be given by,
\begin{align*}
\mathsf{\hat{F}}_n(z)&=\sum_{y\in W^{n,n}(z)}^{}\prod_{i=1}^{n}\mathsf{g}(y_i)\mathsf{F}_n(y),\ z \in W^n,\\
\mathsf{\bar{F}}_{n+1}(x)&=\sum_{z\in W^{n,n+1}(x)}^{}\prod_{i=1}^{n}\mathsf{f}(z_i)\mathsf{\hat{F}}_n(z), \ z\in W^{n+1}.
\end{align*}
Moreover, we define the following Markov kernels $\mathfrak{L}^{\textbf{BC}}_{n,n}$ from $W^n$ to $W^n$, and $\mathfrak{L}^{\textbf{BC}}_{n,n+1}$ from $W^{n+1}$ to $W^{n}$ respectively by,
\begin{align*}
(\mathfrak{L}^{\textbf{BC}}_{n,n}f)(z)&=\frac{1}{\mathsf{\hat{F}}_n(z)}\sum_{y\in W^{n,n}(z)}^{}f(y)\prod_{i=1}^{n}\mathsf{g}(y_i)\mathsf{F}_n(y),\ z \in W^n,\\
(\mathfrak{L}^{\textbf{BC}}_{n,n+1}f)(x)&=\frac{1}{\mathsf{\bar{F}}_{n+1}(x)}\sum_{z\in W^{n,n+1}(x)}^{}f(z)\prod_{i=1}^{n}\mathsf{f}(z_i)\mathsf{\hat{F}}_n(z), \ x\in W^{n+1}.
\end{align*}
Observe that, we have the composition property,
\begin{align*}
\mathfrak{L}^{\textbf{BC}}_{n\to n+1}=\mathfrak{L}^{\textbf{BC}}_{n,n+1}\circ\mathfrak{L}^{\textbf{BC}}_{n,n}
\end{align*}
and from comparing the two expressions in order to get the right normalization constant, we have,
\begin{align*}
\mathsf{\bar{F}}_{n+1}(x)=\frac{\Gamma(a+1)}{n!\Gamma(n+a+1)}\mathsf{F}_{n+1}(x), \ x \in W^{n+1}.
\end{align*}
Finally, we denote by $\left(\mathsf{P}_n^{u,u',\mathsf{\hat{F}}_n}(t);t\ge 0\right)$ the Karlin-McGregor semigroup associated with $n$ birth and death chains with \textit{birth rate},
\begin{align*}
\frac{\mathsf{g}(x)\beta_{u,u'}(x)}{\mathsf{g}(x+1)}, \ x \in \mathbb{N},
\end{align*}
and \textit{death rate},
\begin{align*}
\frac{\mathsf{g}(x+1)\delta_{u,u'}(x+1)}{\mathsf{g}(x)}, \ x \in \mathbb{ N},
\end{align*}
that is moreover Doob's $h$-transformed by $\mathsf{\hat{F}}_n$. The fact that, this is indeed an eigenfunction of $n$ copies of such birth and death chains follows (recursively) from relation (\ref{BCIntertwinings1}) of Proposition \ref{BCIntertwinings} below. This semigroup, $\left(\mathsf{P}_n^{u,u',\mathsf{\hat{F}}_n}(t);t\ge 0\right)$ that is driving the evolution of $n$ non-intersecting birth and death chains is the "hidden" dynamic alluded to above. Now, Proposition \ref{CuencaIntertwining} is an immediate consequence of the following result.
\begin{prop}\label{BCIntertwinings} For $n\ge 1$ and $t\ge 0$, we have the intertwining relations:
\begin{align}
\mathsf{P}_n^{u,u',\mathsf{\hat{F}}_n}(t)\mathfrak{L}^{\textbf{BC}}_{n,n}&=\mathfrak{L}^{\textbf{BC}}_{n, n}P_n^{u,u',\mathsf{F}_n}(t)\label{BCIntertwinings1}\\
P^{u+1,u'+1,\mathsf{F}_{n+1}}_{n+1}(t)\mathfrak{L}^{\textbf{BC}}_{n,n+1}&=\mathfrak{L}^{\textbf{BC}}_{n, n+1}\mathsf{P}_n^{u,u',\mathsf{\hat{F}}_n}(t) \label{BCIntertwinings2}
\end{align}
\end{prop}
\begin{proof}
In the setting of Theorem \ref{Master2}, with $n$ and $n$ particles on each of the $X$ and $Y$ levels, we choose the $\mathcal{D}$-chains (the $Y$-level) to have rates given by,
\begin{align*}
\lambda(x)&=\frac{\mathsf{g}(x+1)\delta_{u,u'}(x+1)}{\mathsf{g}(x)}, \ x\in \mathbb{N},\\
\mu(x)&=\frac{\mathsf{g}(x-1)\beta_{u,u'}(x-1)}{\mathsf{g}(x)}, \ x \in \mathbb{N}.
\end{align*}
Observe that, by performing an $h$-transform by the function $\prod_{i=1}^{n}\pi^{-1}(y_i)\mathsf{g}(y_i)\mathsf{F}_n(y)$ the evolution of these chains is driven by $\left(P_n^{u,u',\mathsf{F}_n}(t);t\ge 0\right)$ and thus we obtain  (\ref{BCIntertwinings1}).

Now, in the setting of Theorem \ref{Master1} with $n$ and $n+1$ particles, let the $\mathcal{D}$-chains (the $X$-level in this new setting, note that these are different from the ones considered above) have birth rate given by $\beta_{u+1,u'+1}(x)$ and death rate given by $\delta_{u+1,u'+1}(x)$. Then, performing an $h$-transform of the corresponding $n$ $\mathcal{\hat{D}}$-chains (the $Y$-level) by the function $\prod_{i=1}^{n}\hat{\pi}^{-1}(z_i)\mathsf{f}(z_i)\mathsf{\hat{F}}_n(z)$ we obtain (\ref{BCIntertwinings2}) after we observe the following compatibility relations between the jump rates,
\begin{align}
\beta_{u+1,u'+1}(x+1)\frac{\mathsf{f}(x+1)}{\mathsf{f}(x)}&=\mu(x+1)=\beta_{u,u'}(x)\frac{\mathsf{g}(x)}{\mathsf{g}(x+1)} \label{compatibilityrates1}, \ x \in \mathbb{N},\\
\delta_{u+1,u'+1}(x)\frac{\mathsf{f}(x-1)}{\mathsf{f}(x)}&=\lambda(x)=\delta_{u,u'}(x+1)\frac{\mathsf{g}(x+1)}{\mathsf{g}(x)} \label{compatibilityrates2}, \ x \in \mathbb{N}.
\end{align}
To see that these relations hold, first note that by making use of $\Gamma(x+1)=x\Gamma(x)$ we obtain the following, for ratios of $\mathsf{f}$ and $\mathsf{g}$ at consecutive points,
\begin{align*}
\frac{\mathsf{f}(x+1)}{\mathsf{f}(x)}&=\frac{(2x+a+b+4)(x+1)(x+b+1)}{(2x+a+b+2)(x+a+b+2)(x+a+2)}, \ x \in \mathbb{N},\\
\frac{\mathsf{g}(x+1)}{\mathsf{g}(x)}&=\frac{(2x+a+b+4)(x+1)(x+b+1)}{(2x+a+b+2)(x+a+b+2)(x+a+2)}, \ x \in \mathbb{N}.
\end{align*}
Similarly, we have relations for ratios of the birth and death rates with different parameters,
\begin{align*}
\frac{\beta_{u+1,u'+1}(x+1)}{\beta_{u,u'}(x)}&=\frac{(x+a+b+2)(x+a+2)(2x+a+b+1)(2x+a+b+2)}{(2x+a+b+3)(2x+a+b+4)(x+a+b+1)(x+a+1)}, \ x \in \mathbb{N},\\
\frac{\delta_{u+1,u'+1}(x)}{\delta_{u,u'}(x+1)}&=\frac{x(x+b)(2x+a+b+3)(2x+a+b+2)}{(2x+a+b+1)(2x+a+b)(x+1)(x+1+b)}, \ x \in \mathbb{N}.
\end{align*}
Using these, (\ref{compatibilityrates1}) and (\ref{compatibilityrates2}) can be readily checked and we are done.
\end{proof}

\paragraph{Strong Stationary Duals} Here, we briefly point out the close connection to the theory of Strong Stationary Duality. The setup is that of $W^{1,1}$ and with $I=\mathbb{N}$ i.e. $X$ and $Y$ each consist of a single particle. We define the cumulative of $\pi$, by $\sum_{0\le y \le x}^{}\pi(y)$. Thus, Theorem \ref*{Master2} gives that if a $\hat{\mathcal{D}}$-chain ($X$-level) is being kept above a (reflecting) $\mathcal{D}$-chain ($Y$-level) via the push-block mechanism we have been studying; then if the $\mathcal{D}$-chain is distributed initially according to $\frac{\pi(y)}{\sum_{0\le y \le x}^{}\pi(y)}1(y\le x)$, the evolution of the projection on the $X$-particle is that of a $\hat{\mathcal{D}}$-chain $h$-transformed by $\sum_{0\le y \le x}^{}\pi(y)$ (see for example Theorem 5.5 of \cite{DiaconisFill} in the discrete time case).

\begin{rmk}
Using the results of this paper, we can also obtain Theorem 2.3 of \cite{WarrenWindridge} which studies a process in a symplectic Gelfand-Tsetlin pattern. Similarly, we could consider pure-birth chains, which strictly speaking are not covered by the results of this work, since we assume that we are dealing with positive death rates $(\mu(x))_{x\in I}>0$, but with entirely analogous considerations Theorem 2.1 of \cite{WarrenWindridge} can also be recovered by the methods that are presented here.
\end{rmk}

\section{Birth and death chain orthogonal polynomials}\label{sectionorthogonalpolynomials}
 We will now recall the well known connection, between the probabilistic world of birth and death chains and the analytic counterpart of their associated orthogonal polynomials on the positive half line. The main references for this subsection will be the seminal papers of Karlin and McGregor, \cite{KarlinMcGregorClassification} and \cite{KarlinMcGregorDifferential}, where most of the theory was laid out. From here onwards, we fix a birth and death chain with generator $\mathcal{D}$, reflecting at $0$, with rates $(\lambda(\cdot),\mu(\cdot))$ and symmetrizing measure $\pi(\cdot)$. As usual we shall also denote by $\hat{\mathcal{D}}$ the generator of its Siegmund dual (which is absorbed at $-1$) with rates $(\hat{\lambda}(\cdot),\hat{\mu}(\cdot))$ and symmetrizing measure $\hat{\pi}(\cdot)$. We will also, often write $\lambda_k$ for $\lambda(k)$, $\pi_k$ for $\pi(k)$ and so on.
 
  We begin by defining the following family of polynomials $\{Q_i\}_{i\ge 0}$ by the three term recursion (note that $\mu(0)=0$),
\begin{align*}
Q_0(x)&=1,\\
-xQ_0(x)&=-(\lambda(0)+\mu(0))Q_0(x)+\lambda(0)Q_1(x),\\
-xQ_n(x)&=\mu(n)Q_{n-1}(x)-(\lambda(n)+\mu(n))Q_n(x)+\lambda(n)Q_{n+1}(x).
\end{align*} 
Then, see Theorem 1 of \cite{KarlinMcGregorDifferential}, there exists at least one measure $\mathfrak{w}(dx)$ on $\mathbb{R}_+=\{0 \le x <\infty\}$, such that these polynomials are orthogonal with respect to $\mathfrak{w}(dx)$,  so that,
\begin{align*}
\int_{0}^{\infty}Q_i(x)Q_j(x)\mathfrak{w}(dx)=\frac{1}{\pi(j)}\delta_{ij}.
\end{align*}
For such a \textit{moment problem} to be \textit{determinate}, so that the measure $\mathfrak{w}$ is unique, when $\mu(0)=0$, as in the case of the $\mathcal{D}$-chain, it suffices for the backwards equation to have a unique solution (see \cite{KarlinMcGregorDifferential}, Theorem 14). In particular, any of the conditions in section \ref{sectioncoalescing} that ensure the well-posedness of the backwards equation are enough for determinacy. In such a case, we have that,
\begin{align*}
\mathfrak{w}(dx)=d\mathfrak{w}(x),
\end{align*}
where $\mathfrak{w}(x)$ is a real valued non-decreasing function, being continuous on the left, with $\mathfrak{w}(x)=0$ for $x\le0$ and $\mathfrak{w}(\infty)=1$. We will denote by $\mathfrak{I}=[I^-,I^+]\subset [0,\infty]$ the support, $\textnormal{supp} (\mathfrak{w})$ of the measure $\mathfrak{w}$. These orthogonal polynomials provide the following spectral expansion of the transition density (see \cite{KarlinMcGregorDifferential} for example) that will be useful for us,
\begin{align}\label{spectralexpansiontransition}
p_t(i,j)=\pi(j)\int_{0}^{\infty}e^{-tx}Q_i(x)Q_j(x)d\mathfrak{w}(x).
\end{align}

\begin{rmk}[Explicit examples]\label{ExplicitPolynomialsRemark}
We give some simple examples for $\lambda(\cdot),\mu(\cdot)$ such that the corresponding orthogonal polynomials $Q_i(x)$ and spectral measures $\mathfrak{w}(dx)$ are explicit. The following rates were considered by Cerenzia and Kuan in \cite{CerenziaKuan}, depending on two real parameters $\alpha,\beta>-1$:
\begin{align*}
\lambda(n)&=\frac{n+\alpha+\beta+1}{2n+\alpha+\beta+1}\frac{2(n+\alpha+1)}{2n+\alpha+\beta+2},\\
\mu(n)&=\frac{n+\beta}{2n+\alpha+\beta}\frac{2n}{2n+\alpha+\beta+1}.
\end{align*}
They give rise to the Jacobi polynomials $Q_i^{\alpha,\beta}(x)$ orthogonal in $[0,2]$ with respect to the weight $\mathfrak{w}(dx)=\mathfrak{w}_{\alpha,\beta}(dx)$:
\begin{align*}
\mathfrak{w}_{\alpha,\beta}(dx)=Z(\alpha,\beta)x^{\alpha}(2-x)^{\beta}dx,
\end{align*}
for some normalization constant $Z(\alpha,\beta)$. For $\alpha=\beta=-\frac{1}{2}$ these specialize to the model studied by Borodin and Kuan in \cite{BorodinKuan} related to $O(\infty)$ while for $-\alpha=\beta=\frac{1}{2}$ they specialize to the model studied by Cerenzia \cite{Cerenzia} related to $Sp(\infty)$. The associated orthogonal polynomials in both cases are the Chebyshev (which are specializations of the Jacobi polynomials).

The following examples are taken from Section 3.1 of \cite{Schoutens}. Further explicit examples can be found in the references therein. In all cases $\mathfrak{w}(dx)$ is actually a discrete measure with atoms of mass $w(n)$ at the positive integers $n\in \mathbb{N}$. The associated (2+1)-dimensional growth and decay processes were not studied before.

The so called $M/M/\infty$ queue is a birth and death process with rates and symmetrizing measure given by:
\begin{align*}
\lambda(n)\equiv\lambda ,\ \ \mu(n)=\mu n,  \ \ \pi(n)=\left(\frac{\lambda}{\mu}\right)^n/n!.
\end{align*}
The orthogonal polynomials associated to it are $Q_n(x)=C_n\left(\frac{x}{\mu};\frac{\lambda}{\mu}\right)$ where $C_n(x;a)$ are the Charlier polynomials defined by:
\begin{align*}
0=C_{n-1}(x;a)+(x-a-n)C_{n}(x;a)+aC_{n+1}(x;a),
\end{align*}
with $C_0(x;a)=1,C_{-1}(x;a)$. These are orthogonal with respect to the Poisson distribution:
\begin{align*}
w(n)=\frac{a^ne^{-a}}{n!}; \ n=0,1,\cdots
\end{align*}
More precisely:
\begin{align*}
\sum_{n=0}^{\infty}C_i\left(n;\frac{\lambda}{\mu}\right)C_j\left(n;\frac{\lambda}{\mu}\right)\frac{\left(\lambda/\mu\right)^n}{n!}e^{-\left(\lambda/\mu\right)}=\frac{\delta_{ij}}{\pi(j)}.
\end{align*}
Moreover, the polynomials associated to the birth and death chain with linear rates:
\begin{align*}
\lambda(n)=(n+\beta)\lambda, \ \ \mu(n)=n\mu,
\end{align*}
are the so called Meixner polynomials (see Section 1.3.2 in \cite{Schoutens}). Finally for finite birth and death chains one can also obtain the dual Hahn, Krawtchouk and Racah polynomials, see \cite{Schoutens}.
\end{rmk}

We also define the polynomials $\{\hat{Q}_i\}_{i\ge0}$, associated to the dual chain with generator $\hat{\mathcal{D}}$. So that, in the recursion above the rates $\left(\lambda,\mu\right)$ are replaced by the dual rates $\left(\hat{\lambda},\hat{\mu}\right)$. In particular, the new recursion is given by,
\begin{align*}
-x\hat{Q}_n(x)=\lambda(n)\hat{Q}_n(x)-(\mu(n+1)+\lambda(n))\hat{Q}_n(x)+\mu(n+1)\hat{Q}_{n+1}(x).
\end{align*}
Since now $\hat{\mu}(0)=\lambda(0)>0$ (recall the $\hat{\mathcal{D}}$-chain gets absorbed at $-1$), in order for the moment problem to be determinate, we need to further require (see \cite{KarlinMcGregorClassification} or \cite{KarlinMcGregorDifferential}),
\begin{align*}
\sum_{j=0}^{\infty}\hat{\pi}(j)\left(\sum_{k=0}^{j}\pi(k)\right)^2=\infty.
\end{align*}
A sufficient, easier to check in practise, condition for this is (see unnumbered display after equation (0.11) on page 367 of \cite{KarlinMcGregorClassification}),
\begin{align*}
\sum_{n=1}^{\infty}\frac{1}{\hat{\mu}(n)}=\sum_{n=1}^{\infty}\frac{1}{\lambda(n)}=\infty.
\end{align*}
In such a case (of determinacy), the dual spectral measure, denoted by $d\hat{\mathfrak{w}}(x)$, satisfies the following key relation (see \cite{KarlinMcGregorClassification} section 6),
\begin{align*}
d\hat{\mathfrak{w}}(x)=\frac{xd\mathfrak{w}(x)}{\lambda(0)}.
\end{align*}
So that in particular, the supports are equal $\textnormal{supp}(\hat{\mathfrak{w}})=\textnormal{supp}(\mathfrak{w})=\mathfrak{I}$. From now on, we assume that both moment problems are determinate with unique solutions $\mathfrak{w}(\cdot)$ and $\hat{\mathfrak{w}}(\cdot)$ respectively.

We will denote by $\langle \cdot , \cdot \rangle_{\mathsf{m}}$ the $L^2$ inner product with measure $\mathsf{m}$. By Corollary 2.3.3 of \cite{Akhiezer} we obtain that, since the solution of the moment problem is unique, the polynomials $\{Q_i\}_{i\ge0}$ are dense in $L^2\left(\mathfrak{I},\mathfrak{w}\right)$. Hence, for $f\in L^2\left(\mathfrak{I},w\right)$,
\begin{align}
f=\sum_{k=0}^{\infty}\langle Q_k , f \rangle_{\mathfrak{w}}Q_k\pi(k),
\end{align}
with the series converging in the $L^2\left(\mathfrak{I},\mathfrak{w}\right)$ sense. We will furthermore, mainly be interested in functions $f\in L^2$ for which this expansion actually converges uniformly. By Theorem 6 of \cite{KarlinMcGregorDifferential}, we have that for $f(x)=Q_i(x)e^{-tx}$ the series,
\begin{align}\label{uniformconvergence}
f(x)=\sum_{k=0}^{\infty}\langle Q_k , f \rangle_{\mathfrak{w}}Q_k(x)\pi(k),
\end{align}
converges absolutely, for $t\ge 0$ and all $x \in \mathbb{C}$, the convergence being uniform over every bounded set, $\big \{(t,x): 0 \le t \le T \textnormal{ and } |x|\le R \big\}$. Moreover, we have the following bound,
\begin{align*}
\sum_{k=0}^{\infty}|\langle Q_k , f \rangle_{\mathfrak{w}}||Q_k(x)|\pi(k)\le e^{t|x|}Q_i\left(-|x|\right).
\end{align*}
It can be easily seen that, in a little bit more generality, the series (\ref{uniformconvergence}) above converges uniformly on compact sets of $(t,x)$ with $0\le t \le T$ and $|x|\le R$, for $f(x)=p_m(x)e^{-tx}$ where $p_m(x)$ is any polynomial of degree $m$. In particular, if $p_m(x)=\sum_{i=0}^{m}c_i^mQ_i(x)$ the previous bound becomes,
\begin{align*}
\sum_{k=0}^{\infty}|\langle Q_k , f \rangle_{\mathfrak{w}}||Q_k(x)|\pi(k)\le e^{t|x|}\sum_{i=0}^{m}|c_i^m|Q_i\left(-|x|\right).
\end{align*}
\begin{rmk}
Under certain regularity and growth assumptions on $w$ at $I^-$  and $\infty$, one can prove that the series in display $(\ref{uniformconvergence})$  converges uniformly on compact intervals of $\mathfrak{I}$ for \textit{bounded variation} functions $f$, such that their derivative satisfies a certain integrability condition (see in particular Theorem 4.17.2 of \cite{Freud} and the references therein).
\end{rmk} 
We need one more property of functions of the form $f(x)=p_m(x)e^{-tx}$, namely that,
\begin{align*}
\langle Q_n , f \rangle_{\mathfrak{w}} \to 0 \textnormal{ as } n \to \infty.
\end{align*}
This can be seen as follows, by writing $p_m(x)=\sum_{i=0}^{m}\tilde{c}_i^mQ_i(x)\pi_i$ we have by (\ref{spectralexpansiontransition}),
\begin{align*}
\langle Q_n , f \rangle_{\mathfrak{w}}=\sum_{i=0}^{m}\tilde{c}_i^mp_t(n,i) \to 0 \textnormal{ as } n \to \infty,
\end{align*}
since, for any $i \in \mathbb{N}$ and $t\ge 0$, $p_t(n,i) \to 0$ as $n \to \infty$.
Finally, we have the following relations between $\{Q_i\}_{i\ge 0}$ and their duals $\{\hat{Q}_i\}_{i\ge 0}$ (see \cite{Vandoorn} or section 6 of \cite{KarlinMcGregorClassification}),
\begin{align}
\pi_{n+1}Q_{n+1}(x)&=\hat{Q}_{n+1}(x)-\hat{Q}_n(x),\\
-x\hat{Q}_n(x)&=\lambda_n\pi_n(Q_{n+1}(x)-Q_n(x)).
\end{align}

We are now in a position to prove the following result, which is modelled on and is essentially a generalization of Proposition 3.1 of \cite{CerenziaKuan}. It is what makes all subsequent calculations work.
\begin{prop}\label{polynomialrelations}
\begin{enumerate}
\item $\sum_{i=0}^{n}\pi_iQ_i(x)=\hat{Q}_n(x)$.
\item $\sum_{k=0}^{n-1}\hat{\pi}_k\hat{Q}_k(x)=\frac{\lambda_0}{x}(1-Q_n(x))$.
\item $\langle \hat{Q}_n , f(0)-f \rangle_{\mathfrak{w}}=\sum_{k=n+1}^{\infty}\langle \pi_kQ_k , f \rangle_{\mathfrak{w}}$ \ , \ for $f$ in $L^{2}\left(\mathfrak{I},\mathfrak{w}\right)$ so that series (\ref{uniformconvergence}) converges pointwise at 0.
\item $\sum_{k=n}^{\infty}\langle \hat{\pi}_k\hat{Q}_k , f \rangle_{\hat{\mathfrak{w}}}=\langle \hat{Q}_n , f \rangle_{\mathfrak{w}}$ \ , \ for $f$ in $L^{2}\left(\mathfrak{I},\mathfrak{w}\right)$ so that $\langle Q_n , f \rangle_{\mathfrak{w}}\to 0$.
\end{enumerate}
\end{prop}
\begin{proof}
To prove (1), note that by telescoping $\sum_{i=1}^{n}\pi_iQ_i(x)=\hat{Q}_n(x)-\hat{Q}_0(x)=\hat{Q}_n(x)-1$ and that $\pi_0Q_0(x)=1$. To prove (2), first note,
\begin{align*}
\hat{\pi}(n)\hat{Q}_n(x)=\lambda(0)\left(\frac{Q_{n+1}(x)-Q_{n}(x)}{-x}\right)
\end{align*}
and hence by summing,
\begin{align*}
\sum_{k=0}^{n-1}\hat{\pi}(k)\hat{Q}_k(x)=\lambda(0)\left(\frac{Q_{n}(x)-1}{-x}\right).
\end{align*}
To prove (3), observe that $\langle \hat{Q}_n , 1 \rangle_{\mathfrak{w}}=\langle \sum_{i=0}^{n}\pi_iQ_i , 1 \rangle_{\mathfrak{w}}=1$. Also note that $Q_{n+1}(0)=Q_n(0)=\cdots=Q_0(0)=1$ and thus from (1) we also get $\hat{Q}_n(0)=\sum_{k=0}^{n}\pi_k$. Moreover, by convergence of the orthogonal decomposition at $0$ we have,
\begin{align*}
\langle \hat{Q}_n , f(0) \rangle_{\mathfrak{w}}&=f(0)=\sum_{k=0}^{\infty}\langle Q_k , f \rangle_{\mathfrak{w}}Q_k(0)\pi(k)=\sum_{k=0}^{\infty}\langle \pi_kQ_k , f \rangle_{\mathfrak{w}},\\
\langle \hat{Q}_n , f \rangle_{\mathfrak{w}}&=\sum_{k=0}^{n}\langle \pi_k Q_k , f \rangle_{\mathfrak{w}}.
\end{align*}
Subtracting the two we get (3). In order to prove (4), we have,
\begin{align*}
\sum_{k=0}^{n-1}\langle \hat{\pi}_k\hat{Q}_k , f \rangle_{\hat{\mathfrak{w}}}=\langle \lambda(0)\left(\frac{Q_{n}(x)-1}{-x}\right) , f \rangle_{\hat{\mathfrak{w}}}=\langle 1- Q_n , f \rangle_{\mathfrak{w}}\overset{n \to \infty}{\longrightarrow}\langle 1 , f \rangle_{\mathfrak{w}},
\end{align*}
where the limit holds by our assumption that $\langle Q_n , f \rangle_{\mathfrak{w}}\to 0$. Hence,
\begin{align*}
\sum_{k=n}^{\infty}\langle \hat{\pi}_k\hat{Q}_k , f \rangle_{\hat{\mathfrak{w}}}=\langle Q_n , f \rangle_{\mathfrak{w}}.
\end{align*}
\end{proof}

\section{Branching rules for multivariate Karlin-McGregor polynomials}\label{multivariatepolynomialssection}
For $\nu \in W^{n}$, we define the $n$-variate Karlin-McGregor polynomials by, with $x=\left(x_1, \cdots , x_n\right)$ in $\mathbb{R}^n$,
\begin{align}
\mathfrak{Q}_{\nu}(x)&=\frac{\det\left(Q_{\nu_i}(x_j)\right)_{i,j=1}^n}{\det\left(x_j^{i-1}\right)_{i,j=1}^n}=\frac{\det\left(Q_{\nu_i}(x_j)\right)_{i,j=1}^n}{\Delta_n(x)},\\
\hat{\mathfrak{Q}}_{\nu}(x)&=\frac{\det\left(\hat{Q}_{\nu_i}(x_j)\right)_{i,j=1}^n}{\det\left(x_j^{i-1}\right)_{i,j=1}^n}=\frac{\det\left(\hat{Q}_{\nu_i}(x_j)\right)_{i,j=1}^n}{\Delta_n(x)}.
\end{align}
The polynomial systems, $\det\left(Q_{\nu_i}(x_j)\right)_{i,j=1}^n$ and $\det\left(\hat{Q}_{\nu_i}(x_j)\right)_{i,j=1}^n$ were first introduced by Karlin and McGregor, in their seminal study of intersection probabilities of birth and death chains in \cite{KarlinMcGregorCoincidence}. Some further properties were also presented in their subsequent brief note \cite{KarlinMcGregorDeterminants}. Observe that in particular, these multivariate polynomials are orthogonal in the continuous chamber $0\le x_1\le x_2 \cdots\le x_n$ (denoted $x \in W^{n}([0,\infty)$), with respect to the weights $\prod_{i=1}^{n}d\mathfrak{w}(x_i)\Delta_n^2(x)$ and $\prod_{i=1}^{n}d\hat{\mathfrak{w}}(x_i)\Delta_n^2(x)$ respectively.

Most importantly, we have the following \textit{two-step} branching rules. The calculations below are in fact more or less implicitly done on page 1116 of \cite{KarlinMcGregorDeterminants}.

\begin{prop}\label{propositionbranching}
\begin{align}
\left.\frac{\det\left(Q_{\nu_i}(x_j)\right)_{i,j=1}^{n+1}}{\det\left(x_j^{i-1}\right)_{i,j=1}^{n+1}}\right \vert_{x_1=0}&=\frac{(-1)^n}{\lambda_0^n}\sum_{k\in W^{n,n+1}(\nu)}^{}\prod_{i=1}^{n}\hat{\pi}_{k_i}\frac{\det\left(\hat{Q}_{k_i}(x_{j+1})\right)_{i,j=1}^n}{\det\left(x_{j+1}^{i-1}\right)_{i,j=1}^n}\label{branching1},\\
\frac{\det\left(\hat{Q}_{\nu_i}(x_j)\right)_{i,j=1}^n}{\det\left(x_j^{i-1}\right)_{i,j=1}^n}&=\sum_{k\in W^{n,n}(\nu)}^{}\prod_{i=1}^{n}\pi_{k_i}\frac{\det\left(Q_{k_i}(x_j)\right)_{i,j=1}^n}{\det\left(x_j^{i-1}\right)_{i,j=1}^n}\label{branching2}.
\end{align}
\end{prop}
\begin{proof}
We prove (\ref*{branching1}) first. In the first equality below we make use of the fact that $Q_k(0)=1$ and in the last one we make use of the relation  $-x\hat{Q}(x)=\lambda_n\pi_n(Q_{n+1}(x)-Q_n(x))$.
\begin{align*}
\left.\frac{\det\left(Q_{\nu_i}(x_j)\right)_{i,j=1}^{n+1}}{\det\left(x_j^{i-1}\right)_{i,j=1}^{n+1}}\right \vert_{x_1=0}&=\frac{\det\left(Q_{\nu_{i+1}}(x_{j+1})-Q_{\nu_{i}}(x_{j+1})\right)_{i,j=1}^{n}}{\det\left(x_{j+1}^{i-1}\right)_{i,j=1}^{n}\prod_{j=1}^{n}x_{j+1}}\\
&=\frac{\det\left(\frac{Q_{\nu_{i+1}}(x_{j+1})-Q_{\nu_{i}}(x_{j+1})}{x_{j+1}}\right)_{i,j=1}^{n}}{\det\left(x_{j+1}^{i-1}\right)_{i,j=1}^{n}}\\
&=\sum_{k\in W^{n,n+1}(\nu)}^{}\frac{\det\left(\frac{Q_{k_{i}+1}(x_{j+1})-Q_{k_{i}}(x_{j+1})}{x_{j+1}}\right)_{i,j=1}^{n}}{\det\left(x_{j+1}^{i-1}\right)_{i,j=1}^{n}}\\
&=\sum_{k\in W^{n,n+1}(\nu)}^{}\frac{\det\left(-\frac{\hat{\pi}_{k_i}}{\lambda_0}\hat{Q}_{k_i}(x_{j+1})\right)_{i,j=1}^n}{\det\left(x_{j+1}^{i-1}\right)_{i,j=1}^n}.
\end{align*}
In order to prove (\ref*{branching2}) we make use of part 1 of Proposition \ref{polynomialrelations} so that (where we set $\nu_0+1=0$),
\begin{align*}
\frac{\det\left(\hat{Q}_{\nu_i}(x_j)\right)_{i,j=1}^n}{\det\left(x_j^{i-1}\right)_{i,j=1}^n}=\frac{\det\left(\sum_{k_i=0}^{\nu_i}\pi_{k_i}Q_{k_i}(x_j)\right)_{i,j=1}^n}{\det\left(x_j^{i-1}\right)_{i,j=1}^n}=\frac{\det\left(\sum_{k_i=\nu_{i-1}+1}^{\nu_i}\pi_{k_i}Q_{k_i}(x_j)\right)_{i,j=1}^n}{\det\left(x_j^{i-1}\right)_{i,j=1}^n}.
\end{align*}
Note that, we can finally pull out the sum $\sum_{k\in W^{n,n}(\nu)}^{}$ by multilinearity.
\end{proof}

Consider the functions,
\begin{align}
h_{n,n+1}(\nu,x)&=(-1)^{\binom{n}{2}}\lambda_0^{\binom{n}{2}}\mathfrak{Q}_{\nu}(x), \textnormal{ for } \nu \in W^{n+1}(\mathbb{N}) \textnormal{ and } x \in W^{n+1}([0,\infty)), \label{normalizedKMpoly1} \\
h_{n,n}(\nu,x)&=(-1)^{\binom{n-1}{2}}\lambda_0^{\binom{n-1}{2}}\hat{\mathfrak{Q}}_{\nu}(x), \textnormal{ for } \nu \in W^{n}(\mathbb{N}) \textnormal{ and } x \in W^{n}([0,\infty)) \label{normalizedKMpoly2}
\end{align}
and define, for $\nu$ in $W^{n+1}$ and $W^n$ respectively,
\begin{align}
h_{n,n+1}(\nu)&=h_{n,n+1}(\nu,0),\\
h_{n,n}(\nu)&=h_{n,n}(\nu,0).
\end{align}

Now, from the branching rules and our original intertwining relations from section \ref{sectionintertwinins} we prove the following:
\begin{prop}
$h_{n,n+1}$ and $h_{n,n}$ are positive harmonic functions for $n+1$ independent copies of $\mathcal{D}$-chains and $n$ independent copies of $\hat{\mathcal{D}}$-chains in $W^{n+1}$ and $W^n$ respectively.
\end{prop}
\begin{proof}
Observe that, from the branching relations we get,
\begin{align*}
h_{n,n}(\nu)&=\left(\Lambda_{n,n}h_{n-1,n}\right)(\nu), \textnormal{ for } \nu \in W^n(\mathbb{N}),\\
h_{n,n+1}(\nu)&=\left(\Lambda_{n,n+1}h_{n,n}\right)(\nu), \textnormal{ for } \nu \in W^{n+1}(\mathbb{N})
\end{align*}
and hence,
\begin{align*}
h_{n,n}(\nu)&=\left(\Lambda_{n,n}\Lambda_{n-1,n}\cdots\Lambda_{1,1}\textbf{1}\right)(\nu), \textnormal{ for } \nu \in W^n(\mathbb{N}),\\
h_{n,n+1}(\nu)&=\left(\Lambda_{n,n+1}\Lambda_{n,n}\cdots\Lambda_{1,1}\textbf{1}\right)(\nu), \textnormal{ for } \nu \in W^{n+1}(\mathbb{N}).
\end{align*}
From relations (\ref{alternatingharmonic1}) and (\ref{alternatingharmonic2}) and the discussion around them, the conclusion is now evident.
\end{proof}

\begin{rmk}
In fact, some more general eigenfunction relations exist. For $x_1<x_2<\cdots<x_n \le 0$ we have,
\begin{align*}
(-1)^{\frac{n(n-1)}{2}}\det\left(Q_{\nu_i}(x_j)\right)_{i,j=1}^n>0
\end{align*}
and it can be readily checked that this is an eigenfunction of $n$ independent $\mathcal{D}$-chains in $W^n$ (see for example displays (19) and (30) respectively in \cite{KarlinMcGregorCoincidence}). These eigenfunctions can also be used to construct consistent dynamics and we will pursue this elsewhere.
\end{rmk}

Before continuing, we briefly recall some well known determinantal conditions for interlacing, namely representations of $\textbf{1}\left(k\in W^{n,n+1}(\nu)\right)$ and $\textbf{1}\left(k\in W^{n,n}(\nu)\right)$ in terms of determinants. First of all, we have the following identity for $\textbf{1}\left(y\in W^{n,n}(x)\right)$,
\begin{align*}
\textbf{1}\left(y_1 \le x_1 < y_2 \le \cdots \le x_n\right)=\det\left(\textbf{1}(y_i\le x_j)\right)_{i,j=1}^n.
\end{align*}
From this, by swapping $x$'s and $y$'s and putting $y_{n+1}=\infty$, or by declaring $y_{n+1}=\textnormal{virt}$, a virtual variable and agreeing that $\textbf{1}(x\le \textnormal{virt})=1$, we obtain the analogous identity for $\textbf{1}\left(y\in W^{n,n+1}(x)\right)$,
\begin{align*}
\textbf{1}\left(x_1 \le y_1 < x_2 \le \cdots \le y_n < x_{n+1}\right)=\det\left(\textbf{1}(x_i\le y_j)\right)_{i,j=1}^{n+1}.
\end{align*}
This can also be written as, after subtracting the last column from each of the rest,
\begin{align*}
\textbf{1}\left(x_1 \le y_1 < x_2 \le \cdots < x_{n+1}\right)=\det\left(f_{i,j}\right)_{i,j=1}^{n+1},
\end{align*}
where,
\begin{align*}
f_{i,j}=\begin{cases}
-\textbf{1}(x_i> y_j) &\textnormal{ if } j \le n\\
1 &\textnormal{ if } j=n+1
\end{cases}.
\end{align*} 
Thus, if we define,
\begin{align*}
\phi(i,j)&=\pi_i \textbf{1}(i\le j),\\
\hat{\phi}(i,j)&=-\hat{\pi}_i \textbf{1}(i < j),\\
\hat{\phi}(\textnormal{virt},j)&=1,
\end{align*}
then from Proposition \ref{propositionbranching}, it is easy to see that:
\begin{cor}
The kernels $\Lambda^{h_{n,n+1}}_{n,n+1}(\nu,\cdot)$ and $\Lambda^{h_{n,n}}_{n,n}(\nu,\cdot)$, for any $\nu \in W^{n+1}$ and $\nu \in W^n$ respectively, that are defined by,
\begin{align}
\Lambda^{h_{n,n}}_{n,n+1}(\nu,k)&=\textbf{1}\left(k\in W^{n,n+1}(\nu)\right)\frac{\prod^n_{i=1}\hat{\pi}_{k_i}h_{n,n}(k)}{h_{n,n+1}(\nu)}=\frac{\det\left(\hat{\phi}(k_i,\nu_j)\right)^{n+1}_{i,j=1}h_{n,n}(k)}{h_{n,n+1}(\nu)},\\
\Lambda^{h_{n-1,n}}_{n,n}(\nu,k)&=\textbf{1}\left(k\in W^{n,n}(\nu)\right)\frac{\prod^n_{i=1}\pi_{k_i}h_{n-1,n}(k)}{h_{n,n}(\nu)}=\frac{\det\left(\phi(k_i,\nu_j)\right)^{n}_{i,j=1}h_{n-1,n}(k)}{h_{n,n}(\nu)},
\end{align}
are Markov.
\end{cor}
Finally, denoting by $\left(P_{n+1}^{h_{n,n+1}}(t);t\ge 0\right)$ and $\left(\hat{P}_n^{h_{n,n}}(t);t\ge 0\right)$ the Karlin-McGregor semigroups associated with $n+1$ $\mathcal{D}$-chains and $n$ $\hat{\mathcal{D}}$-chains, $h$-transformed by $h_{n,n+1}$ and $h_{n,n}$ respectively, we immediately get the following corollary of Theorems \ref{Master1} and \ref{Master2}.

\begin{cor}\label{corinter} For $t\ge0$, we have the intertwining relations,
\begin{align}
P_{n+1}^{h_{n,n+1}}(t)\Lambda^{h_{n,n}}_{n,n+1}&=\Lambda^{h_{n,n}}_{n,n+1}\hat{P}_n^{h_{n,n}}(t),\\
\hat{P}_n^{h_{n,n}}(t)\Lambda^{h_{n-1,n}}_{n,n}&=\Lambda^{h_{n-1,n}}_{n,n}P_{n}^{h_{n-1,n}}(t).
\end{align}
\end{cor}

\section{Coherent measures}\label{sectioncoherentmeasures}
We now move on towards defining, in displays (\ref{coherentmeasure1})  and (\ref{coherentmeasure2}), measures denoted by $\mathcal{M}^{\psi}_{n,n+1}$ and $\mathcal{M}^{\psi}_{n,n}$, depending on a function $\psi$, that are coherent with respect to the Markov links $\Lambda^{h_{n,n}}_{n,n+1}$ and $\Lambda^{h_{n-1,n}}_{n,n}$. We first need some definitions and technical preliminaries. 

Consider the Taylor remainder for a function $f$, that is $(n-1)$-times differentiable at $0$, given by,
\begin{align*}
\mathsf{R}_n^f(x)=\begin{cases}
f(x) & n \le 0\\
f(x)-\sum_{k=0}^{n-1}\frac{f^{(k)}(0)}{k!}x^k \ & n\ge 1
\end{cases}.
\end{align*} 

Now, define for $f$ that is $(j-n)$-times or $(j-(n+1))$-times continuously differentiable at $0$ respectively, the following functions on $\mathbb{N}$, $\Psi^{n,n+1}_{n+1-j}(\cdot)$ and $\Psi^{n,n}_{n-j}(\cdot)$ (their dependence on $f$ will be suppressed),
\begin{align}
\Psi^{n,n+1}_{n+1-j}(i)&=\langle \pi_iQ_i , (-x)^{n+1-j}\mathsf{R}_{j-(n+1)}^f \rangle_{\mathfrak{w}}, \ i \in \mathbb{N},\\
\Psi^{n,n}_{n-j}(i)&=\langle \hat{\pi}_i\hat{Q}_i , (-x)^{n-j}\mathsf{R}_{j-n}^f \rangle_{\hat{\mathfrak{w}}}, \ i \in \mathbb{N}.
\end{align}

We also define the discrete convolution for functions $h_1,h_2:\mathbb{N}\times \mathbb{N} \to \mathbb{C}$ and $h_3:\mathbb{N}\to \mathbb{C}$ as follows,
\begin{align*}
\left(h_1* h_2\right)(u,v)&=\sum_{k \ge 0}^{}h_1(u,k)h_2(k,v),\\
\left(h_1* h_3\right)(u)&=\sum_{k \ge 0}^{}h_1(u,k)h_3(k).
\end{align*}

The lemma below states that, alternating convolutions of $\phi$ and $\hat{\phi}$ with $\Psi^{n,n}_{n-j}$ and $\Psi^{n,n+1}_{n+1-j}$ respectively are nicely consistent. This will be useful in the computations performed in Proposition \ref{coherencyprop} that proves that the measures introduced below are indeed coherent.
\begin{lem}Assume that $f(x)=p(x)e^{-tx}$, where $p(x)$ is a fixed polynomial of arbitrary degree. Then, we have,
\begin{enumerate} 
\item $\left( \phi * \Psi^{n,n}_{n-j} \right)(i)=\Psi^{n-1,n}_{n-j}(i)$.
\item $\left( \hat{\phi} * \Psi^{n,n+1}_{n+1-j} \right)(i)=-\lambda_0\Psi^{n,n}_{n-j}(i)$.
\end{enumerate}
\end{lem}
\begin{proof}
To prove (1) note,
\begin{align*}
\left( \phi * \Psi^{n,n}_{n-j} \right)(i)&=\sum_{k \ge 0}^{}\pi_i \textbf{1}(i \le k)\Psi^{n,n}_{n-j}(k)\\
&=\sum_{k \ge i}^{}\pi_i \langle \hat{\pi}_k\hat{Q}_k , (-x)^{n-j}\mathsf{R}_{j-n}^f \rangle_{\hat{\mathfrak{w}}}\\
&=\pi_i\langle Q_i , (-x)^{n-j}\mathsf{R}_{j-n}^f \rangle_{\mathfrak{w}}=\Psi^{n-1,n}_{n-j}(i).
\end{align*}
Now to prove (2) first observe that with $\mathsf{T}^f_m(x)=(-x)^{-m}\mathsf{R}_m^f(x)$ then, $\mathsf{T}^f_m(0)=\lim_{x \to 0} \mathsf{T}^f_m(x)=\frac{f^{(m)}(0)}{m!}(-1)^m$ and so,
\begin{align*}
\left( \hat{\phi} * \Psi^{n,n+1}_{n+1-j} \right)(i)&=-\sum_{k \ge 0}^{}\hat{\pi}_i \textbf{1}(i < k)\Psi^{n,n+1}_{n+1-j}(k)\\
&=\sum_{k \ge i+1}^{}\hat{\pi}_i \langle \pi_kQ_k , (-x)^{n+1-j}\mathsf{R}_{j-(n+1)}^f \rangle_{\mathfrak{w}}\\
&=-\hat{\pi}_i\langle \hat{Q}_i , ((-x)^{n+1-j}\mathsf{R}_{j-(n+1)}^f)(0)-(-x)^{n+1-j}\mathsf{R}_{j-(n+1)}^f \rangle_{\mathfrak{w}}.
\end{align*}
Moreover, since $d\hat{\mathfrak{w}}=\frac{xd\mathfrak{w}}{\lambda(0)}$ and $\frac{1}{x}\left(((-x)^{n+1-j}\mathsf{R}_{j-(n+1)}^f)(0)-(-x)^{n+1-j}\mathsf{R}_{j-(n+1)}^f\right)=(-x)^{n-j}\mathsf{R}^f_{n-j}$ we get,
\begin{align*}
\left( \hat{\phi} * \Psi^{n,n+1}_{n+1-j} \right)(i)=-\lambda_0  \langle \hat{\pi}_i\hat{Q}_i , (-x)^{n-j}\mathsf{R}_{j-n}^f \rangle_{\hat{\mathfrak{w}}}.
\end{align*}
\end{proof}
\begin{rmk}\label{relaxremark}
Of course, the condition that $f(x)=p(x)e^{-tx}$ is unnecessarily  restrictive. All that is needed, other than the necessary differentiability assumptions on $f$, in order to prove (1) is that $\langle Q_k , (-x)^{n-j}\mathsf{R}_{j-n}^f \rangle_{\mathfrak{w}}\to 0$ as $k \to \infty$ and for (2) that the orthogonal decomposition of  $\mathsf{T}^f_m$ converges pointwise at 0.
\end{rmk}
We now, define the \textit{coherent measures} $\mathcal{M}^{\psi}_{n,n+1}$ and $\mathcal{M}^{\psi}_{n,n}$, for $\psi$ in $L^2\left(\mathfrak{I},\mathfrak{w}\right)$ or $L^2\left(\mathfrak{I},\hat{\mathfrak{w}}\right)$ respectively as follows,
\begin{align}
\mathcal{M}^{\psi}_{n,n+1}(\nu)&=\frac{(-1)^{\binom{n}{2}}}{\lambda_0^{\binom{n}{2}}}\det\left(\langle \pi_{\nu_i}Q_{\nu_i} , (-x)^{n+1-j}\psi \rangle_{\mathfrak{w}}\right)^{n+1}_{i,j=1} h_{n,n+1}(\nu), \textnormal{ for } \nu \in W^{n+1},\label{coherentmeasure1}\\
\mathcal{M}^{\psi}_{n,n}(\nu)&=\frac{(-1)^{\binom{n-1}{2}}}{\lambda_0^{\binom{n-1}{2}}}\det\left(\langle \hat{\pi}_{\nu_i}\hat{Q}_{\nu_i} , (-x)^{n-j}\psi \rangle_{\hat{\mathfrak{w}}}\right)^{n}_{i,j=1}h_{n,n}(\nu)\label{coherentmeasure2}, \textnormal{ for } \nu \in W^n.
\end{align}
Note that, by simply unpacking the notation and observing that the powers of $(-1)$'s actually cancel out, these can be written as,
\begin{align*}
\mathcal{M}^{\psi}_{n,n+1}(\nu_1,\cdots,\nu_{n+1})&=\frac{1}{\lambda_0^{\binom{n}{2}}}\det\left(\int_{\mathfrak{I}}^{} \pi_{\nu_i}Q_{\nu_i}(x) x^{n+1-j}\psi(x) d\mathfrak{w}(x)\right)^{n+1}_{i,j=1} h_{n,n+1}(\nu_1,\cdots,\nu_{n+1}),\\
\mathcal{M}^{\psi}_{n,n}(\nu_1,\cdots,\nu_n)&=\frac{1}{\lambda_0^{\binom{n-1}{2}}}\det\left(\int_{\mathfrak{I}}^{} \hat{\pi}_{\nu_i}\hat{Q}_{\nu_i}(x)x^{n-j}\psi(x) d\hat{\mathfrak{w}}(x)\right)^{n}_{i,j=1}h_{n,n}(\nu_1,\cdots,\nu_{n}).
\end{align*}

The measures $\mathcal{M}^{\psi}$ are real (not necessarily positive) measures and as we see in Lemma \ref{factorizationforcoherent} below their mass is explicit. Moreover, Lemma \ref{factorizationforcoherent}, shows that the "generating functions" (with respect to the corresponding multivariate orthogonal polynomials) of these measures are \textit{multiplicative}. This property, under some extra assumptions (see Appendix), implies that these coherent measures, when they are positive and normalized to be probability measures (see subsection \ref{subsectionpositivityofcoherent}), are in fact \textit{extremal} (and thus, they correspond to points of the boundary of the branching graph coming from the alternating construction, see subsection \ref{examplesofbranchingsubsection}). 
\begin{lem} \label{factorizationforcoherent}With $\star=n,n+1$, let $\psi\in L^2$ be such that each of the functions $\{(-x)^{n+1-i}\psi(x)\}_{i=1}^{n+1}$ has an orthogonal decomposition converging pointwise at the points $\{x_j\}^{n+1}_{j=1}$. Then,
\begin{align}
\sum_{\nu \in W^{\star}}^{}\mathcal{M}^{\psi}_{n,\star}(\nu)\frac{h_{n,\star}(x,\nu)}{h_{n,\star}(\nu)}&=\prod_{i=1}^{\star}\psi(x_i),
\end{align}
where the functions $h_{n,\star}$ where defined in (\ref{normalizedKMpoly1}), (\ref{normalizedKMpoly2}). In particular, the measures $\mathcal{M}^{\psi}_{n,\star}$ have mass $\psi(0)^{\star}$.  Moreover, if $\psi\equiv1$ then $\mathcal{M}^{\psi}_{n,\star}(\nu)=\textbf{1}(\nu=(0,\cdots,\star-1))$.
\end{lem}
\begin{proof}
We apply the Cauchy-Binet formula (for infinite sums, see for example Lemma 2.1 of \cite{CerenziaKuan}), to obtain with $\star=n+1$ (the case $\star=n$ is exactly the same with only changes in notation),
\begin{align*}
\sum_{\nu \in W^{n+1}}^{}\mathcal{M}^{\psi}_{n,n+1}(\nu)\frac{h_{n,n+1}(x,\nu)}{h_{n,n+1}(\nu)}&=\frac{\det\left(\sum_{k\ge 0}^{}\langle \pi_kQ_k , (-x)^{n+1-i}\mathsf{R}_{i-(n+1)}^{\psi} \rangle_{\mathfrak{w}}Q_k(x_j)\right)^{n+1}_{i,j=1}}{\det\left(x_i^{j-1}\right)^{n+1}_{i,j=1}}\\
&=\frac{\det\left( (-x_j)^{n+1-i}\psi(x_j)\right)^{n+1}_{i,j=1}}{\det\left(x_i^{j-1}\right)^{n+1}_{i,j=1}}=\prod_{i=1}^{n+1}\psi(x_i).
\end{align*}
We have also used the fact that,
\begin{align*}
\frac{\det\left( (-x_i)^{n+1-j}\right)^{n+1}_{i,j=1}}{\det\left(x_i^{j-1}\right)^{n+1}_{i,j=1}}=(-1)^{\binom{n}{2}}\frac{\det\left( x_i^{n+1-j}\right)^{n+1}_{i,j=1}}{\det\left(x_i^{j-1}\right)^{n+1}_{i,j=1}}=(-1)^{\binom{n}{2}}(-1)^{\lfloor \frac{n+1}{2}\rfloor}\equiv 1.
\end{align*}
Moreover, we have,
\begin{align*}
\mathcal{M}^{\psi}_{n,n+1}(0,\cdots,n)&=\left.\frac{\det\left(\sum_{k\ge 0}^{}\langle \pi_kQ_k , (-x)^{n+1-i} \rangle_{\mathfrak{w}}Q_k(x_j)\right)^{n+1}_{i,j=1}}{\det\left(x_i^{j-1}\right)^{n+1}_{i,j=1}}\right \vert_{x_1,\cdots,x_n=0}\\
&=\left.\frac{\det\left( (-x_j)^{n+1-i}\right)^{n+1}_{i,j=1}}{\det\left(x_i^{j-1}\right)^{n+1}_{i,j=1}}\right \vert_{x_1,\cdots,x_n=0}=1.
\end{align*}
\end{proof}

Our interest in these measures, as already anticipated, stems from the fact that they are coherent/consistent with respect to the intertwining kernels.
\begin{prop}\label{coherencyprop} Let $\psi(x)=p(x)e^{-tx}$, where $p(x)$ is a polynomial of arbitrary degree. Then with $k\in W^n$,
\begin{align}
\mathcal{M}^{\psi}_{n,n}(k)&=(\mathcal{M}^{\psi}_{n,n+1}\Lambda^{h_{n,n}}_{n,n+1})(k) \ , \ \textnormal{for}  \ \psi(0)=1, \label{coherency1}\\
\mathcal{M}^{\psi}_{n-1,n}(k)&=(\mathcal{M}^{\psi}_{n,n}\Lambda^{h_{n-1,n}}_{n,n})(k)\label{coherency2}.
\end{align}
\end{prop}
\begin{proof}
We prove (\ref{coherency2}) first, using the Cauchy-Binet formula for the passage to the second equality,
\begin{align*}
\sum_{\nu \in W^{n}}^{}\mathcal{M}^{\psi}_{n,n}(\nu)\Lambda^{h_{n-1,n}}_{n,n}(\nu,k)&=\frac{(-1)^{\binom{n-1}{2}}}{\lambda_0^{\binom{n-1}{2}}}h_{n-1,n}(k)\sum_{\nu \in W^{n}}^{}\det\left(\phi(k_i,\nu_j)\right)^n_{i,j=1}\det\left(\langle \hat{\pi}_{\nu_i}\hat{Q}_{\nu_i} , (-x)^{n-j}\mathsf{R}_{j-n}^{\psi} \rangle_{\hat{\mathfrak{w}}}\right)^{n}_{i,j=1}\\
&=\frac{(-1)^{\binom{n-1}{2}}}{\lambda_0^{\binom{n-1}{2}}}h_{n-1,n}(k)\det\left((\phi*\Psi^{n,n}_{n-j})(k_i)\right)^n_{i,j=1}\\
&=\frac{(-1)^{\binom{n-1}{2}}}{\lambda_0^{\binom{n-1}{2}}}h_{n-1,n}(k)\det\left(\Psi^{n-1,n}_{n-j}(k_i)\right)^n_{i,j=1}=\mathcal{M}^{\psi}_{n-1,n}(k).
\end{align*}
We now turn to the proof of (\ref{coherency1}) and calculate, again using the Cauchy-Binet formula for the second equality,
\begin{align*}
\sum_{\nu \in W^{n+1}}^{}\mathcal{M}^{\psi}_{n,n+1}(\nu)\Lambda^{h_{n,n}}_{n,n}(\nu,k)&=\frac{(-1)^{\binom{n}{2}}}{\lambda_0^{\binom{n}{2}}}h_{n,n}(k)\sum_{\nu \in W^{n}}^{}\det\left(\hat{\phi}(k_i,\nu_j)\right)^{n+1}_{i,j=1}\det\left(\langle {\pi}_{\nu_i}{Q}_{\nu_i} , (-x)^{n-j}\mathsf{R}_{j-n}^{\psi} \rangle_{{\mathfrak{w}}}\right)^{n+1}_{i,j=1}\\
&=\frac{(-1)^{\binom{n}{2}}}{\lambda_0^{\binom{n}{2}}}h_{n,n}(k)\det\left((\hat{\phi}*\Psi^{n,n+1}_{n+1-j})(k_i)\right)^{n+1}_{i,j=1}\\
&=\frac{(-1)^{\binom{n-1}{2}}}{\lambda_0^{\binom{n-1}{2}}}h_{n,n}(k)\det\left(\Psi^{n,n}_{n-j}(k_i)\right)^n_{i,j=1}=\mathcal{M}^{\psi}_{n,n}(k).
\end{align*}
The penultimate equality, follows from $\left( \hat{\phi} * \Psi^{n,n+1}_{n+1-j} \right)(i)=-\lambda_0\Psi^{n,n}_{n-j}(i)$ and the fact that the last row of $\{(\hat{\phi}*\Psi^{n,n+1}_{n+1-j})(k_i)\}^{n+1}_{i,j=1}$ is given by, with $k_{n+1}=\textnormal{virt}$ (recall for $j\le n+1$ that $\mathsf{R}_{j-(n+1)}^{\psi}=\psi$),
\begin{align*}
(\hat{\phi}*\Psi^{n,n+1}_{n+1-j})(\textnormal{virt})=\sum_{i \ge 0}^{}\Psi^{n,n+1}_{n+1-j}(i)=\sum_{i \ge 0}^{}\langle \pi_iQ_i , (-x)^{n+1-j}\mathsf{R}_{j-(n+1)}^{\psi} \rangle_{\mathfrak{w}}Q_i(0)=((-x)^{n+1-j}\psi)(0)=\delta_{j,n+1},
\end{align*}
where we have assumed $\psi(0)=1$ and also used the fact that $Q_i(0)=1$.
\end{proof}

\begin{rmk}
Again, conditions on $\psi$ can be relaxed c.f. Remark \ref{relaxremark}.
\end{rmk}

\section{Evolution of coherent measures}\label{sectionevolutionoperators}

\subsection{Evolution operators for coherent measures and their basic properties}

We now define some kind of evolution operators acting on the coherent measures, that generalize the $h$-transformed Karlin-McGregor semigroups. For $\psi$ in $L^2\left(\mathfrak{I},\mathfrak{w}\right)$ and $L^2\left(\mathfrak{I},\hat{\mathfrak{w}}\right)$ respectively, define $\mathfrak{P}^{\psi}_{n,n+1}$ and $\mathfrak{P}^{\psi}_{n,n}$ by,
\begin{align}
\mathfrak{P}^{\psi}_{n,n+1}(k,\nu)&=\frac{h_{n,n+1}(\nu)}{h_{n,n+1}(k)}\det\left(\langle Q_{k_i} , \pi_{\nu_j}Q_{\nu_j}\psi \rangle_{\mathfrak{w}}\right)^{n+1}_{i,j=1}, \textnormal{ for } k,\nu \in W^{n+1},\\
\mathfrak{P}^{\psi}_{n,n}(k,\nu)&=\frac{h_{n,n}(\nu)}{h_{n,n}(k)}\det\left(\langle \hat{Q}_{k_i} , \hat{\pi}_{\nu_j}\hat{Q}_{\nu_j}\psi \rangle_{\hat{\mathfrak{w}}}\right)^{n}_{i,j=1}, \textnormal{ for } k,\nu \in W^{n}.
\end{align}
Note that,
\begin{align}\label{coherentmeasuresevolutionrelation}
\mathfrak{P}^{\psi}_{\bullet ,\star}(k_0,\nu)=\mathcal{M}_{\bullet,\star}^{\psi}(\nu) ,\ \textnormal{where} \ k_0=(0,1,\cdots,\star-1).
\end{align}
This is because, by row and column operations both sides are the same up to a multiplicative constant and since, from the following lemma they both sum to $\psi(0)^{\star}$, they must in fact be equal. 

Moreover, observe that by (\ref{spectralexpansiontransition}) for $\psi(x)=\phi_t(x)=e^{-tx}$ then $\left(\mathfrak{P}^{\phi_t}_{n,n+1};t\ge 0\right)$ and $\left(\mathfrak{P}^{\phi_t}_{n,n};t\ge 0\right)$ are exactly the $h$-transformed Karlin-McGregor semigroups $\left(P_{n+1}^{h_{n,n+1}}(t);t \ge 0\right)$ and $\left(\hat{P}_n^{h_{n,n}}(t);t\ge 0\right)$ respectively. We will now study their properties. The non-trivial issue of positivity will be dealt with at the end of this subsection. First, we have the following lemma regarding their normalization.

\begin{lem} If, $\psi$ is such that its orthogonal decomposition converges pointwise in a neighbourhood of $0$, we then have,
\begin{align*}
\sum_{\nu \in W^{n+1}}^{}\mathfrak{P}^{\psi}_{n,n+1}(k,\nu)&=\psi(0)^{n+1},\  \forall k \in W^{n+1},\\
\sum_{\nu \in W^n}^{}\mathfrak{P}^{\psi}_{n,n}(k,\nu)&=\psi(0)^n, \ \forall k \in W^n.
\end{align*}
\end{lem}
\begin{proof}
We only prove the first equality, as the second is analogous,
\begin{align*}
\sum_{\nu \in W^{n+1}}^{}\mathfrak{P}^{\psi}_{n,n+1}(k,\nu)&=\frac{1}{h_{n,n+1}(k)}\sum_{\nu \in W^{n+1}}^{}\det\left(\langle Q_{k_i} , \pi_{\nu_i}Q_{\nu_i}\psi \rangle_{\mathfrak{w}}\right)^{n+1}_{i,j=1}(-1)^{\binom{n}{2}}\lambda_0^{\binom{n}{2}}\left.\frac{\det\left(Q_{\nu_i}(x_j)\right)_{i,j=1}^{n+1}}{\det\left(x_j^{i-1}\right)_{i,j=1}^{n+1}}\right \vert_{x_1,\cdots,x_{n+1}=0}\\
&=\frac{(-1)^{\binom{n}{2}}\lambda_0^{\binom{n}{2}}}{h_{n,n+1}(k)}\left.\frac{\det\left(\sum_{m\ge 0}^{}\langle Q_{k_i} , \pi_{m}Q_{m}\psi \rangle_{\mathfrak{w}}Q_{m}(x_j)\right)_{i,j=1}^{n+1}}{\det\left(x_j^{i-1}\right)_{i,j=1}^{n+1}}\right \vert_{x_1,\cdots,x_{n+1}=0}\\
&=\frac{(-1)^{\binom{n}{2}}\lambda_0^{\binom{n}{2}}}{h_{n,n+1}(k)}\left.\frac{\det\left(Q_{k_i}(x_j)\right)_{i,j=1}^{n+1}}{\det\left(x_j^{i-1}\right)_{i,j=1}^{n+1}}\prod_{i=1}^{n+1}\psi(x_i)\right \vert_{x_1,\cdots,x_{n+1}=0}=\psi(0)^{n+1}.
\end{align*}
\end{proof}

The simple, but important proposition below, describes the evolution of coherent measures. Its proof is an easy consequence of the Cauchy-Binet formula and of uniform convergence of the orthogonal decomposition on compact sets for functions of the form $p(x)e^{-tx}$, with $p(x)$ a polynomial.

\begin{prop} \label{evolutionprop}Assume $\mathfrak{I}$ is compact or equivalently $I^+<\infty$ and moreover suppose $\psi_1(x)=p_1(x)e^{-t_1x}$ and $\psi_2(x)=p_2(x)e^{-t_2 x}$, where $p_1,p_2$ are arbitrary polynomials and $t_1,t_2 \ge 0$. We then have the following equalities,  
\begin{align}
\sum_{k \in W^{n+1}}^{}\mathcal{M}^{\psi_1}_{n,n+1}(k)\mathfrak{P}^{\psi_2}_{n,n+1}(k,\nu)&=\mathcal{M}^{\psi_1\psi_2}_{n,n+1}(\nu) \label{evolution1}, \ \forall \nu \in W^{n+1}, \\
\sum_{k \in W^n}^{}\mathcal{M}^{\psi_1}_{n,n}(k)\mathfrak{P}^{\psi_2}_{n,n}(k,\nu)&=\mathcal{M}^{\psi_1\psi_2}_{n,n}(\nu) \label{evolution2}, \ \forall \nu \in W^n.
\end{align}
\end{prop}

\begin{proof}
We only prove (\ref{evolution1}), as (\ref{evolution2}) is completely analogous. The passage to the second equality below first uses the Cauchy-Binet formula and secondly the uniform convergence of the orthogonal decomposition on compacts, in order to justify the interchange $\sum \langle \cdot , \cdot \rangle_{\mathfrak{w}}=\langle \sum \cdot , \cdot \rangle_{\mathfrak{w}}$, of summation and integration,
\begin{align*}
\sum_{k \in W^{n+1}}^{}\mathcal{M}^{\psi_1}_{n,n+1}(k)\mathfrak{P}^{\psi_2}_{n,n+1}(k,\nu)&=(-1)^{\binom{n}{2}}\lambda_0^{\binom{n}{2}}h_{n,n+1}(\nu)\sum_{k \in W^{n+1}}^{}\det\left(\langle \pi_{k_i}Q_{k_i} , (-x)^{n+1-j}\psi_1 \rangle_{\mathfrak{w}}\right)^{n+1}_{i,j=1} \det\left(\langle Q_{k_i} , \pi_{\nu_j}Q_{\nu_j}\psi_2 \rangle_{\mathfrak{w}}\right)^{n+1}_{i,j=1}\\
&=(-1)^{\binom{n}{2}}\lambda_0^{\binom{n}{2}}h_{n,n+1}(\nu)\det\left(\langle \sum_{m\ge 0}^{}\langle \pi_m Q_m , \pi_{\nu_i}Q_{\nu_i}\psi_2 \rangle_{\mathfrak{w}}Q_m, (-x)^{n+1-j}\psi_1 \rangle_{\mathfrak{w}}\right)_{i,j=1}^{n+1}\\
&=(-1)^{\binom{n}{2}}\lambda_0^{\binom{n}{2}} h_{n,n+1}(\nu)\det\left(\langle \pi_{\nu_i}Q_{\nu_i} , (-x)^{n+1-j}\psi_1\psi_2 \rangle_{\mathfrak{w}}\right)^{n+1}_{i,j=1}=\mathcal{M}^{\psi_1\psi_2}_{n,n+1}(\nu).
\end{align*}
\end{proof}

\begin{rmk}
In fact, the argument above gives,
\begin{align*}
\sum_{k \in W^{n+1}}^{}\mathfrak{P}^{\psi_1}_{n,n+1}(\mu,k)\mathfrak{P}^{\psi_2}_{n,n+1}(k,\nu)&=\mathfrak{P}^{\psi_1\psi_2}_{n,n+1}(\mu,\nu) , \ \forall \mu,\nu \in W^{n+1}, \\
\sum_{k \in W^{n}}^{}\mathfrak{P}^{\psi_1}_{n,n}(\mu,k)\mathfrak{P}^{\psi_2}_{n,n}(k,\nu)&=\mathfrak{P}^{\psi_1\psi_2}_{n,n}(\mu,\nu) , \ \forall \mu,\nu \in W^{n}.
\end{align*}
Then (\ref{evolution1}) and (\ref{evolution2}) become a consequence of (\ref{coherentmeasuresevolutionrelation}).
\end{rmk}

\begin{rmk}
The assumptions that $\mathfrak{I}$ is compact and that $\psi_1,\psi_2$ are of the special form $p(x)e^{-tx} $ could of course be removed as long as the interchange of summation and integration in the second equality above can be justified.
\end{rmk}

Finally, we give a linear algebraic proof of the following intertwining relations. Although, we have already obtained these equalities in the special case $\psi_t(x)=e^{-tx}$ in Corollary \ref{corinter} by other means and for general functions $\psi$ will not be used in the sequel; we decided to present it, since it sheds some light on the relations between the dual Karlin-McGregor polynomials that are essential for these commutation relations to hold.

\begin{prop}\label{generalfunctionintertwining}
Let $\psi$ be as in the statement of Proposition \ref{evolutionprop} and moreover assume $\psi(0)=1$. Then,
\begin{align*}
\mathfrak{P}^{\psi}_{n,n+1}\Lambda^{h_{n,n}}_{n,n+1}&=\Lambda^{h_{n,n}}_{n,n+1}\mathfrak{P}^{\psi}_{n,n},\\
\mathfrak{P}^{\psi}_{n,n}\Lambda^{h_{n-1,n}}_{n,n}&=\Lambda^{h_{n-1,n}}_{n,n}\mathfrak{P}^{\psi}_{n-1,n}.
\end{align*}
\end{prop}

\begin{proof}
We only prove the first relation, as the second is analogous. Observe that (noting also that the dummy variable on the left is $(n+1)$-dimensional while on the left $n$-dimensional),
\begin{align*}
\sum_{z \in W^{n+1}}^{}\mathfrak{P}^{\psi}_{n,n+1}(k,z)\Lambda^{h_{n,n}}_{n,n+1}(z,\nu)&=\sum_{z \in W^{n}}^{}\Lambda^{h_{n,n}}_{n,n+1}(k,z)\mathfrak{P}^{\psi}_{n,n}(z,\nu),
\end{align*}
is equivalent to,
\begin{align*}
\sum_{z \in W^{n+1}}^{}\det\left(\langle Q_{k_i} , \pi_{z_j}Q_{z_j}\psi \rangle_{\mathfrak{w}}\right)^{n+1}_{i,j=1}\det\left(\hat{\phi}(\nu_j,z_i)\right)^{n+1}_{i,j=1}&=\sum_{z \in W^{n}}^{}\det\left(\hat{\phi}(z_j,k_i)\right)^{n+1}_{i,j=1}\det\left(\langle \hat{Q}_{z_i} , \hat{\pi}_{\nu_j}\hat{Q}_{\nu_j}\psi \rangle_{\hat{\mathfrak{w}}}\right)^{n}_{i,j=1}.
\end{align*}
The left hand side, by the Cauchy-Binet formula is equal to,
\begin{align*}
\det \left(\langle Q_{k_i} , \sum_{z \ge 0}\pi_{z}Q_{z}\hat{\phi}(\nu_j,z)\psi \rangle_{\mathfrak{w}}\right)_{i,j=1}^{n+1}.
\end{align*}
For $j \le n$, the entries of the matrix are given by (recall that $Q_{k_i}(0)=\psi(0)=1$),
\begin{align*}
\sum_{z=\nu_{j}+1}^{}\langle\pi_{z}Q_{z} , -\hat{\pi}_{\nu_j}Q_{k_i} \psi \rangle_{\mathfrak{w}}=\langle\hat{Q}_{\nu_j} , \hat{\pi}_{\nu_j}Q_{k_i} \psi -\hat{\pi}_{\nu_j}Q_{k_i}(0) \psi(0)\rangle_{\mathfrak{w}}=\langle\hat{Q}_{\nu_j} , \hat{\pi}_{\nu_j}Q_{k_i} \psi\rangle_{\mathfrak{w}} -\langle\hat{Q}_{\nu_j} , \hat{\pi}_{\nu_j}\rangle_{\mathfrak{w}}=\mathfrak{a}_{ij}+\mathfrak{b}_j.
\end{align*}
While, the entries of the last column $j=n+1$ are,
\begin{align*}
\langle Q_{k_i} , \sum_{z\ge 0}^{}\pi_z Q_z \psi \rangle_{\mathfrak{w}}=Q_{k_i}(0)\psi(0)=1.
\end{align*}
To work on the right hand side, we first expand $\det\left(\hat{\phi}(z_j,k_i)\right)^{n+1}_{i,j=1}$ in the last column which consists of all $1$'s. The $l^{th}$-summand in this expansion is given by,
\begin{align*}
(-1)^{n+1+l}\sum_{z \in W^n}^{}\det\left(\hat{\phi}(z_j,k_i)\right)_{1\le i \neq l \le n+1,1\le j \le n}\det\left(\langle \hat{Q}_{z_i} , \hat{\pi}_{\nu_j}\hat{Q}_{\nu_j}\psi \rangle_{\hat{\mathfrak{w}}}\right)^{n}_{i,j=1}\\
=(-1)^{n+1+l}\det\left(\sum_{z \ge 0}^{}\hat{\phi}(z,k_i)\langle \hat{Q}_{z} , \hat{\pi}_{\nu_j}\hat{Q}_{\nu_j}\psi \rangle_{\hat{\mathfrak{w}}}\right)_{1\le i \neq l \le n+1,1\le j \le n}.
\end{align*}
The entries of the matrix in the determinant are given by,
\begin{align*}
-\langle \sum_{z=0}^{k_{i}-1}\hat{\pi}_z\hat{Q}_{z} , \hat{\pi}_{\nu_j}\hat{Q}_{\nu_j}\psi \rangle_{\hat{\mathfrak{w}}}=\langle (Q_{k_i}-1) \frac{\lambda_0}{x}, \hat{\pi}_{\nu_j}\hat{Q}_{\nu_j}\psi \rangle_{\hat{\mathfrak{w}}}=\langle (Q_{k_i}-1), \hat{\pi}_{\nu_j}\hat{Q}_{\nu_j}\psi \rangle_{{\mathfrak{w}}}=\mathfrak{a}_{ij}+\mathfrak{c}_j.
\end{align*}
Now, by summing over $l$, we obtain the determinant of an $(n+1) \times (n+1)$ matrix with the last column being all $1$'s and the other entries being $\mathfrak{a}_{ij}+\mathfrak{c}_j$.

By column operations, more precisely by subtracting a multiple of the last all $1$'s column from the each of the rest, the equality of the left and right hand sides is immediate.
\end{proof}

\begin{rmk}
Proposition \ref{coherencyprop} can also be seen as a corollary of Proposition \ref{generalfunctionintertwining} using (\ref{coherentmeasuresevolutionrelation}).
\end{rmk}

\subsection{Positivity of evolution operators and coherent measures}\label{subsectionpositivityofcoherent}

We now arrive at the question of positivity of the coherent measures. It will in fact be easier to consider a more general problem, namely to address this question first for the evolution operators. 

As already observed by (\ref*{spectralexpansiontransition}), for $\psi(z)=\phi_t(z)=e^{-tz}$ the determinants $\det\left(\langle Q_{k_i} , \pi_{\nu_i}Q_{\nu_i}\phi_t \rangle_{\mathfrak{w}}\right)^{n+1}_{i,j=1}$ and $\det\left(\langle \hat{Q}_{k_i} , \hat{\pi}_{\nu_i}\hat{Q}_{\nu_i}\phi_t \rangle_{\hat{\mathfrak{w}}}\right)^{n}_{i,j=1}$ are exactly the transition densities of the Karlin-McGregor semigroups associated to $n+1$ birth and death chains with generator $\mathcal{D}$ and $n$ birth and death chains with generator $\hat{\mathcal{D}}$ respectively, killed when they collide and so they are positive. Hence, since $h_{n,n}$ and $h_{n,n+1}$ are positive as well we obtain:
\begin{lem}
$\mathfrak{P}^{\phi_t}_{n,n+1}$ and $\mathfrak{P}^{\phi_t}_{n,n}$ are positive, $\forall t \ge 0$.
\end{lem} 
Our goal now, is to find conditions on $a$ so that with $\psi_a(z)=1-az$ the operator $\mathfrak{P}^{\psi_a}_{n,n+1}$ is positive. We make use of an argument found in Proposition 5.1 of \cite{CerenziaKuan}, that is recalled briefly here (see Proposition 5.1 part (4) of \cite{CerenziaKuan}, in particular the paragraph between equations (23) and (24) therein, for the details). Our computations below, are quite simple (compared to \cite{CerenziaKuan}, although we do follow the same argument) taking advantage of the relation between the normalization constants and the rates of the chain. First, we calculate for $i,j \in \mathbb{N}$,
\begin{align*}
\langle Q_{i} , \pi_{j}Q_{j}(1-az) \rangle_{\mathfrak{w}}&=\delta_{i,j}+a\pi_{j}\langle Q_{i} , \mu_{j}Q_{j-1}-(\lambda_j+\mu_j)Q_j+\lambda_jQ_{j+1} \rangle_{\mathfrak{w}}\\
&=\delta_{i,j}+a\delta_{i,j-1}\frac{1}{\pi_{j-1}}\pi_j\mu_j-a(\lambda_j+\mu_j)\delta_{i,j}+a\delta_{i,j+1}\lambda_j\frac{1}{\pi_{j+1}}\pi_j\\
&=\delta_{i,j}+a\lambda_{j-1}\delta_{i,j-1}-a(\mu_j+\lambda_j)\delta_{i,j}+a\mu_{j+1}\delta_{i,j+1},
\end{align*}
since $\frac{\pi_j}{\pi_{j-1}}=\frac{\lambda_{j-1}}{\mu_j}$.

We now, reduce the problem as in Proposition 5.1 of \cite{CerenziaKuan}. First, note that if $y_i>x_i+1$ for some $i$ then we get $\det\left(\langle Q_{x_i} , \pi_{y_j}Q_{y_j}\psi \rangle_{\mathfrak{w}}\right)_{i,j=1}^n=0$, since the resulting matrix has a $2 \times 2$ block form consisting of an off diagonal block of $0$'s and a diagonal block of $0$'s and the same happens for $x_i>y_i+1$. Thus, we must have $|x_i-y_i|\le 1$ and we can further restrict to the case $|x_i-x_{i+1}|\le 1$, for otherwise $\det\left(\langle Q_{x_i} , \pi_{y_j}Q_{y_j}\psi \rangle_{\mathfrak{w}}\right)_{i,j=1}^n$ breaks into a product of determinants with entries so that $|x_i-x_{i+1}|\le 1$. Hence, we are led to the case $x_i=\mathsf{x},x_{i+1}=\mathsf{x}+1,\cdots,$ which is the same as considering whether the determinant of the tridiagonal matrix $\{A_{i,j}\}_{i,j=\mathsf{x}}^{\mathsf{x}+m}$ with entries, for some $m\le n$,
\begin{align*}
A_{i,j}=\delta_{i,j}+a\lambda_{j-1}\delta_{i,j-1}-a(\mu_j+\lambda_j)\delta_{i,j}+a\mu_{j+1}\delta_{i,j+1}
\end{align*}
is positive. In order to answer this, we recall the following nice property of tridiagonal matrices (see page 5 of \cite{Totallynonnegativebook}): If each diagonal entry is greater than or equal to the sum of the off-diagonal entries in that row then, all its principal minors are non-negative. So, it suffices to find conditions on $a$ such that,
\begin{align*}
A_{i,i}\ge A_{i,i-1}+A_{i,i+1},
\end{align*}
or more explicitly,
\begin{align*}
1-a(\mu_i+\lambda_i)\ge a\mu_i+ a\lambda_{i}.
\end{align*}
So we need,
\begin{align*}
a\le \frac{1}{2}(\lambda_i+\mu_i)^{-1}, \forall i.
\end{align*}
Thus, by letting $C=\underset{i\ge 0}{\sup}\left(\lambda_i+\mu_i\right)$ we have proven that:
\begin{lem} \label{positivity1}
If $a\le \frac{1}{2C}$ then, $\mathfrak{P}^{\psi_a}_{n,n+1}$ is positive.
\end{lem}

\begin{rmk}\label{remarkendpoint}
We note here, the close connection between the condition $a\le \frac{1}{2C}$ and the true interval of orthogonality. Namely, if the support of the measure $\mathfrak{w}$ is given by $\textnormal{supp}(\mathfrak{w})=[I^-,I^+]$, with $0\le I^- < I^+\le \infty$, then Theorem 14 of \cite{Vandoornoscillation} gives that ( $c_n$ therein is equal to, in our notation, $\mu_n+\lambda_n$),
\begin{align*}
\frac{1}{2}\left(I^-+I^+\right)\le \underset{n \to \infty}{\limsup}\{\lambda_n+\mu_n\}
\end{align*}
and thus,
\begin{align*}
I^+\le 2\underset{n \to \infty}{\limsup}\{\lambda_n+\mu_n\}\le 2 C.
\end{align*}
In particular, since $2C \le \frac{1}{a}$ the root of $\psi_a(z)=1-az$ is not in $[I^-,I^+]$.
\end{rmk}

Moreover, with analogous considerations if we let $\hat{C}=\underset{i\ge 0}{\sup}\left(\hat{\lambda}_i+\hat{\mu}_i\right)$ we obtain the following lemma:

\begin{lem} \label{positivity2}
If $b\le \frac{1}{2\hat{C}}$ then $\mathfrak{P}^{\psi_b}_{n,n}$ is positive.
\end{lem}

Finally, from Lemma \ref{positivity1} and Lemma \ref{positivity2} and Proposition \ref{evolutionprop} we obtain as a corollary the positivity of the coherent measures:

\begin{cor}\label{positivitymeasures}
Assume $\mathfrak{I}$ is compact and let $a\le \frac{1}{2C},b\le \frac{1}{2\hat{C}}$ then, $\mathcal{M}^{\psi_a}_{n,n+1}$ and $\mathcal{M}^{\psi_b}_{n,n}$ are positive.
\end{cor}

\section{Correlation kernels}\label{sectioncorrelation}

\subsection{Computation of the correlation kernel}

In this subsection we assume that $\textnormal{supp}(\mathfrak{w})=\mathfrak{I}$ is compact and that $\psi$ is of the form,
\begin{align}\label{formpsi}
\psi(x)=\psi_{t,\vec{\alpha}}(x)=\prod_{i=1}^{\mathfrak{N}}(1-\alpha_ix)e^{-tx},
\end{align}
for some $\mathfrak{N} \in \mathbb{N}$ and $\frac{1}{2}\left(\frac{1}{C}\wedge\frac{1}{\hat{C}}\right)\ge \alpha_1 \ge \alpha_2 \ge \cdots \ge 0$ and $t \ge 0$. We denote by $\mathbb{GT}_\textbf{s}(\infty)$ the set of all infinite symplectic Gelfand-Tsetlin patterns, namely infinite interlacing sequences of the following form:
\begin{align*}
\mathbb{GT}_\textbf{s}(\infty)=\bigg\{\mathbb{X}=\left(\mathbb{X}^{(0,1)},\mathbb{X}^{(1,1)},\mathbb{X}^{(1,2)},\cdots\right):\mathbb{X}^{(i-1,i)}\in W^{i,i}\left(\mathbb{X}^{(i,i)}\right),\mathbb{X}^{(i,i)}\in W^{i,i+1}\left(\mathbb{X}^{(i,i+1)}\right)\bigg\}.
\end{align*}
Define for $n \in \mathbb{N}$, the following cylinder sets $\mathfrak{C}_{n,n}\left(\mathfrak{x}^{(0,1)},\cdots,\mathfrak{x}^{(n,n)}\right),\mathfrak{C}_{n,n+1}\left(\mathfrak{x}^{(0,1)},\cdots,\mathfrak{x}^{(n,n+1)}\right)$ in $\mathbb{GT}_\textbf{s}(\infty)$, given by,
\begin{align*}
\mathfrak{C}_{n,n}\left(\mathfrak{x}^{(0,1)},\cdots,\mathfrak{x}^{(n,n)}\right)&=\bigg\{\mathbb{X}\in\mathbb{GT}_\textbf{s}(\infty):\mathbb{X}^{(0,1)}=\mathfrak{x}^{(0,1)},\cdots, \mathbb{X}^{(n,n)}=\mathfrak{x}^{(n,n)}\bigg\},\\
\mathfrak{C}_{n,n+1}\left(\mathfrak{x}^{(0,1)},\cdots,\mathfrak{x}^{(n,n+1)}\right)&=\bigg\{\mathbb{X}\in\mathbb{GT}_\textbf{s}(\infty):\mathbb{X}^{(0,1)}=\mathfrak{x}^{(0,1)},\cdots, \mathbb{X}^{(n,n+1)}=\mathfrak{x}^{(n,n+1)}\bigg\}.
\end{align*}
We consider the random variable $\mathsf{X}^{\psi}$, taking values in $\mathbb{GT}_\textbf{s}(\infty)$, with distribution $\Xi^{\psi}$ defined by its values on the cylinder sets as follows,

\begin{align}
\Xi^{\psi}\left[\mathfrak{C}_{n,n}\left(\mathfrak{x}^{(0,1)},\cdots,\mathfrak{x}^{(n,n)}\right)\right]&=\mathcal{M}^{\psi}_{n,n}\left(\mathfrak{x}^{(n,n)}\right)\Lambda_{n,n}^{h_{n-1,n}}\left(\mathfrak{x}^{(n,n)},\mathfrak{x}^{(n-1,n)}\right)\times \cdots \times \Lambda_{1,1}^{h_{0,1}}\left(\mathfrak{x}^{(1,1)},\mathfrak{x}^{(0,1)}\right)\nonumber\\
&=\prod_{k=1}^{n-1}\det\left(\phi(\mathfrak{x}_i^{(k-1,k)},\mathfrak{x}_i^{(k,k)})\right)^k_{i,j=1}\det\left(\hat{\phi}(\mathfrak{x}_i^{(k,k)},\mathfrak{x}_i^{(k,k+1)})\right)^{k+1}_{i,j=1}\nonumber\\
&\times \det\left(\phi(\mathfrak{x}^{(n-1,n)}_i,\mathfrak{x}^{(n,n)}_i)\right)^n_{i,j=1}\frac{(-1)^{\binom{n-1}{2}}}{\lambda_0^{\binom{n-1}{2}}}\det\left(\langle \hat{\pi}_{\mathfrak{x}^{(n,n)}_i}\hat{Q}_{\mathfrak{x}^{(n,n)}_i} , (-x)^{n-j}\psi \rangle_{\hat{\mathfrak{w}}}\right)^{n}_{i,j=1},\\
\Xi^{\psi}\left[\mathfrak{C}_{n,n+1}\left(\mathfrak{x}^{(0,1)},\cdots,\mathfrak{x}^{(n,n+1)}\right)\right]&=\mathcal{M}^{\psi}_{n,n+1}\left(\mathfrak{x}^{(n,n+1)}\right)\Lambda_{n,n+1}^{h_{n,n}}\left(\mathfrak{x}^{(n,n+1)},\mathfrak{x}^{(n,n)}\right)\times \cdots \times \Lambda_{1,1}^{h_{0,1}}\left(\mathfrak{x}^{(1,1)},\mathfrak{x}^{(0,1)}\right)\nonumber\\
&=\prod_{k=1}^{n}\det\left(\phi(\mathfrak{x}_i^{(k-1,k)},\mathfrak{x}_i^{(k,k)})\right)^k_{i,j=1}\det\left(\hat{\phi}(\mathfrak{x}_i^{(k,k)},\mathfrak{x}_i^{(k,k+1)})\right)^{k+1}_{i,j=1}\nonumber\\
&\times\frac{(-1)^{\binom{n}{2}}}{\lambda_0^{\binom{n}{2}}}\det\left(\langle \pi_{\mathfrak{x}^{(n,n+1)}_i}Q_{\mathfrak{x}^{(n,n+1)}_i} , (-x)^{n+1-j}\psi \rangle_{\mathfrak{w}}\right)^{n+1}_{i,j=1}.
\end{align}

Note that, $\mathsf{X}^{\psi}$ is well defined by the coherency property of Proposition \ref{coherencyprop} and positivity of Corollary \ref{positivitymeasures}.
Moreover, observe that for $\psi(x)=\psi_{t,\vec{0}}(x)=\phi_t(x)=e^{-tx}$ then (see Proposition \ref{CoherentdynamicsproptypeB} and the discussion following it), $\Xi^{\phi_t}$ gives the distribution at time $t$ of $\mathcal{D}$-chains on odd levels and $\hat{\mathcal{D}}$-chains on even levels in $\mathbb{GT}_{\textbf{s}}(\infty)$ interacting via the push-block dynamics, started from the fully packed initial condition.

Equivalently, we can view $\mathsf{X}^{\psi}$ as a random point configuration in $\mathbb{N}\times \mathbb{N}$, so that $\Xi^{\psi}$ determines a probability measure on $2^{\mathbb{N}\times \mathbb{N}}$. Abusing notation, we will also denote this by $\Xi^{\psi}$. Our goal, is to calculate explicitly the correlation functions (defined below) $\{\rho_k^{\psi}\}_{k\ge 0}$ of this point process in Theorem \ref{correlationkernelmain}. As above, we will denote by $(n_1,n_2)\in\{(n,n),(n,n+1)\}$ the levels of $\mathbb{GT}_\textbf{s}(\infty)$. For example, $(0,1)$ denotes the first level, $(1,1)$ the second level, $(1,2)$ the third level and so on. For a point $z$ of the form $\left((n_1,n_2),x\right)$ with $(n_1,n_2)$ as above and $x\in \mathbb{N}$ we will say that $z\in \mathsf{X}^{\psi}$, if $z$ belongs to the point configuration corresponding to $\mathsf{X}^{\psi}$. 

In what follows, we will denote by $\mathsf{C}(\mathfrak{I})$, a positively oriented (counter-clockwise) loop around $[0,I^+]$ (and \textit{not just} around $\mathfrak{I}=[I^-,I^+]$, recall $I^-\ge 0$) that is chosen in such a way that it contains \textit{no zeros} of $\psi$. Observe that, this is always possible by Remark \ref{remarkendpoint}. Our method of proof is essentially an application (of a variant) of the famous Eynard-Mehta theorem (see \cite{BorodinRains}). 

We begin with some technical preliminaries but first a comment on notations. In all that follows, all the real weighted integrals over the interval $\mathfrak{I}$, for which we use the notation $\langle \cdot , \cdot \rangle_{\mathsf{m}}$, will be in the $x$-variable, while all the contour integrals over $\mathsf{C}(\mathfrak{I})$ will be in the variable $u$.

\begin{lem} We have the following contour integral expressions for alternating convolutions of $\phi$ and $\hat{\phi}$. In the $1^{st}$ and $3^{rd}$ equalities below we have a total of $2n$ terms in the convolutions, in the $2^{nd}$ a total of $2n+1$ terms and in the $4^{th}$ one $2n-1$ terms.
\begin{align*}
\left(\phi*\frac{\hat{\phi}}{\lambda_0}*\cdots*\phi*\frac{\hat{\phi}}{\lambda_0}\right)(i,j)&=-\frac{1}{2\pi \mathsf{i}}  \oint_{\mathsf{C}(\mathfrak{I})} \langle \pi_i Q_i ,\frac{Q_j(u)}{x-u}\rangle_{\mathfrak{w}}\frac{1}{u^n}du,\\
\left(\phi*\frac{\hat{\phi}}{\lambda_0}*\cdots*\frac{\hat{\phi}}{\lambda_0}*\phi\right)(i,j)&=-\frac{1}{2\pi \mathsf{i}}  \oint_{\mathsf{C}(\mathfrak{I})} \langle \pi_i Q_i ,\frac{\hat{Q}_j(u)}{x-u}\rangle_{\mathfrak{w}}\frac{1}{u^n}du,\\
\left(\frac{\hat{\phi}}{\lambda_0}*\phi*\cdots*\frac{\hat{\phi}}{\lambda_0}*\phi\right)(i,j)&=-\frac{1}{2\pi \mathsf{i}}  \oint_{\mathsf{C}(\mathfrak{I})} \langle \hat{\pi}_i \hat{Q}_i ,\frac{\hat{Q}_j(u)}{x-u}\rangle_{\hat{\mathfrak{w}}}\frac{1}{u^n}du,\\
\left(\frac{\hat{\phi}}{\lambda_0}*\phi*\cdots*\phi*\frac{\hat{\phi}}{\lambda_0}\right)(i,j)&=-\frac{1}{2\pi \mathsf{i}}  \oint_{\mathsf{C}(\mathfrak{I})} \langle \hat{\pi}_i \hat{Q}_i ,\frac{Q_j(u)}{x-u}\rangle_{\hat{\mathfrak{w}}}\frac{1}{u^n}du.
\end{align*}
\end{lem}

\begin{proof}
We begin by writing,
\begin{align*}
\phi(i,j)&=\pi_i \textbf{1}(i\le j) =\pi_i\langle Q_i , \sum_{k=0}^{j}\pi_k Q_k \rangle_{\mathfrak{w}}=\langle \pi_i Q_i ,\hat{Q}_j\rangle_{\mathfrak{w}}\\
&=-\frac{1}{2\pi \mathsf{i}} \oint_{\mathsf{C}(\mathfrak{I})} \langle \pi_i Q_i ,\frac{\hat{Q}_j(u)}{x-u}\rangle_{\mathfrak{w}}du
\end{align*}
and in a similar fashion,
\begin{align*}
\hat{\phi}(i,j)&=-\pi_i \textbf{1}(i< j) =-\hat{\pi}_i\langle \hat{Q}_i , \sum_{k=0}^{j-1}\hat{\pi}_k \hat{Q}_k \rangle_{\hat{\mathfrak{w}}}=\langle \hat{\pi}_i\hat{Q}_i , \frac{\lambda_0}{x}(Q_j(x)-1) \rangle_{\hat{\mathfrak{w}}}\\
&=-\frac{1}{2\pi \mathsf{i}} \oint_{\mathsf{C}(\mathfrak{I})} \langle \hat{\pi}_i\hat{Q}_i , \lambda_0\frac{Q_j(u)}{x-u} \rangle_{\hat{\mathfrak{w}}}\frac{1}{u}du.
\end{align*}
The last equality holds because,
\begin{align*}
\frac{Q_j(x)-1}{x}=-\frac{1}{2\pi \mathsf{i}}\oint_{\mathsf{C}(\mathfrak{I})}\frac{Q_j(u)}{u(x-u)}du \ \textnormal{ for } x \in [0,I^+].
\end{align*}
Moreover,
\begin{align*}
(\phi * \hat{\phi})(i,j)&=-\frac{1}{2\pi \mathsf{i}} \sum_{k \ge 0}^{}-\hat{\pi}_k \textbf{1}(k<j)\oint_{\mathsf{C}(\mathfrak{I})} \langle \pi_i Q_i ,\frac{\hat{Q}_k(u)}{x-u}\rangle_{\mathfrak{w}}du\\
&=-\frac{1}{2\pi \mathsf{i}}  \oint_{\mathsf{C}(\mathfrak{I})} \langle \pi_i Q_i ,-\frac{\sum_{k=0}^{j-1}\hat{\pi}_k\hat{Q}_k(u)}{x-u}\rangle_{\mathfrak{w}}du\\
&=-\frac{1}{2\pi \mathsf{i}}  \oint_{\mathsf{C}(\mathfrak{I})} \langle \pi_i Q_i ,\frac{\lambda_0}{u}\frac{Q_j(u)-1}{x-u}\rangle_{\mathfrak{w}}du\\
&=-\frac{1}{2\pi \mathsf{i}}  \oint_{\mathsf{C}(\mathfrak{I})} \langle \pi_i Q_i ,\lambda_0\frac{Q_j(u)}{x-u}\rangle_{\mathfrak{w}}\frac{1}{u}du,
\end{align*}
where the last equality follows from the fact that for all $n\ge 1$,
\begin{align*}
\oint_{\mathsf{C}(\mathfrak{I})} \frac{1}{u^n(x-u)}du=0\textnormal{ for } x \in [0,I^+].
\end{align*}
Similarly,
\begin{align*}
(\hat{\phi} * \phi)(i,j)&=-\frac{1}{2\pi \mathsf{i}} \sum_{k \ge 0}^{}\pi_k \textbf{1}(k \le j)\oint_{\mathsf{C}(\mathfrak{I})} \langle \hat{\pi}_i\hat{Q}_i , \lambda_0\frac{Q_k(u)}{x-u} \rangle_{\hat{\mathfrak{w}}}\frac{1}{u}du\\
&=-\frac{1}{2\pi \mathsf{i}} \oint_{\mathsf{C}(\mathfrak{I})} \langle \hat{\pi}_i\hat{Q}_i , \lambda_0\frac{\sum_{k = 0}^{j}\pi_kQ_k(u)}{x-u} \rangle_{\hat{\mathfrak{w}}}\frac{1}{u}du\\
&=-\frac{1}{2\pi \mathsf{i}} \oint_{\mathsf{C}(\mathfrak{I})} \langle \hat{\pi}_i\hat{Q}_i , \lambda_0\frac{\hat{Q}_j(u)}{x-u} \rangle_{\hat{\mathfrak{w}}}\frac{1}{u}du.
\end{align*}
By induction, we easily obtain the statement of the lemma.
\end{proof}
We now define the following functions $\mathsf{\Phi}^{(n_1,n_2)}_{(k_1,k_2)}(\cdot, \cdot)$, that will come up in the computation of the correlation kernel, on $\mathbb{N}\times\mathbb{N}$ for $(n_1,n_2)\ge (k_1,k_2)$ given by the convolutions in the Lemma above, but with $\frac{\hat{\phi}}{\lambda_0}$ replaced by $-\frac{\hat{\phi}}{\lambda_0}$ (we just put the factors $(-1)^{\binom{n}{2}}$ and $(-1)^{\binom{n-1}{2}}$ from the cylinder set distributions in the $\hat{\phi}$'s). More explicitly, we define,
\begin{align*}
\mathsf{\Phi}^{(n,n+1)}_{(k,k+1)}(i,j)&=\left(\phi*\left(-\frac{\hat{\phi}}{\lambda_0}\right)*\cdots*\phi*\left(-\frac{\hat{\phi}}{\lambda_0}\right)\right)(i,j)=(-1)^{n-k}\left(-\frac{1}{2\pi \mathsf{i}}\right)  \oint_{\mathsf{C}(\mathfrak{I})} \langle \pi_i Q_i ,\frac{Q_j(u)}{x-u}\rangle_{\mathfrak{w}}\frac{1}{u^{n-k}}du,\\
\mathsf{\Phi}^{(n,n+1)}_{(k,k)}(i,j)&=\left(\left(-\frac{\hat{\phi}}{\lambda_0}\right)*\phi*\cdots*\phi*\left(-\frac{\hat{\phi}}{\lambda_0}\right)\right)(i,j)=(-1)^{n+1-k}\left(-\frac{1}{2\pi \mathsf{i}}\right)  \oint_{\mathsf{C}(\mathfrak{I})} \langle \hat{\pi}_i \hat{Q}_i ,\frac{Q_j(u)}{x-u}\rangle_{\hat{\mathfrak{w}}}\frac{1}{u^{n+1-k}}du,\\
\mathsf{\Phi}^{(n,n)}_{(k,k)}(i,j)&=\left(\left(-\frac{\hat{\phi}}{\lambda_0}\right)*\phi*\cdots*\left(-\frac{\hat{\phi}}{\lambda_0}\right)*\phi\right)(i,j)=(-1)^{n-k}\left(-\frac{1}{2\pi \mathsf{i}}\right)  \oint_{\mathsf{C}(\mathfrak{I})} \langle \hat{\pi}_i \hat{Q}_i ,\frac{\hat{Q}_j(u)}{x-u}\rangle_{\hat{\mathfrak{w}}}\frac{1}{u^{n-k}}du,\\
\mathsf{\Phi}^{(n,n)}_{(k-1,k)}(i,j)&=\left(\phi*\left(-\frac{\hat{\phi}}{\lambda_0}\right)*\cdots*\left(-\frac{\hat{\phi}}{\lambda_0}\right)*\phi\right)(i,j)=(-1)^{n-k}\left(-\frac{1}{2\pi \mathsf{i}}\right)  \oint_{\mathsf{C}(\mathfrak{I})} \langle \pi_i Q_i ,\frac{\hat{Q}_j(u)}{x-u}\rangle_{\mathfrak{w}}\frac{1}{u^{n-k}}du
\end{align*}
and note that, when $(n_1,n_2)=(k_1,k_2)$ then,
\begin{align*}
\mathsf{\Phi}^{(n,n+1)}_{(n,n+1)}(i,j)&=\delta_{i,j},\\
\mathsf{\Phi}^{(n,n)}_{(n,n)}(i,j)&=\delta_{i,j}.
\end{align*}
Moving on, for $\psi$ as in (\ref{formpsi}) we define the following functions for $n,j,i \in \mathbb{N}$,
\begin{align}
\mathcal{E}^{n,n+1}_{n+1-j}(i)&=-\frac{1}{2\pi \mathsf{i}}\oint_{\mathsf{C}(\mathfrak{I})}\frac{Q_i(u)}{\psi(u)(-u)^{n+1-j+1}}du,\\
\mathcal{E}^{n,n}_{n-j}(i)&=-\frac{1}{2\pi \mathsf{i}}\oint_{\mathsf{C}(\mathfrak{I})}\frac{\hat{Q}_i(u)}{\psi(u)(-u)^{n-j+1}}du.
\end{align}
Then, we have the following \textit{biorthogonality} relations between the $\Phi$'s and $\Psi$'s as functions of $i \in \mathbb{N}$.
\begin{lem} \label{orthogonality}
\begin{align*}
\sum_{i\ge0}^{}\Psi^{n,n+1}_{n+1-k}(i)\mathcal{E}^{n,n+1}_{n+1-l}(i)&=\delta_{k,l} ,\textnormal{ for } k,l \le n+1,\\
\sum_{i\ge0}^{}\Psi^{n,n}_{n-k}(i)\mathcal{E}^{n,n}_{n-l}(i)&=\delta_{k,l} ,\textnormal{ for } k,l \le n.
\end{align*}
\end{lem} 
\begin{proof}
We only prove the first equality, as the second is entirely analogous,
\begin{align*}
\sum_{i\ge0}^{}\Psi^{n,n+1}_{n+1-k}(i)\mathcal{E}^{n,n+1}_{n+1-l}(i)&=-\frac{1}{2\pi \mathsf{i}}\sum_{i\ge0}^{}\langle \pi_iQ_i , (-x)^{n+1-k}\psi \rangle_{\mathfrak{w}}\oint_{\mathsf{C}(\mathfrak{I})}\frac{Q_i(u)}{\psi(u)(-u)^{n+1-l+1}}du\\
&=-\frac{1}{2\pi \mathsf{i}}\oint_{\mathsf{C}(\mathfrak{I})}\sum_{i\ge0}^{}\langle \pi_iQ_i , (-x)^{n+1-k}\psi \rangle_{\mathfrak{w}}\frac{Q_i(u)}{\psi(u)(-u)^{n+1-l+1}}du\\
&=-\frac{1}{2\pi \mathsf{i}}\oint_{\mathsf{C}(\mathfrak{I})} \frac{1}{(-u)^{k-l+1}}du=\delta_{k,l}.
\end{align*}
\end{proof}

The last technical ingredient that we need is:

\begin{lem}\label{basis}
For all $n \in \mathbb{N}$, the functions $\mathcal{E}^{n,n+1}_{1}(\cdot),\cdots,\mathcal{E}^{n,n+1}_{n+1}(\cdot)$ form a basis for the linear span of the functions $\left(\hat{\phi}*\mathsf{\Phi}^{(n,n+1)}_{(0,1)}\right)(\textnormal{virt},\cdot),\cdots,\left(\hat{\phi}*\mathsf{\Phi}^{(n,n+1)}_{(n,n+1)}\right)(\textnormal{virt},\cdot)$ and similarly $\mathcal{E}^{n,n}_{1}(\cdot),\cdots,\mathcal{E}^{n,n}_{n}(\cdot)$ form a basis for the linear span of $\left(\hat{\phi}*\mathsf{\Phi}^{(n,n)}_{(0,1)}\right)(\textnormal{virt},\cdot),\cdots,\left(\hat{\phi}*\mathsf{\Phi}^{(n,n)}_{(n-1,n)}\right)(\textnormal{virt},\cdot)$.
\end{lem}

\begin{proof}
Write $Q_i(x)=\sum_{k=0}^{i}a_k(i)x^k$. By using residue calculus and moreover since we only have a singularity at $0$, we obtain that,
\begin{align*}
\mathcal{E}^{n,n+1}_{n+1-j}(i)&=-\frac{1}{2\pi \mathsf{i}}\oint_{\mathsf{C}(\mathfrak{I})}\frac{Q_i(u)}{\psi(u)(-u)^{n+1-j+1}}du=-\frac{(-1)^{n+1-j+1}}{(n+1-j)!}\frac{d^{n+1-j}}{du^{n+1-j}}\left(\frac{Q_i(u)}{\psi(u)}\right)\bigg\rvert_{u=0}\\
&=-\frac{(-1)^{n+1-j+1}}{(n+1-j)!}\sum_{l=0}^{n+1-j}f_l^{n+1-j}\frac{d^l}{du^l}Q_i(u)\bigg\rvert_{u=0}\\
&=\sum_{l=0}^{n+1-j}\tilde{f}_l^{n+1-j}a_l(i),
\end{align*}
where the coefficients $\{f_l^{n+1-j}\}^{n+1-j}_{l=1}$ only depend on the derivatives of $1/\psi(u)$ at $u=0$. In particular $f_{n+1-j}^{n+1-j}=\frac{1}{\psi(0)}=1\ne 0$ and hence also the leading coefficient $\tilde{f}_{n+1-j}^{n+1-j} \ne 0$. Thus we have,
\begin{align*}
span\{\mathcal{E}^{n,n+1}_{1}(\cdot),\cdots,\mathcal{E}^{n,n+1}_{n+1}(\cdot)\}=span\{a_0(\cdot),\cdots, a_n(\cdot)\}.
\end{align*}
Similarly, if we write $\hat{Q}_i(x)=\sum_{k=0}^{i}\hat{a}_k(i)x^k$ then,
\begin{align*}
\mathcal{E}^{n,n}_{n-j}(i)=\sum_{l=0}^{n-j}\tilde{g}_l^{n-j}\hat{a}_l(i),
\end{align*}
with $\tilde{g}_{n-j}^{n-j} \ne 0$. Hence,
\begin{align*}
span\{\mathcal{E}^{n,n}_{1}(\cdot),\cdots,\mathcal{E}^{n,n}_{n}(\cdot)\}=span\{\hat{a}_0(\cdot),\cdots,\hat{a}_{n-1}(\cdot)\}.
\end{align*}
On the other hand, for $0\le k \le n$, we have that,
\begin{align*}
\left(\hat{\phi}*\mathsf{\Phi}^{(n,n+1)}_{(k,k+1)}\right)(\textnormal{virt},j)&=\sum_{i\ge 0}^{}(-1)^{n-k}\left(-\frac{1}{2\pi \mathsf{i}}\right)  \oint_{\mathsf{C}(\mathfrak{I})} \langle \pi_i Q_i ,\frac{Q_j(u)}{x-u}\rangle_{\mathfrak{w}}\frac{1}{u^{n-k}}du\\
&=(-1)^{n-k}\left(-\frac{1}{2\pi \mathsf{i}}\right)  \oint_{\mathsf{C}(\mathfrak{I})} \frac{Q_j(u)}{-u}\frac{1}{u^{n-k}}du\\
&=(-1)^{n-k}a_{n-k}(j).
\end{align*}
Hence,
\begin{align*}
span\{\left(\hat{\phi}*\mathsf{\Phi}^{(n,n+1)}_{(0,1)}\right)(\textnormal{virt},\cdot),\cdots,\left(\hat{\phi}*\mathsf{\Phi}^{(n,n+1)}_{(n,n+1)}\right)(\textnormal{virt},\cdot)\}=span\{a_0(\cdot),\cdots,a_n(\cdot)\}.
\end{align*}
Similarly, for $1 \le k \le n$, we have,
\begin{align*}
\left(\hat{\phi}*\mathsf{\Phi}^{(n,n)}_{(k-1,k)}\right)(\textnormal{virt},j)=\textnormal{const}_{n,k}a_{n-k}(j)
\end{align*}
and thus,
\begin{align*}
span\{\left(\hat{\phi}*\mathsf{\Phi}^{(n,n)}_{(0,1)}\right)(\textnormal{virt},\cdot),\cdots,\left(\hat{\phi}*\mathsf{\Phi}^{(n,n)}_{(n-1,n)}\right)(\textnormal{virt},\cdot)\}=span\{\hat{a}_0(\cdot),\cdots,\hat{a}_{n-1}(\cdot)\}.
\end{align*}
The statement of the lemma is now evident.
\end{proof}

We finally arrive at our main result, that $\Xi^{\psi}$ is a determinantal point process with an explicit kernel given in terms of the orthogonal polynomials $\{Q_i\}_{i\ge 0},\{\hat{Q}_i\}_{i\ge 0}$ and the spectral measures $\mathfrak{w},\hat{\mathfrak{w}}$.

\begin{thm}\label{correlationkernelmain}
Let $\mathfrak{I}$ be compact and $\psi$ be of the form (\ref{formpsi}). Then, the correlation functions $\{\rho^{\psi}_k\}_{k\ge0}$ of $\Xi^{\psi}$ are determinantal,
\begin{align}
 \rho^{\psi}_k(z_1,\cdots,z_k)\overset{\textnormal{def}}{=}\Xi^{\psi}(\{E \in \mathbb{GT}_{\mathbf{s}}(\infty) \textnormal{ s.t. } \{z_1,\cdots,z_k\} \subset E\})=\det\left(\mathcal{K}^{\psi}(z_i,z_j)\right)^k_{i,j=1}
\end{align}
where $\mathcal{K}^{\psi}$ is given by,
\begin{align}\label{correlationcontourintegral}
\mathcal{K}^{\psi}\left(((n_1,n_2),i),(m_1,m_2),j)\right)=\frac{1}{2\pi \mathsf{i}}\oint_{\mathsf{C}(\mathfrak{I})} \tilde{\mathcal{P}_j}(u)\langle \bar{\mathcal{P}}_i(x) , \frac{x^{n_2}}{u^{m_2}}\frac{\psi(x)}{(x-u)\psi(u)} \rangle_{\mathfrak{m}}du\nonumber\\
+\textbf{1}\left((n_1,n_2)\ge (m_1,m_2)\right)\langle \bar{\mathcal{P}}_i(x) , x^{n_2-m_2}\tilde{\mathcal{P}_j}(x) \rangle_{\mathfrak{m}}
\end{align}
and,
\begin{align}
(\bar{\mathcal{{P}}},\tilde{\mathcal{P}},\mathfrak{m})=\begin{cases}
(\pi_iQ_i,Q_j,\mathfrak{w}) \textnormal{ if } (n_1,n_2),(m_1,m_2)=(n,n+1),(m,m+1)\\
(\pi_iQ_i,\hat{Q}_j,\mathfrak{w}) \textnormal{ if } (n_1,n_2),(m_1,m_2)=(n,n+1),(m,m)\\
(\hat{\pi}_i\hat{Q}_i,Q_j,\hat{\mathfrak{w}}) \textnormal{ if } (n_1,n_2),(m_1,m_2)=(n,n),(m,m+1)\\
(\hat{\pi}_i\hat{Q}_i,\hat{Q}_j,\hat{\mathfrak{w}}) \textnormal{ if } (n_1,n_2),(m_1,m_2)=(n,n),(m,m)
\end{cases}.
\end{align}
\end{thm}

\begin{proof}
This is an application of a variant of the Eynard-Mehta Theorem, more specifically Proposition A.2 of \cite{Cerenzia}. Identifying the functions therein from Lemma \ref{orthogonality} and Lemma \ref{basis} we get that, 
\begin{align}
\mathcal{K}^{\psi}\left(((n_1,n_2),i),(m_1,m_2),j)\right)=-\mathsf{\Phi}^{(m_1,m_2)}_{(n_1,n_2)}(i,j)\textbf{1}\left((n_1,n_2)< (m_1,m_2)\right)+\sum_{k=1}^{m_2}\Psi_{n_2-k}^{n_1,n_2}(i)\mathcal{E}_{m_2-k}^{m_1,m_2}(j).
\end{align}
So, we need to calculate $\sum_{k=1}^{m_2}\Psi_{n_2-k}^{n_1,n_2}(i)\mathcal{E}_{m_2-k}^{m_1,m_2}(j)$. The calculation of this sum is elementary but rather tedious. Moreover, all the sums that are encountered in the sequel are finite, so there are no further issues with convergence other than the ones encountered already. We can assume $(n_1,n_2)=(n,n+1),(m_1,m_2)=(m,m+1)$, as all other cases are analogous; we just need to change $Q_i$'s to $\hat{Q}_i$'s and $\mathfrak{w}$ to $\hat{\mathfrak{w}}$, note that in particular we are not using any specific properties of the $Q_i$'s or $\mathfrak{w}$ below.

We first assume that $m \le n$. Then (note that, for $k\le m+1$ we have $\mathsf{R}_{k-(n+1)}^{\psi}=\psi$),
\begin{align*}
\sum_{k=1}^{m+1}\Psi_{n+1-k}^{n,n+1}(i)\mathcal{E}_{m+1-k}^{m,m+1}(j)&=-\frac{1}{2\pi \mathsf{i}}\sum_{k=1}^{m+1}\langle \pi_iQ_i , (-x)^{n+1-k}\psi \rangle_{\mathfrak{w}}\oint_{\mathsf{C}(\mathfrak{I})}\frac{Q_j(u)}{\psi(u)(-u)^{m+2-k}}du\\
&=-\frac{1}{2\pi \mathsf{i}}\oint_{\mathsf{C}(\mathfrak{I})}\langle \pi_iQ_i , \sum_{k=1}^{m+1}\frac{(-x)^{n+1-k}}{(-u)^{m+2-k}}\psi \rangle_{\mathfrak{w}}\frac{Q_j(u)}{\psi(u)}du.
\end{align*}
By using,
\begin{align*}
\sum_{k=1}^{m+1}\frac{(-x)^{n+1-k}}{(-u)^{m+2-k}}=\frac{u}{x-u}\left(1-\left(\frac{u}{x}\right)^{m+1}\right)\frac{(-x)^{n+1}}{(-u)^{m+2}},
\end{align*}
we get,
\begin{align*}
\sum_{k=1}^{m+1}\Psi_{n+1-k}^{n,n+1}(i)\mathcal{E}_{m+1-k}^{m,m+1}(j)&=-\frac{1}{2\pi \mathsf{i}}\oint_{\mathsf{C}(\mathfrak{I})}\langle \pi_iQ_i , \frac{u}{x-u}\left(1-\left(\frac{u}{x}\right)^{m+1}\right)\frac{(-x)^{n+1}}{(-u)^{m+2}}\psi \rangle_{\mathfrak{w}}\frac{Q_j(u)}{\psi(u)}du\\
&=\frac{1}{2\pi \mathsf{i}}\oint_{\mathsf{C}(\mathfrak{I})}\langle \pi_iQ_i , \frac{1}{x-u}\frac{(-x)^{n+1}}{(-u)^{m+1}}\psi \rangle_{\mathfrak{w}}\frac{Q_j(u)}{\psi(u)}du+\langle \pi_iQ_i ,(-x)^{(n+1)-(m+1)}Q_j \rangle_{\mathfrak{w}},
\end{align*}
where we have taken the residue at $u=x$ in the second term.

We now assume that $m\ge n+1$. We split the sum into two,
\begin{align}
\sum_{k=1}^{m+1}\Psi_{n+1-k}^{n,n+1}(i)\mathcal{E}_{m+1-k}^{m,m+1}(j)=\sum_{k=1}^{n+1}\Psi_{n+1-k}^{n,n+1}(i)\mathcal{E}_{m+1-k}^{m,m+1}(j)+\sum_{k=n+2}^{m+1}\Psi_{n+1-k}^{n,n+1}(i)\mathcal{E}_{m+1-k}^{m,m+1}(j).
\end{align}
We calculate the first summand as before,
\begin{align}
\sum_{k=1}^{n+1}\Psi_{n+1-k}^{n,n+1}(i)\mathcal{E}_{m+1-k}^{m,m+1}(j)=-\frac{1}{2\pi \mathsf{i}}\oint_{\mathsf{C}(\mathfrak{I})}\langle \pi_iQ_i , \frac{u}{x-u}\left(1-\left(\frac{u}{x}\right)^{n+1}\right)\frac{(-x)^{n+1}}{(-u)^{m+2}}\psi \rangle_{\mathfrak{w}}\frac{Q_j(u)}{\psi(u)}du.
\end{align}
For the second summand first recall that $\Psi_{n+1-k}^{n,n+1}(i)=\langle \pi_iQ_i , (-x)^{n+1-k}\mathsf{R}_{k-(n+1)}^{\psi} \rangle_{\mathfrak{w}}$ where $\mathsf{R}_{k-(n+1)}^{\psi}(x)=\psi(x)-\sum_{l=0}^{k-(n+1)-1}\frac{\psi^{(l)}(0)}{l!}(-x)^l(-1)^l$ and thus,
\begin{align}
\sum_{k=n+2}^{m+1}\Psi_{n+1-k}^{n,n+1}(i)\mathcal{E}_{m+1-k}^{m,m+1}(j)=-\frac{1}{2\pi \mathsf{i}}\oint_{\mathsf{C}(\mathfrak{I})}\langle \pi_iQ_i , \sum_{k=n+2}^{m+1}\frac{(-x)^{n+1-k}}{(-u)^{m+2-k}}\left[\psi(x)-\sum_{l=0}^{k-(n+1)-1}\frac{\psi^{(l)}(0)}{l!}(-x)^l(-1)^l\right] \rangle_{\mathfrak{w}}\frac{Q_j(u)}{\psi(u)}du.
\end{align}
So, we need to calculate,
\begin{align*}
\sum_{k=n+2}^{m+1}\frac{(-x)^{n+1-k}}{(-u)^{m+2-k}}\left[\psi(x)-\sum_{l=0}^{k-(n+1)-1}\frac{\psi^{(l)}(0)}{l!}(-x)^l(-1)^l\right]=\frac{1}{(-u)^{m+2}}\left[\sum_{k=n+2}^{m+1}(\psi(x)-\psi(0))\frac{(-u)^k}{(-x)^{k-(n+1)}}\right.\\
-\left.\sum_{k=n+3}^{m+1}\sum_{l=0}^{k-(n+2)}\frac{\psi^{(l)}(0)}{l!}\frac{(-u)^k(-1)^l}{(-x)^{k-(n+1)-l}}\right].
\end{align*}
Repeatedly using the geometric summation identity we get that this is equal to,
\begin{align*}
\frac{1}{(-u)^{m+2}}\left[(\psi(x)-\psi(0))\frac{(-1)(-u)^{n+2}}{x-u}\left(1-\left(\frac{u}{x}\right)^{(m+1)-(n+1)}\right)\right.\\
-\left.\sum_{r=1}^{(m+1)-(n+1)-1}\frac{\psi^{(r)}(0)}{r!}\frac{(-1)(-u)^{n+2+r}(-1)^r}{x-u}\left(1-\left(\frac{u}{x}\right)^{(m+1)-(n+1)-r}\right)\right]\\
=-\frac{(-u)^{(n+1)-(m+1)}}{x-u}\left[\psi(x)-\psi(0)-\sum_{r=1}^{(m+1)-(n+1)-1}\frac{\psi^{(r)}(0)}{r!}(-u)^r(-1)^r\right]\\
+\frac{(-x)^{(n+1)-(m+1)}}{x-u}\left[\psi(x)-\psi(0)-\sum_{r=1}^{(m+1)-(n+1)-1}\frac{\psi^{(r)}(0)}{r!}(-x)^r(-1)^r\right]\\
=-\frac{(-u)^{(n+1)-(m+1)}}{x-u}\left[\mathsf{R}^{\psi}_{(m+1)-(n+1)}(u)-\psi(u)+\psi(x)\right]+\frac{(-x)^{(n+1)-(m+1)}}{x-u}\mathsf{R}^{\psi}_{(m+1)-(n+1)}(x).
\end{align*}
Hence,
\begin{align*}
\sum_{k=n+2}^{m+1}\Psi_{n+1-k}^{n,n+1}(i)\mathcal{E}_{m+1-k}^{m,m+1}(j)=\frac{1}{2\pi \mathsf{i}}\oint_{\mathsf{C}(\mathfrak{I})}\langle \pi_iQ_i , \frac{(-u)^{(n+1)-(m+1)}}{x-u}\left[\mathsf{R}^{\psi}_{(m+1)-(n+1)}(u)-\psi(u)+\psi(x)\right] \rangle_{\mathfrak{w}}\frac{Q_j(u)}{\psi(u)}du\\
-\frac{1}{2\pi \mathsf{i}}\oint_{\mathsf{C}(\mathfrak{I})}\langle \pi_iQ_i ,\frac{(-x)^{(n+1)-(m+1)}}{x-u}\mathsf{R}^{\psi}_{(m+1)-(n+1)}(x) \rangle_{\mathfrak{w}}\frac{Q_j(u)}{\psi(u)}du.
\end{align*}
Now, by taking the residue at $u=x$, in both contour integrals in the terms involving $\mathsf{R}^{\psi}_{(m+1)-(n+1)}$ we get (note that there is no pole at $u=0$ in the first contour integral),
\begin{align*}
&\frac{1}{2\pi \mathsf{i}}\oint_{\mathsf{C}(\mathfrak{I})}\langle \pi_iQ_i , \frac{(-u)^{(n+1)-(m+1)}}{x-u}\mathsf{R}^{\psi}_{(m+1)-(n+1)}(u) \rangle_{\mathfrak{w}}\frac{Q_j(u)}{\psi(u)}du
-\frac{1}{2\pi \mathsf{i}}\oint_{\mathsf{C}(\mathfrak{I})}\langle \pi_iQ_i ,\frac{(-x)^{(n+1)-(m+1)}}{x-u}\mathsf{R}^{\psi}_{(m+1)-(n+1)}(x) \rangle_{\mathfrak{w}}\frac{Q_j(u)}{\psi(u)}du\\
&=-\langle \pi_iQ_i ,\frac{(-x)^{(n+1)-(m+1)}}{\psi(x)}\mathsf{R}^{\psi}_{(m+1)-(n+1)}(x)Q_j(x) \rangle_{\mathfrak{w}}+\langle \pi_iQ_i ,\frac{(-x)^{(n+1)-(m+1)}}{\psi(x)}\mathsf{R}^{\psi}_{(m+1)-(n+1)}(x)Q_j(x) \rangle_{\mathfrak{w}}=0.
\end{align*}
So,
\begin{align*}
&\sum_{k=n+2}^{m+1}\Psi_{n+1-k}^{n,n+1}(i)\mathcal{E}_{m+1-k}^{m,m+1}(j)=\frac{1}{2\pi \mathsf{i}}\oint_{\mathsf{C}(\mathfrak{I})}\langle \pi_iQ_i , \frac{(-u)^{(n+1)-(m+1)}}{x-u}\left[-\psi(u)+\psi(x)\right] \rangle_{\mathfrak{w}}\frac{Q_j(u)}{\psi(u)}du\\
&=-\frac{1}{2\pi \mathsf{i}}\oint_{\mathsf{C}(\mathfrak{I})}\langle \pi_iQ_i , \frac{(-u)^{(n+1)-(m+1)}}{x-u} \rangle_{\mathfrak{w}}Q_j(u)du+\frac{1}{2\pi \mathsf{i}}\oint_{\mathsf{C}(\mathfrak{I})}\langle \pi_iQ_i , \frac{(-u)^{(n+1)-(m+1)}}{x-u}\psi(x) \rangle_{\mathfrak{w}}\frac{Q_j(u)}{\psi(u)}du.
\end{align*}
Thus, combining with the first summand we get that for $m>n$,
\begin{align}
\sum_{k=1}^{m+1}\Psi_{n+1-k}^{n,n+1}(i)\mathcal{E}_{m+1-k}^{m,m+1}(j)&=-\frac{1}{2\pi \mathsf{i}}\oint_{\mathsf{C}(\mathfrak{I})}\langle \pi_iQ_i , \frac{(-u)^{(n+1)-(m+1)}}{x-u} \rangle_{\mathfrak{w}}Q_j(u)du\nonumber\\
&+\frac{1}{2\pi \mathsf{i}}\oint_{\mathsf{C}(\mathfrak{I})}\langle \pi_iQ_i , \frac{(-x)^{n+1}\psi(x)}{(-u)^{m+1}(x-u)\psi(u)} \rangle_{\mathfrak{w}}Q_j(u)du.
\end{align}
To obtain the correlation kernel for $m>n$, recall that there is also a contribution from $\mathsf{\Phi}^{(m,m+1)}_{(n,n+1)}$ which is given by,
\begin{align*}
\mathsf{\Phi}^{(m,m+1)}_{(n,n+1)}(i,j)&=(-1)^{n-m}\left(-\frac{1}{2\pi \mathsf{i}}\right)  \oint_{\mathsf{C}(\mathfrak{I})} \langle \pi_i Q_i ,\frac{Q_j(u)}{x-u}\rangle_{\mathfrak{w}}\frac{1}{u^{m-n}}du\\
&=-\frac{1}{2\pi \mathsf{i}}  \oint_{\mathsf{C}(\mathfrak{I})} Q_j(u)\langle \pi_i Q_i ,\frac{(-u)^{(n+1)-(m+1)}}{x-u}\rangle_{\mathfrak{w}}du.
\end{align*}
Putting it all together, we get that,
\begin{align}
\mathcal{K}^{\psi}\left(((n,n+1),i),(m,m+1),j)\right)=\frac{1}{2\pi \mathsf{i}}\oint_{\mathsf{C}(\mathfrak{I})} Q_j(u)\langle \pi_i Q_i , \frac{(-x)^{n+1}}{(-u)^{m+1}}\frac{\psi(x)}{(x-u)\psi(u)} \rangle_{\mathfrak{w}}du\nonumber\\
+\textbf{1}\left(n \ge m\right)\langle \pi_i Q_i , (-x)^{(n+1)-(m+1)}Q_j \rangle_{\mathfrak{w}}.
\end{align}
Multiplying by the conjugating factor $(-1)^{(n+1)-(m+1)}$ (these do not alter the correlation kernel since they vanish when we take the determinant), we obtain the statement of the Theorem.
\end{proof}

\subsection{Large time and finite distance from wall limit}\label{SubsectionScalingLimit} We now take $\psi(u)=\psi_t(u)=e^{-tu}$ so that we are considering the push-block dynamics and we want to take a large time limit while zooming in and looking at particles being at a finite distance from the wall.

More precisely, let $t\sim N \tau$ and $m,n \sim N\eta$ so that moreover, the differences between the different levels $m-n$ is constant. Furthermore note, that $i,j$ which govern the position of the particles will be fixed and not scaled with $N$. This of course, avoids any delicate asymptotics involving the orthogonal polynomials $Q_i,\hat{Q}_i$ or the spectral measures $\mathfrak{w}, \hat{\mathfrak{w}}$. The exact statement of the result is as follows:

\begin{thm} Let $t(N)=N \tau$ and
\begin{align*}
\left(\tilde{m}_1(N),\tilde{m}_2(N)\right)&=\left(\lfloor N \eta\rfloor+m_1,\lfloor N \eta\rfloor+m_2\right),\\ \left(\tilde{n}_1(N),\tilde{n}_2(N)\right)&=\left(\lfloor N \eta\rfloor+n_1,\lfloor N \eta\rfloor+n_2\right),
\end{align*}
with $\alpha=\frac{\eta}{\tau}$. Then we have:
\begin{align*}
\lim_{N\to \infty}\mathcal{K}^{\psi_{t(N)}}\left(\left(\left(\tilde{n}_1(N),\tilde{n}_2(N)\right),i),\left(\tilde{m}_1(N),\tilde{m}_2(N)\right),j)\right)\right)=\mathfrak{K}_{\alpha}\left(\left(\left(n_1,n_2\right),i),\left(m_1,m_2\right),j)\right)\right)\\
=\int_{I^-}^{I^+}\left[-\textbf{1}(x\ge \alpha)+\textbf{1}\left((n_1,n_2)\ge (m_1,m_2)\right)\right]\bar{\mathcal{P}}_i(x) x^{n_2-m_2}\tilde{\mathcal{P}_j}(x)d\mathfrak{m}(x).
\end{align*}
\end{thm}

\begin{proof} First, note that the term:
\begin{align*}
\textbf{1}\left(\left(\tilde{n}_1(N),\tilde{n}_2(N)\right)\ge \left(\tilde{m}_1(N),\tilde{m}_2(N)\right)\right)\langle \bar{\mathcal{P}}_i(x) , x^{\tilde{n}_2(N)-\tilde{m}_2(N)}\tilde{\mathcal{P}_j}(x) \rangle_{\mathfrak{m}}=\textbf{1}\left((n_1,n_2)\ge (m_1,m_2)\right)\langle \bar{\mathcal{P}}_i(x) , x^{n_2-m_2}\tilde{\mathcal{P}_j}(x) \rangle_{\mathfrak{m}}
\end{align*}
remains constant in $N$. We hence, focus on the double integral term of the kernel $\mathcal{K}^{\psi_{t(N)}}$ and write it as (recall $\mathfrak{I}=[I^-,I^+]$),
\begin{align*}
\frac{1}{2\pi \mathsf{i}}\int_{I^-}^{I^+}\oint_{\mathsf{C}(\mathfrak{I})} \frac{e^{-t(N)x+\tilde{n}_2(N)log(x)}}{e^{-t(N)u+\tilde{m}_2(N)log(u)}}\frac{\tilde{\mathcal{P}_j}(u)\bar{\mathcal{P}}_i(x)}{(x-u)} d\mathfrak{m}(x)du.
\end{align*}
Write the term involving exponentials as,
\begin{align*}
\frac{e^{-t(N)x+\tilde{n}_2(N)log(x)}}{e^{-t(N)u+\tilde{m}_2(N)log(u)}}=\frac{e^{-N(\tau x-\eta log(x))}}{e^{-N(\tau u-\eta log(u))}}+o_N(1).
\end{align*}
Let $f(z)=\tau z-\eta log(z)$. Then $f'(z)=\tau-\frac{\eta}{z}$ and so $z=\alpha\overset{\textnormal{def}}{=}\frac{\eta}{\tau}$ is a critical point. Write,
\begin{align*}
\frac{e^{-N(\tau x-\eta log(x))}}{e^{-N(\tau u-\eta log(u))}}=\frac{e^{-N(f(x)-f(\alpha))}}{e^{-N(f(u)-f(\alpha))}}.
\end{align*}
We would like to deform the ${\mathsf{C}(\mathfrak{I})}$ contour to a contour $\mathsf{C}_{\mathsf{s}}$ so that,
\begin{align*}
& \Re\left(f(x)-f(\alpha)\right)\ge 0 \ , \ \text{for} \ \ x \in [0,I^+],\\
& \Re\left(f(u)-f(\alpha)\right)< 0 \ , \ \text{for u on the } \mathsf{C}_{\mathsf{s}}  \text{ contour} \ \ 
\end{align*}
and thus, the double integral will converge uniformly to zero as $N\to \infty$. In the process however, we might pick some residues from the pole of $\frac{1}{x-u}$ depending on how $\alpha$ compares with $I^+$. First note that for $x \in \mathbb{R},  \ \Re(f(x)-f(\alpha))\le 0$ is equivalent to,
\begin{align*}
\alpha e^{\frac{x}{\alpha}-1}\ge |x|.
\end{align*}
Hence, there exists $\beta<0$ so that $\Re(f(x)-f(\alpha))< 0$ for $x < \beta$ and $\Re(f(x)-f(\alpha))> 0$ for $x>\beta$ except at $\alpha$. Similarly, with $u=x+iy$ the inequality $\Re\left(f(u)-f(\alpha)\right)< 0$ is then equivalent to,
\begin{align*}
\alpha e^{\frac{x}{\alpha}-1}< (x^2+y^2)^{\frac{1}{2}}
\end{align*}
and note that $\sup_{\beta \le x \le \alpha}\alpha e^{\frac{x}{\alpha}-1}=\alpha$. We can thus deform the ${\mathsf{C}(\mathfrak{I})}$ contour to a contour $\mathsf{C}_{\mathsf{s}}$ that is equal to a rectangle with sides parallel to the real and imaginary axes so that the two sides that are parallel to the imaginary axis have real parts $r_1=\alpha$ and $r_2<\beta$ and the two sides that are parallel to the real axis have imaginary parts $im_1>\alpha$ and $im_2<-\alpha$. Then, on this contour we have $\Re\left(f(u)-f(\alpha)\right)< 0$ except at $\alpha$, where it vanishes. If $\alpha \le I^+$ in the course of this deformation we also pick the residue at $u=x$ which gives the single integral,
\begin{align*}
-\int_{I^-}^{I^+}\textbf{1}(x\ge \alpha)\bar{\mathcal{P}}_i(x) x^{n_2-m_2}\tilde{\mathcal{P}_j}(x)d\mathfrak{m}(x).
\end{align*}
Thus, for $\alpha > I^+$ the kernel $\mathcal{K}^{\psi_{t(N)}}$ converges to a triangular matrix whose diagonal entries are $1$. This corresponds to the frozen or fully packed region; the particles at high levels haven't had time to move yet since $\eta>\tau I^+$. On the other hand, for $\alpha \le I^+$, in the scaling regime considered here, $\mathcal{K}^{\psi_{t(N)}}$  converges to a kernel $\mathfrak{K}_{\alpha}$ with entries,
\begin{align}\label{discreteensemblekernel}
\mathfrak{K}_{\alpha}\left(((n_1,n_2),i),((m_1,m_2),j)\right)=\int_{I^-}^{I^+}\left[-\textbf{1}(x\ge \alpha)+\textbf{1}\left((n_1,n_2)\ge (m_1,m_2)\right)\right]\bar{\mathcal{P}}_i(x) x^{n_2-m_2}\tilde{\mathcal{P}_j}(x)d\mathfrak{m}(x).
\end{align}

\end{proof}

\begin{rmk}\label{DiscreteEnsemblesRemark}[Multilevel extension of discrete ensembles]
As already mentioned in the introduction Borodin and Olshanski in Section 3 of \cite{BorodinOlshanskiAsep} introduced the so called discrete determinantal ensembles associated to continuous orthogonal polynomials.

Their definition goes as follows: suppose $\mathcal{W}(dx)$ is a weight on $\mathbb{R}$ for which the moment problem is determinate (see \cite{BorodinOlshanskiAsep} for the precise statements). Let $P_k^*(x)$ be the $k^{th}$ orthonormal polynomial with respect to this weight with positive leading coefficient. The discrete ensemble associated to the weight $\mathcal{W}(dx)$ (or equivalently to the polynomials $P_k^*(x)$) is the determinantal point process with the following kernel $\mathsf{K}^{\mathcal{W}}_{r}\left(i,j\right)$:
\begin{align*}
\mathsf{K}^{\mathcal{W}}_{r}\left(i,j\right)=\int_{r}^{\infty}P_i^*(x)P_j^*(x)\mathcal{W}(dx).
\end{align*}

It is easy to see that if restricted to single levels $\mathfrak{K}_{\alpha}\left(\left(\left(n,n+1\right),i),\left(n,n+1\right),j)\right)\right)$ gives rise to the determinantal ensemble with kernel $\mathsf{K}_{\alpha}^{\mathfrak{w}}(i,j)$ and also $\mathfrak{K}_{\alpha}\left(\left(\left(n,n\right),i),\left(n,n\right),j)\right)\right)$ gives rise to the ensemble governed by the kernel $\mathsf{K}_{\alpha}^{\hat{\mathfrak{w}}}(i,j)$; since conjugation by a function does not alter the correlation functions and thus the determinantal measure. 

Thus, $\mathfrak{K}_{\alpha}\left(((n_1,n_2),i),(m_1,m_2),j)\right)$ provide a novel multilevel determinantal extension of these discrete ensembles, so that particles on consecutive levels interlace (by construction). Moreover, in this generality, it is the first time that these ensembles appear in a concrete interacting particle system.
\end{rmk}

\section{Appendix}

\subsection{Technical results}

\begin{proof}[Proof of Lemma \ref*{ConjugacyLemma}]
We will show that for $x,y \in \mathbb{Z}$ and $t\ge 0$,
\begin{align*}
p_t(x,y)=-\bar{\nabla}_y \sum_{w=x}^{\infty}\hat{p}_t(y,w),
\end{align*}
from which the statement of Lemma \ref*{ConjugacyLemma} follows. It will be more convenient to write this equality in matrix form. Define the doubly infinite matrices $U,V$ as follows,
\begin{align*}
A_{ij}=\begin{cases}
1 \ \ j \ge i \\
0 \ \ \textnormal{otherwise}
\end{cases}, \ \ \ \ B_{ij}=\begin{cases}
1 \ \ & i = j \\
-1 \ \ & j=i+1\\
0 \ \ & \textnormal{otherwise}
\end{cases}.
\end{align*}
Observe that, $AB=BA=Id$ and moreover and this is the key relation, $B\mathcal{D}=\hat{\mathcal{D}}^{\textnormal{T}}B$ where $\hat{\mathcal{D}}^{\textnormal{T}}$ denotes the transpose of $\hat{\mathcal{D}}$. Then, with this notation in place we want to show,
\begin{align*}
P(t)=A\hat{P}^{\textnormal{T}}(t)B\overset{\textnormal{def}}{=}P_*(t),  \textnormal{ for } t \ge 0.
\end{align*}
First note that $P_*(0)=Id$ and moreover,where in the first equality we interchange $\frac{d}{dt}$ and an infinite sum which will be justified below, and in the second we use the backwards equation, for $t>0$,
\begin{align*}
\frac{d}{dt}P_*(t)&=A\left(\frac{d}{dt}\hat{P}(t)\right)^{\textnormal{T}} B\\
&=A\left(\hat{\mathcal{D}}\hat{P}(t)\right)^{\textnormal{T}} B\\
&=A\hat{P}^{\textnormal{T}}(t)\hat{\mathcal{D}}^TB\\
&=A\hat{P}^{\textnormal{T}}(t)B\mathcal{D}=P_*(t)\mathcal{D}.
\end{align*}
Finally, note that $-\bar{\nabla}_y \sum_{w=x}^{\infty}\hat{p}_t(y,w) \ge 0$ and $\sum_{y \in \mathbb{Z}}^{}-\bar{\nabla}_y \sum_{w=x}^{\infty}\hat{p}_t(y,w) =1$. Hence, by uniqueness of solutions to the forwards equation we obtain that for $t\ge 0$,  $P_{*}(t)=P(t)$. Now, in order to justify the interchange of summation and differentiation it suffices to show that the series,
\begin{align*}
\sum_{w=x}^{\infty}\frac{d}{dt}\hat{p}_t(y,w)
\end{align*}
converges uniformly on compact intervals of $t$, where $x,y \in \mathbb{Z}$ are fixed. First, note that for $n\ge 1$ we have,
\begin{align}\label{sumderivativesrelation}
\sum_{w=x}^{x+n}\frac{d}{dt}\hat{p}_t(y,w)=\hat{\lambda}(y)\sum_{w=x}^{x+n}\hat{p}_t(y-1,w)-\left(\hat{\lambda}(y)+\hat{\mu}(y)\right)\sum_{w=x}^{x+n}\hat{p}_t(y,w)+\hat{\mu}(y)\sum_{w=x}^{x+n}\hat{p}_t(y+1,w).
\end{align}
Hence, $\sum_{w=x}^{x+n}\frac{d}{dt}\hat{p}_t(y,w)$ converges on $0\le t <\infty$ and moreover, has uniformly bounded partial sums. More specifically,
\begin{align*}
\sum_{w=x}^{x+n}\left|\frac{d}{dt}\hat{p}_t(y,w)\right|\le 2 \left(\hat{\lambda}(y)+\hat{\mu}(y)\right)\ , \ \forall t \ge 0, \forall n \ge 1.
\end{align*}
Thus, the partial sums of,
\begin{align*}
\sum_{w=x}^{\infty}\hat{p}_t(y,w)
\end{align*}
are uniformly bounded and equicontinuous, which can be seen as follows. If we define, for fixed $x,y \in \mathbb{Z}$, $f_n(t)=\sum_{w=x}^{x+n}\hat{p}_t(y,w)$ we obviously have $|f_n(t)| \le 1, \forall t \ge 0$ and $ n \ge 1$. Moreover, for $s\le t$ in $[0,T]$ we have by the Mean Value Theorem, for some $u \in (s,t)$,
\begin{align*}
f_n(t)-f_n(s)=(t-s)\frac{d}{du}f_n(u)
\end{align*}
and hence,
\begin{align*}
|f_n(t)-f_n(s)| \le \left|\sum_{w=x}^{x+n}\frac{d}{du}\hat{p}_u(x,y)\right| &\le |t-s| \underset{u \in [0,T]}{\sup}\sum_{w=x}^{x+n}\left|\frac{d}{du}\hat{p}_u(y,w)\right|\\
&\le 2 \left(\hat{\lambda}(y)+\hat{\mu}(y)\right) |t-s| \ , \ \forall n \ge 1.
\end{align*}
So, by the Arzela Ascoli Theorem we obtain that the series $\sum_{w=x}^{\infty}\hat{p}_t(y,w)$ converges uniformly on every finite interval in $t$ and hence by equality (\ref*{sumderivativesrelation}) the series $\sum_{w=x}^{\infty}\frac{d}{dt}\hat{p}_t(y,w)$ does so as well. By iterating the same argument, we also see that this holds for $\sum_{w=x}^{\infty}\frac{d^k}{dt^k}\hat{p}_t(y,w)$ for any $k\ge 1$.
\end{proof}

\begin{proof}[Proof of Proposition \ref*{Pathwise}]
The result is implied from the following two claims, for $s\le t$, $x,x',x'',w \in I$:

\textbf{1}. If $F_{s,t}(w)= x' \le x$ then $G_{s,t}(x) \ge w$.

\textbf{2}. If $F_{s,t}(w)= x'' > x$ then $G_{s,t}(x) < w$.

To show the first one, observe that without loss of generality we can assume that $F_{s,t}(w)=x$. Then, \textit{attempt} to follow the original/forwards path starting from $w$ at time $s$ and that ends at $x$ at time $t$ backwards in time, using only the $\color{red}{red}$ arrows, until the first time this is no longer possible. This happens iff the original/forwards path/chain came up using an up $\uparrow$ arrow or the chain running backwards encounters a $\color{red}{red}$ up $\color{red}{\uparrow}$ arrow. The claim then follows, since the backwards path always stays above the original/forwards path.

To show the second one, note that without loss of generality we can assume that $F_{s,t}(w)=x+1$. Consider the last instance (if they never meet the claim is trivial) $\tau<t$ the forwards path starting from $w$ at time $s$ and moving according to the original arrows and the backwards path starting from $x$ at time $t$ and using the $\color{red}{red}$ arrows are together. This is equivalently, the first instance (cf. right continuity) they meet, with time running backwards from $t$. This can only happen if the forwards path encounters an up $\uparrow$ arrow which means the backwards path encountered a down $\color{red}{red \downarrow}$ arrow, which gives a contradiction. This is since the paths would split at $\tau$, with time running backwards in such cases.
\end{proof}

\subsection{Projective chains from branching of functions}
Suppose we are given $\forall n \in \mathbb{N}$, indexing sets $I_n \subset \mathbb{Z}^n$, Polish spaces $\mathcal{X}^n=\mathcal{X}\times \cdots \times \mathcal{X}$, a distinguished point $\bar{u}\in \mathcal{X}$, Borel measures $w_n$ on $\mathcal{X}^n$ and finally families of functions $\{F_n\left(x;u_1,\cdots,u_n\right)\}_{x\in I_n}$ orthogonal in $L^2\left(\mathcal{X}^n,w_n\right)$ normalized so that $F_n\left(x;\bar{u},\cdots,\bar{u}\right)=1$, $\forall n \in \mathbb{N}, x \in I_n$. Consider the convex set, denoted by $\mathcal{Y}_n$, consisting of functions $\mathcal{F}_n$ such that the following series converges uniformly in $ \mathcal{X}^n$ (this can be relaxed) and in $L^2\left(\mathcal{X}^n,w_n\right)$,
\begin{align}
\mathcal{F}^{M_n}_n(u_1,\cdots,u_n)=\sum_{x\in I_n}^{}M_n(x)F_n\left(x;u_1,\cdots,u_n\right),
\end{align}
where,
\begin{align}
M_n(x)\ge 0 \ , \ \forall x \in I_n \ \textnormal{and} \ \sum_{x\in I_n}^{}M_n(x)=1.
\end{align}
Note that, by the orthogonality of the $\{F_n\left(x;\cdot\right)\}_{x\in I_n}$ we obtain that the $\{M_n(x)\}_{x\in I_n}$ are determined uniquely by the $\mathcal{F}_n(\cdot)$ as follows,
\begin{align}
M_n(x)=\frac{\langle\mathcal{F}_n\left(\cdot\right),F_n\left(x;\cdot\right)\rangle_{w_n}}{\langle F_n\left(x;\cdot\right),F_n\left(x;\cdot\right)\rangle_{w_n}}.
\end{align}
Now, further assume that,
\begin{align}\label{generatingbranching}
F_n(x;u_1,\cdots,u_{n-1},\bar{u})=\sum_{y\in I_{n-1}}^{}\Lambda^n_{n-1}(x,y)F_{n-1}\left(y;u_1,\cdots,u_{n-1}\right),
\end{align}
for some Markov kernels, $\Lambda^n_{n-1}$ from $I_n$ to $I_{n-1}$ i.e.
\begin{align*}
\Lambda^n_{n-1}(x,y)\ge 0 \ , \ \forall x \in I_n, y \in I_{n-1} \ \textnormal{and (necessarily)} \ \sum_{y\in I_{n-1}}^{}\Lambda^n_{n-1}(x,y)=1.
\end{align*}
Moreover, we assume that for any fixed $x \in I_n$ the measure $\Lambda_{n-1}^n(x,\cdot)$ is supported on \textit{finitely many} $y\in I_{n-1}$. Observe that, this is always the case for branching graphs by definition. In particular, the functions $\{F_n\left(x;\cdot\right)\}_{x\in I_n,n\ge 1}$ generate a projective chain with levels $\{I_n\}_{n\ge 1}$ and Markov links from $I_n$ to $I_{n-1}$ given by $\Lambda_{n-1}^{n}(x,y)$ with $x\in I_n$ and $y\in I_{n-1}$. 
\begin{rmk}\label{RemarkMarkovKernelFromBranchingRule}
In the case of the alternating construction, $I_n=W^n(\mathbb{N})$, $\mathcal{X}=[0,I^+]$ and $\bar{u}=0$. For $\nu \in I_n$ and $u_1,\cdots,u_n \in [0,I^+]$, the functions $F_n(\nu;u_1,\cdots,u_n)$ are given by (cf. (\ref{normalizedKMpoly1})),
\begin{align*}
F_n(\nu;u_1,\cdots,u_n)=\frac{h_{n-1,n}(\nu;u_1,\cdots,u_n)}{h_{n-1,n}(\nu;0,\cdots,0)}=\frac{h_{n-1,n}(\nu;u_1,\cdots,u_n)}{h_{n-1,n}(\nu)}
\end{align*}
and the Markov kernels $\Lambda^{n}_{n-1}(\nu,\kappa)$, for $\nu\in W^{n}$ and $\kappa \in W^{n-1}$, as follows,
\begin{align*}
\Lambda^{n}_{n-1}(\nu,\kappa)=\left(\Lambda_{n-1,n}^{h_{n-1,n-1}}\Lambda_{n-1,n-1}^{h_{n-2,n-1}}\right)(\nu,\kappa).
\end{align*}
\end{rmk}
Moving on to coherent measures, the fact that $M_n\Lambda^n_{n-1}=M_{n-1}$ is equivalent to,
\begin{align} \label{generatingcoherency}
\mathcal{F}^{M_n}_n(u_1,\cdots,u_{n-1},\bar{u})=\sum_{y\in I_{n-1}}^{}M_{n-1}(y)F_
{n-1}(y;u_1,\cdots,u_{n-1}).
\end{align}
This can be seen as follows. If $M_n\Lambda_{n-1}^{n}=M_{n-1}$, we multiply both sides of (\ref{generatingbranching}) by $M_n(x)$ and sum over $x \in I_n$ first (there is only one infinite sum here so we can interchange them without any issues) to arrive at (\ref{generatingcoherency}). On the other hand, if (\ref{generatingcoherency}) holds we can again multiply (\ref{generatingbranching}) by $M_n(x)$ and sum over $x \in I_n$ to obtain using (\ref{generatingcoherency}),
\begin{align*}
\sum_{y\in I_{n-1}}^{}M_{n-1}(y)F_
{n-1}(y;u_1,\cdots,u_{n-1})=\sum_{y\in I_{n-1}}^{}\sum_{x \in I_n}^{}M_n(x)\Lambda^n_{n-1}(x,y)F_{n-1}\left(y;u_1,\cdots,u_{n-1}\right),
\end{align*}
with both series converging uniformly and in $L^{2}\left(\mathcal{X}^{n-1},w_{n-1}\right)$ and by taking the inner product with $F_{n-1}(z; \cdot)$ we get,
\begin{align*}
M_{n-1}(z)= \sum_{x \in I_n}^{}M_n(x) \Lambda_{n-1}^n(x,z).
\end{align*}
Thus (truncated) coherent measures up to level $N$, namely sequences of probability measures $\{M_n\}_{n\le N}$ such that $M_n\Lambda^n_{n-1}=M_{n-1}$ for $n \le N$ are in bijection with sequences $\{\mathcal{F}_n\}_{n\le N}$ such that $\mathcal{F}_n\in \mathcal{Y}_n$ with $\mathcal{F}_n(u_1,\cdots,u_n)=\mathcal{F}_N(u_1,\cdots,u_n,\bar{u},\cdots,\bar{u})$. Thus, if we define $\left(\mathcal{S}\mathcal{F}_n\right)\left(u_1,\cdots,u_{n-1}\right)=\mathcal{F}_n\left(u_1,\cdots,u_{n-1},\bar{u}\right)$ which is an affine map from $\mathcal{Y}_n$ to $\mathcal{Y}_{n-1}$ and consider the projective limit,
\begin{align}
\mathcal{Y}=\underset{\leftarrow}{\lim}\mathcal{Y}_n
\end{align}
consisting of functions $\mathcal{F}_{\infty}$ on the space $\mathcal{X}_0^{\infty}=\left(u_1,u_2,\cdots\right)\in \mathcal{X} \times \mathcal {X} \times \cdots$ (having only finitely many coordinates not equal to $\bar{u}$) such that,
\begin{align}
\mathcal{F}^{\mathcal{F}_{\infty}}_n(u_1,\cdots,u_n)\overset{\textnormal{def}}{=}\mathcal{F}_{\infty}(u_1,\cdots,u_n,\bar{u},\bar{u},\cdots) \in \mathcal{Y}_n  \ , \forall n \in \mathbb{N},
\end{align} 
then studying the extremal coherent measures is equivalent to the study of $\textnormal{Ex}\left(\mathcal{Y}\right)$.

\subsection{Factorization implies extremality}\label{SubsectionFactorizationImpliesExtremality}
We now aim to prove under several assumptions that if $\mathcal{F}_{\infty}$ factorizes then, the corresponding coherent measure is extremal. We will reduce the problem to an application of de Finetti's theorem, following an argument which in this particular setting, as far as we know, originates with Okounkov's and Olshanski's paper \cite{OkounkovOlshanskiJack}. 
 
We assume that, $\forall n \in \mathbb{N}$ and $ x \in I_n$, the functions $F_n\left(x;u_1,\cdots,u_n\right)$ are symmetric polynomials on $[0,I^+]^n$, orthogonal with respect to a weight $w_n$ and $\bar{u}=0$. It will be more convenient to work on the $n$-dimensional torus $\mathbb{T}^n=\{(z_1,\cdots,z_n)\subset \mathbb{C}:|z_i|=1\}$ rather than the cube. We let $\mathfrak{W}$ denote the $\textnormal{BC}_n$ Weyl group namely,
\begin{align*}
\mathfrak{W}=S(n) \ltimes \mathbb{Z}_2^n,
\end{align*}
where the symmetric group $S(n)$ acts by permuting the variables and $\mathbb{Z}_2^n$ acts as follows,
\begin{align*}
f(z_1,\cdots,z_n) \mapsto f(z^{\pm 1}_1,\cdots,z^{\pm 1}_n).
\end{align*}
We will be interested in $\mathfrak{W}$-invariant Laurent polynomials in $n$ variables on $\mathbb{T}^n$. It is a well known fact, that the algebra of $n$-variable $\mathfrak{W}$-invariant Laurent polynomials can be identified with the standard algebra of symmetric polynomials in $n$ variables (see first paragraph of Section 2 of \cite{Veselov} for a discussion). More concretely, under the change of variables,
\begin{align*}
u_i=\frac{I^+}{2}\left(1-\frac{z_i+z_i^{-1}}{2}\right)=\mathfrak{g}(z_i),
\end{align*}
we can map symmetric polynomials on the cube $[0,I^+]^n$ to $\mathfrak{W}$-invariant Laurent polynomials on $\mathbb{T}^n$ and vice versa and note that the distinguished point $\bar{u}=0$ gets mapped to $z=1$. We can thus, consider the corresponding $\mathfrak{W}$-invariant Laurent polynomial to $F_n(x;u_1,\cdots,u_n)$, denoted by $G_n(x;z_1,\cdots,z_n)=F_n(x;\mathfrak{g}(z_1),\cdots,\mathfrak{g}(z_n))$, orthogonal in $L^2\left(\mathbb{T}^n,\tilde{w}_n\right)$ where $\tilde{w}_n$ is obtained by the change of variables formula. Finally, we denote the corresponding convex set $\tilde{\mathcal{Y}}_n$ consisting of functions $\mathcal{G}_n(z_1,\cdots,z_n)=\mathcal{F}_n(\mathfrak{g}(z_1),\cdots,\mathfrak{g}(z_n))$ so that,
\begin{align}\label{series}
\mathcal{G}_n(z_1,\cdots,z_n)&=\sum_{x\in I_n}^{}M_n(x)G_n\left(x;z_1,\cdots,z_n\right),\\
G_n(x;z_1,\cdots,z_{n-1},1)&=\sum_{y\in I_{n-1}}^{}\Lambda^n_{n-1}(x,y)G_{n-1}\left(y;z_1,\cdots,z_{n-1}\right) \nonumber.
\end{align}

We make the following essential (and rather non-trivial to check) \textit{positive definiteness} assumption, namely that $\forall x \in I_n$,
\begin{align*}
G(x;z_1,\cdots,z_n)=\underset{\lambda_1,\cdots,\lambda_n \in \mathbb{Z}}{\sum}\mathfrak{a}(x;\lambda_1,\cdots\lambda_n)z_1^{\lambda_1} \cdots z_n^{\lambda_n} \ , \ \textnormal{with} \  \mathfrak{a}(x;\lambda_1,\cdots\lambda_n)\ge 0 \ ,  \forall{\lambda_1,\cdots,\lambda_n}\in \mathbb{Z}.
\end{align*}
Note that, since $G(x;z_1,\cdots,z_n)=1$ this implies that,
\begin{align*}
\underset{\lambda_1,\cdots,\lambda_n \in \mathbb{Z}}{\sum}\mathfrak{a}(x;\lambda_1,\cdots\lambda_n)=1
\end{align*}
and so by the positivity of the $\mathfrak{a}(x;\lambda_1,\cdots\lambda_n)$ for $(z_1,\cdots,z_n) \in \mathbb{T}^n$, $|G_n(x;z_1,\cdots,z_n)|\le 1$ and in particular the series (\ref{series}) converges uniformly. Thus, $\mathcal{G}_n$ is a continuous,normalized, positive definite, symmetric function on $\mathbb{T}^n$. 

Hence, and this is the key observation, the convex set $\tilde{\mathcal{Y}}_n$ is a subset of the convex set of characteristic functions of measures on $\mathbb{Z}^n$ invariant under the action of $S(n)$. Thus, $\tilde{\mathcal{Y}}=\underset{\leftarrow}{\lim}\tilde{\mathcal{Y}}_n$ the set of functions $\mathcal{G}_{\infty}$ on $(z_1,z_2,\cdots)\in \mathbb{T}_0^{\infty}$ such that,
\begin{align}
\mathcal{G}_n(z_1,\cdots,z_n)\overset{\textnormal{def}}{=}\mathcal{G}_{\infty}(z_1,\cdots,z_n,1,1,\cdots) \in \tilde{\mathcal{Y}}_n  \ , \forall n \in \mathbb{N},
\end{align} 
is a (convex) subset of the convex set $\mathcal{Z}$ of characteristic functions of probability measures on $\mathbb{Z}^{\infty}=\mathbb{Z}\times\mathbb{Z}\times\cdots$, invariant under the action of $S(\infty)$. We have thus arrived at the following result.

\begin{prop}
Under the assumptions above, for $\mathcal{G}_{\infty} \in \tilde{\mathcal{Y}}$ further assume that there exists $\mathcal{G}_1\in\tilde{\mathcal{Y}}_1$ such that $\forall n \ge 1$,
\begin{align}\label{factorization}
\mathcal{G}_{\infty}(z_1,\cdots,z_n,1,1,\cdots)=\prod_{i=1}^{n}\mathcal{G}_1(z_i).
\end{align}
Then, $\mathcal{G}_{\infty}\in \textnormal{Ex}(\tilde{\mathcal{Y}})$.
\end{prop}
\begin{proof}
By de Finetti's theorem and the factorization property  (\ref{factorization}) we have $\mathcal{G}_{\infty}\in \textnormal{Ex}(\mathcal{Z})$. Since  $\tilde{\mathcal{Y}}$ is a convex subset of $\mathcal{Z}$ we get $\mathcal{G}_{\infty}\in \textnormal{Ex}(\tilde{\mathcal{Y}})$.
\end{proof}

\begin{rmk}
We have a Markov kernel $\Lambda^{\infty}_{n}:\textnormal{Ex}(\tilde{\mathcal{Y}}) \to I_n$, defined for $\mathcal{G}_{\infty}\in \textnormal{Ex}(\tilde{\mathcal{Y}})$ such that (\ref*{factorization}) holds, that is given as follows,
\begin{align}
\Lambda_n^{\infty}\left(\mathcal{G}_1,x\right)\overset{\textnormal{def}}{=}M^{\mathcal{G}_1}_n(x)\overset{\textnormal{def}}{=}\frac{\langle\prod_{}^{n}\mathcal{G}_1\left(\cdot\right),G_n\left(x;\cdot\right)\rangle_{\tilde{w}_n}}{\langle G_n\left(x;\cdot\right),G_n\left(x;\cdot\right)\rangle_{\tilde{w}_n}}.
\end{align}
\end{rmk}
\begin{rmk}\label{RemarkPositiveDefiniteness}
Note that, the assumptions considered in this section are satisfied in the case of general $\beta$ normalized Jack (see \cite{OkounkovOlshanskiJack}) and Jacobi (see \cite{OkounkovOlshanskiJacobi}) polynomials. Checking the positive definiteness of $G_n(\nu;\cdot)$ corresponding to $F_n(\nu;\cdot)=\frac{h_{n-1,n}(\nu;\cdot)}{h_{n-1,n}(\nu)}$, cf. (\ref{normalizedKMpoly1}) which would imply the extremality of $M_n=\mathcal{M}_{n-1,n}^{\psi}$ for $\psi(x)=p(x)e^{-tx}$ where $p(x)$ is an arbitrary polynomial is in general non-trivial.
\end{rmk}

\bigskip
\noindent
{\sc Mathematics Institute, University of Warwick, Coventry CV4 7AL, U.K.}\newline
\href{mailto:T.Assiotis@warwick.ac.uk}{\small T.Assiotis@warwick.ac.uk}

\end{document}